\documentclass[a4paper,10pt]{article} 
\usepackage{amsmath}
\allowdisplaybreaks[4]
\usepackage{amssymb,esint}
\usepackage{amscd}
\usepackage{mathrsfs}
\usepackage{xspace}
\usepackage{fancyhdr}
\usepackage{verbatim}
\usepackage{cite}
\usepackage{appendix}
\usepackage[hidelinks]{hyperref}
\hypersetup{pdftitle={Scaling-Critical Theory for the Boltzmann and Landau Equations},
 pdfauthor={Ke Chen, Quoc-Hung Nguyen, Tong Yang}}
\usepackage{bm}
\usepackage{needspace}
\usepackage{srcltx} 

\def\R{\mathbb{R} }
\def\1{{{\mathbf 1}}}

\def\bK{\mathbf{K}}

\def\G{\mathbb{G}}

\newcommand{\jap}[1]{\langle #1 \rangle}

\newcommand{\oast}{\circledast}

\newcommand{\bT}{\mathbf{\Theta}}

\def\dif{{\rm d}}

\newcommand{\keywords}{\textbf{Keywords}: }

\newtheorem{theorem}{Theorem}[section]

\newtheorem{corollary}[theorem]{Corollary}

\newtheorem{definition}[theorem]{Definition}

\newtheorem{lemma}[theorem]{Lemma}

\newtheorem{proposition}[theorem]{Proposition}
\newtheorem{remark}[theorem]{Remark}

\newenvironment{proof}[1][Proof]{\textbf{#1.} }{\hfill\rule{0.5em}{0.5em}}
\makeatletter
\let\AddToReset=\@addtoreset
\makeatother
	\AddToReset{equation}{section}
	
	\AddToReset{theorem}{section}

\begin{document}
		\title{Scaling-Critical Theory for the Boltzmann and Landau Equations}
		\author{
			{\bf Ke Chen\thanks{Email address: ke.chen@unibocconi.it, Department of Decision Sciences, Bocconi University, 20136 Milan, Italy.}},
			{\bf Quoc-Hung Nguyen\thanks{Email address: qhnguyen@amss.ac.cn, Academy of Mathematics and Systems Science,
					Chinese Academy of Sciences,
					Beijing, 100190, P. R. China.} },
			{\bf Tong Yang\thanks{Email address: t.yang@polyu.edu.hk, Department of Applied Mathematics, The Hong Kong Polytechnic University, Kowloon, Hong Kong, P. R. China.
			}}
		}
		\date{}  
		\maketitle
\begin{abstract}
	This paper develops a scaling-critical theory for the spatially inhomogeneous non-cutoff
Boltzmann and Landau equations with very soft potentials. 
For sufficiently small initial perturbations in a localized, weighted, anisotropic Riesz-potential norm, we prove global well-posedness near a Maxwellian in the whole space. This critical phase-space norm captures the Boltzmann--Landau scaling, the velocity-dependent anisotropy, and the hypoelliptic transport structure. The proof combines frozen-operator estimates, a critical fixed-point argument, and weighted hypocoercive energy estimates.

We also develop a short-time pointwise Green-function theory for variable-coefficient kinetic equations with a nonnegative H\"older background. We first construct the small-jump Green function by freezing coefficients along kinetic characteristics and then recover the full kernel through a convergent parametrix expansion. The resulting bounds capture the fractional Kolmogorov geometry near characteristics and rapid decay away from them. These estimates form the analytic core of the scaling-critical theory and
provide a robust pointwise framework for further regularity analysis in the soft-potential regime.
\end{abstract}

	\noindent	\keywords{non-cutoff Boltzmann equation, Landau equation, potential spaces, short-time well-posedness, regularization, Green functions}
            
		\clearpage
\tableofcontents
\clearpage
		\section{Introduction}
		
We consider the Cauchy problem
\begin{equation}\label{eqlandau}
\begin{aligned}
	&\partial_t F+v\cdot\nabla_x F=\mathcal{Q}_s(F,F),
	\qquad
	(t,x,v)\in \mathbb R^+\times \mathbb R^d\times \mathbb R^d,
	\qquad s\in(0,1],\\
	&F|_{t=0}=F_0 .
\end{aligned}
\end{equation}
{For} \(s\in(0,1)\){, the operator} \(\mathcal{Q}_s\)
{is the non-cutoff Boltzmann collision operator}
\begin{equation}\label{BL}
\mathcal{Q}_s(f_1,f_2)(v)
:=
\int_{\mathbb R^d}\int_{\mathbb S^{d-1}}
\bigl(
f_1(v_\star')f_2(v')-f_1(v_\star)f_2(v)
\bigr)
B(|v_\star-v|,\cos\theta)
\,\dif \sigma\,\dif v_\star .
\end{equation}
{The pre- and post-collisional velocities are related by the}
$\sigma${-representation}
\[
v'
=
\frac{v+v_\star}{2}
+
\frac{|v-v_\star|}{2}\sigma,
\qquad
v_\star'
=
\frac{v+v_\star}{2}
-
\frac{|v-v_\star|}{2}\sigma,
\]
{where}
\[
\cos\theta
=
\frac{v-v_\star}{|v-v_\star|}\cdot\sigma .
\]
{We consider inverse-power-law interactions without an angular cutoff. The
collision kernel is assumed to have the form}
\begin{equation}\label{defB}
B(|v_\star-v|,\cos\theta)
=
|v_\star-v|^\gamma b(\cos\theta),
\qquad
b(\cos\theta)
\sim
\left|\sin\frac{\theta}{2}\right|^{-(d-1)-2s},
\end{equation}
with
\[
\gamma>-d-2s,
\qquad
s\in(0,1).
\]
{Using the cancellation lemma and the Carleman representation from}
\cite{ADVW,S} {(see also Lemma~}\ref{lemcanc}{), the
Boltzmann equation can be rewritten as}
		\begin{equation}\label{eq1}
			\partial_t F+v\cdot \nabla_x F=-\int_{\mathbb{R}^d}(F(v)-F(v+h))	\mathbf{C}_{F}(v,h)\frac{\dif h}{|h|^{d+2s}}+{c_{d,\gamma,s,b}} F (\Lambda_v^{-d-\gamma}
F).		\end{equation}
{Here} 
$\Lambda_v^{\alpha}=(-\Delta_v)^\frac{\alpha}{2}$ {is the Fourier
multiplier with symbol} $|\xi|^{\alpha}${, while the nonnegative
coefficient} $\mathbf{C}_{F}$ {depends nonlocally on} $F$.
Under the half-sphere convention introduced in
Remark~\ref{rem:half-sphere-convention}, Lemma~\ref{lemrefo} gives
\[
 \mathbf C_F(v,h)=2^{d-1}
 \int_{\substack{w\perp h\\ |w|\geq|h|}}
 F(v+w)|w|^{\gamma+2s+1}
 A\!\left(\frac{|h|^2}{|w|^2}\right)\,\dif\mathcal H^{d-1}(w),
 \qquad h\neq0,
\]
where $A$ is the bounded angular factor defined in that lemma.
For $F\geq0$, averaging over directions therefore yields
\begin{align}\label{eq1-co}
 \int_{\mathbb S^{d-1}}\mathbf C_F(v,\rho\omega)
 \,\dif\mathcal H^{d-1}(\omega)
 &\lesssim\int_{|w|\geq\rho}F(v+w)|w|^{\gamma+2s}\,\dif w\notag\\
 &\leq\int_{\mathbb R^d}F(v+w)|w|^{\gamma+2s}\,\dif w,
 \qquad \rho>0.
\end{align}
Thus, the non-cutoff Boltzmann operator behaves like a fractional diffusion
of order $2s$ {in} $v${, with a rough, nonlocal coefficient
determined by} $F${. At the endpoint} $s=1${, the collision
operator becomes the Landau operator, given by}
	\begin{align*}
			\mathcal{Q}_1(f_1,f_2)(v):=	
			\operatorname{div}\left[\int_{\mathbb{R}^d}a(v-w)\left(f_1(w)\nabla f_2(v)-f_2(v)\nabla f_1(w)\right) \dif w\right].
		\end{align*}
{The matrix-valued function} $a$ {is nonnegative, symmetric, and
determined by the particle interaction. For inverse-power-law interactions,}
\begin{align}\label{defa}
			a(z)=|z|^{\gamma+2}\left(\mathrm{Id}-\frac{z}{|z|}\otimes\frac{z}{|z|}\right),
		\end{align}
{where} $\mathrm{Id}$ {is the identity matrix and} $\otimes$
{is the tensor product. Equivalently, the Landau operator has the convolution
form}
\begin{align*}
			\mathcal{Q}_1(f_1,f_2)(v)=\operatorname{div}\left[a\star f_1(v)\nabla f_2(v)\right]-	\operatorname{div}\left[(a\star\nabla f_1)(v)f_2(v)\right].
		\end{align*}
Under a suitable normalization of the angular kernel,
$(1-s)\mathcal Q_s$ formally converges to a constant multiple of the Landau
operator as $s\uparrow1$; see Remark~\ref{Z} for the precise normalization
and scope of this comparison.

Our concern is the range
\[
-d-2s<\gamma\leq-2s,
\]
with $s=1$ denoting the Landau equation. In dimension three, this includes the
Landau--Coulomb case $\gamma=-3$. 

The nonlinear equation \eqref{eqlandau} has the scaling
\begin{equation}\label{scaling}
F_\lambda(t,x,v)
=\lambda^{-\nu(d+\gamma)-1}
 F\left(\frac{t}{\lambda},
        \frac{x}{\lambda^{1+\nu}},\frac{v}{\lambda^\nu}\right),
\qquad \nu\in\mathbb R.
\end{equation}
The choice $\nu=1/(2s)$ matches the transport and diffusion scales of
$\partial_t+v\cdot\nabla_x+(-\Delta_v)^s$.

We seek a well-posedness theory in which the initial data is measured by a
potential norm built from a quantity invariant under the data scaling at
$\nu=1/(2s)$. The coefficient estimates and the solution norms
must accommodate the collision anisotropy and its velocity-dependent
kinetic scales.
More generally, we seek pointwise Green-function estimates for equations
with prescribed nonnegative backgrounds, which effectively reflects
the same kinetic scaling in space and velocity.

\subsection{Main results}		\label{secnota}
{Throughout the paper,} {\(d\geq3\)} {is an integer, and we fix}
\[
    s\in(0,1],
    \qquad
    -d-2s<\gamma\leq -2s,
\]
{and set}
\[
    \kappa:=-(\gamma+2s)\in[0,d).
\]
{For} \(x,v\in\mathbb R^d\){, we write}
\begin{equation}\label{notaP}
    \langle x\rangle=(1+|x|^2)^{1/2},
    \qquad
    \langle x,v\rangle=(1+|x|^2+|v|^2)^{1/2},
    \qquad
    [x]_v=|x|+|x\cdot v|.
\end{equation}

{Before stating the main theorem, we specify the assumptions on the collision
kernel. In the Boltzmann case} \(0<s<1\){, we adopt the following
standard half-sphere convention.}

\begin{remark}[Half-sphere convention]\label{rem:half-sphere-convention}
{Temporarily denote the angular kernel in} \eqref{defB} {by}
\(\widetilde b\){, and define}
\[
\mathbb S^{d-1}_+
:=\{\sigma\in\mathbb S^{d-1}:\cos\theta\geq0\},
\qquad
b_{\rm sym}(r):=\widetilde b(r)+\widetilde b(-r),
\quad r\in[0,1].
\]
{Split the angular integral in} \(\mathcal Q_s(F,F)\)
{over the two half-spheres and apply} \(\sigma\mapsto-\sigma\)
{on} \(\{\cos\theta<0\}\){. This transformation exchanges}
\(v'\) {and} \(v_\star'\) {but leaves}
\(F(v_\star')F(v')\) {unchanged. Consequently, restricting the
angular integration to} \(\mathbb S^{d-1}_+\) {and replacing}
\(\widetilde b\) {by} \(b_{\rm sym}\) {does not change the
quadratic collision operator.}

{With a slight abuse of notation, we henceforth write} \(b\)
{for} \(b_{\rm sym}\) {and} \(\mathcal Q_s\)
{for the corresponding half-sphere bilinear extension:}
\[
\mathcal Q_s(f_1,f_2)(v)
:=
\int_{\mathbb R^d}\int_{\mathbb S^{d-1}_+}
\bigl(f_1(v_\star')f_2(v')-f_1(v_\star)f_2(v)\bigr)
|v_\star-v|^\gamma b(\cos\theta)
\,\dif\sigma\,\dif v_\star .
\]
\end{remark}

{Under this convention,}
\[
    B(|v-v_\star|,\cos\theta)
    =
    |v-v_\star|^\gamma b(\cos\theta),
    \qquad \cos\theta\in[0,1].
\]
{We impose the following quantitative form of the non-cutoff angular
singularity; see, for example,} \cite{Villani02,GS2011}. {There exist
constants}
\[
    0<c_b\leq C_b<\infty,
    \qquad
    \delta_0\in(0,1),
\]
{such that}
\begin{equation}\label{assofb}
    0\leq
    (1-r)^{\frac{d-1+2s}{2}}b(r)
    \leq C_b,
    \qquad r\in[0,1],
\end{equation}
and
\begin{equation}\label{assofb-lower}
    (1-r)^{\frac{d-1+2s}{2}}b(r)
    \geq c_b,
    \qquad r\in[1-\delta_0,1].
\end{equation}
{Moreover, for every} \(k\in \mathbb{N}\){,}
\begin{equation}\label{assofb-derivative}
    \left|
    \frac{\dif^k}{\dif r^k}
    \left[
    (1-r)^{\frac{d-1+2s}{2}}b(r)
    \right]
    \right|
    \leq C_{b,k},
    \qquad r\in[0,1].
\end{equation}
{Here} \(\delta_0\) {specifies a fixed neighborhood of grazing
collisions on which the non-cutoff singularity is nondegenerate.}

Throughout the Boltzmann case $0<s<1$, every use of coercivity assumes
\eqref{assofb-lower} with fixed $c_b>0$ and $\delta_0>0$.
The structural constants in coercivity estimates and their consequences
include these two parameters, together with the angular upper and
regularity bounds wherever used. Uniform estimates over angular kernels
require common positive lower bounds for $c_b$ and $\delta_0$ and common
upper bounds for the relevant angular constants.

{The Landau case} \(s=1\) {has no angular kernel; we use the
Landau collision operator with the matrix} \(a\) {defined in}
\eqref{defa}{.}
\subsubsection{Perturbative data and well-posedness}
			We write the distribution as a perturbation of the normalized Maxwellian
		$$\mu(v)=(2\pi)^{-\frac{d}{2}}\exp\left(-\frac{|v|^2}{2}\right),$$
			{with zero bulk velocity, unit density, and unit temperature. We
			define the perturbation} $f(t,x,v)$ {by}
		$$
		F=f+\mu.
		$$
			{The linearized collision operator is}
		\begin{align}\label{defL}
			\mathbf{L}_s(f)=	-\mathcal{Q}_s(\mu,f)-	\mathcal{Q}_s(f,\mu).
		\end{align}
			{Equation} \eqref{eqlandau} {then becomes the following
			equation for} $f=F-\mu${:}
		\begin{equation}\label{eqperbo}
			\begin{aligned}
				&\partial_t f+v\cdot \nabla_x f+\mathbf{L}_s (f)=	\mathcal{Q}_s(f,f),\quad\quad \quad s\in(0,1],\\
				&f|_{t=0}=f_0:=F_0-\mu.
			\end{aligned}
		\end{equation}
For $\kappa=-(\gamma+2s)$ as fixed above and $0<\kappa<d$, the homogeneous quantity
\[
\sup_{x,v}\int_{\mathbb R^d}|w|^{-\kappa}|f_0(x,v+w)|\,\dif w
\]
is \emph{invariant under the data scaling} induced by \eqref{scaling} with
$\nu=1/(2s)$. This is the Riesz potential of order $d-\kappa$, up to
normalization. At $\kappa=0$, we use the integral itself; the endpoint is
not interpreted as an ordinary Riesz kernel of order $d$.
The norm used in the local estimates truncates this potential and resolves
the velocity anisotropy.

{For a measurable function} $f:\mathbb{R}^d\times
			\mathbb{R}^d\to\mathbb{R}${, define}
		\begin{equation}\label{defI12}
			\begin{aligned}
				&\mathbf{I} f(x,v)=\sup_{v_0\in\mathbb{R}^d}\langle v_0\rangle\int_{\mathbb{R}^d}\mathbf{1}_{[w]_{v_0}\leq 1} [w]_{v_0}^{-\kappa}\operatorname*{ess\,sup}_{|x_1|\leq 1}|f|(x+x_1,v+w) \dif w,
			\end{aligned}
		\end{equation}
			where $\mathbf1$ is the indicator function. Suprema over phase-space variables
in data norms are understood as essential suprema; the supremum over the
anisotropy parameter $v_0$ is an ordinary supremum. At positive times, the
pointwise estimates are applied to continuous representatives. At $\kappa=0$,
the factor $[w]_{v_0}^{-\kappa}$ is interpreted as $1$.

		For fixed exponents $\varkappa_1,\varkappa_2\geq0$, set
\begin{equation}\label{normcr}
 \|f_0\|_{\mathrm{cr}}
 :=\operatorname*{ess\,sup}_{x,v\in\mathbb R^d}
       \langle x,v\rangle^{\varkappa_1}
       \langle v\rangle^{\varkappa_2}\mathbf I f_0(x,v).
\end{equation}
We call \eqref{normcr} the scaling-critical data norm: its underlying
homogeneous potential is invariant under the data scaling \eqref{scaling}
with $\nu=1/(2s)$. The truncation, anisotropy, and weights adapt this
critical quantity to the Maxwellian perturbation problem.

\begin{theorem}[Local well-posedness for scaling-critical data]\label{Th1}
Let $d\geq3$, $0<s\leq1$, and $-d-2s<\gamma\leq-2s$.
For $0<s<1$, assume the collision-kernel hypotheses
\eqref{assofb}--\eqref{assofb-derivative}; for $s=1$, use the
Landau operator \eqref{defa}. Let $N_0\in\mathbb N$ and
$\varkappa_1,\varkappa_2\geq0$ satisfy
\[
 N_0\geq10d,\qquad
 \varkappa_2>10s^{-1}(d+2)N_0.
\]
There exist structural constants $\delta_{\mathrm{loc}}\in(0,1)$,
$T_0\in(0,1)$, and $C>0$ with the following property.
If $f_0$ is measurable, $\mu+f_0\geq0$ almost everywhere, and
\begin{equation}\label{conini}
 \|f_0\|_{\mathrm{cr}}\leq\delta_{\mathrm{loc}},
\end{equation}
then \eqref{eqperbo} has a distributional solution $f$ on $(0,T_0]$
such that $\mu+f\geq0$ and
\[
 f(t)\longrightarrow f_0\quad\hbox{in }\mathcal D'(\mathbb R^{2d})
       \quad(t\downarrow0),\qquad
 \|f\|_{T_0}\leq C\|f_0\|_{\mathrm{cr}}.
\]
Here $\|\cdot\|_T$ is the solution norm
\eqref{defnormT}--\eqref{defT2}. Uniqueness holds among distributional
solutions with the same initial trace and
$\|h\|_{T_0}\leq4C_1\|f_0\|_{\mathrm{cr}}$, where $C_1\geq1$ is
the fixed structural mapping constant chosen in the proof.
The constants, including $T_0$, are independent of the size of $f_0$
within the stated ball.
\end{theorem}

In particular, write
$\tilde t_v=t\langle v\rangle^{-\kappa}$ and
$\hat t_v=t\langle v\rangle^\gamma$. The pointwise part of the
solution norm gives
\begin{equation}\label{short-time-pointwise}
 \sup_{\substack{0<t\leq T_0\\x,v\in\mathbb R^d}}
 \sum_{|\alpha|+|\beta|\leq N_0}
 \langle x,v\rangle^{\varkappa_1}\langle v\rangle^{\varkappa_2}
 \tilde t_v^{\frac{d-\kappa}{2s}}\,
 t^{|\alpha|}\hat t_v^{\frac{|\alpha|+|\beta|}{2s}}
 |\partial_x^\alpha\partial_v^\beta f(t,x,v)|
 \leq C\|f_0\|_{\mathrm{cr}}.
\end{equation}
The time powers quantify regularization of the initial velocity
concentration and the derivative costs on the kinetic scales.
The full norm additionally controls fractional velocity differences and
the coefficient potentials.

We use the Fourier convention
\[
\widehat h(\xi,\eta)=\iint_{\mathbb R^{2d}}
 e^{-i(x\cdot\xi+v\cdot\eta)}h(x,v)\,\dif x\,\dif v,\qquad
h(x,v)=(2\pi)^{-2d}\iint_{\mathbb R^{2d}}
 e^{i(x\cdot\xi+v\cdot\eta)}\widehat h(\xi,\eta)\,\dif\xi\,\dif\eta.
\]
Thus $\langle\nabla_v\rangle^m$ has symbol $\langle\eta\rangle^m$.

Under the weight assumptions below, the short-time estimates of
Theorem~\ref{Th1} supply the weighted regularity and smallness required
for a restart at a fixed positive time. Global continuation and long-time
decay are then reduced to weighted energy estimates for the restarted
solution, as recorded in the following corollary.

\begin{corollary}[Global continuation and long-time estimates]\label{thmglo}
        	{Let \(N_0\in\mathbb N\) and
			\(\varkappa_1,\varkappa_2\geq0\) satisfy}
			\[
			N_0,\varkappa_1\ge 10d,
			\qquad
			\varkappa_2>40s^{-1}(N_0+\varkappa_1)^2.
			\]
            {Assume \(s\in(0,1]\) and}
\[
    -d-2s<\gamma\leq -2s .
\]
	{For \(0<s<1\), suppose that the collision kernel satisfies
	\eqref{defB} and \eqref{assofb}--\eqref{assofb-derivative}; for \(s=1\),
	use the Landau collision operator defined in \eqref{defa}.}
			{
			There exist constants $\varepsilon_{\mathrm{data}}>0$ and $C>0$, together with a
			solution norm $\|\cdot\|_{\mathcal X}$ defined in the proof, with the
			following property: if the initial datum
			$f_0:\mathbb{R}^d\times\mathbb{R}^d\to\mathbb{R}$ satisfies
			$f_0+\mu\geq0$ and
			}
			\begin{equation}\label{conini-global}
 \|f_0\|_{\mathrm{cr}}\leq\varepsilon_{\mathrm{data}},
\end{equation}
			{
			then the Cauchy problem \eqref{eqperbo} admits a global solution $f$
			such that $f+\mu\geq0$ and
			$f(t)\to f_0$ in $\mathcal D'(\mathbb R^{2d})$ as $t\downarrow0$.
			Moreover,
			}
			\[
				\|f\|_{\mathcal X}\leq C\|f_0\|_{\mathrm{cr}},
			\]
			{
			and $f$ is unique among solutions $h$ with the same initial datum
			satisfying $\|h\|_{\mathcal X}\leq2C\|f_0\|_{\mathrm{cr}}$. In addition,
			}
\begin{equation}\label{main}
\begin{aligned}
&\sup_{t>0}\sum_{m+n\leq N_0}
 \min\{t,1\}^{n+\frac{d-\kappa+m}{2s}}\\[-2pt]
&\quad\times\left\|
 \left(\langle t\rangle^{d/2}
       +\langle x/\langle t\rangle\rangle^{\varkappa_1/(2d)}\right)
 \langle v\rangle^{\varkappa_2/2}
 \langle\nabla_v\rangle^m
 \left\langle\langle t\rangle\lambda(\nabla_x)\right\rangle^n f(t)
 \right\|_{L^\infty_{x,v}}\\
&\qquad\leq C\|f\|_{\mathcal X}\leq C^2\|f_0\|_{\mathrm{cr}}.
\end{aligned}
\end{equation}
			{
			Here $\lambda(\nabla_x)$ denotes the Fourier multiplier with symbol
			}
			\[
				\lambda(\xi)=|\xi|^2
				\langle \xi\rangle^{-\frac{2(1+s)}{1+2s}}.
			\]
		\end{corollary}

Estimate \eqref{main} separates short-time smoothing from large-time decay.
The factor $\min\{t,1\}^{n+(d-\kappa+m)/(2s)}$ compensates for the initial
velocity concentration and the derivative costs. The spatial multiplier obeys
\[
\lambda(\xi)\sim |\xi|^2\quad (|\xi|\leq1),
\qquad
\lambda(\xi)\sim |\xi|^{2s/(1+2s)}\quad (|\xi|\geq1).
\]
These frequency scales are consistent with the low-frequency Taylor
dispersion and high-frequency enhanced dissipation analysed in
\cite{BCD}, with the collision frequency fixed.
For $t\geq1$, the estimate gives $t^{-d/2}$ decay and polynomial localization
on the transport scale $x/\langle t\rangle$.

\begin{remark}[Localization and nonnegativity]\label{rem:localized-criticality}
The truncation $[w]_{v_0}\leq1$ in \eqref{defI12} is essential.
Removing it while retaining the supremum over $v_0$ can give an infinite
quantity even for smooth compactly supported data when $0\leq\kappa\leq1$.
The fixed cutoff, the local spatial supremum, and the polynomial weights
in \eqref{normcr} are additional assumptions, so the full norm is not
exactly invariant under \eqref{scaling}.
Nonnegativity is imposed separately by $\mu+f_0\geq0$; it is not a
consequence of the norm.
\end{remark}

We outline the proof of the local result, Theorem~\ref{Th1}, and its global
continuation, Corollary~\ref{thmglo}. The short-time construction adapts a
Schauder-type potential-space strategy, related to \cite{CHN,Hung}, to the
kinetic transport and collision geometry. We freeze the principal collision
operator at a velocity $v_0$ and study
\begin{equation}\label{eqmo}
\partial_t f+v\cdot\nabla_x f+\mathcal L_\mu^{v_0}f=F_{v_0}.
\end{equation}
For the Landau equation ($s=1$), the principal frozen coefficient is
$(a\star\mu)(v_0)$. Since $\mu$ is radial, the eigenvalues of this matrix
in directions orthogonal and parallel to $v_0$ have respective orders
$\langle v_0\rangle^{\gamma+2}$ and $\langle v_0\rangle^\gamma$.
Equivalently, in the sense of quadratic forms,
\[
(a\star\mu)(v_0)\asymp
\langle v_0\rangle^{\gamma+2}\Pi_{\hat v_0}
+\langle v_0\rangle^\gamma\mathrm P_{\hat v_0}
=\langle v_0\rangle^{-\kappa}\mathcal O_{v_0}^2,
\qquad \kappa=-(\gamma+2).
\]
Here $\mathrm P_{\hat v_0}$, $\Pi_{\hat v_0}$, and the stretching matrix
$\mathcal O_{v_0}$ are defined in Section~\ref{secgeometry}. Thus the
principal symbol is comparable to
$\langle v_0\rangle^{-\kappa}|\mathcal O_{v_0}\xi|^2$.
For the non-cutoff Boltzmann equation, the principal fractional symbol is
modeled by transverse and parallel terms of orders
$\langle v_0\rangle^{-\kappa}|\Pi_{\hat v_0}\xi|^{2s}$ and
$\langle v_0\rangle^\gamma|\mathrm P_{\hat v_0}\xi|^{2s}$,
respectively. Since $\gamma=-\kappa-2s$, both models lead to the adapted
derivatives and time
\[
\mathcal D_{v_0}^{(\diamond)}=\mathcal O_{v_0}\nabla_\diamond,
\quad \diamond\in\{x,v\},
\qquad \tilde t_{v_0}=t\langle v_0\rangle^{-\kappa}.
\]
The anisotropic distance satisfies
$[z]_{v_0}\sim|\mathcal O_{v_0}^{-1}z|$.
The commuting field
$\mathcal D_{v_0}^{(v)}+t\mathcal D_{v_0}^{(x)}$ transfers velocity smoothing
to space. The resulting velocity and spatial derivative scales are formally
$\tilde t_{v_0}^{-1/(2s)}$ and
$t^{-1}\tilde t_{v_0}^{-1/(2s)}$, respectively; these are the scales used in
\eqref{defyz}--\eqref{defnormT}.

The frozen Fourier representation \eqref{defH0} contains the phase
$\int_0^t E(\sigma\xi-\eta,v_0)\,\dif\sigma$, which records the transport
coupling between spatial and velocity frequencies. The fractional
Kolmogorov estimates of \cite{HouZhang} guide the near-characteristic
kernel profile, while Fourier estimates retain the velocity anisotropy and
control its off-diagonal decay; see also \cite{Grube,CHNfp,haina}.
For Maxwellian coefficients, low-frequency symbol regularity and tail
estimates give additional decay beyond the generic fractional
Kolmogorov profile. The resulting pointwise kernel bounds control the
solution and its collision coefficients in the local norm $\|f\|_T$.
Together with the bilinear estimate, they close a contraction for small
initial data and prove Theorem~\ref{Th1}. The same frozen estimates also
serve as inputs to the variable-coefficient Green-function construction
below.

For Corollary~\ref{thmglo}, the local bounds provide weighted regularity
and polynomial velocity decay at a fixed time $t_*=T_0/2$. Restarting at
$t_*$, we write $g(t)=f(t_*+t)$ and use the decomposition
$g=g_1+\sqrt\mu\,g_2$ of \cite{Caf,CG24}. The first component carries the
polynomial tail; the second is controlled in a Maxwellian-weighted space by
the frequency-dependent hypocoercive energies of \cite{BCD}, with the
strong-singularity estimates of \cite{YZ}. At high spatial frequencies,
mixed transport--collision energies yield enhanced dissipation. At low
frequencies, the micro--macro decomposition and Kawashima coupling transfer
microscopic dissipation to hydrodynamic modes, producing Taylor dispersion.
For strictly soft potentials, higher velocity moments are used to recover
decay through interpolation. The commuting field
$Z=\nabla_v+t\nabla_x$ tracks phase mixing, while the weighted estimates
propagate spatial localization. The collision, commutator, and moment
estimates for this continuation step are given in
Subsection~\ref{secglo}.

Finally, the fixed-point space does not itself enforce $\mu+f\geq0$.
We approximate the initial distribution $\mu+f_0$ by smooth nonnegative
data while controlling the critical norm. For each approximation, a
positive linear iteration preserves nonnegativity by the maximum principle
and converges under the local contraction estimates. Compactness at
positive times passes this property to the limiting solution, which is
identified with the fixed point by uniqueness. This separate step proves
the nonnegativity asserted in Theorem~\ref{Th1}.
\vspace{0.2cm}\\
{\bf Related works.}~~ The perturbative theory near
a Maxwellian begins with collision coercivity and weighted energy estimates.
For the non-cutoff Boltzmann operator, entropy-dissipation and anisotropic
collision estimates appear in \cite{ADVW,GS2011}; near-equilibrium existence
and qualitative results are developed in
\cite{AMUXY-1,AMUXY2011-AA,AMUXY2011-CMP,AMUXY2012JFA,GS2011}.
The global mild-solution framework of \cite{DLSS21} treats both equations.
For Landau, \cite{Guo} treats the periodic problem, \cite{HsiaoYu} treats soft
potentials in the whole space, and \cite{CM17} treats very soft and Coulomb
potentials near equilibrium. The torus theory \cite{CTW16} covers hard,
Maxwellian, and moderately soft potentials with weaker velocity weights.
The large-data theory of renormalized weak solutions
for the Boltzmann equation originates with DiPerna and Lions
\cite{DiPernaLions1989}; for non-cutoff interactions, see Alexandre and
Villani \cite{AlexandreVillani2002}. 

Polynomial velocity tails present a related issue. For the non-cutoff
Boltzmann equation near a global Maxwellian, \cite{AMSY} obtains perturbative
solutions using weighted collision estimates. The decomposition into a
polynomially decaying component and a Maxwellian-weighted component goes back
to Caflisch \cite{Caf}. It has been used for the cutoff equation in the whole
space \cite{DLL26}, for non-cutoff soft potentials near a global Maxwellian
\cite{CG24}, and near local Maxwellians \cite{DL26}. Polynomially weighted
results for the coupled Vlasov--Poisson--Boltzmann and
Vlasov--Poisson--Landau systems are obtained in \cite{CDL24} on a torus or a
union of cubes.

Local theories allow other classes of rough initial data. Henderson,
Snelson, and Tarfulea construct solutions from bounded, measurable,
polynomially decaying data, subject to positivity assumptions, for Landau
and non-cutoff Boltzmann respectively \cite{HSTLandau,HSTBoltzmann}.
Henderson and Wang prove local well-posedness for very soft non-cutoff
Boltzmann potentials in a weighted Sobolev class with polynomial velocity
decay \cite{HW22}. He and Ji show that certain rough data with slowly
decaying tails admit only finite Sobolev smoothing for the non-cutoff
soft-potential equation \cite{HeJi}.

In the spatially homogeneous setting, Guillen and Silvestre
\cite{GuillenSilvestre2025} proved that Fisher information is nonincreasing
for the Landau equation with a class of interactions including the Coulomb
case, ruling out finite-time blow-up. For the Boltzmann equation, Imbert,
Silvestre, and Villani \cite{ImbertSilvestreVillani2026} established the
corresponding monotonicity for a large class of collision kernels and
obtained global smooth solutions in the very-soft-potential regime. 

The terminology of criticality also requires care. Global theories in
critical Besov spaces \cite{DLX,MoSa,LZ,CLXX} use ``critical'' for borderline
spatial regularity in perturbative energy frameworks. In this paper,
``scaling-critical'' refers to the homogeneity of the underlying velocity
potential under the kinetic data scaling \eqref{scaling} with $\nu=1/(2s)$, see
Remark~\ref{rem:localized-criticality}. The weighted $L^{3/2}$ result for
the spatially homogeneous Landau--Coulomb equation \cite{GGL25} concerns a
scaling without spatial transport. For results on strong kinetic
singularities and the grazing-collision limit, see \cite{YZ}.

\subsubsection{Green functions for prescribed backgrounds}
We next consider the one-sided collision operator with prescribed background:
	\begin{equation}\label{li1}
		\partial_t F+v\cdot \nabla_x F-\mathcal{Q}_s(g,F)=0,
	\end{equation}
		{where} \(g=g(t,x,v)\ge 0\) {is prescribed on}
		\([0,T]\times\R^d\times\R^d\){, and}
		\(\varkappa\in\mathbb N\) {is sufficiently large. For}
\(a\in\mathbb R\){, let} \(\lfloor a\rfloor\) {denote the
greatest integer not exceeding} \(a\){:}
\[
\lfloor a\rfloor:=\max\{n\in\mathbb Z:n\le a\}.
\]
{For} $a>0$ {with} $a\notin\mathbb{N}${, set}
$k=\lfloor a\rfloor$ {and define the weighted H\"{o}lder norm by}
\begin{align}
\|h\|_{\mathbf{C}^a}:=
\sum_{|\alpha|\le k}\sup_{x,v}\langle v\rangle^\varkappa
|\partial_{x,v}^\alpha h(x,v)|
+\sum_{|\alpha|=k}\sup_{\substack{x,v\\0<|(y,w)|\le1}}
\langle v\rangle^\varkappa
\frac{|\partial_{x,v}^\alpha h(x,v)-
\partial_{x,v}^\alpha h(x-y,v-w)|}{|(y,w)|^{a-k}}.
\end{align}
{We also set}
$\|h\|_{\mathbf{C}^0}:=\sup_{x,v}\langle v\rangle^\varkappa |h(x,v)|${.
We assume that the mass density of} \(g\) {has a positive lower bound}
	\begin{align}\label{lbmass}
	 \int_{\R^d} g(t,x,v)\,\dif v\geq \mathsf m_0,
	\end{align}
		{uniformly for} \((t,x)\in [0,T]\times\R^d${. For}
\(0<s<1\){, choose}
			\[
			0<\varepsilon_1<\min\{s,1-s\},
			\qquad b_0:=s+\varepsilon_1.
			\]
	      {Assume further that}
\begin{equation}\label{Ig}
    \begin{aligned}
    \|g\|_{L^\infty_T\mathbf{C}^{b_0}}<\infty.
    \end{aligned}
\end{equation}
	  	{If} \(\varkappa>d+2\){, then} \eqref{Ig}
				{implies uniform upper bounds on the local mass and energy
				densities. The same weighted} \(L^\infty\) {control, together
				with the lower mass bound, also gives the entropy bound.}
 
	{Because} \(\mathcal{Q}_s(g,\cdot)\) {is linear in its second
	argument, we introduce the fundamental solution}
	\[
	\G=\G_g(t,x,v;\tau,y,w),
	\]
		{defined distributionally by}
	\begin{equation}\label{eq:Green_def}
		\left\{
		\begin{aligned}
			&\partial_t \G+v\cdot\nabla_x \G-\mathcal{Q}_s(g,\G)=0,
			\qquad t>\tau,\\
			&\G(\tau,x,v;\tau,y,w)=\delta(x-y)\delta(v-w).
		\end{aligned}
		\right.
	\end{equation}
		{Here} \(t\ge \tau\) {and}
		\((x,v),(y,w)\in\R^d\times\R^d\){; the initial datum is the
		phase-space Dirac mass at} \((y,w)\){.}
		
		{We also consider the inhomogeneous Cauchy problem with source}
		$\mathscr F(t,x,v)$ {and initial datum} $F_{in}(x,v)${:}
		\begin{equation}\label{eq:inho}
			\left\{
			\begin{aligned}
				& \partial_t F + v \cdot \nabla_x F - \mathcal{Q}_s(g, F) = \mathscr F, \quad t > 0, \\
				& F(0, x, v) = F_{in}(x, v).
			\end{aligned}
			\right.
		\end{equation}
		{On every interval where the kernel is defined, Duhamel's formula reads}
		\begin{equation}\label{eq:representation_formula}
			\begin{aligned}
				F(t, x, v) &= \iint_{\mathbb{R}^{2d}} \G(t, x, v; 0, y, w) F_{in}(y, w) \, \dif y \, \dif w \\
				&\quad + \int_0^t \iint_{\mathbb{R}^{2d}} \G(t, x, v; \tau, y, w) \mathscr F(\tau, y, w) \, \dif y \, \dif w \, \dif \tau .
			\end{aligned}
		\end{equation}

For a prescribed background $g$, the operator
\[
\partial_t+v\cdot\nabla_x-\mathcal Q_s(g,\cdot)
\]
may be viewed as a variable-coefficient, anisotropic deformation of the
fractional Kolmogorov operator. This viewpoint motivates comparison of its
Green function with the fundamental solution of the constant-coefficient
model. The reference profile is defined by
\begin{equation}\label{defN0}
\mathcal{N}(x,v) := \langle |x|+|v|\rangle^{-(d+2s)} \int_0^1 \langle x-\sigma v\rangle^{-(d+2s)} \, \dif\sigma, \qquad (x,v)\in\mathbb{R}^{2d}.
\end{equation}
For $0<s<1$, let $p_t$ be the fundamental solution of the fractional Kolmogorov operator
$\partial_t+v\cdot\nabla_x+(-\Delta_v)^s$ with $p_0=\delta_{(0,0)}$.
Hou and Zhang \cite[Theorem~1.1 and Lemma~2.7]{HouZhang} proved that
\begin{equation}\label{eq:kolmogorov-unit-profile}
	p_1(x,v)\asymp\mathcal N(x,v).
\end{equation}
The integral over $\sigma$ records the coupling between spatial displacement
and velocity along kinetic characteristics. Accordingly, we seek an upper
bound for the variable-coefficient Green function with the same kinetic
profile, supplemented by a polynomial factor reflecting the decay of $g$.

{To measure Green-function regularity, we introduce dimensionless spatial
and velocity increments adapted to the kinetic scaling. For} $t>0$
{and} $h\in\mathbb R^d${, define}
		\begin{align}\label{defPP2}
				\ell_{t}^{x}(h)=\frac{|h|}{t ^{1+\frac{1}{2s}}},\quad\quad\quad  \ell_{t}^{v}(h)=\frac{|h|}{t^\frac{1}{2s}}.
		\end{align}
{For} $b\in(0,1)$, $0<\tau<t${, the corresponding normalized pointwise
H\"older seminorms are}
\begin{align}
	[F]_{x,b}
	(t,x,v;\tau,y,w)
	&:=
	\sup_{\substack{h\neq0\\
			\ell_{t-\tau}^{x}(h)\leq1}}
	\frac{
		\left|
		\delta_h^xF(t,x,v;\tau,y,w)
		\right|
	}{
		(\ell_{t-\tau}^{x}(h))^{b}
	},
	\label{eq:def-normalized-x-holder}\\
	[F]_{v,b}
	(t,x,v;\tau,y,w)
	&:=
	\sup_{\substack{h\neq0\\
			\ell_{t-\tau}^{v}(h)\leq1}}
	\frac{
		\left|
		\delta_h^vF(t,x,v;\tau,y,w)
		\right|
	}{
		(\ell_{t-\tau}^{v}(h))^{b}
	}.
	\label{eq:def-normalized-v-holder}
\end{align}

The following theorem constructs the Green function for \eqref{li1} and gives
short-time pointwise upper bounds for the kernel and its kinetic H\"older increments.

		\Needspace{18\baselineskip}
\begin{theorem}\label{Greshort}
		     {Let} \(s\in(0,1)\){, assume that} \(g\)
		     {satisfies} \eqref{lbmass} {and} \eqref{Ig}{.
There exists a structural exponent}
	$M_\sharp>0${, depending only on} $d${,} $s${,
	and} $\gamma${, such that the following conclusion holds whenever}
	$\varkappa>M_\sharp${.}
    
	  {Define the Green-function envelope}
            \begin{align}
				\mathbf{B}(t,x,v,w)= \langle w\rangle^{M_\sharp}t^{-d-d/s}\mathcal{N}(t^{-1-\frac{1}{2s}}(x-tw),t^{-\frac{1}{2s}}(v-w))\left\langle\frac{\langle t^{-1}|x-tv|+|v|\rangle}{\langle w\rangle} \right\rangle^{-(\varkappa-M_\sharp)}.\label{defbT}
			\end{align}
		{Then there exists a Green function $\G_g$ satisfying}
\eqref{eq:Green_def}{, unique in the affine parametrix class
$\mathfrak A_{T_*}$ defined in the proof, and there exist}
\(T_*\in(0,\min\{1,T\}]\) {and}
		\(C_*>0\) {such that, for every}
		\(x,y,v,w\in\mathbb R^d\){,} \(0\le\tau<t\le T\){,
		and} \(0<t-\tau\le T_*\){,}
	\begin{equation}\label{ptwGreen}
					\begin{aligned}
						|\G_g(t,x,v;\tau,y,w)|+[\G_g]_{x,\frac{b_0}{1+2s}}(t,x,v;\tau,y,w)+&[\G_g]_{v,b_0}(t,x,v;\tau,y,w) \\
	                    &\leq C_*\,\mathbf{B}(t-\tau,x-y,v,w).
					\end{aligned}
				\end{equation}
		{The constants} \(T_*\) {and} \(C_*\)
		{depend only on}
		\(d,s,\gamma,b_0,\varkappa,T,\mathsf m_0\){,}
			\(\|g\|_{L^\infty_T\mathbf C^{b_0}}\){, and the angular-kernel
			constants, including $c_b$ and $\delta_0$ in
\eqref{assofb-lower}}.
				\end{theorem}
		
The source characteristic is $x=y+(t-\tau)w$, so the profile is written in
$X=x-y-(t-\tau)w$ and $V=v-w$. We first freeze the small-jump coefficient
along a kinetic characteristic. The lower mass bound,
together with the angular lower bound \eqref{assofb-lower}, yields averaged
coercivity, and the H\"older bound on $g$ controls the freezing error.
A transported cutoff reduces the variable-coefficient problem to a localized
Duhamel estimate.

The full operator is recovered by adding the large jumps and the cancellation
remainder. {For the cancellation remainder, a weak identity puts one finite difference on the}
coefficient and another on a product of kernels. The condition $b_0>s$ makes
the resulting jump integral integrable. Each parametrix step improves the time
power and consumes velocity weights; after finitely many steps, the remainder
is treated by contraction.

{Uniqueness from the final contraction is
in the affine class $\mathfrak A_{T_*}$ defined in the proof of
Theorem~\ref{Greshort}.}

The Green-function problem has several precedents with different operators
or coefficient hypotheses. The spectral analysis of Liu and Yu
\cite{LiuYu04,LiuYu06} concerns the Boltzmann equation linearized around a
global Maxwellian. For nonlocal kinetic model operators, \cite{HouZhang,Grube}
obtain fractional Kolmogorov-type pointwise kernel profiles; related
Schauder and fundamental-solution estimates appear in
\cite{HaoWuZhang,AuscherImbertNiebel}. In Theorem~\ref{Greshort}, the
coefficients are generated by a prescribed nonnegative background depending
on $(t,x,v)$, and the conclusion is a short-time pointwise Green-function
estimate under weighted H\"older control and a positive lower mass bound.

A related nonlinear regularity theory uses macroscopic bounds on the
solution. For $\gamma>-d$ and $\gamma+2s\leq0$, Silvestre's upper-bound
mechanism \cite{S} assumes additional weighted velocity integrability.
Weak Harnack and Schauder estimates for kinetic integral equations appear
in \cite{ImbertSilvestre-JEMS,ImbertSilvestre-APDE}; the global regularity
theorem \cite{IS0} assumes $\gamma+2s\geq0$. For the spatially
inhomogeneous Landau equation with very soft and Coulomb potentials,
\cite{SnelsonSolomon} gives a continuation criterion involving mass density,
a velocity moment, and velocity integrability. These nonlinear results have
different hypotheses and conclusions from the prescribed-background
Green-function estimate above.

Section~\ref{secgeometry} develops the collision geometry, coefficient bounds,
and frozen kernels. Section~\ref{secwp} uses these estimates in the local
potential-space construction and proves Theorem~\ref{Th1}.
Subsection~\ref{secglo} presents the continuation and long-time argument
for Corollary~\ref{thmglo}. Section~\ref{secGreen} establishes short-time
pointwise estimates for the Green function of the prescribed-background
kinetic equation.

\section{Collision geometry and frozen kinetic kernels}\label{secgeometry}
\subsection{Notation}

\begin{itemize}
	\item Throughout the paper,
$\mathbb N=\{0,1,2,\ldots\}$ and
$\mathbb N_+:=\{1,2,3,\ldots\}$.
			\item {For} $v_0\in\mathbb{R}^d${, define}
			$$\hat v_0=\begin{cases}
				\frac{v_0}{|v_0|},\ \ v_0\neq 0,\\
				0,\ \ \quad v_0=0.
			\end{cases}$$
			\item {For} $e\in\mathbb{S}^{d-1}\cup\{0\}${,
			define the projections} $\mathrm{P}_e$ {and} $\Pi_e$ {by}
			\begin{align}\label{defproj}
				\mathrm{P}_e= e\otimes e,\quad\quad \quad \Pi_e=\mathrm{Id}-\mathrm{P}_e.
			\end{align}
			\item	{For} $v_0\in\mathbb{R}^d${, define the
			stretching matrix} $\mathcal{O}_{v_0}$ {by}
			\begin{align}\label{defmatr}
				\mathcal{O}_{v_0}=\langle v_0\rangle^{-1}\mathrm{P}_{\hat v_0}+\Pi_{\hat v_0}.
			\end{align}
				{Its inverse is}
$\mathcal{O}_{v_0}^{-1}=\langle v_0\rangle\mathrm{P}_{\hat v_0}+\Pi_{\hat v_0}$.
	{The anisotropic metric} $[z]_{v_0}$ {from}
	\eqref{notaP} {satisfies}
$[z]_{v_0}\sim |\mathcal{O}_{v_0}^{-1}z|$.
			\item	{Let} $v_0\in\R^d${. For}
			$\diamond\in\{x,v\}${, define the anisotropic weighted derivatives}
			\begin{align}\label{anideri}
			\mathcal D^{(\diamond)}_{v_0}:=\mathcal O_{v_0}\nabla_\diamond,\ \ \  	\mathcal D^{(\diamond),k}_{v_0}:=\big(\mathcal D^{(\diamond)}_{v_0}\big)^{\otimes k},\ \ \ k\in\mathbb{N}.
			\end{align}
			{For} $t>0${, define the scaled anisotropic kinetic
derivative, an} \((x,v)\){-block operator, by}
			\[
			\mathcal D_{t,v_0}^{(x,v)}:=\big(\mathcal D^{(x)}_{t,v_0},\,\mathcal D^{(v)}_{t,v_0}\big)=\big(t\,\tilde t_{v_0}^\frac{1}{2s}\,\mathcal D^{(x)}_{v_0},\,\tilde t_{v_0}^\frac{1}{2s}\,\mathcal D^{(v)}_{v_0}\big),\quad\quad \text{with}\ \tilde t_{v_0}=\langle v_0\rangle^{-\kappa}t.
			\]
			{The factors}
			$t\tilde t_{v_0}^\frac{1}{2s}$ {and}
			$\tilde t_{v_0}^\frac{1}{2s}$ {match the kinetic scaling in}
			$(x,v)${. Similarly, set}
			\[
			\nabla^{(x,v)}_{t,v_0}:=\big(t\,\hat  t_{v_0}^\frac{1}{2s}\,\nabla_{x},\, \hat t_{v_0}^\frac{1}{2s}\nabla_v\big),\quad\quad \text{with}\ \hat t_{v_0}=\langle v_0\rangle^{-2s} \tilde t_{v_0}=\langle v_0\rangle^{\gamma}t.
			\]
			{For} $m\in\mathbb{N}${, write}
			\[
			\mathcal D_{t,v_0}^{(x,v),m}:=\bigl(\mathcal D_{t,v_0}^{(x,v)}\bigr)^{\otimes m},\qquad
			\nabla_{t,v_0}^{(x,v),m}:=\bigl(\nabla_{t,v_0}^{(x,v)}\bigr)^{\otimes m}.
			\]

			\item 
{For} $(y,w)\in \mathbb R^{2d}${, define the finite difference}
\begin{align*}
	\delta_{(y,w)}^{(x,v)} f(x,v)
	:= f(x,v)-f(x-y,v-w).
\end{align*}
{In particular, for} $h\in\mathbb{R}^d${,}
\begin{align*}
	\delta_h^x f(x,v):=\delta_{(h,0)}^{(x,v)} f(x,v),
	\qquad
	\delta_h^v f(x,v):=\delta_{(0,h)}^{(x,v)} f(x,v).
\end{align*}

			\item {We distinguish convolutions in different variables:}
			\(f\star g\) {denotes convolution in} \(v\){,
			whereas} \(f*g\) {denotes convolution in} \((x,v)\){.
			For two kernels}
			\[
			G_i:\{(t,x,v;\tau,y,w)\in\R^+\times\R^d\times\R^d\times\R^+\times\R^d\times\R^d:\ t>\tau\}\to\R,
			\qquad i=1,2,
			\]
			{their time--space--velocity convolution is defined by}
			\begin{align}\label{deftsvconvo}
			(G_1\circledast G_2)(t,x,v;\tau,y,w)
			:=\int_{\tau}^{t}\iint_{\R^{2d}}
			G_1(t,x,v;t',z,u)\,G_2(t',z,u;\tau,y,w)\,\dif z\,\dif u\,\dif t',
			\qquad t>\tau.
			\end{align}
			\item 	{We write} $A\lesssim B$ {if}
			$A\leq CB$ {for a generic constant} $C>0${, and}
			$A\sim B$ {if both} $A\lesssim B$ {and}
			$B\lesssim A${. The notation} $\lesssim_a$ {allows
			the implicit constant to depend on} $a${.}
	            \item {For} $a,b\in\mathbb{R}${, write}
	            $a\wedge b=\min\{a,b\}$ {and}
	            $a_+=\max\{a,0\}${.}
	            \item {For} $x,y\in\mathbb{R}^d${, denote
	            the Euclidean inner product by}
	            $(x,y)_{\mathbb{R}^d}=x\cdot y${. For a matrix}
	            $\mathsf{A}\in\mathbb{R}^{d\times d}${, write}
 $(\mathsf{A} x,x)_{\mathbb{R}^d}=x^\top \mathsf{A}\ x$
 {for the quadratic form associated with} $\mathsf{A}${.}
			\item {We write} $f(x)\lesssim \langle x\rangle^{-\infty}$
			{if, for every} $N\in\mathbb N${, there exists}
			$C_N>0$ {such that}
			\begin{align}\label{def-inf}
			|f(x)|\le C_N\langle x\rangle^{-N}\qquad \text{for all }x\in\mathbb R^d.	\end{align}
			{Equivalently,} $f$ {decays faster than every
			polynomial. In particular, for} \(m\in\mathbb N\){,}
\(\langle x\rangle^{m}f(x)\lesssim \langle x\rangle^{-\infty}\).

		\end{itemize}
        \begin{remark}[Comparison of derivative scales]
For bounded $v_0$, both scaled derivatives are comparable to
$(t^{1+1/(2s)}\nabla_x,t^{1/(2s)}\nabla_v)$. At large velocity,
$\mathcal D_{t,v_0}^{(\diamond)}$ retains the distinct parallel and transverse
scales. The Euclidean derivative satisfies
\[
\nabla_{t,v_0}^{(\diamond)}
=\langle v_0\rangle^{-1}\mathcal O_{v_0}^{-1}
 \mathcal D_{t,v_0}^{(\diamond)},
\qquad \diamond\in\{x,v\}.
\]
The matrix on the right has eigenvalue $1$ in the parallel direction and
$\langle v_0\rangle^{-1}$ in the transverse directions. Hence an anisotropic
derivative bound implies its Euclidean counterpart. The latter records less
transverse regularity.
{The comparison holds with $v_0$ fixed,
before evaluation at $v_0=v$.}
\end{remark}
	\begin{remark}[Finite differences associated with \(\mathcal D_{t,v_0}^{(x,v)}\)]
{For} \((x,v)\in\mathbb R^d\times\mathbb R^d\){,}
	\[
	x\cdot\nabla_x+v\cdot\nabla_v
	=
	\left(\frac{\mathcal O_{v_0}^{-1} x}{t\tilde t_{v_0}^{1/(2s)}}, 	\frac{\mathcal O_{v_0}^{-1}v}{\tilde t_{v_0}^{1/(2s)}}\right)\cdot 	\mathcal D_{t,v_0}^{(x,v)}.
	\]
{Since} \(|\mathcal{O}_{v_0}^{-1}h|\sim [h]_{v_0}\){,
the natural anisotropic lengths of the} \(x\){- and}
\(v\){-increments associated with}
	\(\mathcal D^{(x)}_{t,v_0}\) {and}
	\(\mathcal D^{(v)}_{t,v_0}\) {are}

    \begin{align}\label{defPP1}
       \ell_{t,v_0}^x(h)
	=\frac{[h]_{v_0}}{t\tilde t_{v_0}^{1/(2s)}},\quad\quad   \ell_{t,v_0}^v(h)
	=\frac{[h]_{v_0}}{\tilde t_{v_0}^{1/(2s)}}.
    \end{align}
  {If} $|v_0|\leq C${, then}
  $\ell_{t,v_0}^x(h)\sim_C\ell_t^x(h)$ {and}
  $\ell_{t,v_0}^v(h)\sim_C\ell_t^v(h)${, where}
  $\ell_t^x$ {and} $\ell_t^v$ {are defined in}
  \eqref{defPP2}{.}
	\end{remark}

	\begingroup
	The Green-kernel estimates use the following intrinsic terminal
seminorms, in addition to the physical seminorms in
	\eqref{eq:def-normalized-x-holder}--\eqref{eq:def-normalized-v-holder}.
	With $\sigma=t-\tau$, define
	\begin{align}\label{eq:def-intrinsic-terminal-holder}
	[F]^{\mathrm{int}}_{\diamond,b}(t,x,v;\tau,y,w)
	:=\sup_{\substack{h\ne0\\\ell_{\sigma,v}^{\diamond}(h)\le1}}
	\frac{|\delta_h^{\diamond}F(t,x,v;\tau,y,w)|}
	{\ell_{\sigma,v}^{\diamond}(h)^b},
	\qquad \diamond\in\{x,v\}.
	\end{align}
	The source seminorms $[F]^{\mathrm{int}}_{y,b}$ and
	$[F]^{\mathrm{int}}_{w,b}$ use $\ell_{\sigma,w}^{x}$ and
	$\ell_{\sigma,w}^{v}$, respectively. In each difference the reference
velocity is fixed at the unshifted endpoint.
	\par\endgroup

		\subsection{Carleman-type representation}		
			{Following} \cite{S}{, we reformulate the collision
			operator} \eqref{BL} {under the half-sphere convention of
			Remark~}\ref{rem:half-sphere-convention}{.}
		\begin{lemma}\label{lemrefo}
				{For} $s\in(0,1)${, the collision operator}
				$\mathcal{Q}_s(f_1,f_2)$ {admits the representation}
			\begin{align}\label{carle}
				\mathcal{Q}_s(f_1,f_2)(t,x,v)=\int_{\mathbb{R}^d}\left( f_2(t,x,v')	\mathbf{C}_{f_1}(t,x,v,v-v')-f_2(t,x,v)	\mathbf{C}_{f_1}(t,x,v',v-v')\right) \frac{\dif v'}{|v-v'|^{d+2s}},
			\end{align}
				{where} $\mathbf{C}_f$ {is defined by}
			\begin{align}\label{defCf1}
				\mathbf{C}_f(t,x,v,z):=	2^{d-1}\int_{z^\perp}f(t,x,v+w)|w|^{-\kappa+1}A\left(\frac{|z|^2}{|w|^2}\right)\mathbf{1}_{|w|\geq |z|} \dif \mathcal H^{d-1}(w),
			\end{align}
			with
			\begin{align*}
				A(\rho)= (1+\rho)^{\frac{\gamma+2-d}{2}}b\left(\frac{1-\rho}{1+\rho}\right) \rho^{\frac{d-1+2s}{2}},\quad 
				\kappa=-\gamma-2s.
			\end{align*}
	   {Here} $z^\perp=\{w\in\mathbb R^d:w\cdot z=0\}$
	   {and} $\mathcal H^{d-1}$ {is the}
	   $(d-1)${-dimensional Hausdorff measure on this hyperplane.}
Under assumptions \eqref{assofb} and
\eqref{assofb-derivative}, for $0\leq\rho\leq1$,
    	\begin{equation}\label{esofA}
			0\le A(\rho)\le C,\quad\quad
|A^{(n)}(\rho)|\lesssim_n 1,\qquad n\in\mathbb N\cup\{0\}.
			\end{equation}
For each $n$, the derivative bound depends on
$d,s,\gamma,C_b$ and $C_{b,k}$ for $1\leq k\leq n$.

Moreover, \eqref{assofb-lower} implies 
\begin{align}
			\label{rholb}
            A(\rho)\ge c_0,
			\qquad 0\le\rho\le\rho_0,
			\end{align}
with
\[
 \rho_0:=\frac{\delta_0}{2-\delta_0}\in(0,1),\qquad
 c_0:=2^{-(d-1+2s)/2}c_b
 \min\{1,(1+\rho_0)^{(1-\kappa)/2}\}>0.
\]

		\end{lemma}
  \begin{proof}
	{We suppress} \((t,x)\) {and use the half-sphere
	interpretation of} \eqref{BL} {from
	Remark~}\ref{rem:half-sphere-convention}{. The standard Carleman
	change of variables}
	\[
	(\sigma,v_\star)\longmapsto (w,v'),
	\qquad
	w:=v_\star'-v,
	\]
	{(see, for example,} \cite[Lemma A.1]{S}{) gives}
	\begin{align}\label{refo}
		\mathcal Q_s(f_1,f_2)(v)
		=
		2^{d-1}&
		\int_{\mathbb R^d}
		\int_{(v-v')^\perp}
		\bigl(
		f_1(v+w)f_2(v')-f_1(v'+w)f_2(v)
		\bigr)
		\nonumber\\
		&\qquad \qquad\qquad\times
		B(r,\cos\theta)r^{-(d-2)}
		\mathbf 1_{\{\cos\theta\ge0\}}
		\frac{\dif \mathcal H^{d-1}(w)\,\dif v'}{|v-v'|},
	\end{align}
	{where}
	\[
	r
	=
	\left(|v-v'|^2+|w|^2\right)^{1/2},
	\qquad
	\cos\frac{\theta}{2}
	=
	\frac{|w|}{r}.
	\]
	{It follows that}
	\begin{align}\label{costheta-vw}
		\cos\theta
		&=
		2\cos^2\frac{\theta}{2}-1
		=
		\frac{|w|^2-|v-v'|^2}
		{|w|^2+|v-v'|^2}.
	\end{align}
	{In particular,}
$
	\cos\theta\ge0$ is equivalent to
$|w|\ge |v-v'|.
$
		{Set}
	\[
	\rho:=\frac{|v-v'|^2}{|w|^2}.
	\]
	{By} \eqref{costheta-vw}{,}
	\[
	\cos\theta=\frac{1-\rho}{1+\rho},
	\qquad
	-\cos\theta=\frac{\rho-1}{1+\rho}.
	\]
	{Therefore, the definition of the symmetrized kernel} \(B\)
	{gives}
	\begin{align*}
		r^{-(d-2)}B(r,\cos\theta)
		&=
		r^{\gamma+2-d}
		b\left(\frac{1-\rho}{1+\rho}\right).
	\end{align*}
{Hence,}
	\begin{align}\label{factor-Carleman}
		r^{-(d-2)}B(r,\cos\theta)
		&=
		|w|^{-\kappa+1}
		|v-v'|^{-(d-1+2s)}
		A\left(\frac{|v-v'|^2}{|w|^2}\right).
	\end{align}
{Substituting} \eqref{factor-Carleman} {into}
\eqref{refo} {and using}
	$
	\mathbf 1_{\{\cos\theta\ge0\}}
	=
	\mathbf 1_{\{|w|\ge |v-v'|\}},
$
	{yields}
	\begin{align*}
		\mathcal Q_s(f_1,f_2)(v)
		&=
		2^{d-1}
		\int_{\mathbb R^d}
		\int_{(v-v')^\perp}
		\bigl(
		f_1(v+w)f_2(v')-f_1(v'+w)f_2(v)
		\bigr)
		\\
		&\qquad\qquad\times
		|w|^{-\kappa+1}
		A\left(\frac{|v-v'|^2}{|w|^2}\right)
		\mathbf 1_{\{|w|\ge |v-v'|\}}
		\frac{\dif \mathcal H^{d-1}(w) \,\dif v'}{|v-v'|^{d+2s}} .
	\end{align*}
	{Consequently,}
	\[
	\mathcal Q_s(f_1,f_2)(v)
	=
	\int_{\mathbb R^d}
	\left(
	f_2(v')\mathbf C_{f_1}(v,v-v')
	-
	f_2(v)\mathbf C_{f_1}(v',v-v')
	\right)
	\frac{\dif v'}{|v-v'|^{d+2s}}.
	\]
{This proves} \eqref{carle}.
Since $1+\rho\sim1$ on $[0,1]$ and the map
$\rho\mapsto(1-\rho)/(1+\rho)$ is smooth with bounded derivatives there,
the upper bound in \eqref{esofA} follows from \eqref{assofb},
and the derivative estimates additionally use \eqref{assofb-derivative}.
The lower bound \eqref{rholb} follows from \eqref{assofb-lower}.
\end{proof}


The following remark records the formal grazing limit $s\to 1$ and the corresponding
normalization.
\begin{remark}[Grazing-collision normalization]\label{Z}
Fix $\gamma>-d-2$ and use the half-sphere convention of
Remark~\ref{rem:half-sphere-convention}. Assume $b_s\geq0$ and
\begin{equation}\label{eq:grazing-angular-normalization}
 (1-s)\int_{\mathbb S^{d-1}_+}
 b_s(\cos\theta)(1-\cos\theta)\,\dif\sigma
 \longrightarrow \Lambda_{\rm ang}\in(0,\infty),
\end{equation}
while, for every $0<\delta<\pi/2$,
\begin{equation}\label{eq:grazing-angular-concentration}
 (1-s)\int_{\mathbb S^{d-1}_+\cap\{\theta\geq\delta\}}
 b_s(\cos\theta)\,\dif\sigma\longrightarrow0.
\end{equation}
Here $\mathcal Q_s$ has kernel $|v-v_*|^\gamma b_s(\cos\theta)$, and the
factor $1-s$ remains outside the collision operator. The exact Carleman
coefficient of Lemma~\ref{lemrefo}, including its angular factor and the
restriction $|w|\geq|h|$, is left unchanged.

With the Landau matrix normalized by \eqref{defa}, the first two angular
moments and the second-order Taylor expansion give, formally,
\[
 (1-s)\mathcal Q_s(f_1,f_2)\longrightarrow
 C_{\rm graz}\mathcal Q_1(f_1,f_2),\qquad
 C_{\rm graz}=\frac{\Lambda_{\rm ang}}{4(d-1)},\qquad s\uparrow1,
\]
{for fixed smooth, rapidly decaying $f_1,f_2$,
in the weak sense.}

For the model kernel
\[
 b_s(\cos\theta)=A\left(\sin\frac\theta2\right)^{-(d-1)-2s},
 \qquad 0<\theta\leq\pi/2,
\]
one has $\Lambda_{\rm ang}=A2^{d-1}|\mathbb S^{d-2}|$ and hence
$C_{\rm graz}=A2^{d-3}|\mathbb S^{d-2}|/(d-1)$. Thus $A=1$ gives
$C_{\rm graz}=\pi$ when $d=3$, as in \cite{YZ}; choosing
$A=(d-1)/(2^{d-3}|\mathbb S^{d-2}|)$ gives $C_{\rm graz}=1$.
\end{remark}

	        {We next recall the cancellation lemma for the non-cutoff Boltzmann
	        operator. Define}
\begin{align}\label{defK}
    \mathcal K(f)(t,x,v)
    :=
    \int_{\mathbb R^d}
    \delta_z^v \mathbf C_f(t,x,v,z)
    \frac{\dif z}{|z|^{d+2s}} .
\end{align}
\begin{lemma}\label{lemcanc}
Let $0<s<1$, $-d-2s<\gamma\leq-2s$, and $f\in\mathcal S(\mathbb R^d)$.
With the symmetric principal-value interpretation in \eqref{defK},
\begin{equation}\label{cancel}
\mathcal K(f)=c_{d,\gamma,s,b}\Lambda_v^{-d-\gamma}f,
\qquad c_{d,\gamma,s,b}>0.
\end{equation}
At $\gamma=-d$, the multiplier on the right is the identity.
We use this coefficient throughout with the fixed half-sphere
and Fourier conventions; $c_{d,\gamma,s,\lambda b}=\lambda c_{d,\gamma,s,b}$
for $\lambda>0$.
\end{lemma}
\begin{proof}
Put $q=d+\gamma\in(-2s,d-2s]$, $L_a=(1+a^2)^{1/2}$,
$S_f(v,r)=\int_{\mathbb S^{d-1}}f(v+r\omega)\,\dif\omega$, and
$C_d=2^{d-1}|\mathbb S^{d-2}|$.
In \eqref{defCf1}, write $w=r\omega$, $z=ar\theta$, with
$\omega\perp\theta$ and $0<a<1$. Rotational invariance of the measure
on orthogonal pairs gives
\begin{equation}\label{eq:cancellation-radial}
\mathcal Kf(v)=C_d\int_0^1\frac{A(a^2)}{a^{1+2s}}
 \int_0^\infty r^{q-1}
 \bigl(S_f(v,r)-S_f(v,L_ar)\bigr)\,\dif r\,\dif a.
\end{equation}
To justify the changes of variables, first average the integrand at $z$ and
$-z$. Its second difference is $O(a^2r^2)$ near $r=0$, and is bounded by
$C_{N,v}a^2r^2\langle r\rangle^{-N}$ for large $r$, since
$|\omega\pm ta\theta|\geq1$ for $0\leq t\leq1$.
Thus the symmetrized double integral is absolutely convergent:
$q>-2s>-2$ and $s<1$ suffice.

For $q\ne0$, set
\[
J_q(v)=
\begin{cases}
\displaystyle\int_0^\infty r^{q-1}S_f(v,r)\,\dif r,&q>0,\\[3pt]
\displaystyle\int_0^\infty r^{q-1}
 (S_f(v,r)-|\mathbb S^{d-1}|f(v))\,\dif r,&q<0.
\end{cases}
\]
Scaling the radial variable in \eqref{eq:cancellation-radial} yields
\[
\mathcal Kf=C_d I_qJ_q,
\qquad
I_q=\int_0^1 A(a^2)a^{-1-2s}(1-L_a^{-q})\,\dif a.
\]
For $q>0$, $J_q$ is a positive multiple of the Riesz potential
$\Lambda^{-q}f$. For $q<0$, it is a negative multiple of
$\Lambda^{-q}f$, by the symmetric second-difference formula for the fractional
Laplacian of order $-q\in(0,2s)$. Since $I_q$ has the sign of $q$,
the coefficient in \eqref{cancel} is positive in both cases.
For $q=0$, integration on $(\varepsilon,R)$ followed by
$\varepsilon\downarrow0$, $R\uparrow\infty$ gives
\[
\int_0^\infty\frac{S_f(v,r)-S_f(v,L_ar)}r\,\dif r
=|\mathbb S^{d-1}|f(v)\log L_a.
\]
Consequently,
$c_{d,-d,s,b}=C_d|\mathbb S^{d-1}|
\int_0^1 A(a^2)a^{-1-2s}\log L_a\,\dif a>0$.
All the angular integrals converge because their factors in parentheses
are $O(a^2)$ at zero. Positivity is strict by \eqref{rholb}.
\end{proof}

			{We now isolate the principal part of the collision operator used in}
			\eqref{defL}{.}

{Lemma~}\ref{lemrefo} {decomposes the collision operator as}
		\begin{equation}\label{qsmr}
			\begin{aligned}
				\mathcal{Q}_{s}(f_1,f_2)(v)=&	-\int_{\mathbb{R}^d} (f_2(v)-f_2(v'))	\mathbf{C}_{f_1}(v,v-v')\frac{\dif v'}{|v-v'|^{d+2s}}\\
				&+f_2(v)\int_{\mathbb{R}^d}\left( \mathbf{C}_{f_1}(v,v-v')-\mathbf{C}_{f_1}(v',v-v')\right)\frac{\dif v'}{|v-v'|^{d+2s}}\\
				:=&\mathcal{Q}_{s,\mathrm{main}}(f_1,f_2)(v)+\mathcal{Q}_{s,\mathrm{rem}}(f_1,f_2)(v).
			\end{aligned}
		\end{equation}
{Lemma~}\ref{lemcanc} {gives}
\begin{align*}
\mathcal{Q}_{s,\mathrm{rem}}(f_1,f_2)=c_{d,\gamma,s,b}f_2\,\Lambda_v^{-d-\gamma} f_1.
\end{align*}
			{For} $s=1${, define}
		\begin{align}\label{q1mr}
			\mathcal{Q}_{1,\mathrm{main}}(f_1,f_2)(v)=\operatorname{div}(a\star f_1(v)\nabla f_2(v)),\quad\quad\quad\mathcal{Q}_{1,\mathrm{rem}}(f_1,f_2)(v)=-\operatorname{div}(a\star \nabla f_1(v) f_2(v)).
		\end{align}
			{The operator} \(\mathcal{Q}_{s,\mathrm{main}}(f_1,f_2)\)
			{acts on} \(f_2\) {with nonlocal coefficients
			determined by} \(f_1\){, while}
			\(\mathcal{Q}_{s,\mathrm{rem}}(f_1,f_2)\) {is lower order.}

			{Define}
		\begin{align*}
			\mathcal{L}_g(f)=-\mathcal{Q}_{s,\mathrm{main}}(g,f).
		\end{align*}
			{The one-sided operator can be written as}
			\begin{align*}
			-\mathcal{Q}_s(g,f)=\mathcal{L}_g(f)(v)-\mathcal{Q}_{s,\mathrm{rem}}(g,f).
		\end{align*}
			{Consequently,} \eqref{li1} {becomes}
			\begin{equation}\label{eqperbo1}
			\begin{aligned}
				&\partial_t F+v\cdot \nabla_x F+\mathcal{L}_g (F)=	\mathcal{Q}_{s,\mathrm{rem}}(g,F),\quad\quad \quad s\in(0,1],\\
				&F|_{t=0}=F_0.
			\end{aligned}
		\end{equation}
			{Similarly, in the perturbative setting,} \eqref{eqperbo}
			{becomes}
		\begin{equation}\label{eqpert2}
			\begin{aligned}
				&	\partial_t f+v\cdot \nabla_x f+\mathcal{L}_\mu (f)=	\mathcal{Q}_{s}(f,f)+\mathcal{Q}_{s}(f,\mu)+\mathcal{Q}_{s,\mathrm{rem}}(\mu,f),\quad\quad \quad s\in(0,1],\\
				&f|_{t=0}=f_0.
			\end{aligned}
		\end{equation}
		\subsection{Frozen-coefficient approximation and the Duhamel formula}
	        {We use the convolution convention from}
	        \eqref{deftsvconvo}{. For}
$(t,x,v)\in \mathbb{R}^+\times \mathbb{R}^d\times \mathbb{R}^d${,
set}
		\[
        X^t_{x,v}=x-tv. 
		\]
	            {Fix} $(t_0,x_0,v_0)\in
	            (0,T]\times\R^d\times\R^d${. For each}
			$t\in[0,t_0]${, freeze the operator along the backward
			free-transport characteristic by setting}
		\begin{equation}\label{deffro}
		\mathcal L^{t_0,x_0,v_0}_{g}(t)\phi(v)
		:=
		\begin{cases}
			\displaystyle
			\int_{\R^d}\bigl(\phi(v)-\phi(v-z)\bigr)\,
			\mathbf C_g\bigl(t,X^{t_0-t}_{x_0,v_0},v_0,z\bigr)\frac{\dif z}{|z|^{d+2s}},
			& s\in(0,1),\\[2ex]
			\displaystyle
			-\operatorname{div}_v\!\Big((a\star g)\bigl(t,X^{t_0-t}_{x_0,v_0},v_0\bigr)\nabla_v \phi(v)\Big),
			& s=1,
		\end{cases}
		\end{equation}
			{where} $\mathbf{C}_g(t,x,v,z)$ {is defined in}
			\eqref{defCf1}{. When} \(g=\mu\){, the coefficients
			are independent of} \((t,x)\){, and we simply write}
			\(\mathcal{L}_g^{t_0,x_0,v_0}(t)=\mathcal L_\mu^{v_0}\){.}

			{Fix} $\tau\in[0,t_0)$ {and consider the Cauchy problem}
		\begin{equation}\label{lineeq_tau_char}
			\left\{
			\begin{aligned}
				&\partial_t f+v\cdot\nabla_x f+\mathcal L^{t_0,x_0,v_0}_{g}(t)f=F,
				\qquad \text{in } (\tau,t_0)\times\R^d\times\R^d,\\
				&f|_{t=\tau}=f_\tau .
			\end{aligned}
			\right.
		\end{equation}
	
			{The next lemma gives a Duhamel formula for this equation. Freezing
			the coefficients along transport characteristics is natural because, without
			collisions or forcing,}
			\[ \partial_t f+v\cdot\nabla_x f=0 \qquad\Longrightarrow\qquad f(t,x,v)=f_\tau(x-(t-\tau)v,v). \]
			{Accordingly, the kernel representation uses the backward transport
			shift} \(x\mapsto x-(t-\sigma)v\){.}
	\begin{lemma}\label{lemlin-char}
		{For every} $t\in(\tau,t_0]${, the solution of}
		\eqref{lineeq_tau_char} {is}
		\begin{align}
			f(t,x,v)
			&=\big(H^{t_0,x_0,v_0}_{g}(t,\tau)\ast f_\tau\big)\big(x-(t-\tau)v,v\big)\notag\\
			&\quad+\int_\tau^t
			\big(H^{t_0,x_0,v_0}_{g}(t,\sigma)\ast F(\sigma)\big)
			\big(x-(t-\sigma)v,v\big)\,\dif \sigma,
			\label{solfor_tau_char}
		\end{align}
		{where}
		\begin{align}
			H^{t_0,x_0,v_0}_{g}(t,\tau,x,v)
			=
			\frac1{(2\pi)^{2d}}\iint_{\R^{2d}}
			\exp\!\Big(
			-\int_\tau^t
			E_g\big((r-\tau)\xi-\eta;\,r,\,X_{x_0,v_0}^{t_0-r},\,v_0\big)\,\dif r
			+i\xi\cdot x+i\eta\cdot v
			\Big)\,\dif\xi\,\dif\eta,
			\label{defH}
		\end{align}
		{and the symbol is}
		\begin{align}
			E_g(z;r,x,v)
			=
			\begin{cases}
				\displaystyle
				\int_{\R^d}(1-e^{iz\cdot h})\,\mathbf C_g\bigl(r,x,v,h\bigr)\,
				\frac{\dif h}{|h|^{d+2s}},
				& s\in(0,1),\\[2ex]
				\big((a\star g)(r,x,v)\,z,z\big)_{\mathbb{R}^d},
				& s=1.
			\end{cases}
			\label{defEs}
		\end{align}
		{Equivalently, the Green kernel of}
		$\partial_t+v\cdot\nabla_x+\mathcal{L}_g^{t_0,x_0,v_0}(t)$
		{is}
		\begin{equation}\label{KH}
			\mathcal G_{t_0,x_0,v_0}(t,x,v;\tau,y,u)
			=
			H^{t_0,x_0,v_0}_{g}\bigl(t,\tau,x-y-(t-\tau)v,v-u\bigr).
		\end{equation}
	\end{lemma}

	\begin{proof}
{Taking the Fourier transform in} \((x,v)\){, with}
\(\xi\) {dual to} \(x\) {and} \(\eta\)
{dual to} \(v\){, gives}\[
		\partial_t \hat f(t,\xi,\eta)-\xi\cdot\nabla_\eta \hat f(t,\xi,\eta)
		+E_g\bigl(\eta;\,t,\,X_{x_0,v_0}^{t_0-t},\,v_0\bigr)\hat f(t,\xi,\eta)
		=\hat F(t,\xi,\eta),
		\qquad
		\hat f(\tau,\xi,\eta)=\hat f_\tau(\xi,\eta).
		\]
		{Introduce the characteristic shift}
		\[
		\hat p(t,\xi,\eta):=\hat f\bigl(t,\xi,\eta-(t-\tau)\xi\bigr),
		\qquad
		\text{equivalently }\
		p(t,x,v)=f\bigl(t,x+(t-\tau)v,v\bigr).
		\]
		{Then} $\hat p$ {satisfies}
		\[
		\partial_t \hat p(t,\xi,\eta)
		+
		E_g\bigl(\eta-(t-\tau)\xi;\,t,\,X_{x_0,v_0}^{t_0-t},\,v_0\bigr)\,\hat p(t,\xi,\eta)
		=
		\hat F\bigl(t,\xi,\eta-(t-\tau)\xi\bigr),
		\]
		{with initial datum}
		\[
		\hat p(\tau,\xi,\eta)=\hat f_\tau(\xi,\eta).
		\]
		{Solving this ODE yields}
		\begin{align*}
			\hat p(t,\xi,\eta)
			&=
			\exp\!\Big(
			-\int_\tau^t
			E_g\bigl(\eta-(r-\tau)\xi;\,r,\,X_{x_0,v_0}^{t_0-r},\,v_0\bigr)\,\dif r
			\Big)\hat f_\tau(\xi,\eta)\\
			&\quad
			+\int_\tau^t
			\exp\!\Big(
			-\int_\sigma^t
			E_g\bigl(\eta-(r-\tau)\xi;\,r,\,X_{x_0,v_0}^{t_0-r},\,v_0\bigr)\,\dif r
			\Big)
			\hat F\bigl(\sigma,\xi,\eta-(\sigma-\tau)\xi\bigr)\,\dif \sigma.
		\end{align*}
		{Taking the inverse Fourier transform in} $(\xi,\eta)$
		{gives}
		\[
		p(t,x,v)
		=
		\big(H^{t_0,x_0,v_0}_{g}(t,\tau)\ast f_\tau\big)(x,v)
		+
		\int_\tau^t
		\big(H^{t_0,x_0,v_0}_{g}(t,\sigma)\ast F(\sigma)\big)
		\bigl(x+(\sigma-\tau)v,v\bigr)\,\dif \sigma.
		\]
		{Finally, since}
		\[
		f(t,x,v)=p\bigl(t,x-(t-\tau)v,v\bigr),
		\]
		{we obtain} \eqref{solfor_tau_char}{. Formula}
		\eqref{KH} {then follows by taking}
		$f_\tau=\delta_{(y,u)}$ {and} $F\equiv0${.}
	\end{proof}

In the notation of the physical density introduced in
\eqref{eq:frozen-physical-density-Fourier}, the frozen Green kernel
\eqref{KH} has the equivalent form
\begin{equation}\label{eq:physical-frozen-Green}
\mathcal G_{t_0,x_0,v_0}(t,x,v;\tau,y,u)
=p_g^{t_0,x_0,v_0}(t,\tau;x-y-\sigma u,v-u),
\qquad \sigma=t-\tau.
\end{equation}
Indeed, substituting $X=x-y-\sigma v$ and $V=v-u$ into
\eqref{eq:H-physical-shear} gives $X+\sigma V=x-y-\sigma u$.
Thus the physical density uses the initial velocity $u$ in its
spatial transport shift.

		{When} $g=\mu${, the frozen coefficients depend only on}
		$v_0$ {and not on} $t$ {or} $x${. Thus,
		modulo free transport, the frozen operator is translation-invariant in}
		$(x,v)${, and its fundamental solution has an explicit Fourier
		representation.}

		\begin{corollary}\label{coroduha}
			{Setting} \(g=\mu\) {and} \(\tau=0\)
			{in Lemma~}\ref{lemlin-char}{, we obtain the solution of the
			frozen Cauchy problem:}
			\begin{align}\label{solfor0}
				f(t,x,v)=(H_\mu^{v_0}(t)\ast	f_0)(x-tv,v)+\int_0^t (H_\mu^{v_0}(t-\tau)\ast  F(\tau))(x-(t-\tau)v,v)\dif \tau ,
			\end{align}
			{where}
			\begin{align}
				&H_\mu^{v_0}(t,x,v)=\frac{1}{(2\pi)^{2d}}\iint_{\mathbb{R}^{2d}}
				\exp\left(-\int_{0}^{t}E(\sigma\xi-\eta,v_0)\dif \sigma +i\xi\cdot x+i\eta\cdot v\right) \dif \xi \dif \eta,\label{defH0}\\ &E(z,v_0)=\left\{\begin{aligned}
					&	\int_{\mathbb{R}^d} (1-e^{iz\cdot h}) {\mathbf{C}}_\mu(v_0, h)\frac{\dif h}{|h|^{d+2s}},\ \ s\in(0,1),\\
					&(a\star\mu(v_0)z,z)_{\mathbb{R}^d},\ \ \ s=1.
				\end{aligned}\right.\label{defEs0}
			\end{align}	
		\end{corollary}
		\subsection{Estimates for the coefficient functions}\label{seculb}
	{We now derive lower and upper pointwise bounds for} \(E_g\){,
	defined in} \eqref{defEs}{. These bounds govern the
	regularity and decay of solutions to the frozen-coefficient equation}
	\eqref{lineeq_tau_char}{.}
	
	\subsubsection{\texorpdfstring{{Lower bound for $E_g$}}{Lower bound for Eg}}
	 {Let}
	 $g:[0,T]\times\mathbb{R}^d\times\mathbb{R}^d\to[0,\infty)$
	 {be given, and assume that}
			\begin{subequations}\label{eq:macro-main}
			\begin{align}
				&\mathsf{m}_0\le \int_{\R^d} g(t,x,v)\,\dif v\le \mathsf M_0,\label{eq:macro-main-a}\\
				&\int_{\R^d}|v|^2 g(t,x,v)\,\dif v\le \mathsf E_0,\label{eq:macro-main-b}\\
				&\int_{\R^d} g(t,x,v)\log(1+g(t,x,v))\,\dif v\le \mathsf H_0,\label{eq:macro-main-c}
			\end{align}
		\end{subequations}
	{hold uniformly for} \((t,x)\in[0,T]\times\mathbb{R}^d\){.
	Direct estimates show that} \eqref{lbmass}--\eqref{Ig}
	{imply} \eqref{eq:macro-main} {provided $\varkappa>d+2$, as
assumed in the Green-function construction. Indeed, if
$K=\|g\|_{L^\infty_T\mathbf C^0}$, then the mass and energy are bounded
by $K\int(1+|v|^2)\langle v\rangle^{-\varkappa}\,\dif v$, and
$\int g\log(1+g)\,\dif v\leq\log(1+K)\int g\,\dif v$.}

		{We begin with the Landau case} $s=1$ {and the
		coefficient matrix} \(\mathsf{A}[g](v)=a\star g(v)\){; see}
		\cite{DL,DV1,Luis2017}{.}
		\begin{lemma} \cite[Lemmas 3.1--3.2]{Luis2017} \label{leA}
{Assume that $g\geq0$ satisfies \eqref{eq:macro-main} and
$\sup_{t\in[0,T]}\|g(t)\|_{\mathbf C^0}<\infty$.
Then $\mathsf A[g]$ is a finite symmetric matrix, and there exists}
\begin{equation}
c=c({d,\gamma,}\mathsf m_0,\mathsf M_0,\mathsf E_0,\mathsf H_0)>0,
\end{equation}
{such that, uniformly for
$(t,x,v)\in[0,T]\times\mathbb R^d\times\mathbb R^d$ and
$e\in\mathbb S^{d-1}$,}
\begin{equation}
\operatorname{det} \mathsf{A}[g]
\ge
c\langle v\rangle^{(d-1)(\gamma+2)+\gamma},
\end{equation}
\begin{equation}
\left(\mathsf{A}[g](v) e, e\right)_{\mathbb{R}^d} \geq c \jap{v}^\gamma (\langle v\rangle^2-|v\cdot e|^2).
\end{equation}
\end{lemma}
		
	{We next consider the Boltzmann case} $s\in(0,1)$
	{and derive a pointwise characterization of the frozen Fourier
	kernel. Related Sobolev coercivity estimates appear in}
	\cite{ADVW,Mou06,MS07,AMUXY2012JFA,AMUXYKJM,AMUXY2011-CMP,GS2011}{.}

{We first reformulate} $E_g$ {through the Carleman
representation.}
	\begin{lemma}\label{lem:hyperplane-representation}
		{Let} $s\in(0,1)$, let $g\geq0$, and assume
\eqref{assofb-lower}. {For every} $t>0$
		{and} $x_0,v_0,z\in\mathbb R^d${,}
		\begin{align}\label{eq:hyperplane-representation}
			E_g(z;t,x_0,v_0)
		{\geq 2^{d-1}c_0}\int_{\mathbb R^d} g(v_0+w)\,|w|^{-\kappa}
			\int_{\substack{h\in w^\perp\\ |h|\le \sqrt{\rho_0}|w|}}
			\bigl(1-\cos(z\cdot h)\bigr)\,|h|^{-d+1-2s}\,\dif \mathcal H^{d-1}(h)\,\dif w,
		\end{align}
	        where $c_0$ and $\rho_0$ are the lower-bound
parameters specified in \eqref{rholb}, with the stated dependence on
$c_b$ and $\delta_0$.
	\end{lemma}
	\begin{proof}
{The definition of} $E_g$ {and the Carleman
representation give}
\begin{align*}
E_g(z;t,x_0,v_0)
={}&2^{d-1}\int_{\mathbb R^d}\int_{h^\perp}
g(v_0+w)\bigl(1-\cos(z\cdot h)\bigr)
|h|^{-d-2s}|w|^{-\kappa+1} \\
&\qquad\qquad\times
A\left(\frac{|h|^2}{|w|^2}\right)
\mathbf 1_{\{|h|\leq |w|\}}
\,\dif\mathcal H^{d-1}(w)\,\dif h .
\end{align*}
{All integrands are nonnegative, so Tonelli's theorem applies. The coarea
formula on the incidence set}
$\{(h,w)\in\mathbb R^d\times\mathbb R^d:h\cdot w=0\}$ gives
\[
\int_{\mathbb R^d}\int_{h^\perp}F(h,w)
\,\dif\mathcal H^{d-1}(w)\,\dif h
=
\int_{\mathbb R^d}\int_{w^\perp}F(h,w)
\frac{|h|}{|w|}
\,\dif\mathcal H^{d-1}(h)\,\dif w .
\]
{Consequently,}
\begin{align*}
E_g(z;t,x_0,v_0)
={}&2^{d-1}\int_{\mathbb R^d}g(v_0+w)|w|^{-\kappa}
\int_{w^\perp}
\bigl(1-\cos(z\cdot h)\bigr)|h|^{-d+1-2s} \\
&\qquad\qquad\times
A\left(\frac{|h|^2}{|w|^2}\right)
\mathbf 1_{\{|h|\leq |w|\}}
\,\dif\mathcal H^{d-1}(h)\,\dif w .
\end{align*}
{By \eqref{assofb-lower} and \eqref{rholb},
$A(\rho)\geq c_0$}
{for} $0\leq\rho\leq\rho_0${. Restricting the
nonnegative inner integral to} $|h|\leq\sqrt{\rho_0}|w|$ {yields}
\[
E_g(z;t,x_0,v_0)
{\geq 2^{d-1}c_0}
\int_{\mathbb R^d}g(v_0+w)|w|^{-\kappa}
\int_{\substack{h\in w^\perp\\
|h|\leq\sqrt{\rho_0}|w|}}
\bigl(1-\cos(z\cdot h)\bigr)|h|^{-d+1-2s}
\,\dif\mathcal H^{d-1}(h)\,\dif w .
\]
{This retains both the amplitude $c_0$ and the
angular interval represented by $\rho_0$.}
\end{proof}

{We also need an elementary lower bound for the truncated fractional symbol
on a hyperplane. Here} $\Pi_\theta z$ {denotes the orthogonal
projection of} $z$ {onto} $\theta^\perp${.}
	\begin{lemma}\label{lem:truncated-hyperplane-symbol}
		{Let} \(s\in(0,1)\){,}
		\(\theta\in\mathbb{S}^{d-1}\){,} \(z\in\mathbb R^d\){,
		and} \(r>0\){. Then}
		\begin{align}\label{eq:truncated-hyperplane-symbol}
			\int_{\substack{h\in \theta^\perp\\ |h|\le r}}
			\bigl(1-\cos(z\cdot h)\bigr)\,|h|^{-d+1-2s}\,\dif \mathcal H^{d-1}(h)
			\gtrsim|\Pi_\theta z|^{2s}\min\{1,r|\Pi_\theta z|\}^{2-2s}.
		\end{align}
	        {The implicit constant depends only on} $d$ {and}
	        $s${.}
	\end{lemma}
	\begin{proof}
			{Since} \(z\cdot h=(\Pi_\theta z)\cdot h\)
			{for} \(h\in\theta^\perp\){, the left-hand side
			depends only on the tangential component} \(\Pi_\theta z\){.
			Set}
		$
		a:=|\Pi_\theta z|.
		$
			{Rotate coordinates so that}
		$
		\Pi_\theta z = a e_1.
		$
			{Writing} \(h=(h_1,h')\in
			\mathbb R\times\mathbb R^{d-2}\){, the left-hand side becomes}
		\[
		I(a,r)
		:=
		\int_{|h|\le r}
		\bigl(1-\cos(ah_1)\bigr)\,|h|^{-d+1-2s}\,\dif h.
		\]
			{Consider the cone}
		\[
		\Gamma:=\{h\in \mathbb R^{d-1}: |h'|\le |h_1|/2\}.
		\]
			{For} \(h\in\Gamma\){,}
		\[
		|h|^{-d+1-2s}\ge c\,|h_1|^{-d+1-2s},
		\]
			{with} \(c=\left(\frac{\sqrt{5}}{2}\right)^{-d+1-2s}\){.
			Moreover,} \(|h_1|\le r/\sqrt5\) {and}
			\(|h'|\le |h_1|/2\) {imply} \(|h|\le r\){.
			Therefore,}
		\begin{align*}
			I(a,r)
			&\ge
			\int_{|h_1|\le r/\sqrt5}\int_{|h'|\le |h_1|/2}
			\bigl(1-\cos(ah_1)\bigr)\,|h|^{-d+1-2s}\,\dif h\\
			&\gtrsim\int_0^{r/\sqrt5}
			\bigl(1-\cos(ah_1)\bigr)\,h_1^{-1-2s}\,\dif h_1.
		\end{align*}
			{It remains to prove}
		\begin{align}\label{eq:1d-symbol-lb}
			\int_0^\rho \bigl(1-\cos(at)\bigr)\,t^{-1-2s}\,\dif t
			\gtrsim a^{2s}\min\{1,a\rho\}^{2-2s},
			\qquad \rho:=\frac{r}{\sqrt5}.
		\end{align}
		
			{We distinguish two cases.}
		
		\smallskip
		\noindent
		{\it Case 1: \(a\rho\le 1\).}
			{For} \(0\le at\le1\){,}
		\[
		1-\cos(at)\gtrsim (at)^2.
		\]
			{Hence,}
		\[
		\int_0^\rho \bigl(1-\cos(at)\bigr)\,t^{-1-2s}\,\dif t
		\gtrsim a^2\int_0^\rho t^{1-2s}\,\dif t
	\sim{a^{2s}(a\rho)^{2-2s}}.
		\]
		
		\smallskip
		\noindent
		{\it Case 2: \(a\rho\ge 1\).}
			{The interval} \([1/(2a),1/a]\) {is then
			contained in} \((0,\rho]\){. On this interval,}
			\(at\in[1/2,1]\){, so}
		\[
		1-\cos(at)\gtrsim 1.
		\]
			{Therefore,}
		\[
		\int_0^\rho \bigl(1-\cos(at)\bigr)\,t^{-1-2s}\,\dif t
		\gtrsim\int_{1/(2a)}^{1/a} t^{-1-2s}\,\dif t
		\gtrsim a^{2s}.
		\]
		
			{Combining the two cases proves}
			\eqref{eq:1d-symbol-lb} {and hence}
			\eqref{eq:truncated-hyperplane-symbol}{.}
	\end{proof}\par\medskip

{To describe the anisotropy of the Fourier symbol} $E_g$
{from} \eqref{defEs}{, introduce the Fourier-space norm}
\begin{align}
\label{defan}
[[z]]_{v_0}
:=
\bigl(
|z|^2-\langle v_0\rangle^{-2}(v_0\cdot z)^2
\bigr)^{1/2}.
\end{align}
{Equivalently,}
\begin{equation}
[[z]]_{v_0}=|\mathcal O_{v_0}z|,
\end{equation}
{where} $\mathcal O_{v_0}$ {is defined in}
\eqref{defmatr}{. Combining the hyperplane representation with a
geometrically nondegenerate set extracted from} \eqref{eq:macro-main}
{gives the following anisotropic coercivity estimate.}
\begin{lemma}\label{lem:anisotropic-coercivity}
{Let} $s\in(0,1)${, and let} $E_g$ {be
defined by} \eqref{defEs}{. Assume that} $g=g(t,x,v)\ge0$
{and that} \eqref{eq:macro-main} {holds.}
{Assume also the angular lower bound
\eqref{assofb-lower}.} {Then there exists}
\begin{equation}
{c=c(d,s,\kappa,\mathsf m_0,\mathsf E_0,\mathsf H_0,c_b,\delta_0)>0},
\end{equation}
{such that, for every}
$(t,x_0,v_0)\in[0,T]\times\mathbb R^d\times\mathbb R^d$
{and} $z\in\mathbb R^d${,}
	\begin{equation}\label{eq:anisotropic-coercivity-final}
		E_g(z;t,x_0,v_0)
		\ge
		c\,\langle v_0\rangle^{-\kappa}\,
		[[z]]_{v_0}^{2s}\min\{1,[[z]]_{v_0}\}^{2-2s}.
	\end{equation}
\end{lemma}
\begin{remark}[{Low-frequency behavior and jump moments}]
Estimate \eqref{eq:anisotropic-coercivity-final} gives quadratic coercivity
for $[[z]]_{v_0}\leq1$. Under \eqref{Ig} with $\varkappa>d+2$,
the upper kernel bound \eqref{z26} also gives finite jump moments of every
integer order $2\leq m<\varkappa-d$ at each fixed $(t,x_0,v_0)$;
see \eqref{bd123}. Since $\mathbf C_g$ is even in $h$, dominated convergence yields
\[
 E_g(z;t,x_0,v_0)
 =\frac12\int_{\mathbb R^d}(z\cdot h)^2
 \mathbf C_g(t,x_0,v_0,h)\frac{\dif h}{|h|^{d+2s}}
 +o(|z|^2),\qquad z\to0.
\]
These moments give $C^m$ regularity at the indicated orders. The derivative
bounds \eqref{derE} provide the additional input for Fourier integration
by parts in the far-field estimates.
\end{remark}
\begin{proof}[Proof of Lemma \ref{lem:anisotropic-coercivity}]
{The proof adapts the mass--energy--entropy truncation argument of}
\cite[Lemma~7.1]{S}{, which extracts a nondegenerate velocity set.
That result does not directly imply} \eqref{eq:anisotropic-coercivity-final}
{because we need an anisotropic lower bound depending on the frozen
velocity} \(v_0\) {through} \([[z]]_{v_0}\){. We
therefore apply its geometric argument to the hyperplane representation of the
Boltzmann symbol.}

{Fix} $(t,x_0)$ {and suppress these variables. Write}
	$E_g(z,v_0)=E_g(z;t,x_0,v_0)$ {and set}
	\begin{equation*}
		\phi(r):=r^{2s}\min\{1,r\}^{2-2s},\qquad r\geq0.
	\end{equation*}
	{We use the elementary scaling property}
	\begin{equation}\label{eq:phi-scaling}
		\phi(\lambda r)
		\geq\phi(\lambda)\phi(r),
		\qquad \lambda,r\geq0.
	\end{equation}
{Following the standard mass--energy--entropy truncation argument
(see, for example,} \cite{S}{), we first extract a nondegenerate
velocity set. Let}
	\begin{equation*}
		R:=1+2\sqrt{\frac{\mathsf E_0}{\mathsf m_0}} .
	\end{equation*}
	{By} \eqref{eq:macro-main-a}--\eqref{eq:macro-main-b}{,}
	\begin{equation*}
		\int_{B_R}g(v)\,\dif v\geq \frac{3\mathsf m_0}{4}.
	\end{equation*}
	{Choose} $M>1$ {sufficiently large that}
	\begin{equation*}
		\frac{\mathsf H_0}{\log(1+M)}
		\leq
		\frac{\mathsf m_0}{4}.
	\end{equation*}
	{Then} \eqref{eq:macro-main-c} {implies}
	\begin{equation*}
		\int_{\{g>M\}}g(v)\,\dif v\leq \frac{\mathsf m_0}{4}.
	\end{equation*}
	{Hence,}
	\begin{equation*}
		\int_{B_R\cap\{g\leq M\}}g(v)\,\dif v
		\geq
		\frac{\mathsf m_0}{2}.
	\end{equation*}
	{Let}
	\begin{equation*}
		\ell:=\frac{\mathsf m_0}{4|B_R|}
		\qquad\text{and}\qquad
		F:=B_R\cap\{\ell\leq g\leq M\}.
	\end{equation*}
	{Then}
	\begin{equation}\label{eq:F-lower}
		|F|\geq \frac{\mathsf m_0}{4M},
		\qquad
		g\geq \ell\quad\text{a.e. on }F.
	\end{equation}
	{Lemmas~}\ref{lem:hyperplane-representation}
	{and} \ref{lem:truncated-hyperplane-symbol} {give}
	\begin{align}
		E_g(z,v_0)
		&\gtrsim
		\int_{\mathbb R^d}
		g(v_0+w)|w|^{-\kappa}
		|\Pi_{\hat w}z|^{2s}
		\min\{1,|w||\Pi_{\hat w}z|\}^{2-2s}\,\dif w
		\notag\\
		&\gtrsim
		\int_{\mathbb R^d}
		g(v_0+w)|w|^{-(\kappa+2s)}
		\phi\bigl(|w||\Pi_{\hat w}z|\bigr)\,\dif w .
		\label{eq:Eg-reduced-before-F}
	\end{align}
The implicit lower-bound constants in \eqref{eq:Eg-reduced-before-F}
contain the factor $c_0\rho_0^{1-s}$. Indeed, the hyperplane estimate
contributes $c_0$, and
\[
 \min\{1,\sqrt{\rho_0}a\}^{2-2s}
 \geq \rho_0^{1-s}\min\{1,a\}^{2-2s},\qquad a\geq0.
\]
All subsequent lower bounds in this proof retain this factor. With the
other parameters fixed, the final coercivity constant can be chosen
proportional to $c_b$.

	{Changing variables by} $u=v_0+w$ {and using}
	\eqref{eq:F-lower} {gives}
	\begin{equation}\label{eq:Eg-reduced-F}
		E_g(z,v_0)
		\gtrsim
		\int_F
		|u-v_0|^{-(\kappa+2s)}
		\phi\bigl(|u-v_0||{\Pi_{\widehat{u-v_0}}}z|\bigr)\,\dif u .
	\end{equation}
{Here $\widehat{u-v_0}=(u-v_0)/|u-v_0|$ for $u\neq v_0$,
with $\widehat0=0$ and $\Pi_{\widehat0}=\mathrm{Id}$.
The integrand may be assigned any value at $u=v_0$.}
	
	{We next fix a small geometric constant. For a line}
	$\Lambda\subset\mathbb R^d${, define}
	\begin{equation*}
		N_r(\Lambda):=
		\{u\in\mathbb R^d:\operatorname{dist}(u,\Lambda)<r\}.
	\end{equation*}
	{Since}
	\begin{equation*}
		|N_r(\Lambda)\cap B_R|
		\leq C_d R r^{d-1},
		\qquad 0<r\leq R,
	\end{equation*}
	{we may choose} $r_*\in(0,1]${, depending only on}
	$d,R,\mathsf m_0,M${, such that}
	\begin{equation}\label{eq:rstar-choice-new}
		|N_{r_*}(\Lambda)\cap B_R|
		\leq
		\frac{\mathsf m_0}{16M}
		\qquad
		\text{for every line }\Lambda\subset\mathbb R^d .
	\end{equation}
{First suppose} $|v_0|\leq4R${. The estimate is immediate
when} $z=0${; otherwise, set}
	\begin{equation*}
		\Lambda:=v_0+\mathbb R z,
		\qquad
		F_*:=F\setminus N_{r_*}(\Lambda).
	\end{equation*}
	{By} \eqref{eq:F-lower} {and}
	\eqref{eq:rstar-choice-new}{,}
	\begin{equation}\label{eq:Fstar-lower-small-v}
		|F_*|\geq \frac{\mathsf m_0}{8M}.
	\end{equation}
	{For} $u\in F_*${,}
	\begin{equation*}
		r_*\leq |u-v_0|\leq 5R,
		\qquad
		|u-v_0|\,|{\Pi_{\widehat{u-v_0}}}z|
		=
		|z|\operatorname{dist}(u,\Lambda)
		\geq r_*|z|.
	\end{equation*}
	{Using} \eqref{eq:Eg-reduced-F}{,}
	\eqref{eq:phi-scaling}{, and}
	\eqref{eq:Fstar-lower-small-v} {gives}
	\begin{equation*}
		E_g(z,v_0)
		\geq
		c\,\phi(|z|).
	\end{equation*}
	{Since} $|v_0|\leq4R${, we have}
	$[[z]]_{v_0}\leq|z|$ {and}
	$\langle v_0\rangle^{-\kappa}\sim1${, with constants depending
	only on} $R$ {and} $\kappa${. Hence,}
	\begin{equation}\label{eq:small-v-final}
		E_g(z,v_0)
		\gtrsim\langle v_0\rangle^{-\kappa}\phi([[z]]_{v_0}).
	\end{equation}
{It remains to consider} $|v_0|>4R${. Set}
	\begin{equation*}
		e:=\frac{v_0}{|v_0|},
		\qquad
		q:=|\Pi_e z|.
	\end{equation*}
	{Assume again that} $z\neq0$ {and define}
	\begin{equation*}
		D:=\operatorname{dist}(0,v_0+\mathbb R z).
	\end{equation*}
	{Then}
	\begin{equation}\label{eq:D-formula-new}
		D=\frac{|v_0|q}{|z|}.
	\end{equation}
	{First suppose} $D\geq2R${. For every}
	$u\in F\subset B_R${,}
	\begin{equation*}
		\operatorname{dist}(u,v_0+\mathbb R z)
		\geq D-R
		\geq \frac{D}{2}.
	\end{equation*}
	{Hence,}
	\begin{equation*}
		|u-v_0||{\Pi_{\widehat{u-v_0}}}z|
		=
		|z|\operatorname{dist}(u,v_0+\mathbb R z)
		\geq
		\frac{|v_0|q}{2}.
	\end{equation*}
	{Moreover, since} $|u|\leq R$ {and}
	$|v_0|>4R${,}
	\begin{equation}\label{eq:w-comparable-v0}
		|u-v_0|\sim \langle v_0\rangle .
	\end{equation}
	{Using} \eqref{eq:Eg-reduced-F}{,}
	\eqref{eq:F-lower}{, and} \eqref{eq:w-comparable-v0}
{gives}
	\begin{equation*}
		E_g(z,v_0)
		\gtrsim\langle v_0\rangle^{-(\kappa+2s)}
		\phi(c|v_0|q).
	\end{equation*}
	{By} \eqref{eq:phi-scaling} {and}
	$|v_0|\sim\langle v_0\rangle${,}
	\begin{equation}\label{eq:q-lower-large-D}
		E_g(z,v_0)
		\gtrsim\langle v_0\rangle^{-\kappa}\phi(q).
	\end{equation}
	{On the other hand,} \eqref{eq:D-formula-new}
	{and} $D\geq2R$ {imply}
$
		\frac{|z|}{|v_0|}
		\leq
		C_R q.
$
	{Therefore,}
	\begin{align*}
		[[z]]_{v_0}^2
		&=
		|\Pi_e z|^2+\frac{(e\cdot z)^2}{1+|v_0|^2}
\leq
		q^2+\frac{|z|^2}{|v_0|^2}
		\leq
		C_R q^2 .
	\end{align*}
	{Combining this estimate with} \eqref{eq:phi-scaling}
	{and} \eqref{eq:q-lower-large-D} {gives}
	\begin{equation}\label{eq:large-D-final}
		E_g(z,v_0)
		\gtrsim\langle v_0\rangle^{-\kappa}\phi([[z]]_{v_0}).
	\end{equation}
	{Now suppose} $D<2R${, and let}
	\begin{equation*}
		\Lambda:=v_0+\mathbb R z,
		\qquad
		F_*:=F\setminus N_{r_*}(\Lambda).
	\end{equation*}
	{Then} \eqref{eq:F-lower} {and}
	\eqref{eq:rstar-choice-new} {imply}
	\begin{equation*}
		|F_*|\geq \frac{\mathsf m_0}{8M}.
	\end{equation*}
	{For} $u\in F_*${,}
	\begin{equation*}
		|u-v_0||{\Pi_{\widehat{u-v_0}}}z|
		=
		|z|\operatorname{dist}(u,\Lambda)
		\geq
		r_*|z|,
	\end{equation*}
	{and again} $|u-v_0|\sim\langle v_0\rangle${.
	Thus,} \eqref{eq:Eg-reduced-F} {gives}
	\begin{equation*}
		E_g(z,v_0)
		\gtrsim\langle v_0\rangle^{-(\kappa+2s)}
		\phi(c|z|).
	\end{equation*}
	{Apply} \eqref{eq:phi-scaling} {with
$r=|z|/\langle v_0\rangle$ and $\lambda=c\langle v_0\rangle$.
If $\lambda<1$, then $\lambda\geq c$ and its scaling factor is bounded
below by a positive constant depending on $c$; if $\lambda\geq1$,
use its $\lambda^{2s}$ lower bound. Thus}
	\begin{equation}\label{eq:small-D-scale}
		\langle v_0\rangle^{-2s}\phi(c|z|)
		\gtrsim
		\phi\left(\frac{|z|}{\langle v_0\rangle}\right).
	\end{equation}
	{Moreover,} \eqref{eq:D-formula-new} {and}
	$D<2R$ {imply}
	$
		q\leq C_R\frac{|z|}{|v_0|}.
	$
	{Consequently,}
	\begin{align*}
		[[z]]_{v_0}^2
		&=
		q^2+\frac{(e\cdot z)^2}{1+|v_0|^2}
	\leq
		C_R\frac{|z|^2}{\langle v_0\rangle^2}.
	\end{align*}
	{By monotonicity of $\phi$ and its doubling bound}
$\phi(C_R r)\lesssim_{C_R}\phi(r)$, \eqref{eq:small-D-scale} {yields}
	\begin{equation}\label{eq:small-D-final}
		E_g(z,v_0)
		\gtrsim\langle v_0\rangle^{-\kappa}\phi([[z]]_{v_0}).
	\end{equation}
	{Combining} \eqref{eq:small-v-final}{,}
	\eqref{eq:large-D-final}{, and} \eqref{eq:small-D-final}
{gives}
	\begin{equation*}
		E_g(z,v_0)
		\gtrsim\langle v_0\rangle^{-\kappa}
		\phi([[z]]_{v_0}),
	\end{equation*}
	{which is} \eqref{eq:anisotropic-coercivity-final}{.
	The constants are uniform in} $(t,x_0)${.}
\end{proof}

\begin{remark}
{When} \(g=\mu\){, the Maxwellian structure substantially
shortens the proofs of Lemmas~}\ref{leA} {and}
\ref{lem:anisotropic-coercivity}{.}
\end{remark}
\begin{remark}[Coercivity is generated by small jumps]
\label{rem:small-jump-coercivity}
{Assume the hypotheses of
Lemma~\ref{lem:anisotropic-coercivity}, including
\eqref{assofb-lower}.} {For every} \(0<\delta\leq1\){,}
\begin{equation*}
\int_{\{[h]_{v_0}\leq\delta\}}
(1-\cos(z\cdot h))
\mathbf C_g(t,x_0,v_0,h)
\frac{\dif h}{|h|^{d+2s}}
\geq
c\langle v_0\rangle^{-\kappa}
[[z]]_{v_0}^{2s}
\min\{1,\delta[[z]]_{v_0}\}^{2-2s},
\end{equation*}
{where the constant depends only on the parameters listed in
Lemma~\ref{lem:anisotropic-coercivity} and is independent of $\delta$.}

Indeed, on the nondegenerate region constructed in the proof of
Lemma~\ref{lem:anisotropic-coercivity}, the relevant collision velocity $u$
is uniformly bounded and $h\perp(u-v_0)$. Hence
$[h]_{v_0}=|h|+|h\cdot u|\lesssim|h|$, so the truncated Carleman region
still contains a hyperplane ball of radius comparable to $\delta$.
Lemma~\ref{lem:truncated-hyperplane-symbol} and the same geometric argument
then give the displayed estimate.

Taking $\delta=1$ and using
$\{[h]_{v_0}\leq1\}\subset\{|h|\leq1\}$ yields
\[
\int_{|h|\leq1}
(1-\cos(z\cdot h))
\mathbf C_g(t,x_0,v_0,h)
\frac{\dif h}{|h|^{d+2s}}
\geq c\langle v_0\rangle^{-\kappa}
[[z]]_{v_0}^{2s}\min\{1,[[z]]_{v_0}\}^{2-2s}.
\]
Thus, large jumps are unnecessary for coercivity.
\end{remark}
	\subsubsection{\texorpdfstring{{Upper bound for $E_g$}}{Upper bound for Eg}}
	    {Upper bounds require stronger assumptions than the lower bounds
	    above. When} $\gamma+2s\geq0${, the collision singularity is
	    mild enough for the macroscopic bounds} \eqref{eq:macro-main}
	    {to control the coefficients pointwise. For example, in the
	    Landau case,}
	    \[ \big(\mathsf A[g](v)e,e\big)_{\mathbb{R}^d} \lesssim\langle v\rangle^\gamma \big(\langle v\rangle^2-|v\cdot e|^2\big), \qquad e\in\mathbb S^{d-1}, \]
	    {where} $\mathsf{A}[g](v)=(a\star g)(v)${. Here,
	    however, we work in the very-soft-potential range}
	    \[ -d<\gamma+2s\leq 0 . \]
	    {The coefficient kernels are then more singular, and macroscopic
	    bounds alone do not provide the pointwise control needed for the frozen
	    analysis. We therefore assume in addition that}
	    $\|g\|_{L^\infty_T\mathbf{C}^0}<\infty${. Under this
	    assumption, Lemma~}\ref{lemcoeff} {gives the Landau upper
	    bound, and Lemma~}\ref{lemCf} {gives its Boltzmann analogue.}
    
	{We now establish pointwise upper bounds for} $E_g$
	{under} \eqref{Ig}{, beginning with the Landau case:}
	$$
		E_g(z;t,x_0,v_0)
	=
		\big(\mathsf{A}[g](t,x_0,v_0)\,z,\,z\big)_{\mathbb{R}^d},\quad\quad \text{and}\ \ \mathsf{A}[g](t,x,v)=(a\star g(t,x,\cdot))(v).
$$

		\begin{lemma}\label{lemcoeff}
		 {Assume that}
		 $\|g\|_{L^\infty_T\mathbf{C}^0}<\infty$ {and}
		 $\varkappa>d+2${. Then}
			
            $$\left(\mathsf{A}[g](v) e, e\right)_{\mathbb{R}^d} \lesssim \|g\|_{L^\infty_T\mathbf{C}^0} \jap{v}^\gamma (\langle v\rangle^2-|v\cdot e|^2)~~\forall e\in \mathbb{S}^{d-1}.$$
		\end{lemma}
	\begin{proof}
	{It suffices to consider} \(|v|>2\){, because the
	bounded-velocity case is immediate.}
	
	{We first record the elementary bounds}
    \begin{align*}
    &(a(w){\hat v},{\hat v})_{\mathbb{R}^d}=|w|^\gamma |\Pi_{\hat v} w|^2,\\&
    |(a(w){\hat v},\eta)_{\mathbb{R}^d}|
	\lesssim |w|^{\gamma+1}|\Pi_{\hat v} w|\,|\eta|,
	\qquad \eta\perp {\hat v},
    \end{align*}
	and
	\[
	|(a(w)\eta_1,\eta_2)_{\mathbb{R}^d}|
	\lesssim |w|^{\gamma+2}|\eta_1|\,|\eta_2|,
	\qquad \eta_1,\eta_2\perp {\hat v}.
	\]
	{Using}
	\[
	|g(t,x,v-w)|
	\le
	\|g\|_{L^\infty_T\mathbf C^0}
	\langle v-w\rangle^{-\varkappa},
	\]
	{it remains to prove, for} \(m=0,1,2\){,}
	\[
	J_m(v):=
	\int_{\mathbb R^d}
	|w|^{\gamma+m}|\Pi_{\hat v} w|^{2-m}
	\langle v-w\rangle^{-\varkappa}\,dw
	\lesssim
	\langle v\rangle^{\gamma+m}.
	\]
	{Set} \(u=v-w\){. Since}
	\(\Pi_{\hat v}v=0\){,}
	\[
	J_m(v)=
	\int_{\mathbb R^d}
	|v-u|^{\gamma+m}|\Pi_{\hat v} u|^{2-m}
	\langle u\rangle^{-\varkappa}\,du .
	\]
	{Decompose the integral over}
	\[
	\{|u|\le |v|/2\},\qquad
	\{|v-u|\le |v|/2\},
	\qquad
	\{|u|>|v|/2,\ |v-u|>|v|/2\}.
	\]
	{On the first region,} \(|v-u|\sim|v|\){, so}
	\[
	J_m^{(1)}
	\lesssim
	|v|^{\gamma+m}
	\int_{\mathbb R^d}|u|^{2-m}\langle u\rangle^{-\varkappa}\,du
	\lesssim
	|v|^{\gamma+m}.
	\]
	{On the second region,}
	\(\langle u\rangle\sim\langle v\rangle\) {and}
	\(|\Pi_{\hat v}u|\le|v-u|\){. Hence,}
	\[
	J_m^{(2)}
	\lesssim
	\langle v\rangle^{-\varkappa}
	\int_{|v-u|\le |v|/2}
	|v-u|^{\gamma+2}\,du
	\lesssim
	\langle v\rangle^{d+\gamma+2-\varkappa}
	\lesssim
	\langle v\rangle^{\gamma+m}.
	\]
	{On the remaining region, both} \(|u|\) {and}
	\(|v-u|\) {are bounded below by a constant multiple of}
	\(|v|\){. Since} \(\varkappa>d+2\){,}
	\[
	J_m^{(3)}
	\lesssim
	\langle v\rangle^{\gamma+m}.
	\]
	{Therefore,} \(J_m(v)\lesssim\langle v\rangle^{\gamma+m}\)
	{for} \(m=0,1,2\){.}
	
	{Now, for} \(e\in\mathbb S^{d-1}\){, write}
	\[
	e=e_\parallel+e_\perp,
	\qquad
	e_\parallel=(e\cdot {\hat v}){\hat v},
	\qquad
	e_\perp=\Pi_{\hat v} e.
	\]
    {Then}
    \begin{align*}
    &|(\mathsf A[g]e_\parallel,e_\parallel)_{\mathbb{R}^d}|
	\lesssim
	\|g\|_{L^\infty_T\mathbf C^0}
	\langle v\rangle^\gamma |e_\parallel|^2,\\&
    |(\mathsf A[g]e_\parallel,e_\perp)_{\mathbb{R}^d}|
	\lesssim
	\|g\|_{L^\infty_T\mathbf C^0}
	\langle v\rangle^{\gamma+1}|e_\parallel|\,|e_\perp|,
    \end{align*}
	and
	\[
	|(\mathsf A[g]e_\perp,e_\perp)_{\mathbb{R}^d}|
	\lesssim
	\|g\|_{L^\infty_T\mathbf C^0}
	\langle v\rangle^{\gamma+2}|e_\perp|^2.
	\]
	{Therefore,}
	\[
	\begin{aligned}
		|(\mathsf A[g]e,e)_{\mathbb{R}^d}|
		&\lesssim
		\|g\|_{L^\infty_T\mathbf C^0}
		\langle v\rangle^\gamma
		\left(
		|e_\parallel|^2
		+\langle v\rangle |e_\parallel|\,|e_\perp|
		+\langle v\rangle^2|e_\perp|^2
		\right)  \\
		&\lesssim
		\|g\|_{L^\infty_T\mathbf C^0}
		\langle v\rangle^\gamma
		\left(
		|e_\parallel|^2+\langle v\rangle^2|e_\perp|^2
		\right).
	\end{aligned}
	\]
	{Finally, since} \(|e|=1\){,}
	\[
	|e_\parallel|^2+\langle v\rangle^2|e_\perp|^2
	=
	1+|v|^2|e_\perp|^2
	=
	\langle v\rangle^2-|v\cdot e|^2.
	\]
	{This proves the claim.}
\end{proof}\par\medskip
			{For} \(a>d\) {and}
			\(v,z\in\mathbb R^d\){, define}
		\begin{align}\label{defom}
			\Omega_a(v,z)
			&:=\mathbf{1}_{\langle \Pi_{\hat z}v\rangle \geq |z|/100}
			\langle\Pi_{\hat z}v\rangle^{1-\kappa}
			\langle\mathrm P_{\hat z}v\rangle^{d-1-a}
			+\langle v,z\rangle^{d-\kappa-a},
		\end{align}
			{where} $\mathrm{P}_{\hat z}$ {and}
			$\Pi_{\hat z}$ {are defined in} \eqref{defproj}{.
			The following lemma quantifies the anisotropic decay of}
			$\mathbf{C}_g${.}
		
		\begin{lemma} \label{lemCf}
				{In the Boltzmann case, assume that} \(g\geq0\)
				{and} \(\|g\|_{L^\infty_T\mathbf C^{b_0}}<\infty\)
				{for some} \(b_0\in(0,1)\){.}
			{Then, for every} $v,z\in\mathbb{R}^d$
			{and} $t\in(0,T)${,}
			\begin{align}\label{z26}
				&
\sup_{t\in[0,T]}\sup_{x\in\R^d}	\mathbf{C}_{g}(t,x,v,z)  \lesssim {{\Omega_{\varkappa}}(v,z)} 
	\|g\|_{L^\infty_T\mathbf{C}^0},\\
				& 
\sup_{t\in[0,T]}\sup_{x\in\R^d}	\sup_{|(y,u)|\leq 1}	\frac{|\delta_{(y,u)}^{(x,v)}		\mathbf{C}_{g}(t,x,v,z)|}{|(y,u)|^{b_0}}\lesssim	{\Omega_{\varkappa}}(v,z) \|g\|_{L^\infty_T\mathbf{C}^{b_0}}. \label{z27}
			\end{align}
		\end{lemma}
		\begin{proof}
			{The definition of} $\mathbf{C}_g$ {in}
			\eqref{defCf1} {and estimate} \eqref{esofA}
			{give}
\begin{align*}			
	&\mathbf{C}_g(t,x,v,z) \lesssim \int_{z^\perp} g(t,x,v+w) |w|^{-\kappa+1}\mathbf{1}_{|w|\geq |z|} \dif \mathcal H^{d-1}(w),\\
	&|\delta_{(y,u)}^{(x,v)}		\mathbf{C}_{g}(t,x,v,z)|\lesssim \int_{z^\perp} |\delta_{(y,u)}^{(x,v)}	g(t,x,v+w)| |w|^{-\kappa+1}\mathbf{1}_{|w|\geq |z|} \dif \mathcal H^{d-1}(w).
	\end{align*}
	{It then follows from} \eqref{Ig} {that}
	\begin{align*}
			\frac{\mathbf{C}_g(t,x,v,z) }{\|g\|_{L^\infty_T\mathbf{C}^{0}}}+\sup_{|(y,u)|\leq 1}	\frac{|\delta_{(y,u)}^{(x,v)}		\mathbf{C}_{g}(t,x,v,z)|}{|(y,u)|^{b_0}\|g\|_{L^\infty_T\mathbf{C}^{b_0}}}&\lesssim \int_{z^\perp} \langle v+w\rangle^{-\varkappa} |w|^{-\kappa+1}\mathbf{1}_{|w|\geq |z|} \dif \mathcal H^{d-1}(w)\\
			&\lesssim  \int_{z^\perp} \langle |\mathrm{P}_{\hat z}v|+|\Pi_{\hat z}v+w|\rangle^{-\varkappa} |w|^{-\kappa+1}\mathbf{1}_{|w|\geq |z|}\dif \mathcal H^{d-1}(w).
	\end{align*}
	{It remains to bound the last integral. We distinguish two cases according
	to} \(\langle \Pi_{\hat z}v\rangle\){. First suppose}
	\[
	\langle \Pi_{\hat z}v\rangle<\frac{|z|}{100}.
	\]
	{Then, on} \(\{|w|\ge|z|\}\){,}
	\[
	|\Pi_{\hat z}v+w|\sim |w|.
	\]
	{Therefore, polar coordinates on} \(z^\perp\) {give}
	\[
	\begin{aligned}
		&\int_{z^\perp}
		\left\langle |\mathrm P_{\hat z}v|+|\Pi_{\hat z}v+w|\right\rangle^{-\varkappa}
		|w|^{-\kappa+1}\mathbf 1_{\{|w|\ge |z|\}}\,\dif \mathcal H^{d-1}(w)  \\
		&\qquad\lesssim
		\int_{|z|}^{\infty}
		\langle |\mathrm P_{\hat z}v|+\rho\rangle^{-\varkappa}
		\rho^{d-\kappa-1}\dif\rho
		\lesssim
		\langle v,z\rangle^{d-\kappa-\varkappa}.
	\end{aligned}
	\]
	{Now suppose}
	\[
	\langle \Pi_{\hat z}v\rangle\ge \frac{|z|}{100}.
	\]
	{Decompose the integral over}
	\[
	|\Pi_{\hat z}v+w|\le \frac12\langle \Pi_{\hat z}v\rangle
	\qquad\text{and}\qquad
	|\Pi_{\hat z}v+w|> \frac12\langle \Pi_{\hat z}v\rangle .
	\]
	{On the first region, local integrability of}
	\(|w|^{-\kappa+1}\) {on} \(z^\perp\) {and rapid decay
	in} \(|\Pi_{\hat z}v+w|\) {give}
	\[
	\lesssim
	\langle \Pi_{\hat z}v\rangle^{-\kappa+1}
	\langle \mathrm P_{\hat z}v\rangle^{d-1-\varkappa}.
	\]
	{The same polar-coordinate estimate gives, on the complementary region,}
	\[
	\lesssim
	\langle v,z\rangle^{d-\kappa-\varkappa}.
	\]
	{Combining the two cases yields}
	\[
	\int_{z^\perp}
	\left\langle |\mathrm P_{\hat z}v|+|\Pi_{\hat z}v+w|\right\rangle^{-\varkappa}
	|w|^{-\kappa+1}\mathbf 1_{\{|w|\ge |z|\}}\,\dif \mathcal H^{d-1}(w)
	\lesssim
	{\Omega_{\varkappa}}(v,z).
	\]
    {This proves the lemma.}
		\end{proof}\medskip
        
{Estimate} \eqref{z26} {implies the following upper bounds
for} \(E_g\) {and its derivatives.}
	\begin{lemma}
\label{deofE}{Let} $s\in(0,1)$ {and $\varkappa>d+2$}{, and assume that}
\(g\) {satisfies} \eqref{Ig}{. Then, for every}
\((t,x_0,v_0)\in[0,T]\times\mathbb R^d\times\mathbb R^d\)
{and} \(z\in\mathbb R^d\){,}
\begin{equation}\label{equpperE}
		E_g(z;t,x_0,v_0)
		\lesssim 
		\langle v_0\rangle^{-\kappa}
		[[z]]_{v_0}^{2s}\min\{1,\langle v_0\rangle[[z]]_{v_0}\}^{2-2s}\|g\|_{L^\infty_T\mathbf{C}^0}.
	\end{equation}
{Moreover, for} \(m\in\mathbb N_+\) {with}
\(m\leq\varkappa-d\){,}
					\begin{align}
						&|(\mathcal{O}_{v_0}^{-1} \nabla_z)^{\otimes m} E_g(z;t,x_0,v_0)|\lesssim  \langle v_0\rangle^{-\kappa+m} \left(([[z]]_{v_0}^{(2s-1)_+}+\mathbf{1}_{s=\frac{1}{2}}\log ([[z]]_{v_0}+1) )\mathbf{1}_{m=1}+\mathbf{1}_{m\geq2}\right)\|g\|_{L^\infty_T\mathbf{C}^0},\label{derE}
					\end{align}
					{where} $\mathcal{O}_{v_0}$ {and}
					$[[z]]_{v_0}$ {are defined in} \eqref{defmatr}
					{and} \eqref{defan}{, respectively.}
\end{lemma}
\begin{proof}
	{Recall the definition of} $E_g$ {in}
	\eqref{defEs}{. To prove} \eqref{equpperE}{, use}
	\eqref{z26} {to obtain}
 \begin{equation}\label{Eup}
        \begin{aligned}
    |E_g(z;t,x_0,v_0)|&\lesssim \int_{\mathbb{R}^d} (1-\cos(z\cdot h)) \mathbf{C}_g(t,x_0,v_0,h)\frac{\dif h}{|h|^{d+2s}}\\
    &\lesssim \|g\|_{L^\infty_T\mathbf{C}^0}\int_{\mathbb{R}^d} \min\{1,([[z]]_{v_0}[h]_{v_0})^2\}{\Omega_{\varkappa}}(v_0,h)\frac{\dif h}{|h|^{d+2s}},
    \end{aligned}
 \end{equation}
	    {where we used}
	    $|1-\cos(z\cdot h)|\lesssim\min\{1,|z\cdot h|^2\}$
	    {and}
	    $|z\cdot h|\leq |\mathcal O_{v_0}z|\,
	    |\mathcal O_{v_0}^{-1}h|\leq [[z]]_{v_0}[h]_{v_0}.$ For
	    \(2\leq m\leq\varkappa-d\){, differentiation under the
	    integral and} \eqref{z26} {give}
	\[
	\bigl|
	(\mathcal{O}_{v_0}^{-1} \nabla_z)^{\otimes m}  E_g(z;t,x_0,v_0)
	\bigr|
	\lesssim \|g\|_{L^\infty_T\mathbf{C}^0}
	\int_{\mathbb R^d}[h]_{v_0}^m\,{\Omega_{\varkappa}}(v_0,h)\frac{\dif h}{|h|^{d+2s}}.
	\]
	{It remains to treat} \(m=1\){. The elementary inequality}
    	\begin{align*}
						|\partial_b (1-\cos(ab))|\lesssim |a| \min\{1,|ab|\},\ \forall a,b\in\mathbb{R},
					\end{align*}
	                    {gives}
                    \begin{align*}
	\bigl|
	\mathcal{O}_{v_0}^{-1} \nabla_z E_g(z;t,x_0,v_0)
	\bigr|
	&\lesssim \|g\|_{L^\infty_T\mathbf{C}^0}
	\int_{\mathbb R^d}[h]_{v_0} \min\{1,[[z]]_{v_0} [h]_{v_0}\}\,{\Omega_{\varkappa}}(v_0,h)\frac{\dif h}{|h|^{d+2s}}.
	\end{align*}
	    {We claim that, for every} $\lambda>0${,}
    \begin{equation}\label{bd123}
    \begin{aligned}
   & I_1:=\int_{\R^d}\min\{1,(\lambda [h]_{v_0})^2\} {\Omega_{\varkappa}}(v_0,h) \frac{\dif h}{|h|^{d+2s}}\lesssim \langle v_0\rangle^{-\kappa}\lambda^{2s}\min\{1,\lambda\langle v_0\rangle\}^{2-2s},\\
    &I_2:=\int_{\R^d}[h]_{v_0}\min\{1,\lambda [h]_{v_0}\} {\Omega_{\varkappa}}(v_0,h) \frac{\dif h}{|h|^{d+2s}}\lesssim \langle v_0\rangle^{-\kappa+1}(\lambda^{(2s-1)_+}+\mathbf{1}_{s=\frac{1}{2}}\log(\lambda+1)),\\
    &I_3:=\int_{\R^d} [h]_{v_0}^m{\Omega_{\varkappa}}(v_0,h) \frac{\dif h}{|h|^{d+2s}}\lesssim \langle v_0\rangle^{-\kappa+m},\ \ m\geq 2.
    \end{aligned}
    \end{equation}
	    {These estimates imply} \eqref{equpperE} {and}
	    \eqref{derE} {upon taking} \(\lambda=[[z]]_{v_0}\){.
	    Recall from} \eqref{defom} {that}
	    ${\Omega_{\varkappa}}(v,h)=\Omega_1(v,h)+\Omega_2(v,h)${, where}
    \begin{align*}
    \Omega_1(v,h)=\mathbf{1}_{\langle \Pi_{\hat h}v\rangle \geq \frac{|h|}{100}}\langle  \Pi_{\hat h} v\rangle^{-\kappa+1}\langle\mathrm{P}_{\hat h} v\rangle^{d-1-\varkappa},\quad\quad\    \Omega_2(v,h)=\langle v,h\rangle^{d-\kappa-\varkappa}.
\end{align*}
{For} $i=1,2,3${, write}
$I_i=I_{i,1}+I_{i,2}${, where the two terms correspond to}
$\Omega_1$ {and} $\Omega_2${.}
Writing $h=\rho\theta$ and using $1+|a|\asymp\langle a\rangle$, for each of
the three functions
\[
\varphi(q)\in\bigl\{\min\{1,\lambda^2q^2\},
q\min\{1,\lambda q\},q^m\bigr\},
\]
polar coordinates give
\begin{align*}
\int_{\mathbb{R}^d} \varphi([h]_{v_0})&\Omega_1(v_0,h)\frac{\dif h}{|h|^{d+2s}} \\
&\lesssim \int_{\mathbb{S}^{d-1}}
\left(\int_0^{100\langle\Pi_\theta v_0\rangle}
\varphi\bigl(\langle v_0\cdot\theta\rangle\rho\bigr)
\frac{\dif\rho}{\rho^{1+2s}}\right)
\langle\Pi_\theta v_0\rangle^{1-\kappa}
\langle v_0\cdot\theta\rangle^{d-1-\varkappa}\dif\theta.
\end{align*}
{Using the one-dimensional estimates}
\begin{align*}
&\int_0^{a}\min\{1,\lambda^2\rho^2\} \frac{\dif \rho}{\rho^{1+2s}} \lesssim \lambda^{2s}\min\{1,\lambda a\}^{2-2s},\quad\quad\quad \forall a, \lambda>0,\\
&\int_0^{a}\rho\min\{1,\lambda\rho\} \frac{\dif \rho}{\rho^{1+2s}} \lesssim \beta(a,\lambda):=\mathbf{1}_{s<\frac{1}{2}} a^{1-2s}\min\{1,a\lambda\}+\mathbf{1}_{s\geq\frac{1}{2}} \lambda^{2s-1}\min\{1,a\lambda\}^{2-2s}+ \mathbf{1}_{s=\frac{1}{2}}\mathbf{1}_{a>\lambda^{-1}} \log(a\lambda),\\
&\int_0^{a}\rho^m\frac{\dif \rho}{\rho^{1+2s}} \lesssim a^{m-2s},
\end{align*}
{gives}
\begin{equation}
\label{i123}
\begin{aligned}
I_{1,1}&\lesssim \int_{\mathbb{S}^{d-1}} (\lambda \langle v_0\cdot \theta\rangle)^{2s}\min\{1,\lambda \langle v_0\cdot \theta\rangle\langle\Pi_\theta v_0\rangle\}^{2-2s} \langle \Pi_\theta v_0\rangle^{-\kappa+1}\langle v_0\cdot \theta\rangle^{d-1-\varkappa} \dif \theta\\
&\lesssim \int_{\mathbb{S}^{d-1}} \lambda ^{2s}\min\{1,\lambda \langle\Pi_\theta v_0\rangle\}^{2-2s} \langle \Pi_\theta v_0\rangle^{-\kappa+1}\langle v_0\cdot \theta\rangle^{d+1-\varkappa} \dif \theta,\\
I_{2,1}&\lesssim \int_{\mathbb{S}^{d-1}} \langle v_0\cdot \theta\rangle \beta(\langle\Pi_\theta v_0\rangle,\lambda \langle v_0\cdot \theta\rangle)\langle \Pi_\theta v_0\rangle^{-\kappa+1}\langle v_0\cdot \theta\rangle^{d-1-\varkappa} \dif \theta\\
&\lesssim  \int_{\mathbb{S}^{d-1}}\Bigl(\mathbf 1_{s<\frac12}\langle\Pi_\theta v_0\rangle^{1-2s}
+\mathbf 1_{s\geq\frac12}\lambda^{2s-1}
+\mathbf{1}_{s=\frac{1}{2}}\log(\langle\Pi_\theta v_0\rangle \lambda+1)\Bigr)
\langle \Pi_\theta v_0\rangle^{-\kappa+1}\langle v_0\cdot \theta\rangle^{d+1-\varkappa} \dif \theta\\
I_{3,1}&\lesssim \int_{\mathbb{S}^{d-1}} \langle \Pi_\theta v_0\rangle^{m-2s-\kappa+1}\langle v_0\cdot \theta\rangle^{d-1-\varkappa+m} \dif \theta.
\end{aligned}
\end{equation}
We use the elementary angular estimate
\begin{equation*}
\int_{\mathbb S^{d-1}}
\langle\Pi_\theta v_0\rangle^\alpha
\langle v_0\cdot\theta\rangle^{-M}\,\dif\theta
\lesssim_{d,\alpha,M}
\langle v_0\rangle^{\alpha-1}
\bigl(1+\mathbf 1_{M=1}\log\langle v_0\rangle\bigr),
\qquad \alpha>-(d-1),\quad M\geq1.
\end{equation*}
This follows by aligning the polar axis with $v_0$; only the endpoint
$M=1$ produces the logarithm.

For $I_{1,1}$ and $I_{2,1}$, take $M=\varkappa-d-1>1$; all relevant
values of $\alpha$ are admissible because $\kappa<d$. At $s=1/2$, the
additional logarithm is absorbed using
$\log(1+AB)\leq\log(1+A)+\log(1+B)$,
$B\log B\lesssim B^2$, and
\[
\langle v_0\rangle^{-1}
\log(1+\lambda\langle v_0\rangle)
\lesssim1+\log(1+\lambda).
\]
For $I_{3,1}$, take
$\alpha=m-2s-\kappa+1>-(d-1)$ and
$M=\varkappa-d+1-m\geq1$. The endpoint logarithm is absorbed by
$\log\langle v_0\rangle\lesssim_s\langle v_0\rangle^{2s}$.
Consequently, \eqref{i123} yields
\begin{align*}
&I_{1,1}\lesssim \langle v_0\rangle^{-\kappa} \lambda^{2s} \min\{1,\lambda\langle v_0\rangle\}^{2-2s},\\
&I_{2,1}\lesssim \langle v_0\rangle^{-\kappa+1}(\lambda^{(2s-1)_+}+\mathbf{1}_{s=\frac{1}{2}}\log(\lambda+1)),\\
&I_{3,1}\lesssim \langle v_0\rangle^{m-\kappa}.
\end{align*}
It remains to estimate the contribution of $\Omega_2$. Set
$L=\langle v_0\rangle\geq1$ and $p=\varkappa+\kappa-d>\kappa+2$.
Since $[h]_{v_0}\leq2L|h|$ and
$\langle v_0,h\rangle\asymp L+|h|$, polar coordinates give
\[
 I_{1,2}\lesssim\int_0^\infty
 \min\{1,(\lambda L\rho)^2\}(L+\rho)^{-p}
 \frac{\dif\rho}{\rho^{1+2s}}.
\]
If $\lambda L\leq1$, retain the square on the whole integral and use
$\int_0^\infty\rho^{1-2s}(L+\rho)^{-p}\,\dif\rho
\lesssim L^{2-2s-p}$; its convergence follows from $p>2-2s$.
If $\lambda L\geq1$, use $(L+\rho)^{-p}\leq L^{-p}$ and
$\int_0^\infty\min\{1,(a\rho)^2\}\rho^{-1-2s}\,\dif\rho
\lesssim a^{2s}$. Since $p>\kappa+2$, the two cases yield
\[
 I_{1,2}\lesssim
 \begin{cases}
 \lambda^2L^{4-2s-p}\leq\lambda^2L^{2-2s-\kappa},&\lambda L\leq1,\\
 \lambda^{2s}L^{2s-p}\leq\lambda^{2s}L^{-\kappa},&\lambda L\geq1.
 \end{cases}
\]
This is exactly the first bound in \eqref{bd123}.
For $I_{2,2}$ the same fixed factor $L$ gives
\[
 I_{2,2}\lesssim L\int_0^\infty
 \rho^{-2s}\min\{1,\lambda L\rho\}(L+\rho)^{-p}\,\dif\rho.
\]
For $s<1/2$, discard the minimum to obtain
$I_{2,2}\lesssim L^{2-2s-p}\leq L^{1-\kappa}$.
For $s>1/2$, use $(L+\rho)^{-p}\leq L^{-p}$ and split at
$\rho=(\lambda L)^{-1}$; this gives
$I_{2,2}\lesssim L^{2s-p}\lambda^{2s-1}
\leq L^{1-\kappa}\lambda^{2s-1}$.
For $s=1/2$, substitution $\rho=Lr$ and splitting at $r=1$ give
\[
 I_{2,2}\lesssim L^{1-p}\bigl(1+\log(1+\lambda L^2)\bigr)
 \lesssim L^{1-\kappa}\bigl(1+\log(1+\lambda)\bigr),
\]
using $\log(1+\lambda L^2)\leq\log(1+\lambda)+2\log L$ and
$p>\kappa+2$. Finally, for $2\leq m\leq\varkappa-d$,
\[
 I_{3,2}\lesssim L^m\int_0^\infty
 \rho^{m-1-2s}(L+\rho)^{-p}\,\dif\rho
 \lesssim L^{2m-2s-p}\leq L^{m-\kappa}.
\]
Here $p>m-2s$ ensures convergence at infinity, and
$m\leq\varkappa-d$ gives the last inequality.
	    {These bounds prove} \eqref{bd123} {and hence the
	    lemma.}
\end{proof}	
		\subsection{Pointwise estimates for frozen-coefficient kernels}\label{secfrozenkernel}
	     {Using the coefficient bounds from Section~}\ref{seculb}{,
we now derive pointwise upper bounds for the frozen kernel}
\[
    H_g^{t_0,x_0,v_0}(t,\tau,x,v),
\]
{and its derivatives. Recall that} \(H_g^{t_0,x_0,v_0}\)
{is the fundamental solution of the frozen-coefficient equation}
\eqref{lineeq_tau_char}{.}

{The fractional Kolmogorov kernel is the model case. Its sharp
time-one bound is} \eqref{eq:kolmogorov-unit-profile}{, with}
\(\mathcal N\) {defined in} \eqref{defN0}{; see}
\cite{HouZhang} {and, in one dimension,}
\cite{Grube,haina}{. Unlike purely nonlocal diffusion kernels (see,
for example,} \cite{CKW,Sato,Ch-Zh4}{), this profile records the
transport-induced coupling between} \(x\) {and} \(v\){.}

{We next consider the frozen linear Boltzmann equation. It has the same
transport structure as the fractional Kolmogorov equation, but an anisotropic
jump kernel. In particular, velocity decay of the background} \(g\)
{induces anisotropic decay of} \(\mathbf C_g(t,x_0,v_0,h)\)
{in the jump variable} \(h\){:}
\[
    \mathbf C_g(t,x_0,v_0,h)\lesssim {\Omega_{\varkappa}}(v_0,h),
\]
{where} ${\Omega_{\varkappa}}$ {is defined in} \eqref{defom}{.
The decay is strongest in the direction selected by} \(v_0\)
{and is naturally measured by}
\[
    [h]_{v_0}:=|h|+|h\cdot v_0|.
\]
{The frozen fundamental solution should therefore retain the
Kolmogorov concentration near the kinetic line while exhibiting the far-field
anisotropic decay dictated by} \({\Omega_{\varkappa}}(v_0,h)\){. We establish the
corresponding pointwise upper bound by adapting the mechanism of}
\cite{HouZhang}{, replacing the isotropic stable tail with the
anisotropic control of} \(\mathbf C_g\){. To state the result,
define the anisotropic kinetic profile}
\begin{align}\label{anisotropic-HZ-profile}
	\mathcal N_{v_0}(x,v)
	:=
\langle [x]_{v_0}+[v]_{v_0}\rangle^{-d-{2s}}
	\int_0^1
	\langle [x-\sigma v]_{v_0}\rangle^{-d-{2s}}
	\,\dif \sigma .
\end{align}
{This is the anisotropic analogue of the Hou--Zhang profile in}
\cite{HouZhang}{. The factor}
\(\langle [x]_{v_0}+[v]_{v_0}\rangle^{-d-{2s}}\)
{measures the largest jump needed to reach} \((x,v)\){,
whereas the integral measures the distance to the kinetic line}
\[
x=\theta v,\qquad 0\le \theta\le1,
\]
{in the anisotropic metric} \([\,\cdot\,]_{v_0}\){.
In particular,}
\begin{align}\label{N-marginal-x}
 &  	\langle v_0\rangle^2 \int_{\mathbb R^d}\mathcal N_{v_0}(x,v)\,\dif x
    \lesssim
    \langle v_0\rangle
    \langle [v]_{v_0}\rangle^{-d-2s},\\
 &   	\langle v_0\rangle^2 \iint_{\mathbb R^{2d}}\mathcal N_{v_0}(x,v)\,\dif x
\dif v    \lesssim
    1.
\end{align}
{The factor} \(\langle v_0\rangle^2\) {is the
Jacobian of the anisotropic change of variables.}

{The estimate below takes the minimum of two complementary bounds.
The first is the anisotropically scaled Kolmogorov profile}
\(\mathcal N_{v_0}\){, which captures spreading generated by
coupling} \(v\cdot\nabla_x\) {with nonlocal velocity diffusion and,
in particular, concentration near the transported kinetic line. The second is
a rapid off-diagonal bound specific to the present linear Boltzmann kernel. It
comes from decay of} \(\mathbf C_g(t,x_0,v_0,h)\) {in} \(h\){,
inherited from velocity decay of} \(g\){. Thus, the
estimate captures both the kinetic geometry and the additional anisotropic
suppression of large jumps.}

Let $p_g^{t_0,x_0,v_0}(t,\tau;x,v)$ denote the physical density of the
frozen kinetic process over $[\tau,t]$, with Fourier transform
\begin{align}\label{eq:frozen-physical-density-Fourier}
\widehat p_g^{t_0,x_0,v_0}(t,\tau;\xi,\eta)
=\exp\!\left[-\int_\tau^t
E_g\bigl(\eta+(t-r)\xi;
r,X_{x_0,v_0}^{t_0-r},v_0\bigr)\,\dif r\right].
\end{align}
The symbol $E_g$ is even in its frequency argument.  Comparing
\eqref{eq:frozen-physical-density-Fourier} with \eqref{defH} gives the
exact transported/physical relation
\begin{equation}\label{eq:H-physical-shear}
H_g^{t_0,x_0,v_0}(t,\tau,x,v)
=p_g^{t_0,x_0,v_0}(t,\tau;x+\sigma v,v),
\qquad \sigma=t-\tau.
\end{equation}
Thus normalized derivative bounds for the physical density transfer
directly to $H_g$ through this shear.

		\begin{lemma}
			\label{Hgeneral}{Let} $s\in(0,1)$, and assume
\eqref{assofb-lower}. {Assume that}
			$g$ {satisfies} \eqref{eq:macro-main} {and}
			$$\|g\|_{L^\infty_T\mathbf{C}^0}\leq C_0<\infty.$$
			{Assume also
$\varkappa\ge 2d+2s+\kappa$.} {Fix} $t_0\in[0,T]$ {and}
			$x_0,v_0\in\mathbb R^d${, and let}
$H_g^{t_0,x_0,v_0}$ {be defined by} \eqref{defH}{.
Then, for each} $m\in\mathbb N${, there exists}
$$
C_m=C(d,s,\gamma,\varkappa,C_0,\mathsf m_0,\mathsf M_0,
\mathsf E_0,\mathsf H_0,m,c_b,\delta_0,C_b,
(C_{b,k})_{1\leq k\leq k_m}),
$$
{where $k_m$ is a finite derivative order determined by the proof.
The constants are independent of $(t_0,x_0,v_0,t,\tau)$.
Then, whenever} $0<\sigma=t-\tau\leq1${,}
			\begin{align}\label{z1}
				\left| \mathcal{D}_{\sigma,v_0}^{(x,v),m}H_g^{t_0,x_0,v_0}(t,\tau,x,v)\right|\leq C_m\sigma^{-d}\tilde \sigma_{v_0}^{-\frac{d}{s}}	\langle v_0\rangle^2\left(\mathcal{N}_{v_0} \left({\frac{x+\sigma v}{\sigma\tilde \sigma_{v_0}^{1/(2s)}}},\frac{v}{\tilde \sigma_{v_0}^\frac{1}{2s}}\right)\wedge \left\langle\frac{[x]_{v_0}/\sigma+[v]_{v_0}}{\langle v_0\rangle}\right\rangle^{-(\varkappa-2d)}\right),
			\end{align}
			{where} $\mathcal{D}_{\sigma,v_0}^{(x,v),m}$
			{is defined in} \eqref{anideri}{.}
\end{lemma}
       {When} \(|v_0|\lesssim1\){, the anisotropic
	quantities in} \eqref{z1} {reduce to the usual kinetic scale:}
\[
    [x]_{v_0}\sim |x|,
    \qquad
    [v]_{v_0}\sim |v|,
    \qquad
    \tilde\sigma_{v_0}\sim \sigma .
\]
{Therefore,} \eqref{z1} {implies, for}
\(m_1,m_2\in\mathbb N\){,}
\begin{align*}
    \left|
    \nabla_x^{m_1}\nabla_v^{m_2}
    H_g^{t_0,x_0,v_0}(t,\tau,x,v)
    \right|
    &\lesssim
    \sigma^{-d-\frac{d}{s}
    -\frac{(2s+1)m_1}{2s}
    -\frac{m_2}{2s}}
    \nonumber\\
    &\quad\times
    \left(
    \mathcal N\left(
        {\frac{x+\sigma v}{\sigma^{\frac{2s+1}{2s}}}},
        \frac{v}{\sigma^{\frac1{2s}}}
    \right)
    \wedge
    \left\langle
        \frac{|x|}{\sigma}+|v|
    \right\rangle^{-(\varkappa-2d)}
    \right),
\end{align*}
{where} \(\sigma=t-\tau\){,} \(\mathcal N\)
{is defined in} \eqref{defN0}{, and the implicit constant
may depend on the fixed bound for} \(|v_0|\){. This is the standard
fractional Kolmogorov scaling:}
\[
    |v|\sim \sigma^{\frac1{2s}},
    \qquad
    |x|\sim \sigma^{1+\frac1{2s}}
    =
    \sigma^{\frac{2s+1}{2s}} .
\]	{Lemma~}\ref{Hgeneral} {immediately gives the
corresponding estimate for a Maxwellian background. Indeed, when} \(g=\mu\){,
condition} \eqref{Ig} {and the bound}
$\|g\|_{L^\infty_T\mathbf{C}^0}<\infty$ {hold for every}
\(\varkappa>0\){.}
		\begin{lemma}
			\label{Hmu}
		{Let} $s\in(0,1)${. For} $H_\mu^{v_0}$
		{defined in} \eqref{defH0} {and} \(0<t\leq1\){,}
			\begin{align*}
				\left| \mathcal{D}_{t,v_0}^{(x,v),m}H_\mu^{v_0}(t,x,v)\right|\lesssim_m t^{-d}\tilde t_{v_0}^{-\frac{d}{s}}	\langle v_0\rangle^2\left(\mathcal{N}_{v_0} \left({\frac{x+tv}{t\tilde t_{v_0}^{1/(2s)}}},\frac{v}{\tilde t_{v_0}^\frac{1}{2s}}\right)\wedge \left\langle\frac{[x]_{v_0}/t+[v]_{v_0}}{\langle v_0\rangle}\right\rangle^{-\infty}\right),\ \ \qquad \forall m\in\mathbb{N},
			\end{align*}
	            {where} $A^{-\infty}$ {is understood in
	            the sense of} \eqref{def-inf}{.}
		\end{lemma}\vspace{0.3cm}
	{For fixed} \((t_0,x_0,v_0)\){, consider the frozen equation}
\begin{align}\label{eqfrozen}
	\partial_t f+v\cdot\nabla_x f
	+\int_{\mathbb R^d}\delta_h f(v)
	\frac{\mathbf C_g(t,X_{x_0,v_0}^{t_0-t},v_0,h)}{|h|^{d+{2s}}}\,\dif h=0 .
\end{align}
{Set}
\[
\nu(t,\dif h)
:=
\frac{\mathbf C_g(t,X_{x_0,v_0}^{t_0-t},v_0,h)}{|h|^{d+{2s}}}\,\dif h .
\]
{The corresponding kinetic process is}
\[
Z_t
=
\left(
\int_0^t L_r\,\dif r,\,
L_t
\right),
\]
{where} \(L\) {is the additive pure-jump process with jump
measure} \(\nu(t,\dif h)\){.}

{We use the following consequences of the structure of}
\(\mathbf C_g\){. First, it is nonnegative and satisfies the
cancellation condition}
\begin{align}\label{frozen-cancellation}
	\int_{R_0\le |h|\le R_1}
	h\,\nu(t,\dif h)=0,
	\qquad
	0<R_0\le R_1<\infty .
\end{align}
{because}
$\mathbf{C}_g(t,x,v,h)=\mathbf{C}_g(t,x,v,-h)${. Second, the
small jumps are nondegenerate in the averaged Fourier sense. For}
$\tau\in[0,1]${, define}
\[
\widetilde\nu(\tau,K)
=t\nu(\tau t,\tilde t_{v_0}^\frac{1}{2s}\mathcal{O}_{v_0}K),\ \ \ \forall K\subset \mathbb{R}^d.
\]
{Equivalently,}
\[
\widetilde\nu(\tau,\dif h)
=\frac{\langle v_0\rangle^\kappa \mathbf{C}_g(\tau t,X_{x_0,v_0}^{t_0-\tau t},v_0,\tilde t_{v_0}^\frac{1}{2s} \mathcal{O}_{v_0}h)}{|\mathcal{O}_{v_0} h|^{d+2s}}|\operatorname{det}\mathcal{O}_{v_0}|\ \dif h.
\]
{Then} \(\widetilde\nu\) {is the Lévy measure of the rescaled
process}
$$
\widetilde L_\tau
:=
\tilde t_{v_0}^{-1/(2s)}\mathcal{O}_{v_0}^{-1}L_{t\tau}.
$$
{In particular,}
$$
\left(
t^{-1}\tilde t_{v_0}^{-1/(2s)}\mathcal{O}_{v_0}^{-1}
\int_0^t L_r\,\dif r,\ 
\tilde t_{v_0}^{-1/(2s)}\mathcal{O}_{v_0}^{-1}L_t
\right)
=
\left(
\int_0^1 \widetilde L_\tau \dif\tau,
\widetilde L_1
\right).
$$
{Moreover,}
\begin{align}\label{normalized-coercivity}
	\inf_{\tau\in[0,1]}
	\int_{|w|\le1}
	\bigl(1-\cos(\xi\cdot w)\bigr)\,
	\widetilde\nu(\tau,\dif w)
	\ge
	c_0\bigl(|\xi|^2\wedge |\xi|^{2s}\bigr),
\end{align}
and
\begin{align}\label{normalized-moment}
	\sup_{\tau\in[0,1]}
	\int_{|w|\le1}|w|^2\,\widetilde\nu(\tau, \dif w)
	\le C_0 .
\end{align}
{The constant in \eqref{normalized-coercivity}
inherits the dependence on $c_b$ and $\delta_0$ from
\eqref{assofb-lower}.}
{Indeed, the change of variables}
$h=\tilde t_{v_0}^{1/(2s)}\mathcal O_{v_0}w$
{and Remark~}\ref{rem:small-jump-coercivity}
{with} $\delta\sim\tilde t_{v_0}^{1/(2s)}$
{give \eqref{normalized-coercivity}; here}
$t\langle v_0\rangle^{-\kappa}\tilde t_{v_0}^{-1}=1${. Moreover,
\eqref{normalized-moment} follows from the upper symbol bound
\eqref{equpperE}, since, for} $|w|\le1${,}
\[
	|w|^2\lesssim\sum_{j=1}^d\bigl(1-\cos w_j\bigr),
\]
{and the same change of variables bounds the normalized symbol uniformly
in each coordinate direction.}

The coefficient envelope also gives, uniformly in $h\ne0$,
\begin{align}
{\Omega_{\varkappa}}(v_0,h)
&\lesssim \langle v_0\rangle^{1-\kappa}
|\mathcal O_{v_0}^{-1}\widehat h|^{-d-2s},
\label{eq:anisotropic-envelope-upper}
\\
\widetilde\nu(\tau,\dif w)
&\le C_*|w|^{-d-2s}\,\dif w,
\qquad 0\le\tau\le1.
\label{eq:normalized-pointwise-upper}
\end{align}
Indeed,
$|\mathcal O_{v_0}^{-1}\widehat h|
=\langle\mathrm P_{\widehat h}v_0\rangle$ and
$\langle v_0\rangle\leq
\langle\Pi_{\widehat h}v_0\rangle+
\langle\mathrm P_{\widehat h}v_0\rangle$.
Thus \eqref{defom} and
$\varkappa\geq2d+2s+\kappa$ give
\begin{align*}
&\langle\Pi_{\widehat h}v_0\rangle^{1-\kappa}
\langle\mathrm P_{\widehat h}v_0\rangle^{d-1-\varkappa}
\lesssim
\langle v_0\rangle^{1-\kappa}
|\mathcal O_{v_0}^{-1}\widehat h|^{-d-2s},\\
&\langle v_0,h\rangle^{d-\kappa-\varkappa}
\lesssim
\langle v_0\rangle^{1-\kappa}
|\mathcal O_{v_0}^{-1}\widehat h|^{-d-2s}.
\end{align*}
This proves \eqref{eq:anisotropic-envelope-upper}. For
$h=\tilde t_{v_0}^{1/(2s)}\mathcal O_{v_0}w$,
$|\mathcal O_{v_0}^{-1}\widehat h|=|w|/|\mathcal O_{v_0}w|$.
Substitution into the formula for $\widetilde\nu$, together with
$|\det\mathcal O_{v_0}|=\langle v_0\rangle^{-1}$, yields the asserted
bound for $\widetilde\nu$ in \eqref{eq:normalized-pointwise-upper}.

{The following proposition gives an anisotropic
upper bound for the frozen kernel.}
\begin{proposition}
	\label{prop-frozen-anisotropic-upper}
	\begingroup
	Assume \eqref{frozen-cancellation}, \eqref{normalized-coercivity},
	\eqref{normalized-moment}, and \eqref{eq:normalized-pointwise-upper}.
	Let $0<\sigma=t-\tau\leq1$. For every pair of multi-indices
	$\alpha,\beta$,
	\begin{align}\label{frozen-anisotropic-upper}
	\Bigl|\bigl(\mathcal D_{\sigma,v_0}^{(x)}\bigr)^\alpha
	\bigl(\mathcal D_{\sigma,v_0}^{(v)}\bigr)^\beta
	p_g^{t_0,x_0,v_0}(t,\tau;x,v)\Bigr|
	&\lesssim_{\alpha,\beta}
	\sigma^{-d}\tilde\sigma_{v_0}^{-d/s}\langle v_0\rangle^2
	\mathcal N_{v_0}\!\left(
	\frac{x}{\sigma\tilde\sigma_{v_0}^{1/(2s)}},
	\frac{v}{\tilde\sigma_{v_0}^{1/(2s)}}
	\right).
	\end{align}
	Consequently, for every $m\in\mathbb N$,
	\begin{align}\label{eq:A10-transported-derivatives}
	\bigl|\mathcal D_{\sigma,v_0}^{(x,v),m}
	H_g^{t_0,x_0,v_0}(t,\tau,x,v)\bigr|
	&\lesssim_m
	\sigma^{-d}\tilde\sigma_{v_0}^{-d/s}\langle v_0\rangle^2
	\mathcal N_{v_0}\!\left(
	\frac{x+\sigma v}{\sigma\tilde\sigma_{v_0}^{1/(2s)}},
	\frac{v}{\tilde\sigma_{v_0}^{1/(2s)}}
	\right).
	\end{align}
	The constants are uniform in $\tau,t,t_0,x_0,v_0$ under the stated
	structural bounds.
	\par\endgroup
\end{proposition}
\begin{proof}
\begingroup
Normalize by
\[
(x,v)=\left(
\sigma\tilde\sigma_{v_0}^{1/(2s)}\mathcal O_{v_0}y,
\tilde\sigma_{v_0}^{1/(2s)}\mathcal O_{v_0}w
\right),
\qquad r=\tau+\sigma\theta,
\]
and set
\[
\widetilde\nu_\theta(K)
:=\sigma\nu\left(
\tau+\sigma\theta,
\tilde\sigma_{v_0}^{1/(2s)}\mathcal O_{v_0}K
\right),
\qquad 0\le\theta\le1.
\]
These measures satisfy \eqref{normalized-coercivity}, \eqref{normalized-moment}, and
\eqref{eq:normalized-pointwise-upper} with constants independent of
$\tau,t,x_0,v_0$.

Split $\widetilde\nu_\theta$ at $|u|=1$ into
$\widetilde\nu^{(0)}_\theta+\widetilde\nu^{(1)}_\theta$.
The corresponding independent kinetic increments are
$\widetilde Z^{(0)}_1$ and $\widetilde Z^{(1)}_1$.
The characteristic exponent of $\widetilde Z^{(0)}_1$ is
\[
\Phi_0(\xi,\eta)
 =\int_0^1\int_{|u|\le1}
 \bigl(1-\cos((\eta+(1-\theta)\xi)\cdot u)\bigr)
 \widetilde\nu_\theta(\dif u)\,\dif\theta.
\]
By \eqref{normalized-coercivity} and scaling,
\[
\Phi_0(\xi,\eta)
\gtrsim\int_0^1
\bigl(|\eta+(1-\theta)\xi|^2
\wedge|\eta+(1-\theta)\xi|^{2s}\bigr)\,\dif\theta
\gtrsim |(\xi,\eta)|^2\wedge|(\xi,\eta)|^{2s}.
\]
The last inequality follows by compactness on the unit sphere.
Together with \eqref{normalized-moment}, this permits repeated Fourier
integration by parts and gives a smooth density
$q_1$ of $\widetilde Z^{(0)}_1$ with
\[
 |\partial^k q_1(y,w)|
 \le C_{k,N}\langle |y|+|w|\rangle^{-N},
 \qquad k\in\mathbb N^{2d},\quad N\ge0.
\]

Let $\Lambda(\dif u)=C\mathbf1_{|u|>1}|u|^{-d-2s}\,\dif u$, with
$C$ chosen so that $\widetilde\nu^{(1)}_\theta\le\Lambda$.
Termwise comparison of the Poisson series shows that, for every Borel
set $B\subset\mathbb R^{2d}$,
\[
\mathbb P(\widetilde Z^{(1)}_1\in B)
\le e^{\Lambda(\mathbb R^d)}\mathbb P(Z^\Lambda_1\in B),
\]
where $Z^\Lambda_1$ is the kinetic compound-Poisson increment with
intensity $\dif\theta\,\Lambda(\dif u)$.
For $(y,w)\in\mathbb R^{2d}$, write
$Q_1(y,w)=(y,w)+[-\frac12,\frac12]^{2d}$.
The upper half of \cite[Lemma~2.3]{HouZhang}, together with its
profile comparison in \cite[Lemma~2.7]{HouZhang}, applies to this
isotropic reference measure and yields
\[
 \mathbb P\bigl(\widetilde Z^{(1)}_1\in Q_1(y,w)\bigr)
 \le C\mathcal N(y,w).
\]

For $z,h\in\mathbb R^{2d}$, the definition \eqref{defN0} gives
\begin{equation}\label{eq:A10-profile-moderation}
 \mathcal N(z-h)\le C\langle h\rangle^{2(d+2s)}\mathcal N(z).
\end{equation}
Cover $\mathbb R^{2d}$ by unit boxes $Q_1(j)$,
$j\in\mathbb Z^{2d}$. Independence and
\eqref{eq:A10-profile-moderation} give, for the density $\widetilde p_1$
of $\widetilde Z^{(0)}_1+\widetilde Z^{(1)}_1$ and
$N>2d+2(d+2s)$,
\begin{align*}
|\partial^k\widetilde p_1(z)|
&\le\sum_{j\in\mathbb Z^{2d}}
 \sup_{z'\in Q_1(j)}|\partial^k q_1(z')|\,
 \mathbb P(\widetilde Z^{(1)}_1\in Q_1(z-j))\\
&\le C_{k,N}\mathcal N(z)
 \sum_{j\in\mathbb Z^{2d}}
 \langle j\rangle^{-N+2(d+2s)}
 \le C_k\mathcal N(z).
\end{align*}
The change of variables has Jacobian
$\sigma^d\tilde\sigma_{v_0}^{d/s}|\det\mathcal O_{v_0}|^2$.
Since $|\det\mathcal O_{v_0}|^{-2}=\langle v_0\rangle^2$ and
$|\mathcal O_{v_0}^{-1}x|\asymp[x]_{v_0}$ uniformly, the order-zero
estimate and its derivatives give \eqref{frozen-anisotropic-upper}.
The shear \eqref{eq:H-physical-shear} turns the normalized $(x,v)$
derivatives into $(\partial_y,\partial_y+\partial_w)$, proving
\eqref{eq:A10-transported-derivatives}.
\par\endgroup
\end{proof}
\begin{remark}[Comparison with \cite{HouZhang}]
	{The preceding estimate replaces the isotropic profile}
	\eqref{eq:kolmogorov-unit-profile} {with}
	\(\mathcal N_{v_0}\) {because} \eqref{z26}
	{controls the jump density in the anisotropic metric}
	\([w]_{v_0}\sim|\mathcal O_{v_0}^{-1}w|\){. Hou--Zhang also
	assume a pointwise lower bound on the Lévy density. Their upper-bound proof
	uses it only to establish coercivity of the real part of the small-jump
	characteristic exponent. Here, the averaged Fourier lower bound}
	\eqref{normalized-coercivity} {supplies this coercivity directly.}
	{The large-jump argument uses the pointwise
normalized domination \eqref{eq:normalized-pointwise-upper},
which follows from the coefficient envelope at the stated
velocity-decay threshold.} {Thus,
	the pointwise lower bound assumed in} \cite{HouZhang} {is
	replaced here by averaged coercivity in the upper-bound argument. A matching
	lower bound is not claimed.}
\end{remark}
{Proposition~}\ref{prop-frozen-anisotropic-upper}
{uses the probabilistic approach of Hou--Zhang} \cite{HouZhang}{,
which captures the anisotropic large-jump profile}
\(\mathcal N_{v_0}\){. To prove the rapid far-field decay in
Lemma~}\ref{lemdecay}{, however, we use the Fourier representation
of the frozen kernel. Repeated Fourier integration by parts requires precise
derivative bounds for} \(E_g\){.}
{The required smoothness comes from the upper jump-moment
estimates \eqref{bd123}, with derivative orders limited by \eqref{derE}.}
		\begin{lemma}\label{lemdecay}
	        {Let} $g,t_0,x_0,v_0$ {and}
	        \(H_g^{t_0,x_0,v_0}\) {be as in Lemma~}\ref{Hgeneral}{.
	        For every} \(m\in\mathbb N\){, there exists}
$$
C_m=C(C_0,\mathsf m_0,\mathsf M_0,\mathsf E_0,\mathsf H_0,m,
{c_b,\delta_0}),
$$
{such that, whenever} $0<\sigma=t-\tau\leq1${,}
			\begin{align}\label{eq:H-rough-decay}
				\left| \mathcal{D}_{\sigma,v_0}^{(x,v),m}H_g^{t_0,x_0,v_0}(t,\tau,x,v)\right|\leq C_m \langle v_0\rangle^2 \sigma^{-d}\tilde  \sigma_{v_0}^{-\frac{d}{s}} \left\langle\frac{[x]_{v_0}/\sigma+[v]_{v_0}}{\langle v_0\rangle}\right\rangle^{-(\varkappa-2d)}.
			\end{align}
		\end{lemma}
\begin{proof}
	{Let}
	\[
	\omega(\xi,\eta)
	:=
	\int_\tau^t
	E_g\bigl((r-\tau)\xi-\eta;\,r,X_{x_0,v_0}^{t_0-r},v_0\bigr)\,\dif r .
	\]
	{Then}
	\[
	H_g^{t_0,x_0,v_0}(t,\tau,x,v)
	=
	(2\pi)^{-2d}
	\iint_{\mathbb R^{2d}}
	e^{-\omega(\xi,\eta)}
	e^{i\xi\cdot x+i\eta\cdot v}\,\dif \xi \dif \eta .
	\]
	{We suppress the harmless normalization constant and define}
	\[
	p:=\sigma\mathcal O_{v_0}\xi,\qquad
	q:=\mathcal O_{v_0}\eta .
	\]
	{Then}
	\[
	\dif \xi\,\dif \eta
	=
	\langle v_0\rangle^2\,\sigma^{-d}\,\dif p\,\dif q .
	\]
	{Moreover,}
	\[
	[[(r-\tau)\xi-\eta]]_{v_0}
	=
	|\theta p-q|,
	\qquad
	\theta:=\frac{r-\tau}{\sigma}.
	\]
	{By Lemma~}\ref{lem:anisotropic-coercivity}{,}
	\[
	\begin{aligned}
\omega(\xi,\eta)
		&\gtrsim\tilde  \sigma_{v_0}
		\int_0^1
		|\theta p-q|^{2s}
		\min\{1,|\theta p-q|\}^{2-2s}\,\dif \theta .
	\end{aligned}
	\]
	{Kinetic nondegeneracy gives}
	\[
	\int_0^1
	|\theta p-q|^{2s}
	\min\{1,|\theta p-q|\}^{2-2s}\,\dif \theta
	\gtrsim
	(|p|+|q|)^{2s}\wedge (|p|+|q|)^2 .
	\]
	{Consequently, for every} \(a\ge0\){,}
	\begin{align}\label{eq:basic-fourier-integral}
		\iint_{\mathbb R^{2d}}
		\bigl(1+\tilde  \sigma_{v_0}^{1/(2s)}(|p|+|q|)\bigr)^a
		e^{-c\omega(\xi,\eta)}
		\,\dif \xi \dif \eta
		\lesssim \langle v_0\rangle^2\sigma^{-d}\tilde  \sigma_{v_0}^{-d/s}.
	\end{align}
	{The phase, on the other hand, becomes}
	\[
	\xi\cdot x+\eta\cdot v
	=
	p\cdot X+q\cdot V,
	\qquad
	X:=\frac{\mathcal O_{v_0}^{-1}  x}{\sigma},
	\qquad
	V:=\mathcal O_{v_0}^{-1} v .
	\]
{Introduce the Fourier integration-by-parts vector field}
	\[
	\mathscr D_{\xi,\eta}
	:=
	\left(
	\sigma^{-1}\mathcal O_{v_0}^{-1}\nabla_\xi,\,
	\mathcal O_{v_0}^{-1}\nabla_\eta
	\right).
	\]
	{It satisfies}
	\[
	\mathscr D_{\xi,\eta}
	e^{i\xi\cdot x+i\eta\cdot v}
	=
	i
	\left(X,V
	\right)
	e^{i\xi\cdot x+i\eta\cdot v}.
	\]
	{Moreover,}
	\begin{align}
	\label{equthe}
	|(X,V)|
	\sim
	\frac{[x]_{v_0}}{\sigma}+[v]_{v_0}.
\end{align}
	{For} \(|\beta|\le\varkappa-d\){, Lemma~}\ref{deofE}
	{gives}
	\[
	|\mathscr D_{\xi,\eta}^{\beta}\omega(\xi,\eta)|
	\lesssim \langle v_0\rangle^{|\beta|}
	\Bigl(
	1+\tilde  \sigma_{v_0}\left(
	(|p|+|q|)^{(2s-1)_+}+\log(1+|p|+|q|)\mathbf{1}_{s=\frac{1}{2}}\right) 
	\Bigr).
	\]

{Hence,}
	\[
	|\mathscr D_{\xi,\eta}^{\beta}e^{-\omega(\xi,\eta)}|
	\lesssim \langle v_0\rangle^{|\beta|}
	\left(1+\tilde  \sigma_{v_0}^{1/(2s)}\left((|p|+|q|)+\log(1+|p|+|q|)\mathbf{1}_{s=\frac{1}{2}}\right)\right)^{C_\beta}
	e^{-c\omega(\xi,\eta)} .
	\]

	{Combining this with} \eqref{eq:basic-fourier-integral}
	{yields}
	\begin{align}\label{eq:derivative-integral-bound}
		\iint_{\mathbb R^{2d}}
		|\mathscr D_{\xi,\eta}^{\beta}e^{-\omega(\xi,\eta)}|
		\,\dif \xi \dif \eta
		\lesssim \langle v_0\rangle^{2+|\beta|}\sigma^{-d}\tilde  \sigma_{v_0}^{-d/s}.
	\end{align}
	{Applying} \(\mathcal{D}_{\sigma,v_0}^{(x,v),m}\)
	{to the kernel produces the Fourier multiplier}
	\[
	\bigl(
	\sigma\tilde  \sigma_{v_0}^{1/(2s)}\mathcal{O}_{v_0}\xi,\,
	\tilde  \sigma_{v_0}^{1/(2s)}\mathcal{O}_{v_0}\eta
	\bigr)^{\otimes m}.
	\]
	{In the} \((p,q)\){-variables,}
	\[
	|\sigma\tilde  \sigma_{v_0}^{1/(2s)}\mathcal O_{v_0}\xi|
	+
	|\tilde  \sigma_{v_0}^{1/(2s)}\mathcal O_{v_0}\eta|
	\lesssim
	\tilde  \sigma_{v_0}^{1/(2s)}(|p|+|q|).
	\]

	{Thus, differentiation introduces at most the factor}
	\[
	\bigl(1+\tilde  \sigma_{v_0}^{1/(2s)}(|p|+|q|)\bigr)^m .
	\]
	{For} $n\leq\varkappa-d${, integrating by parts}
	\(n\) {times with} \(\mathscr D_{\xi,\eta}\) {gives}
	\[
	\begin{aligned}
		\langle X,V\rangle^{n}
		\left|
		\mathcal{D}_{\sigma,v_0}^{(x,v),m}H_g^{t_0,x_0,v_0}(t,\tau,x,v)
		\right|
		&\lesssim
		\sum_{|\beta|\le n}
		\iint_{\mathbb R^{2d}}
		\left|
		\mathscr D_{\xi,\eta}^{\beta}
		\left[
			\bigl(
			\sigma\tilde  \sigma_{v_0}^{1/(2s)}\mathcal O_{v_0}\xi,\,
			\tilde \sigma_{v_0}^{1/(2s)}\mathcal O_{v_0}\eta
			\bigr)^{\otimes m}
		e^{-\omega(\xi,\eta)}
		\right]
		\right|
		\,\dif \xi \dif \eta  \\
		&\lesssim
		\langle v_0\rangle^{2+n}
		\sigma^{-d}\tilde \sigma_{v_0}^{-d/s}.
	\end{aligned}
	\]

{Combining this estimate with} \eqref{equthe} {gives}
	\[
	\left|
	\mathcal{D}_{\sigma,v_0}^{(x,v),m}H_g^{t_0,x_0,v_0}(t,\tau,x,v)
	\right|
	\le
	C_{m}
	\langle v_0\rangle^2
	\sigma^{-d}\tilde  \sigma_{v_0}^{-d/s}
\left\langle\frac{[x]_{v_0}/\sigma+[v]_{v_0}}{\langle v_0\rangle}\right\rangle^{-n}.
	\]
	{Taking $n=\lfloor\varkappa-d\rfloor$ proves}
\eqref{eq:H-rough-decay}{, since $n\geq\varkappa-2d$.
Thus the stated exponent is a deliberate weakening of the bound obtained above.}
\end{proof}\par\medskip
{Combining Lemma~\ref{lemdecay} with
Proposition~\ref{prop-frozen-anisotropic-upper} and the shear
\eqref{eq:H-physical-shear} proves Lemma~\ref{Hgeneral}.}

					{The anisotropic bounds in Lemma~}\ref{leA}
					{yield the following Gaussian pointwise estimate for}
					$H_g^{t_0,x_0,v_0}$ {in the Landau case.}
					\begin{lemma}\label{pwthm}
{Let $s=1$ and $\varkappa>d+2$.
Assume that $g\geq0$ satisfies \eqref{eq:macro-main} and
$\sup_{t\in[0,T]}\|g(t)\|_{\mathbf C^0}\leq C_0$.}
{For every}
$m\in\mathbb N${, there exist} $c_1,c_2>0$
{depending only on}
\(d,\gamma,C_0,\mathsf m_0,\mathsf M_0,\mathsf E_0,\mathsf H_0\){,
and} \(m\){, such that, for} \(\sigma=t-\tau>0\){,}
					\begin{align*}
						&						H_g^{t_0,x_0,v_0}(t,\tau,x,v)\gtrsim 		\langle v_0\rangle^{2}	\sigma^{-d}\tilde \sigma_{v_0}^{-d}	\exp\left(-c_1\left(\frac{[x]_{v_0}^2}{\sigma^2\tilde \sigma_{v_0}}+\frac{[v]_{v_0}^2}{\tilde   \sigma_{v_0}}\right)\right),\\
                &|\mathcal{D}^{(x,v), m}_{\sigma, v_0}	H_g^{t_0,x_0,v_0}(t,\tau,x,v)|
						\lesssim 	\langle v_0\rangle^{2}	\sigma^{-d}\tilde \sigma_{v_0}^{-d} \exp\left(-c_2\left(\frac{[x]_{v_0}^2}{\sigma^2\tilde \sigma_{v_0}}+\frac{[v]_{v_0}^2}{\tilde \sigma_{v_0}}\right)\right).	
					\end{align*}
					
				\end{lemma}
\begin{proof}
{Use the notation and changes of variables}
$(\xi,\eta)\mapsto(p,q)$ {and} $(x,v)\mapsto(X,V)$
{from the proof of Lemma~}\ref{lemdecay}{, and set}
\(\mathsf{A}(r)=\mathsf{A}[g](r,X_{x_0,v_0}^{t_0-r},v_0)\){.
Then} $\omega(\xi,\eta)
	=
	\tilde \sigma_{v_0} \beta(p,q)${, where}
	\[
	\beta(p,q)=\int_0^1
	\bigl(\langle v_0\rangle^{\kappa}
	\mathcal O_{v_0} ^{-1}\mathsf{A}(\tau+\theta\sigma)\mathcal O_{v_0}^{-1}(\theta p-q),\theta p-q\bigr)_{\mathbb{R}^d}\,\dif \theta .
	\]
	{Lemmas~}\ref{leA} {and} \ref{lemcoeff}
	{give the two-sided anisotropic ellipticity bound}
	\[
	c\mathrm{Id}\le 	\langle v_0\rangle^{\kappa}
	\mathcal O_{v_0} ^{-1}\mathsf{A}(\tau+\theta\sigma)\mathcal O_{v_0}^{-1}\le C\mathrm{Id} .
	\]
	{Thus,} $\beta(p,q)$ {is a positive quadratic form
comparable to} \(|p|^2+|q|^2\){, and its Fourier transform is
Gaussian. Consequently, there exist} \(c,C>0\){, independent of}
\((t,\tau,x,v)\){, such that}
\[
\begin{aligned}
&(\det\mathcal O_{v_0}^{-1})^2
\sigma^{-d}\tilde \sigma_{v_0}^{-d}
\exp\left(
    -C\frac{|X|^2+|V|^2}{\tilde \sigma_{v_0}}
\right) \\
&\qquad\lesssim
H_g^{t_0,x_0,v_0}(t,\tau,x,v)
\lesssim (\det\mathcal O_{v_0}^{-1})^2
\sigma^{-d}\tilde \sigma_{v_0}^{-d}
\exp\left(
    -c\frac{|X|^2+|V|^2}{\tilde \sigma_{v_0}}
\right).
\end{aligned}
\]
{Since}
	\[
	|X|\sim \frac{[x]_{v_0}}{\sigma},
	\qquad
	|V|\sim [v]_{v_0},
	\qquad
	(\det\mathcal O_{v_0} ^{-1})^2\sim\langle v_0\rangle^2,
	\]
	{the claimed lower bound follows, as does the upper bound for}
	\(m=0\){.}
	
	{For the derivative estimates, applying}
	\[
	\mathcal D_{\sigma,v_0}^{(x,v)}
	=
	\left(
	\sigma\tilde\sigma_{v_0}^{1/2}\mathcal O_{v_0} \nabla_x,\,
	\tilde\sigma_{v_0}^{1/2}\mathcal O_{v_0} \nabla_v
	\right),
	\qquad
	\tilde\sigma_{v_0}:=\langle v_0\rangle^{-\kappa}\sigma ,
	\]
	{corresponds in the} \((p,q)\){-variables to
	multiplication by}
	\[
	\tilde\sigma_{v_0}^{1/2}( p,q).
	\]
An order-$m$ scaled derivative has Fourier multiplier given by a linear
combination of monomials $\tilde\sigma_{v_0}^{m/2}p^{\alpha}q^{\zeta}$,
with $|\alpha|+|\zeta|=m$. Since
$\beta(p,q)$ is a uniformly positive quadratic form, differentiation of
its Gaussian Fourier transform gives, for every such pair of multi-indices,
\[
\left|
\iint_{\mathbb R^{2d}}
\tilde\sigma_{v_0}^{m/2}p^\alpha q^\zeta
e^{-\tilde\sigma_{v_0}\beta(p,q)}
e^{ip\cdot X+iq\cdot V}\,dp\,dq
\right|
\leq C_m\tilde\sigma_{v_0}^{-d}
e^{-c_m(|X|^2+|V|^2)/\tilde\sigma_{v_0}}.
\]
The constants depend only on $m$ and the ellipticity bounds; polynomial
factors are absorbed by decreasing the Gaussian exponent.
	{Multiplying by}
$
	(\det\mathcal O_{v_0}^{-1} )^2\sigma^{-d}
=
	\langle v_0\rangle^2\sigma^{-d}
$ {gives}
	\[
	|\mathcal D_{\sigma,v_0}^{(x,v),m}H_g^{t_0,x_0,v_0}(t,\tau,x,v)|
	\lesssim_m
	\langle v_0\rangle^{2}
	\sigma^{-d}\tilde \sigma_{v_0}^{-d}
	\exp\left(
	-c_1\left(
	\frac{[x]_{v_0}^2}{\sigma^2\tilde\sigma_{v_0}}
	+
	\frac{[v]_{v_0}^2}{\tilde\sigma_{v_0}}
	\right)\right).
	\]
    {This proves the lemma.}
\end{proof}

{The final lemma isolates a decay-transfer principle used below. The
kernel's anisotropic factor transfers polynomial decay from the initial to the
terminal variables, with explicit losses that measure the mismatch between the
frozen velocity and the relevant kinetic velocities.}
\begin{lemma}\label{lemtail}
	{Let} \(t\in(0,1)\){, and set}
	\begin{align*}
		J:=
		\frac{[x/t]_{v_0}+[v]_{v_0}}{\langle v_0\rangle}.
	\end{align*}
	{Then, for every}
	\(x,x_1,v,v_0,v_1\in\mathbb R^d\){,}
	\begin{align}
		\langle J\rangle^{-1}
		&\lesssim
		\langle |v-v_1|+|v_0-v_1|\rangle
		\langle v_1\rangle^{-1},
		\label{eq:A-velocity-bound}
		\\
		\langle J\rangle^{-1}
		&\lesssim
		\langle |x-x_1|+|v-v_1|+|v_0-v_1|\rangle^{1/2}
		\langle |x_1|+|v_1|\rangle^{-1/2}.
		\label{eq:A-space-velocity-bound}
	\end{align}
\end{lemma}
\begin{proof}
	{We first prove} \eqref{eq:A-velocity-bound}{. Set}
	\begin{align*}
		R:=|v-v_1|+|v_0-v_1|.
	\end{align*}
	{We claim that}
	\begin{align}
		\langle v_1\rangle
		\lesssim
		\langle J\rangle \langle R\rangle .
		\label{eq:v1-control}
	\end{align}
	{If} \(R\ge|v_1|/4\){, the claim is immediate. If}
	\(R<|v_1|/4\){, then}
	 \(|v|\sim|v_0|\sim|v_1|\){. Moreover,}
	\begin{align*}
		|v\cdot v_0|
		&=
		|(v_1+(v-v_1))\cdot (v_1+(v_0-v_1))|  \\
		&\ge
		|v_1|^2
		- |v-v_1|\,|v_1|
		- |v_0-v_1|\,|v_1|
		- |v-v_1|\,|v_0-v_1|  \\
		&\gtrsim
		|v_1|^2 .
	\end{align*}
	{Therefore,}
	\begin{align*}
		J
		\ge
		\frac{|v\cdot v_0|}{\langle v_0\rangle}
		\gtrsim
		\frac{|v_1|^2}{\langle v_1\rangle}.
	\end{align*}
	{Hence,} \(\langle v_1\rangle\lesssim\langle J\rangle\){.
	This proves} \eqref{eq:v1-control} {and therefore}
	\eqref{eq:A-velocity-bound}{.}
	
	{To prove} \eqref{eq:A-space-velocity-bound}{, set}
	\begin{align*}
		B:=|x-x_1|+|v-v_1|+|v_0-v_1|.
	\end{align*}
	{By} \eqref{eq:v1-control} {and} \(R\le B\){,}
	\begin{align}
		\langle v_1\rangle
		\lesssim
		\langle J\rangle \langle B\rangle .
		\label{eq:v1-control-B}
	\end{align}
	{Consequently,}
	\begin{align}
		\langle v_0\rangle
		&\lesssim
		\langle v_1\rangle+\langle v_0-v_1\rangle  \lesssim
		\langle J\rangle \langle B\rangle .
		\label{eq:v0-control}
	\end{align}
	{Since} \(t\in(0,1)\){,}

	\begin{align*}
		|x|+|v|
		\le
		\left[\frac{x}{t}\right]_{v_0}+[v]_{v_0}
		\le
		J\langle v_0\rangle .
	\end{align*}
	{Therefore,}
	\begin{align*}
		|x_1|+	|v_1|
		&\le
		|x|+|x-x_1|+	|v|+|v-v_1|
		\lesssim
		\langle J\rangle^2\langle B\rangle.
	\end{align*}
	{This proves} \eqref{eq:A-space-velocity-bound}{.}
\end{proof}
		\section{Short-time well-posedness and continuation}\label{secwp}
		This section establishes the short-time well-posedness theory in weighted
potential spaces, which is the main contribution of the nonlinear analysis.
We first develop the solution norms and linear estimates, and then prove
Theorem~\ref{Th1} by contraction and a separate nonnegativity argument.
The continuation and long-time estimates of Corollary~\ref{thmglo}
are treated subsequently in Subsection~\ref{secglo}, using the established
hypocoercive framework and retaining the main estimates needed for the
positive-time restart.
We begin with the local solution norms for the Boltzmann equation,
$0<s<1$, and the Landau equation, $s=1$. Fix
$\varepsilon_0\in((2s-1)_+,s)$ when $0<s<1$, and
$\varepsilon_0\in(0,1)$ when $s=1$.
		Set
		\begin{equation}\label{def:sigma0}
			\sigma_0:=\min\left\{\frac{d-\kappa}{2s},
			\frac{\varepsilon_0}{2s},\frac1{10}\right\}.
		\end{equation}
        {For real $\varkappa_1,\varkappa_2\geq0$ and $x,v\in\mathbb{R}^d$, define the weight}
		\begin{align}
			&	\mathbf{w}(x,v)=\langle x, v\rangle^{\varkappa_1} \langle  v\rangle^{\varkappa_2}.\label{defwe1}
		\end{align}
The polynomial weight estimates do not require integral exponents: for every
real $a$ and multi-index $\alpha$,
$|\partial^\alpha\langle z\rangle^a|
\leq C_{a,\alpha}\langle z\rangle^{a-|\alpha|}$.

		{Recall that}
		$$\tilde t_v=\langle v\rangle^{-\kappa}t.$$
        {For $m,k\in\mathbb{N}$ and $f:[0,T]\times\mathbb{R}^d\times \mathbb{R}^d\to \mathbb{R}$, define}
		\begin{equation}\label{defyz}
			\begin{aligned}
				&	\|f(t)\|_{ Y_{m}^k}=\sup_{x,v}	\mathbf{w}(x,v)\tilde t_v^{\frac{d-\kappa}{2s}}|\mathcal{D}_{t,v_0}^{(v),k}\nabla_{t,v_0}^{(x,v),{m}}f|(t,x,v)\Big|_{v_0=v},\\
				&	\|f(t)\|_{ Y_{m}^{k+\varepsilon_0}}=\sup_{x,v}	\mathbf{w}(x,v)\tilde t_v^{\frac{d-\kappa}{2s}} \sup_{0<[z]_v\leq 1}\Big({(\tilde t_v^\frac{1}{2s}[z]_v^{-1})^{\varepsilon_0}|\mathcal{D}_{t,v_0}^{(v),k}\nabla_{t,v_0}^{(x,v),m}\delta_z^vf|(t,x,v)}\Big)\!\Big|_{v_0=v},\\
				&	\|f(t)\|_{ Z_{m}^k}=\sup_{x,v}	\mathbf{w}(x,v)|\mathbf{I}\mathcal{D}_{t,v_0}^{(v),k}\nabla_{t,v_0}^{(x,v),m}f|(t,x,v)\Big|_{v_0=v},\\
				&	\|f(t)\|_{ Z_{m}^{k+\varepsilon_0}}=\sup_{x,v}\mathbf{w}(x,v)	 \sup_{0<[z]_v\leq 1}\Big((\tilde  t_v^{\frac{1}{2s}}[z]_v^{-1})^{\varepsilon_0}|\mathbf{I}\mathcal{D}_{t,v_0}^{(v),k}\nabla_{t,v_0}^{(x,v),m}\delta_z^vf|(t,x,v)\Big)\!\Big|_{v_0=v}.
			\end{aligned}
		\end{equation}
{In \eqref{defyz}, $v_0$ specifies the
derivative frame, held fixed under $\mathbf I$ and then set to $v$.
The supremum in \eqref{defI12} uses an independent anisotropy parameter;
the coefficient-freezing center for the Duhamel families is specified below.}
		{When \(k=0\), we write \(Y_m:=Y_m^0\) and \(Z_m:=Z_m^0\). Fix $N_0\geq 10d$. For $f:[0,T]\times\mathbb{R}^d\times \mathbb{R}^d\to \mathbb{R}$, define}
		\begin{align}\label{defnormT}
			\|f\|_T:=	\|f\|_{T,1}+	\|f\|_{T,2},
		\end{align}
		with
		\begin{align}
			\|f\|_{T,1}:=\sup_{t\in (0,T]}\sum_{m=0}^{ N_0}\left(\|f(t)\|_{Y_{m}}+\|f(t)\|_{Y_{m}^{\lfloor s\rfloor+\varepsilon_0}}\right),\label{defT1}\\	\label{defT2}
			\|f\|_{T,2}:=\sup_{t\in (0,T]}\sum_{m=0}^{ N_0}\left(\|f(t)\|_{Z_{m}}+\|f(t)\|_{Z_{m}^{\lfloor s\rfloor+\varepsilon_0}}\right).
		\end{align}
\begin{remark}
\begin{enumerate}
    \item The time powers follow the homogeneous kinetic scaling. The fixed localization and polynomial weights impose additional conditions; see Remark~\ref{rem:localized-criticality}.
    \item {The norm \(\|f\|_{T,1}\) controls the pointwise size of the solution and its derivatives, whereas \(\|f\|_{T,2}\) controls the coefficient fields \(a\star f\) in the Landau case and \(\mathbf C_f\) in the Boltzmann case.}
\end{enumerate}
\end{remark}

Theorem~\ref{Th1} is proved in Subsection~\ref{secproof}; the preceding
norms are used throughout its linear and nonlinear estimates.
		
		\subsection{Weighted convolution estimates}
For each integer $N\geq0$, define $\Phi_{t,v_0}^{[N]}$ below.
We write $\Phi_{t,v_0}=\Phi_{t,v_0}^{[N]}$ with $N$ fixed sufficiently
large for the weight and derivative order under consideration.
Constants may depend on $N$.

			Define
		\begin{align}\label{defPh}
	{\Phi_{t,v_0}^{[N]}(x,v)}=t^{-d}\tilde t_{v_0}^{-\frac{d}{s}}\langle v_0\rangle^2\left(\mathcal{N}_{v_0} \left({\frac{x+tv}{t\tilde t_{v_0}^{1/(2s)}}},\frac{v}{\tilde t_{v_0}^\frac{1}{2s}}\right)\wedge \left\langle\frac{[x]_{v_0}/t+[v]_{v_0}}{\langle v_0\rangle}\right\rangle^{{-N}}\right).
		\end{align}
{Here $\mathcal N_{v_0}$ is defined in
\eqref{anisotropic-HZ-profile}.
{Lemma~\ref{Hmu} for $s<1$, and the Gaussian
bound in Lemma~\ref{pwthm} for $s=1$, give}}
		\begin{align}\label{Hmues}
			\left| \mathcal{D}_{t,v_0}^{(x,v),m}H_\mu^{v_0}(t,x,v)\right|{\lesssim_{m,N}}	\Phi_{t,v_0}(x,v),\ \ \ \ \forall m\in\mathbb{N}.
		\end{align}
        {The elementary inequality $A\wedge B\leq A^{\theta}B^{1-\theta}$, valid for every $\theta\in[0,1]$, yields}
    \begin{align*}
    		\left| \Phi_{t,v_0}(x,v)\right|{\lesssim_{\theta,N}} t^{-d}&\tilde t_{v_0}^{-\frac{d}{s}}\langle v_0\rangle^2\left(\mathcal{N}_{v_0} \left({\frac{x+tv}{t\tilde t_{v_0}^{1/(2s)}}},\frac{v}{\tilde t_{v_0}^\frac{1}{2s}}\right)\right)^\theta \left\langle\frac{[x]_{v_0}/t+[v]_{v_0}}{\langle v_0\rangle}\right\rangle^{{-N(1-\theta)}}.
    \end{align*}    
For $1-s/(d+2s)<\theta<1$, the unit-Jacobian shear
$x\mapsto x+tv$ gives the following estimate.
\begin{align}
\label{intNthe} 
 t^{-d}&\tilde t_{v_0}^{-\frac{d}{s}}\langle v_0\rangle^2\iint_{\mathbb{R}^{2d}}\left(\mathcal{N}_{v_0} \left({\frac{x+tv}{t\tilde t_{v_0}^{1/(2s)}}},\frac{v}{\tilde t_{v_0}^\frac{1}{2s}}\right)\right)^\theta  \dif x\dif v {\lesssim_\theta} 1,\quad\quad \forall t>0,\ v_0\in\mathbb{R}^d.
\end{align}
\begingroup
Indeed, for $0<\delta<2s$, the isotropic profile
$\mathcal N(x,v)=\langle |x|+|v|\rangle^{-d-2s}
\int_0^1\langle x-\lambda v\rangle^{-d-2s}\,\dif\lambda$ obeys
\begin{align*}
\iint\langle x,v\rangle^\delta\mathcal N(x,v)\,\dif x\dif v
&\lesssim\int_0^1\iint
\langle z,v\rangle^{\delta-d-2s}\langle z\rangle^{-d-2s}
\,\dif z\dif v\dif\lambda<\infty,
\qquad z=x-\lambda v.
\end{align*}
Since $\mathcal N(x,v)\gtrsim\langle x,v\rangle^{-2(d+2s)}$,
choosing $\delta=2(d+2s)(1-\theta)<2s$ gives
$\mathcal N^\theta\in L^1(\mathbb R^{2d})$.
Uniformly in $v_0$,
$\mathcal N_{v_0}(x,v)\asymp
\mathcal N(\mathcal O_{v_0}^{-1}x,\mathcal O_{v_0}^{-1}v)$;
the change of variables has Jacobian
$|\det\mathcal O_{v_0}|^2=\langle v_0\rangle^{-2}$.
This change of variables and the shear prove \eqref{intNthe}.
\par\endgroup
		\begin{definition}
			[Admissible weight] 
			\label{defad}
            	A positive function \(\mathbf q:\mathbb R^{2d}\to(0,\infty)\)
			{is called admissible if there exist
$N\geq0$ and comparison constants, depending only on $\mathbf q$,
such that the following estimates hold uniformly for
$t\in(0,1)$ and $x,v,y,w\in\mathbb R^d$:}
			\begin{align*}
				&	{\mathbf{q}(x-tv,v)\asymp \mathbf{q}(x,v)},\\
                &\left\langle\frac{[x-y]_{v}/t}{\langle v\rangle}\right\rangle^{{-N}}\mathbf{q}^{-1}(y,v)\lesssim \mathbf{q}^{-1}(x,v),\\
            &\left\langle\frac{[v-w]_{v}}{\langle v\rangle}\right\rangle^{{-N}}\mathbf{q}^{-1}(x,w)\lesssim \mathbf{q}^{-1}(x,v).
			\end{align*}
		\end{definition}
\begingroup
The weights $\langle x,v\rangle^\alpha\langle v\rangle^\beta$ used
below, with $\alpha,\beta\geq0$, are admissible by the triangle inequality.
\par\endgroup
{Increasing $N$ if necessary, every admissible
weight $\mathbf q$ satisfies}
    \begin{align*}
    \left\langle\frac{[x-y]_{v}/t+[v-w]_{v}}{\langle v\rangle}\right\rangle^{{-N}}\mathbf{q}^{-1}(y,w)\lesssim \mathbf{q}^{-1}(x,v).
    \end{align*}
    \begingroup
Using $A\wedge B\leq A^\theta B^{1-\theta}$ in \eqref{defPh},
choose $\theta$ as in \eqref{intNthe}. The rapid factor absorbs the
weight comparisons in Definition~\ref{defad}, and \eqref{intNthe}
bounds the remaining integral uniformly at $v_0=v$. Thus
\endgroup
\begin{align}\label{weighttrans}
 \iint_{\mathbb{R}^{2d}} \Phi_{t,v_0}(x-y,v-u)(\mathbf{q}^{-1}(y, u)+ \mathbf{q}^{-1}(y, v)+\mathbf{q}^{-1}(x, u))\ \dif y \dif u \Big|_{v_0=v} \lesssim \mathbf{q}^{-1}(x, v).
\end{align}  
\vspace{0.3cm}
		
The following lemma gives the weighted convolution estimates used in the
$Y_m^k$- and $Z_m^k$-norms.
		\begin{lemma}\label{le11}
			Let $\mathbf q$ be an admissible weight and choose the exponent
$N$ in \eqref{defPh} sufficiently large depending on $\mathbf q$.
For a measurable function $h:\mathbb R^{2d}\to\mathbb R$, define
			\begin{align}\label{defGv}
				\mathcal{G}_{v_0}h(t,x,v)=\iint_{\mathbb{R}^{2d}} \Phi_{t,v_0}(y,u)|h|(x-tv-y,v-u) \,\dif y\dif u.
			\end{align}
			Then, for $0<t\leq\frac12$,
			\begin{align}
				&\sup_{x,v}\left(\mathbf{q}(x,v)	\mathcal{G}_{v_0}h(t,x,v)|_{v_0=v}\right)\lesssim \sup_{x,v}\big(\mathbf{q}(x,v)|h(x,v)|\big),\label{G11}\\
				\label{G00}
				&\sup_{x,v}\left(\tilde  t_{v_0}^{\frac{d-\kappa}{2s}}\mathbf{q}(x,v)	\mathcal{G}_{v_0}h(t,x,v)|_{v_0=v}\right)\lesssim \sup_{x,v}\big(\mathbf{q}(x,v)\mathbf{I}h(x,v)\big),\\
				&\sup_{x,v}\left(\mathbf{q}(x,v)(\mathbf{I}	\mathcal{G}_{v_0}h)(t,x,v)|_{v_0=v}\right)\lesssim \sup_{x,v}\big(\mathbf{q}(x,v)\mathbf{I}h(x,v)\big).\label{I1G}
			\end{align}
			{Moreover, define}
			\begin{align*}
				&	\mathcal{Z}_{v_0}h(t,x,v)=\iint_{\mathbb{R}^{2d}} \Phi_{t,v_0}(y,u)|\delta^v_uh|(x-tv-y,v)\,\dif y\dif u.
			\end{align*}
			{Then, for every $a\in(0,1)$,}
			\begin{align}
				\sup_{x,v}&\left(\mathbf{q}(x,v)	\mathcal{Z}_{v_0}h(t,x,v)|_{v_0=v}\right)\nonumber\\
				&	\lesssim \sup_{x,v}\mathbf{q}(x,v)\Bigg(|h|(x,v)+\sup_{0<[z]_{v}\leq 1}(\tilde t_{v}^{-\frac{1}{2s}}[z]_{v})^{-a}{|\delta^v_zh|(x,v)}\Bigg),\label{G11d}\\
				\sup_{x,v}&\left(\mathbf{q}(x,v)(\mathbf{I}	\mathcal{Z}_{v_0}h)(t,x,v)|_{v_0=v}\right)\nonumber\\
				&\lesssim \sup_{x,v}\mathbf{q}(x,v)\Bigg(|\mathbf{I}h|(x,v)+\sup_{0<[z]_{v}\leq 1}\Big((\tilde t_{v}^{-\frac{1}{2s}}[z]_{v})^{-a}{|\mathbf{I}\delta^v_zh|(x,v)}\Big)\Bigg).\label{I1Gd}
			\end{align}
			The implicit constants depend on $\mathbf q$, $N$, and the fixed
parameters, but are independent of $t$ and $h$. In the estimates involving
$\mathbf I$, the parameter $v_0$ is held fixed until evaluation at $v_0=v$.
		\end{lemma}
\begin{proof}
\begingroup
Fix $t,x,v$, with $0<t\leq\frac12$. Throughout the proof, the
subscript $v$ in $\Phi_{t,v}$, $\mathcal G_v$, and $\mathcal Z_v$
denotes the fixed parameter $v_0=v$, including when $\mathbf I$ is applied.
By \eqref{weighttrans} and the transport comparison in
Definition~\ref{defad},
\[
\mathcal G_vh(t,x,v)
\leq\|\mathbf qh\|_{L^\infty}
\iint\Phi_{t,v}(y,u)\mathbf q^{-1}(x-tv-y,v-u)\,\dif y\dif u
\lesssim\|\mathbf qh\|_{L^\infty}\mathbf q^{-1}(x,v),
\]
which proves \eqref{G11}.

To prove \eqref{G00}, partition velocity space into the boxes
\[
Q_k=\varepsilon\tilde t_v^{1/(2s)}\mathcal O_v
\bigl(k+[-\tfrac12,\tfrac12]^d\bigr),
\qquad k\in\mathbb Z^d,
\]
with centers $u_k$, where $\varepsilon>0$ is a sufficiently small
dimensional constant. On each box $tQ_j\times Q_k$, the normalized
arguments of $\mathcal N_v$ in \eqref{defPh} vary by bounded amounts.
Hence \eqref{eq:A10-profile-moderation} and the local comparability of
the polynomial factor in \eqref{defPh} give
\[
\sup_{tQ_j\times Q_k}\Phi_{t,v}
\lesssim\inf_{tQ_j\times Q_k}\Phi_{t,v}.
\]
Admissibility also gives, for $(y,u)\in tQ_j\times Q_k$,
\[
\mathbf q^{-1}(x-tv-tu_j,v-u_k)
\asymp\mathbf q^{-1}(x-tv-y,v-u).
\]
Indeed, $|y-tu_j|/t+|u-u_k|\lesssim\varepsilon$, and
$[z]_{v'}/\langle v'\rangle\leq2|z|$ for every $v'$.
Moreover, choosing the reference velocity $v$ in \eqref{defI12} yields
\[
\int_{Q_k}|h|(x-tv-y,v-u)\,\dif u
\lesssim\frac{\tilde t_v^{\kappa/(2s)}}{\langle v\rangle}
(\mathbf I h)(x-tv-tu_j,v-u_k),
\qquad y\in tQ_j.
\]
Here $[u-u_k]_v\leq\tilde t_v^{1/(2s)}\leq1$ and
$|y-tu_j|\leq1$, so both localizations in \eqref{defI12} apply.
Since $|Q_k|=\varepsilon^d\tilde t_v^{d/(2s)}/\langle v\rangle$,
\[
\begin{aligned}
\tilde t_v^{(d-\kappa)/(2s)}\mathcal G_vh(t,x,v)
&\lesssim\sum_{j,k}|tQ_j|\,|Q_k|
\sup_{tQ_j\times Q_k}\Phi_{t,v}\,
(\mathbf I h)(x-tv-tu_j,v-u_k)\\
&\lesssim\|\mathbf q\mathbf I h\|_{L^\infty}
\iint\Phi_{t,v}(y,u)\mathbf q^{-1}(x-tv-y,v-u)\,\dif y\dif u\\
&\lesssim\|\mathbf q\mathbf I h\|_{L^\infty}\mathbf q^{-1}(x,v).
\end{aligned}
\]
This proves \eqref{G00}.

For \eqref{I1G}, the spatial argument of $h$ in
$\mathbf I\mathcal G_vh$ is $x-tv-y+x_1-tw$, where
$|x_1|\leq1$ and $|w|\leq1$. Choose a fixed finite covering
$B_{3/2}(0)\subset\bigcup_\ell B_1(c_\ell)$.
Positivity and \eqref{defI12} then imply
\begin{align}\label{eq:potential-transport-cover}
(\mathbf I\mathcal G_vh)(t,x,v)
&\lesssim\sum_\ell\iint\Phi_{t,v}(y,u)
(\mathbf I h)(x-tv-y+c_\ell,v-u)\,\dif y\dif u.
\end{align}
Apply \eqref{weighttrans} at $(x-tv+c_\ell,v)$ and use
\[
\mathbf q^{-1}(x-tv+c_\ell,v)
\lesssim\mathbf q^{-1}(x-tv,v)
\lesssim\mathbf q^{-1}(x,v),
\]
which follows from admissibility and the boundedness of $c_\ell$.
This proves \eqref{I1G}.

Finally, the triangle inequality gives
\[
|\delta_u^v h|(x,v)\leq |h|(x,v)+|h|(x,v-u),
\qquad
(\mathbf I\delta_u^v h)(x,v)
\leq(\mathbf I h)(x,v)+(\mathbf I h)(x,v-u),
\]
since velocity translations commute with $\mathbf I$.
Using \eqref{weighttrans} for the two terms gives
$\mathbf q(x,v)\mathcal Z_vh(t,x,v)
\lesssim\|\mathbf qh\|_{L^\infty}$.
The same covering argument as in \eqref{eq:potential-transport-cover},
applied to $\delta_u^v h$ for each fixed $u$, gives
$\mathbf q(x,v)(\mathbf I\mathcal Z_vh)(t,x,v)
\lesssim\|\mathbf q\mathbf I h\|_{L^\infty}$.
These bounds imply \eqref{G11d} and \eqref{I1Gd}.
\par\endgroup
\end{proof}

\begingroup
We next establish the recentered convolution bounds for $\mathbf I$.
\begin{lemma}\label{lem:recentered-convolution}
Let $\mathbf q$ be an admissible weight and $0<t\leq1$.
Choose the exponent $N$ in \eqref{defPh} sufficiently large depending on
$\mathbf q$ (and on $\beta$ for the second estimate). Then
\begin{align}\label{eq:recentered-positive-convolution}
\sup_{x,v}\mathbf q(x,v)\,
\mathbf I\!\left(\left.\mathcal G_{v_0}h(t,x,u)
\right|_{v_0=u}\right)(x,v)
&\lesssim \|\mathbf q\mathbf I h\|_{L^\infty}.
\end{align}
For $0<\beta\leq1$, $\beta<2s$, one also has
\begin{align}\label{eq:recentered-potential-moment}
&\sup_{x,v}\mathbf q(x,v)\,
\mathbf I\!\left(
\iint\Phi_{t,u}(y,z)(1\wedge|z|)^\beta
|h|(x-tu-y,u-z)\,\dif y\dif z\right)(x,v)\notag\\
&\hspace{35mm}\lesssim
t^{\beta/(2s)}\|\mathbf q\mathbf I h\|_{L^\infty}.
\end{align}
Scaled kernel derivatives of fixed order may be taken in the external
frame centered at $v$. The same bounds hold for uniformly bounded
velocity shifts, with the moment factor translated together with the
kernel. Constants may depend on $\mathbf q$, $N$, $\beta$, the derivative
order, the shift bound, and the structural parameters, but not on $t$ or $h$.
\end{lemma}
\begin{proof}
Every sample $u=v+w$ in \eqref{defI12} satisfies $|w|\leq1$.
Uniformly for such $w$ and every $z$,
\[
\langle v+w\rangle\asymp\langle v\rangle,\qquad
[z]_{v+w}\asymp[z]_v,\qquad
\tilde t_{v+w}\asymp\tilde t_v,\qquad
\Phi_{t,v+w}(y,z)\lesssim\Phi_{t,v}(y,z).
\]
The last comparison follows directly from \eqref{defPh} and
\eqref{anisotropic-HZ-profile}; the derivative frames obey the same
local comparison. Dominate the sampled positive kernel by
$\Phi_{t,v}$ before applying the covering argument in
\eqref{eq:potential-transport-cover}. This proves
\eqref{eq:recentered-positive-convolution} (use a covering of $B_2$
when $t>1/2$).

For \eqref{eq:recentered-potential-moment}, retain
$(1\wedge|z|)^\beta\leq[z]_v^\beta$ in that argument and use the
$\beta$-moment calculation in \eqref{eq:A27-subcritical-frozen-moment}.
Since $\tilde t_v\leq t$, this gives the stated gain. For a shifted
envelope, translate the kernel and moment factor together; the resulting bounded
shifts of the source are handled by the same finite covering and
admissibility of $\mathbf q$.
\end{proof}
\par\endgroup

		\begin{lemma}\label{lemi1i2}
			{For $g,h:\mathbb{R}^d\times\mathbb{R}^d\to\mathbb{R}$, one has}
			\begin{align*}
				&\mathbf{I} (gh)(x,v)\leq \mathbf{I} h(x,v)\sup_{|x_1|,|w|\leq 1}|g(x+x_1,v+w)|.
			\end{align*}
		\end{lemma}
		{Lemma \ref{lemi1i2} follows directly from \eqref{defI12}, so we omit its proof.}
		
		{The following lemma records the interpolation inequalities for the $Y_m$- and $Z_m$-norms.}
		\begin{lemma}\label{lemitp} {For any $m,k\in\mathbb{N}$, $\varepsilon\in(0,1)$, and $h:[0,T]\times \mathbb{R}^d\times \mathbb{R}^d\to\mathbb{R}$, one has}
			\begin{align*}
				&\|h(t)\|_{Y_{m}^{k+\varepsilon}}\lesssim\|h(t)\|_{Y_{m}^k}+\|h(t)\|_{Y_{m}^{k+1}},\ \ \ ~~ \|h(t)\|_{Z_{m}^{k+\varepsilon}}\lesssim\|h(t)\|_{Z_{m}^k}+\|h(t)\|_{Z_{m}^{k+1}}.
			\end{align*}
		\end{lemma}
			\begin{proof}
			Fix $t,x,v$ and set
			\[
				F_{v_0}:=\mathcal D_{t,v_0}^{(v),k}\nabla_{t,v_0}^{(x,v),m}h,
				\qquad \rho:=\tilde t_v^{-\frac1{2s}}[z]_v.
			\]
			If $[z]_v\leq1$, then the weights, time scales, and frozen derivative frames at $v$ and $v-\theta z$, $0\leq\theta\leq1$, are uniformly comparable.  Hence the triangle inequality and the fundamental theorem of calculus give
			\begin{align*}
				\mathbf w(x,v)\tilde t_v^{\frac{d-\kappa}{2s}}|\delta_z^vF_{v_0}(x,v)|_{v_0=v}|
				&\lesssim\min\big\{\|h(t)\|_{Y_m^k},\rho\|h(t)\|_{Y_m^{k+1}}\big\},\\
				\mathbf w(x,v)\mathbf I(\delta_z^vF_{v_0})(x,v)|_{v_0=v}
				&\lesssim\min\big\{\|h(t)\|_{Z_m^k},\rho\|h(t)\|_{Z_m^{k+1}}\big\}.
			\end{align*}
			{Indeed, the identity}
			\[
				z\cdot\nabla_v=\tilde t_v^{-1/(2s)}(\mathcal O_v^{-1}z)\cdot\mathcal D_{t,v}^{(v)},
				\qquad |\mathcal O_v^{-1}z|\sim[z]_v,
			\]
			{gives the derivative estimate. The same argument applies within the positive, subadditive operator $\mathbf I$. Since ${\rho^{-\varepsilon}}\min\{A,\rho B\}\leq A+B$, taking the supremum over $(x,v,z)$ proves the result.}
		\end{proof}

		\subsection{\texorpdfstring{{Estimates for the linear equation}}{Estimates for the linear equation}} \label{seclinear}
		{To establish local existence, we first consider the linear problem}
		\begin{equation}\label{linlo}
			\begin{aligned}
				&\partial_t g+v\cdot \nabla_x g+\mathcal{L}_\mu(g)=G,\quad\quad \text{in}\ (0,T)\times \mathbb{R}^d\times\mathbb{R}^d,\\
				&g|_{t=0}=f_0.
			\end{aligned}
		\end{equation}
		{Here \(G:[0,T]\times\mathbb{R}^d\times\mathbb{R}^d\to \mathbb{R}\) depends on prescribed functions \(h_i\), \(1\leq i\leq4\), and has the collision structure}
		\begin{align}\label{GQ} G = \mathcal{Q}_{s}(h_1,h_2) + \mathcal{Q}_{s}(h_3,\mu) + \mathcal{Q}_{s,\mathrm{rem}}(\mu,h_4), \end{align}
		{where \(\mathcal{Q}_{s,\mathrm{rem}}\) is defined in \eqref{qsmr}.}
{We first establish the a priori estimate;
the linear solutions used in the fixed-point argument are constructed
at the beginning of Subsection~\ref{secproof}.}
		{The following lemma provides the key estimate for \eqref{linlo}.}\par\medskip
\begin{lemma}\label{lemloc} {Let \(0<T\leq1\), and suppose that \(g\) solves \eqref{linlo} on
\([0,T]\times\mathbb R^d\times\mathbb R^d\) with the source term \(G\) given by \eqref{GQ}. Then there exists a constant \(C>0\), depending only on the fixed structural parameters, such that}
\begin{align}\label{largt} \|g\|_T \leq C \Big( \|f_0\|_{\mathrm{cr}} + \|h_1\|_{T}\|h_2\|_T + T^{\sigma_0}\sum_{i=3,4}\|h_i\|_T \Big). \end{align} \end{lemma}

{To estimate the $\|\cdot\|_T$-norm, fix $v_0\in\mathbb{R}^d$ and rewrite \eqref{linlo} as}
		\begin{equation*}
			\begin{aligned}
				&\partial_t g+v\cdot \nabla_x g+\mathcal{L}^{v_0}_\mu g=G+\mathcal{R}^{v_0}(g),\\
				&	g|_{t=0}=f_0,
			\end{aligned}
		\end{equation*}
		where
		\begin{align}\label{defRe0}
			\mathcal{R}^{v_0}(g)=\mathcal{L}^{v_0}_\mu g-\mathcal{L} _\mu g.
		\end{align}
		{Corollary \ref{coroduha} gives}
		\begin{align}
			\label{forg}
			g=g_{1,v_0}+g_{2,v_0}+g_{3,v_0},
		\end{align}
		where
		\begin{align*}
			&g_{1,v_0}(t,x,v)=(H_\mu^{v_0}(t)\ast	f_0)(x-tv,v),\\
			&g_{2,v_0}(t,x,v)=\int_0^t (H_\mu^{v_0}(t-\tau)\ast  G(\tau))(x-(t-\tau)v,v)\dif \tau ,\\
			&g_{3,v_0}(t,x,v)=\int_0^t \left(H_\mu^{v_0}(t-\tau)\ast  \mathcal{R}^{v_0}(g)(\tau)\right)(x-(t-\tau)v,v)\dif \tau .
		\end{align*} 
		{Here $H_\mu^{v_0}(t,x,v)$ is the linear kernel defined in \eqref{defH0}.}\par\medskip
{We regard the right-hand side of \eqref{forg} as a function of the additional parameter \(v_0\). {Let}}
\[ F:[0,T]\times(\mathbb R^d)^3\to\mathbb R, \qquad F=F(t,x,v,v_0). \]
{Take derivatives and velocity differences
at fixed $v_0$, then set $v_0=v$ in the pointwise norms. For example,}
\begin{align}\label{tilYn} \|F(t)\|_{\tilde Y_m^k} := \sup_{x,v} \mathbf w(x,v)\, \tilde t_v^{\frac{d-\kappa}{2s}} | \mathcal{D}_{t,v_0}^{(v),k} \nabla_{t,v_0}^{(x,v),m} F| (t,x,v,v_0) \Big|_{v_0=v}. \end{align}
		{With a slight abuse of notation, for $h:[0,T]^2\times (\mathbb{R}^d)^3\to \mathbb{R}$, define}
		\begin{align}\label{tilYntt} 
			&	\|h(t,\tau)\|_{\tilde Y_{m}^k}=\sup_{x,v}	\mathbf{w}(x,v)\sum_{t'\in\{t,\tau\}}\tilde t_v'^{\frac{d-\kappa}{2s}}|\mathcal{D}_{t',v_0}^{(v),k}\nabla^{(x,v),m}_{t',v_0}h|(t,\tau,x,v,v_0)\Big|_{v_0=v}.
		\end{align}
\begingroup
For the potential norms, keep the derivative frame at the external
velocity $v$ and freeze coefficients at the sampled velocity $u$.
The superscript $(u)$ denotes differentiation in $u$. Define
		\begin{equation*}
	\begin{aligned}
		\|h(t,\tau)\|_{\widetilde Y_m^{k+\varepsilon_0}}
		&:=\sup_{x,v}\mathbf w(x,v)
		\sum_{t'\in\{t,\tau\}}\widetilde t'_v{}^{\frac{d-\kappa}{2s}}
		\sup_{0<[z]_v\le1}
		\big(\widetilde t'_v{}^{\frac1{2s}}[z]_v^{-1}\big)^{\varepsilon_0}
		\big|\mathcal D_{t',v_0}^{(v),k}
		\nabla_{t',v_0}^{(x,v),m}\delta_z^v h\big|
		(t,\tau,x,v,v_0)\Big|_{v_0=v},\\
		\|h(t,\tau)\|_{\widetilde Z_m^k}
		&:=\sup_{x,v}\mathbf w(x,v)
		\sum_{t'\in\{t,\tau\}}
		\mathbf I\!\left(
        \left.\mathcal D_{t',v}^{(u),k}
        \nabla_{t',v}^{(x,u),m}h(t,\tau,x,u,v_0)
        \right|_{v_0=u}\right)(x,v),\\
		\|h(t,\tau)\|_{\widetilde Z_m^{k+\varepsilon_0}}
		&:=\sup_{x,v}\mathbf w(x,v)
		\sum_{t'\in\{t,\tau\}}\sup_{0<[z]_v\le1}
		\big(\widetilde t'_v{}^{\frac1{2s}}[z]_v^{-1}\big)^{\varepsilon_0}
		\mathbf I\!\left(
        \left.\mathcal D_{t',v}^{(u),k}
        \nabla_{t',v}^{(x,u),m}\delta_z^u h(t,\tau,x,u,v_0)
        \right|_{v_0=u}\right)(x,v).
		\end{aligned}
		\end{equation*}
Here the external $v$ in the derivative frame is held fixed under
$\mathbf I$. In particular, both endpoints of $\delta_z^u h$ use the
same freezing center $v_0=u$. The one-time norms are obtained by omitting
the sum and taking $t'=t$.
\par\endgroup
		{Set}
		\begin{align}\label{defnormT1}
			\lceil F\rfloor_T:=		\lceil F\rfloor_{T,1}+	\lceil F\rfloor_{T,2},
		\end{align}
		with
		\begin{equation}\label{defT12}
			\begin{aligned} 
			\lceil F\rfloor_{T,1}
			:=
			\sup_{0<t\le T}
			\sum_{m=0}^{N_0}
			\left(
			\|F(t)\|_{\widetilde Y_m}
			+\|F(t)\|_{\widetilde Y_m^{\lfloor s\rfloor+\varepsilon_0}}
			\right),
			\end{aligned}
		\end{equation}
		\begin{equation}\label{defT22}
			\begin{aligned}
				\lceil F\rfloor_{T,2}
			:=
			\sup_{0<t\le T}
			\sum_{m=0}^{N_0}
			\left(
			\|F(t)\|_{\widetilde Z_m}
			+\|F(t)\|_{\widetilde Z_m^{\lfloor s\rfloor+\varepsilon_0}}
			\right).
			\end{aligned}
		\end{equation}
{We apply this family norm directly to
$g_{i,v_0}(t,x,v)$ and write $\lceil g_{i,v_0}\rfloor_T$.}
{Identity \eqref{forg} holds for every fixed
$v_0$, so after differentiation or differencing we may set $v_0=v$
pointwise and $v_0=v+w$ under $\mathbf I$. Hence}
\begin{align}\label{esSf1}
    \|g\|_T\leq \sum_{i=1}^3 \lceil g_{i,v_0}\rfloor_T .
\end{align}
	{We now estimate the three terms on the right-hand side of \eqref{forg}.}\par\medskip
		{\bf 1. Estimate  $g_{1,v_0}$.}\par\medskip
		{We begin with the contribution from the initial data:}
		\begin{align*}
			g_{1,v_0}(t,x,v)&=(H_\mu^{v_0}(t)\ast f_0)(x-tv,v)\\
			&=\iint_{\mathbb{R}^{2d}}H_\mu^{v_0}(t,x-y-tv,v-u)f_0(y,u)\dif y\dif u.
		\end{align*}
			{The following proposition estimates $\lceil g_{1,v_0}\rfloor_T$, defined in \eqref{defnormT1}.}

		\begin{proposition} \label{Pro1} 	{For every $T\in (0,1)$,}
			\begin{align*}	\lceil 	g_{1,v_0}\rfloor_{T}\lesssim 	\|f_0\|_{\mathrm{cr}}.
			\end{align*}
		\end{proposition}
		\begin{proof}
\begingroup
The integer-order bounds follow from \eqref{Hmues}, \eqref{G00}, and
Lemma~\ref{lem:recentered-convolution}. To prove the fractional bounds,
fix $m\leq N_0$, put $k=\lfloor s\rfloor$, and keep $v_0$ fixed.
In the convolution defining $g_{1,v_0}$, the terminal difference
$\delta_z^v$ acts on the kernel as
\[
 H_{\mu}^{v_0}(t,A,B)-H_{\mu}^{v_0}(t,A+tz,B-z),
 \qquad A=x-tv-y,\quad B=v-u.
\]
Set $q=[z]_{v_0}/\tilde t_{v_0}^{1/(2s)}$.
For $q\leq1$, the fundamental theorem of calculus in the direction
$t z\cdot\nabla_x-z\cdot\nabla_v$ and \eqref{Hmues} at one higher
order give a factor $q$. The intervening normalized displacements
are bounded, so \eqref{eq:A10-profile-moderation} compares their
positive envelopes with the endpoint envelopes. For $q>1$, the
triangle inequality gives the same estimate with factor $1$.
Consequently,
\begin{align*}
 &|\mathcal D_{t,v_0}^{(v),k}\nabla_{t,v_0}^{(x,v),m}
 [H_{\mu}^{v_0}(t,A,B)-H_{\mu}^{v_0}(t,A+tz,B-z)]|\\
 &\qquad\lesssim(1\wedge q)
 [\Phi_{t,v_0}(A,B)+\Phi_{t,v_0}(A+tz,B-z)].
\end{align*}
Apply \eqref{G00} to the first envelope. In the second one, change
variables to shift the source by $(tz,-z)$; since $t\leq1$ and
$[z]_{v_0}\leq1$, admissibility of $\mathbf w$ and the finite
covering used in \eqref{eq:potential-transport-cover} control this
bounded shift. For the potential norm, do this for each fixed sample
$u=v+w$, with freezing center $v_0=u$ and external derivative frame $v$.
The comparisons $|w|\leq1$, $[z]_u\asymp[z]_v$, and
$\tilde t_u\asymp\tilde t_v$ allow
Lemma~\ref{lem:recentered-convolution} at both endpoints. Thus the
normalized pointwise and potential expressions are bounded by
$C(1\wedge q)\|f_0\|_{\mathrm{cr}}$.
Finally,
\[
 q^{-\varepsilon_0}(1\wedge q)\leq1,
 \qquad 0<\varepsilon_0<1,
\]
proves the $\widetilde Y_m^{k+\varepsilon_0}$ and
$\widetilde Z_m^{k+\varepsilon_0}$ bounds. For $s=1$, all kernel
bounds used here are \eqref{Hmues} from Lemma~\ref{pwthm}.
\par\endgroup
		\end{proof}\par\vspace{0.3cm}
		{\bf 2. Estimate $g_{2,v_0}$.}\par\medskip
		{We next estimate the contribution of the source term $G$ in \eqref{GQ}. Define}
		\begin{equation}
			\begin{aligned}\label{defsv0}
				\mathcal{S}_{v_0} (f_1,f_2)(t,\tau,x,v)=&(H_\mu^{v_0}(t)\ast \mathcal{Q}_{s}(f_1,f_2)(\tau))(x-tv,v)\\
				:=&	\mathcal{S}_{1,v_0} (f_1,f_2)(t,\tau,x,v)+	\mathcal{S}_{2,v_0} (f_1,f_2)(t,\tau,x,v),
			\end{aligned}
		\end{equation}
		where
		\begin{equation*}
			\begin{aligned}
				&	\mathcal{S}_{1,v_0} (f_1,f_2)(t,\tau,x,v)=(H_\mu^{v_0}(t)\ast \mathcal{Q}_{s,\mathrm{main}}(f_1,f_2)(\tau))(x-tv,v),\\
				&	\mathcal{S}_{2,v_0} (f_1,f_2)(t,\tau,x,v)=(H_\mu^{v_0}(t)\ast \mathcal{Q}_{s,\mathrm{rem}}(f_1,f_2)(\tau))(x-tv,v).
			\end{aligned}
		\end{equation*}
		{Consequently,}
		\begin{align}\label{intg2}
			g_{2,v_0}(t,x,v)=\int_0^t  ({\mathcal{S}}_{v_0} (h_1,h_2)+\mathcal{S}_{v_0} (h_3,\mu)+\mathcal{S}_{2,v_0} (\mu,h_4))(t-\tau,\tau,x,v)\dif \tau .
		\end{align}
		{Because both the kernel and the source terms have time singularities, we split the integral in \eqref{intg2} over $[0,t/2]$ and $[t/2,t]$. To estimate $\mathcal{S}_{i,v_0}(f_1,f_2)$ precisely, we introduce the following non-endpoint norm:}
		\begin{equation*}
			\begin{aligned}
				\|f\|_{T,*}:=&\sup_{t\in (0,T]}\sup_{x,v} \langle  v\rangle^{\varkappa}\sum_{m=0}^{N_0} \tilde t_{v}^{\frac{d-\kappa}{2s}}\left(|\nabla^{(x,v),m}_{t,v_0}f|(t,x,v)\big|_{v_0=v}+ \sup_{0<[z]_v\leq 1}\frac{|\mathcal{D}^{(v),\lfloor s\rfloor}_{t,v_0}\nabla_{t, v_0}^{(x,v),m}\delta_z^vf|(t,x,v)\big|_{v_0=v}}{(\tilde  t_v^{-\frac{1}{2s}}[z]_v)^{\varepsilon_0}}\right)\\
				&+\sup_{t\in (0,T]}\sup_{x,v} \langle  v\rangle^{\varkappa}\sum_{m=0}^{N_0}  \sup_{0<[z]_v\leq 1}\Big({(\tilde  t_v^\frac{1}{2s}[z]_v^{-1})^{\varepsilon_0}|\mathbf{I}\mathcal{D}^{(v),\lfloor s\rfloor}_{t,v_0}\nabla_{t,v_0}^{(x,v),m}\delta_z^vf|}(t,x,v)\Big)\bigg|_{v_0=v}\\
				&+T^{\frac{1}{10}}\sup_{t\in(0,T]}\sup_{x,v}  \langle v\rangle^{\varkappa}\mathbf{I}f(t,x,v),
			\end{aligned}
		\end{equation*}
		{where $\varkappa=\varkappa_1+\varkappa_2$. This velocity weight is controlled by the original weight because}
		\[
		\langle v\rangle^{\varkappa}
		\leq \langle x,v\rangle^{\varkappa_1}
		\langle v\rangle^{\varkappa_2}=\mathbf w(x,v).
		\]
		{Relative to $\|\cdot\|_T$ in \eqref{defnormT}, this norm omits the spatial weight and includes the factor $T^{1/10}$ in the zeroth-order critical term $|\mathbf I f(t,x,v)|$. Consequently, the Maxwellian is small in this norm for small $T$:}
		\begin{align}\label{musm}
			\|\mu\|_{T,*}\lesssim T^{\sigma_0}.
		\end{align}
		
		{We first estimate the nonlocal terms in $\mathcal{Q}_s(f_1,f_2)$. Corollary \ref{lemcoeff1} and Lemma \ref{lemA1} treat the convolution coefficient $a\star f$ in the Landau case, while Lemma \ref{lemCf1} treats $\mathbf C_f$ in the Boltzmann case. The following corollary is an immediate consequence of Lemma \ref{lemcoeff}.}\par\medskip
			\begin{corollary}\label{lemcoeff1}
			{Set $\tilde a_{v_0}(v)=\mathcal{O}_{v_0}^{-1}a(v)\mathcal{O}_{v_0}^{-1}$ and}
			\begin{align*}
				\mathcal{P}_{v_0}f(t,x,v)=(\tilde a_{v_0}\star f(t,x,\cdot))(v).
			\end{align*}
			{Then}
			\begin{align*}
				\sum_{m=0}^{N_0}\sup_{t\in[0,T]}\sup_{x,v}&\langle x \rangle^{\varkappa_1} \langle v\rangle^{\kappa}	\left|\nabla_{t,v_0}^{(x,v),m}		\mathcal{P}_{v_0}f\right|(t,x,v)
				\Big|_{v_0=v}\lesssim \|f\|_{T}.
			\end{align*}
		\end{corollary}
		{The next lemma estimates}
\[ \tilde t_v^{\frac12} \mathbf I \big(\tilde a_{v_0}\star \mathcal D_{v_0}^{(v)}f\big). \]
{The factor}
\[
    \tilde t_v^{\frac12}
    \big(\tilde a_{v_0}\star \mathcal D_{v_0}^{(v)}f\big)
\]
{is scaling-critical, with size \(O(1)\) in the \(\|\cdot\|_T\)-norm. Applying \(\mathbf I\), however, makes it lower order: \(\mathbf I\) behaves like a Riesz potential of order \(d-\kappa\), equivalently a derivative of order \(-(d-\kappa)\). Thus}
\[ \tilde t_v^{\frac12} \mathbf I\big(\tilde a_{v_0}\star \mathcal D_{v_0}^{(v)}f\big) \]
{is perturbative in the critical norm.}
\begin{lemma}\label{lemA1}
Let $s=1$. For every $0<t<T\leq1$ and $m\leq N_0$,
\begin{align}\label{afl}
&\tilde t_v^{1/2}
\left(\mathbf I\nabla_{t,v_0}^{(x,v),m}
 (\tilde a_{v_0}\star\mathcal D_{v_0}^{(v)}f)\right)(t,x,v)
 \Big|_{v_0=v}\notag\\
&\qquad\lesssim\langle v\rangle\langle x\rangle^{-\varkappa_1}
\left(\tilde t_v^{(d-\kappa)/2}
 +\mathbf1_{\kappa\leq d-1}\tilde t_v^{1/2}
 (1+|\log\tilde t_v|)\right)\|f\|_T.
\end{align}
\end{lemma}
\begin{proof}
Keep $v_0=v$ fixed, set $L=\langle v\rangle$, and write
$A_v(u)=\mathcal O_v^{-1}a(u)\mathcal O_v^{-1}$.
For $u\neq0$, the projection identity $\Pi_{\hat u}u=0$ gives
\[
 \|\mathcal O_v^{-1}\Pi_{\hat u}\|
 \lesssim1+|\Pi_{\hat u}v|\lesssim\langle v-u\rangle,
 \qquad \|\mathcal O_v^{-1}\hat u\|\leq L.
\]
Since $a(u)=|u|^{-\kappa}\Pi_{\hat u}$, it follows that
\begin{align}\label{eq:Landau-projection-derivative}
 |A_v(u)|&\lesssim |u|^{-\kappa}\langle v-u\rangle^2,\\
 |\mathcal D_v^{(u)}A_v(u)|
 &\lesssim L|u|^{-\kappa-1}\langle v-u\rangle^2.\notag
\end{align}
Indeed, for $b=\mathcal O_ve_j$,
\[
 \partial_b\Pi_{\hat u}
 =-|u|^{-1}\bigl(\Pi_{\hat u}b\otimes\hat u
 +\hat u\otimes\Pi_{\hat u}b\bigr).
\]
Sandwich this identity between $\mathcal O_v^{-1}$ and use the two
bounds above; differentiating $|u|^{-\kappa}$ gives the remaining term.
Thus the source factor $\langle v-u\rangle^2$ must be retained before
integration. The corresponding bounds for every sample $v+w$ in
$\mathbf I$ are unchanged up to a constant, because $|w|\leq1$.

For $m=0$, the $Z$ bounds and the local comparison of derivative frames give
\[
 \mathbf I f(t,x,v-u)\lesssim
 \langle x\rangle^{-\varkappa_1}\langle v-u\rangle^{-\varkappa_2}\|f\|_T.
\]
If $\kappa<d-1$, integration by parts in $u$ is legitimate at $u=0$;
the boundary term there is $O(\delta^{d-\kappa-1})$. Therefore
\begin{align*}
 \mathbf I(A_v\star\mathcal D_v^{(v)}f)(t,x,v)
 &\lesssim L\langle x\rangle^{-\varkappa_1}\|f\|_T
 \int |u|^{-\kappa-1}\langle v-u\rangle^{2-\varkappa_2}\,\dif u\\
 &\lesssim L\langle x\rangle^{-\varkappa_1}\|f\|_T.
\end{align*}
The last integral is uniformly bounded: on $|u|\leq1$ use
$\kappa+1<d$, and on $|u|>1$ use $\varkappa_2>d+2$.

If $d-1\leq\kappa<d$, choose a smooth cutoff $\chi_\delta$, equal to
one on $|u|\leq\delta$ and zero on $|u|\geq2\delta$, $0<\delta\leq1$.
Split the convolution into $A_1=A_v\chi_\delta\star\mathcal D_vf$ and
$A_2=A_v(1-\chi_\delta)\star\mathcal D_vf$.
On $|u|\leq2\delta$, the frames and time scales at $v$ and $v-u$ are
comparable. The first bound in \eqref{eq:Landau-projection-derivative}
and the $Z^1$ bound (obtained by averaging on an intrinsic unit segment
and using the $Z^0$ and $Z^{1+\varepsilon_0}$ bounds) yield
\[
 \tilde t_v^{1/2}\mathbf I A_1(t,x,v)
 \lesssim\langle x\rangle^{-\varkappa_1}\delta^{d-\kappa}\|f\|_T.
\]
Here the source weight absorbs $\langle v-u\rangle^2$.
Integrate by parts for $A_2$. The second bound in
\eqref{eq:Landau-projection-derivative}, including
$|\mathcal D_v\chi_\delta|\lesssim\delta^{-1}$, gives
\[
 \mathbf I A_2(t,x,v)\lesssim
 L\langle x\rangle^{-\varkappa_1}
 \left(1+\delta^{d-\kappa-1}
 +\mathbf1_{\kappa=d-1}|\log\delta|\right)\|f\|_T.
\]
The constant $1$ bounds $|u|>1$; the other terms come from
$\int_\delta^1r^{d-\kappa-2}\,\dif r$ and the cutoff annulus.
Set $\delta=\tilde t_v^{1/2}\leq1$. For $\kappa>d-1$,
$\tilde t_v^{1/2}\leq\tilde t_v^{(d-\kappa)/2}$; for $\kappa=d-1$
retain $1+|\log\tilde t_v|$. This proves \eqref{afl} for $m=0$.

For $m>0$, move the constant-coefficient derivatives onto $f$ before
making this decomposition. Since $\gamma=-\kappa-2$,
comparison of their scales at $v$ and $v-u$ costs at most
$\langle v\rangle^{-\eta_m}\langle v-u\rangle^{\eta_m}$,
$\eta_m=|\gamma|m/2$. On $|u|\leq2\delta$ this comparison is bounded;
on the complementary region use
$\varkappa_2-\eta_m>d+2$, which follows from Theorem~\ref{Th1}.
The same integrals therefore apply at each derivative order.
For nonsmooth $f$, regularize first; the integrable majorants and Fatou's
lemma pass the estimate to the limit.
\end{proof}

		{The next lemma uses the coefficient envelopes \(\Omega_a\) defined in \eqref{defom}. Set}
		\[
			\varkappa:=\varkappa_1+\varkappa_2,
			\qquad
			\eta_m:=\frac{|\gamma|m}{2s},
			\qquad
			\xi_m:=\eta_m+\frac{\kappa\varepsilon_0}{2s},
		\]
		and define the capped anisotropic increment
		\begin{equation}\label{def:truncated-increment}
			\mathfrak r_{\sigma,v}(z)
			:=1\wedge\ell_{\sigma,v}^v(z)
			=1\wedge\frac{[z]_v}{\tilde\sigma_v^{1/(2s)}}.
		\end{equation}
		{For \(0\leq m\leq N_0\), define}
		\begin{align}\label{def:weighted-C-majorant}
			&{\Omega_{\mathbf w,m}(x,v,z)
			:=\langle v\rangle^{-\eta_m}\Big(
			\langle x\rangle^{-\varkappa_1}
			\Omega_{\varkappa_2-\eta_m}(v,z)\wedge
			\Omega_{\varkappa-\eta_m}(v,z)\Big),}\\
		&
			\Omega_{\mathbf w}:=\sum_{m=0}^{N_0}\Omega_{\mathbf w,m}.
		\end{align}
		{For the aligned finite difference of the coefficient, set}
		\begin{equation}\label{def:difference-C-majorant}
			\Omega_{\mathrm{diff},m}(v,z)
			:=\langle v\rangle^{-\xi_m}
			\Omega_{\varkappa-\xi_m}(v,z).
		\end{equation}
		{Here \(\mathrm P_{\hat z}\) and \(\Pi_{\hat z}\) are defined in \eqref{defproj}.}
		\begin{lemma}\label{lemCf1}
			Assume that
			\[
				\varkappa_2-\eta_{N_0}>d,
				\qquad
				\varkappa-\xi_{N_0}>d.
			\]
			For every \(x,v\in\mathbb R^d\), \(t\in(0,T)\), and
			\(m\leq N_0\), one has
			\begin{align}
				\left|
				\nabla_{t,v_0}^{(x,v),m}\mathbf C_f(t,x,v,z)
				\big|_{v_0=v}
				\right|
				&\lesssim
				\Omega_{\mathbf w,m}(x,v,z)\|f\|_T,
				\label{Cfs1}\\
				\left|
				\delta_z^v\nabla_{t,v_0}^{(x,v),m}
				\mathbf C_f(t,x,v,z)\big|_{v_0=v}
				\right|
				&\lesssim
				\ell_{t,v}^v(z)^{\varepsilon_0}
				\Omega_{\mathrm{diff},m}(v,z)
				\|f\|_{T,*},
				\qquad 0<[z]_v\leq1.
				\label{delCfs2}
			\end{align}
		\end{lemma}
			\begin{proof}
			{We begin with \eqref{Cfs1}. First, we record a hyperplane estimate encoded in the definition of $\mathbf I$. Let $F\geq0$ be continuous. For every $\theta\in\mathbb S^{d-1}$ and $u\in\mathbb R^d$,}
	\begin{align}
		\int_{\theta^\perp\cap B_{1/2}}
		F(x,u+h)
		\bigl(1+|h|^{1-\kappa}\bigr)
		\,\dif\mathcal H^{d-1}(h)
		\lesssim
		\mathbf I F(x,u).
		\label{eq:I-hyperplane-trace}
	\end{align}
	
	To prove \eqref{eq:I-hyperplane-trace}, take $v_0=R\theta$ in
\eqref{defI12}, write $h=h_\perp+a\theta$, and set $r=\langle R\rangle a$.
Then
\[
[h]_{R\theta}
=\sqrt{|h_\perp|^2+r^2/\langle R\rangle^2}
  +\frac{R}{\langle R\rangle}|r|
\longrightarrow |h_\perp|+|r|.
\]
The Jacobian cancels the prefactor $\langle R\rangle$. On the open set
$|h_\perp|+|r|<1$, continuity of $F$ and Fatou's lemma give
\[
\mathbf I F(x,u)\geq
\int_{|h_\perp|<1}F(x,u+h_\perp)K_\kappa(|h_\perp|)
 \,\dif\mathcal H^{d-1}(h_\perp),
\qquad
K_\kappa(a):=2\int_0^{1-a}(a+r)^{-\kappa}\,\dif r.
\]
For $0<a\leq1/2$, integration over $1/4\leq r\leq1/2$ yields
$K_\kappa(a)\gtrsim1$. If $\kappa>1$, integration over $0\leq r\leq a$
also gives $K_\kappa(a)\gtrsim a^{1-\kappa}$. If $\kappa\leq1$, then
$a^{1-\kappa}\leq1$. These bounds prove \eqref{eq:I-hyperplane-trace}.
The value at $a=0$ is immaterial to the hyperplane integral.

{We next derive the weighted hyperplane estimate used below. Suppose that, for some $\beta>\max\{d-1,d-\kappa\}$,}
	\begin{align}
		\mathbf I F(x,u)
		\lesssim 
		M\langle u\rangle^{-\beta},
		\qquad u\in\mathbb R^d.
		\label{eq:I-weighted-assumption}
	\end{align}
	{Then}
	\begin{align}
		&\int_{\theta^\perp}
		F(x,v+w)|w|^{1-\kappa}
		\mathbf 1_{|w|\geq\rho}
		\,\dif\mathcal H^{d-1}(w)
		\notag\\
		&\quad\lesssim
		M\Big(
		\mathbf 1_{
			\langle\Pi_\theta v\rangle\geq\rho/100
		}
		\langle\Pi_\theta v\rangle^{1-\kappa}
		\langle\mathrm P_\theta v\rangle^{d-1-\beta}
		+
		\langle v,\rho\rangle^{d-\kappa-\beta}
		\Big).
		\label{eq:weighted-hyperplane-estimate}
	\end{align}
	{To prove this estimate, choose a lattice $\Lambda_\theta\subset\theta^\perp$ and cover $\theta^\perp$ by fixed-radius balls centered at its lattice points, with uniformly bounded overlap. Estimate the ball centered at the origin using the weighted part of \eqref{eq:I-hyperplane-trace}. On every remaining ball centered at $\zeta\in\Lambda_\theta$, one has $|w|\sim |\zeta|$; the unweighted part of \eqref{eq:I-hyperplane-trace}, together with \eqref{eq:I-weighted-assumption}, then gives}
	\begin{align}
		&\int_{\theta^\perp}
		F(x,v+w)|w|^{1-\kappa}
		\mathbf 1_{|w|\geq\rho}
		\,\dif\mathcal H^{d-1}(w)
\lesssim
		M
		\sum_{\zeta\in\Lambda_\theta(\rho)}
		\langle\zeta\rangle^{1-\kappa}
		\langle v+\zeta\rangle^{-\beta},
		\label{eq:hyperplane-lattice-sum}
	\end{align}
	{where no restriction is imposed on $\zeta$ when $\rho\leq4$, whereas $|\zeta|\geq\rho/2$ when $\rho>4$.}

	If
	$\langle\Pi_\theta v\rangle<\rho/100$, then every term in
	\eqref{eq:hyperplane-lattice-sum} satisfies
	\begin{align*}
		|\Pi_\theta v+\zeta|
		\gtrsim
		|\Pi_\theta v|+|\zeta|.
	\end{align*}
	{A dyadic summation in the $(d-1)$-dimensional hyperplane therefore gives}
	\begin{align*}
		\sum_{\zeta\in\Lambda_\theta(\rho)}
		\langle\zeta\rangle^{1-\kappa}
		\langle v+\zeta\rangle^{-\beta}
		\lesssim
		\langle v,\rho\rangle^{d-\kappa-\beta}.
	\end{align*}
	{If $\langle\Pi_\theta v\rangle\geq\rho/100$, split the sum into}
	\begin{align*}
		|\Pi_\theta v+\zeta|
		\leq
		\frac12\langle\Pi_\theta v\rangle
		\qquad\text{and}\qquad
		|\Pi_\theta v+\zeta|
		>
		\frac12\langle\Pi_\theta v\rangle.
	\end{align*}
	{On the first region, $|\zeta|\asymp\langle\Pi_\theta v\rangle$, and summation in $\Pi_\theta v+\zeta\in\theta^\perp$ yields}
	\begin{align*}
		\sum_{
			|\Pi_\theta v+\zeta|
			\leq\frac12\langle\Pi_\theta v\rangle
		}
		\langle\zeta\rangle^{1-\kappa}
		\langle v+\zeta\rangle^{-\beta}
		\lesssim
		\langle\Pi_\theta v\rangle^{1-\kappa}
		\langle\mathrm P_\theta v\rangle^{d-1-\beta}.
	\end{align*}
	{The preceding dyadic argument bounds the complementary region by $\langle v,\rho\rangle^{d-\kappa-\beta}$. This proves \eqref{eq:weighted-hyperplane-estimate}. In particular, the local cube $\zeta=0$ has been retained and estimated separately.}

	{We now apply \eqref{eq:weighted-hyperplane-estimate} to $\mathbf C_f$. By the definition of the scaled derivatives,}
	\begin{align*}
		\left|
		\nabla_{t,v_0}^{(x,v),m}f(t,x,u)
		\big|_{v_0=v}
		\right|
		\lesssim
		\langle v\rangle^{-\eta_m}
		\langle u\rangle^{\eta_m}
		\left|\nabla_{t,u}^{(x,v),m}f(t,x,u)\right|.
	\end{align*}
	{Moreover,}
	\[
		\mathbf w^{-1}(x,u)
		\leq\langle x\rangle^{-\varkappa_1}
		\langle u\rangle^{-\varkappa_2},
		\qquad
		\mathbf w^{-1}(x,u)\leq\langle u\rangle^{-\varkappa}.
	\]
	{Consequently, the $Z_m$-component of $\|f\|_T$ gives}
	\begin{align}
		&\left(
		\mathbf I
		\left|
		\nabla_{t,v_0}^{(x,v),m}f
		\right|
		\right)(t,x,u)\Big|_{v_0=v}
		\lesssim
		{\langle v\rangle^{-\eta_m}
		\Big(
		\langle x\rangle^{-\varkappa_1}
		\langle u\rangle^{-(\varkappa_2-\eta_m)}
		\wedge
		\langle u\rangle^{-(\varkappa-\eta_m)}
		\Big)\|f\|_T.}
		\label{eq:Cf-Zm-input}
	\end{align}
	{Both exponents in \eqref{eq:Cf-Zm-input} exceed $d$, since}
	\(
	\varkappa-\eta_m\geq\varkappa_2-\eta_m>d
	\).
	{Writing $\rho=|z|$ and $\theta=\hat z$, the definition of $\mathbf C_f$ yields}
	\begin{align*}
		&\left|
		\nabla_{t,v_0}^{(x,v),m}
		\mathbf C_f(t,x,v,z)
		\big|_{v_0=v}
		\right|\lesssim
		\int_{\theta^\perp}
		\left|
		\nabla_{t,v_0}^{(x,v),m}f(t,x,v+w)
		\big|_{v_0=v}
		\right|
		|w|^{1-\kappa}
		\mathbf 1_{|w|\geq\rho}
		\,\dif\mathcal H^{d-1}(w).
	\end{align*}
	{Apply \eqref{eq:weighted-hyperplane-estimate} to both global bounds in \eqref{eq:Cf-Zm-input}, with \(\beta=\varkappa_2-\eta_m\) and \(\beta=\varkappa-\eta_m\), and take their minimum. This gives}
	\begin{align*}
		\left|
		\nabla_{t,v_0}^{(x,v),m}
		\mathbf C_f(t,x,v,z)
		\big|_{v_0=v}
		\right|
		\lesssim
		\Omega_{\mathbf w,m}(x,v,z)\|f\|_T.
	\end{align*}
	{This proves \eqref{Cfs1}.}

	{We next prove the aligned finite-difference estimate \eqref{delCfs2}. In the Carleman integral, write \(u=v+w\), where \(w\perp z\). Then}
	\begin{align*}
		[z]_u=[z]_{v+w}=[z]_v,
		\qquad
		\tilde t_u^{-1/(2s)}
		=
		\tilde t_v^{-1/(2s)}
		\left(\frac{\langle u\rangle}{\langle v\rangle}\right)^{\kappa/(2s)}.
	\end{align*}
	{Because the frozen derivatives have constant coefficients, both they and \(\delta_z^v\) commute with the Carleman integral. The \(Z_m^{\varepsilon_0}\)-component of \(\|f\|_{T,*}\), together with the preceding scaled-derivative comparison, therefore gives, for \(0<[z]_v\leq1\),}
	\begin{align*}
		&\left(
		\mathbf I
		\left|
		\delta_z^v\nabla_{t,v_0}^{(x,v),m}f
		\right|
		\right)(t,x,u)\Big|_{v_0=v}
		\lesssim
		\ell_{t,v}^v(z)^{\varepsilon_0}
		\langle v\rangle^{-\xi_m}
		\langle u\rangle^{-(\varkappa-\xi_m)}
		\|f\|_{T,*}.
	\end{align*}
	{Since \(\varkappa-\xi_m>d\), applying \eqref{eq:weighted-hyperplane-estimate} with \(\beta=\varkappa-\xi_m\) proves \eqref{delCfs2}.}
		\end{proof}\par\medskip

        {As in Lemma~\ref{lemA1}, we estimate \(t\,\mathbf I\mathcal K(f)\), where $\mathcal{K}(f)$ is defined in \eqref{defK}. The term}
\[
    t\,\mathcal K(f)
\]
{has scaling-critical size in the \(\|\cdot\|_T\)-norm. Lemma~\ref{lemcanc} gives \(\mathcal K(f)={c_{d,\gamma,s,b}}\Lambda_v^{-d-\gamma}f\); thus the cancellation structure converts the singular integral into an operator of nonpositive order when $\gamma\in[-d,-2s]$, or into a velocity operator of order strictly less than $2s$ when $\gamma\in(-d-2s,-d)$. This structure reduces the small-time loss incurred by estimating \eqref{defK} directly. Consequently, $t\,\mathbf I\mathcal K(f)$ is lower order and gains a positive power of the time scale. We prove the following lemma for $\gamma\in[-d,-2s]$.
For $\gamma\in(-d-2s,-d)$, the positive-order cancellation operator is
treated instead by fractional integration by parts in Lemma~\ref{lemnew};
see the case $\alpha=-d-\gamma\in(0,2s)$ in its proof.}
\medskip
	\begin{lemma}
	\label{lemCfl}
	Assume that $\varkappa_2>\frac{|\gamma| N_0}{2s}+d$.  Suppose moreover that $-d-\gamma\leq 0$.
	Then, for every $0<t<T<1$ and every $m\leq N_0$,
	\begin{align}
	t	\left(
		\mathbf I
		\left|
		\nabla_{t,v_0}^{(x,v),m}\mathcal K(f)
		\right|
		\right)(t,x,v)\Big|_{v_0=v}
		\lesssim
	t	\langle x\rangle^{-\varkappa_1}
		\|f\|_T .
		\label{cance}
	\end{align}
\end{lemma}
\begin{proof}
	Fix $m\leq N_0$ and, with $v_0$ kept fixed, set
	\begin{align*}
		F_{m,v_0}
		:=
		\nabla_{t,v_0}^{(x,v),m}f.
	\end{align*}
	{Since $\mathcal K$ acts only in the velocity variable and $\nabla_{t,v_0}^{(x,v),m}$ has constant coefficients for fixed $v_0$, Lemma \ref{lemcanc} gives}
	\begin{align}
		\nabla_{t,v_0}^{(x,v),m}\mathcal K(f)
		=
		{c_{d,\gamma,s,b}}\Lambda_v^{-d-\gamma} F_{m,v_0}.
		\label{eq:commute-cancellation}
	\end{align}
	{If $a:=d+\gamma\in(0,d)$, then $\Lambda_v^{-a}$ is the Riesz potential}
	\begin{align*}
		\Lambda_v^{-a}F(v)
		=
		c_a\int_{\mathbb R^d}
		F(v-h)|h|^{a-d}\,\dif h.
	\end{align*}
	{Therefore,}
	\begin{align}
		\left(
		\mathbf I
		\left|
		\Lambda_v^{-a}F_{m,v_0}
		\right|
		\right)(t,x,v)\Big|_{v_0=v}
		&\lesssim
		{\|f\|_T\langle v\rangle^{\frac{\gamma m}{2s}}\int_{\mathbb{R}^d} \langle v-h\rangle^{-\frac{\gamma m}{2s}}\mathbf{w}^{-1}(x,v-h)|h|^{a-d}\dif h}\nonumber\\
        &\lesssim {\langle x\rangle^{-\varkappa_1}\|f\|_T}.
		\label{eq:Riesz-cancellation-bound}
	\end{align}
		{With $\eta_m=-\gamma m/(2s)\geq0$, use
$\mathbf w^{-1}(x,v-h)\leq
\langle x\rangle^{-\varkappa_1}\langle v-h\rangle^{-\varkappa_2}$.
For $0<a<d$, split at $|h|=1$: the Riesz kernel is locally
integrable and bounded outside the ball, while
$\varkappa_2-\eta_m>d$. For $a=0$, use $\Lambda_v^0F=F$.}
	\end{proof}
		{The bilinear estimates below contain two truncated increments at different
		cutoff scales: a first-order difference from the propagated kernel and a
		fractional difference from the source or coefficient. The following elementary
		radial estimate quantifies their interaction and provides the one-dimensional
		input for the anisotropic jump integrations.}
		\begin{lemma}\label{lemW}
			{Let \(\alpha\in(0,2s)\) satisfy \(1+\alpha>2s\). If $0<\delta_1\leq \delta_2\leq 1$, then}
			\begin{align*}
				\int_0^1\left(\min\{\frac{\rho }{\delta_1},1\}\min\{\frac{\rho}{\delta_2},1\}^{\alpha}\right)\frac{\dif \rho}{\rho^{1+2s}}\sim 	\frac{1}{\delta_1^{2s-\alpha}\delta_2^{\alpha}}.
			\end{align*}
		\end{lemma}
		\begin{proof}
			If $0<\alpha<2s$ and $1+\alpha>2s$, we have
			\begin{align*}
				&	\int_0^1\left(\min\{\frac{\rho}{\delta_1},1\}\min\{\frac{\rho}{\delta_2},1\}^{\alpha}\right)\frac{\dif \rho}{\rho^{1+2s}}\\&\sim
				\left(\int_{0}^{\delta_1/4}\frac{\rho^{1+\alpha}}{\delta_1\delta_2^\alpha}\frac{d \rho}{\rho^{1+2s}} +	\int_{\delta_1/4}^{\delta_2/2}\frac{\rho^{\alpha}}{\delta_2^{\alpha}}\frac{\dif \rho}{\rho^{1+2s}} +	\int_{\delta_2/2}^1\frac{\dif \rho}{\rho^{1+2s}} \right)
				\\&\sim
				\frac{\delta_1^{1+\alpha-2s}}{\delta_1\delta_2^\alpha}	 +	\frac{\delta_1^{\alpha-2s}}{\delta_2^{\alpha}} + \delta_2^{-2s}\sim  \frac{\delta_1^{\alpha-2s}}{\delta_2^{\alpha}}.
			\end{align*}
			{This proves the lemma.}
			\end{proof}\par\medskip
			{We now combine the preceding radial calculation with the anisotropic
			coefficient envelopes $\Omega_a$ and $\Omega_{\mathbf w}$. The next lemma
			converts the two finite-difference cutoffs into the precise time powers used
			in the local bilinear estimates, while retaining the velocity decay
			$\langle v\rangle^{-\kappa}$. It also records the capped-moment and
			finite-rate variants needed for the remainder terms.}
			\begin{lemma}
				\label{lem:global-Omega-integration}
				Let \(s\in(0,1)\) and \(0<t,\tau\leq1\).  If
				\(a>d+2s\), \(0<b<2s\), and \(1+b>2s\), then
				\begin{align}
					&\int_{\mathbb R^d}
					\mathfrak r_{t,v}(z)\mathfrak r_{\tau,v}(z)^b
					\big(\Omega_a(v,z)+\Omega_a(v-z,z)\big)
					\frac{\dif z}{|z|^{d+2s}}
				\lesssim
					\tilde t_v^{-\frac{2s-b}{2s}}
					\tilde\tau_v^{-\frac{b}{2s}}
					\langle v\rangle^{-\kappa}.
					\label{eq:global-Omega-unweighted}
				\end{align}
				{If, in addition,}
				\(\varkappa_2-|\gamma|N_0/(2s)>d+2s\), then
				\begin{align}
					&\int_{\mathbb R^d}
					\mathfrak r_{t,v}(z)\mathfrak r_{\tau,v}(z)^{\varepsilon_0}
					\big(\langle v\rangle^{-\varkappa}
					+\langle v-z\rangle^{-\varkappa}\big)
					\Omega_{\mathbf w}(x,v,z)
					\frac{\dif z}{|z|^{d+2s}}
					\lesssim
					\tilde t_v^{-\frac{2s-\varepsilon_0}{2s}}
					\tilde\tau_v^{-\frac{\varepsilon_0}{2s}}
					\mathbf w^{-1}(x,v)\langle v\rangle^{-\kappa}.
					\label{eq:global-Omega-weighted}
				\end{align}
				{For \(s<b<\min\{1,2s\}\), one also has}
				\begin{align}
					&\int_{|z|\leq1}|z|^b\mathfrak r_{t,v}(z)^b
					\big(\Omega_a(v,z)+\Omega_a(v-z,z)\big)
					\frac{\dif z}{|z|^{d+2s}}
					\lesssim
					\langle v\rangle^{-\kappa}
					\tilde t_v^{-1+\frac b{2s}}.
					\label{eq:global-Omega-capped-moment}
				\end{align}
				{Finally, for each fixed \(r_0\in(0,1)\),}
				\begin{align}
					\sup_{v\in\mathbb R^d}\int_{|z|\geq r_0}
					\big(\Omega_a(v,z)+\Omega_a(v-z,z)\big)
					\frac{\dif z}{|z|^{d+2s}}
					\lesssim_{a,r_0}1.
					\label{eq:finite-large-jump-rate}
				\end{align}
			\end{lemma}
			\begin{proof}
				{Write \(z=\rho\theta\). Since \([\rho\theta]_v\sim\rho\langle v\cdot\theta\rangle\), Lemma~\ref{lemW} gives}
				\begin{align}
					\int_0^1
					\mathfrak r_{t,v}(\rho\theta)
					\mathfrak r_{\tau,v}(\rho\theta)^b
					\frac{\dif\rho}{\rho^{1+2s}}
					\lesssim
					\langle v\cdot\theta\rangle^{2s}
					\tilde t_v^{-\frac{2s-b}{2s}}
					\tilde\tau_v^{-\frac{b}{2s}}.
					\label{eq:radial-Carleman-bound}
				\end{align}
				{If the two intrinsic radii occur in the opposite order, splitting at both radii and using \(1+b>2s\) gives the same upper bound. The angular reduction is}
				\[
					\int_{\mathbb S^{d-1}}
					\langle v\cdot\theta\rangle^{2s+d-1-a}
					\langle\Pi_\theta v\rangle^{1-\kappa}\,\dif\theta
					+\langle v\rangle^{d-\kappa-a}
					\int_{\mathbb S^{d-1}}
					\langle v\cdot\theta\rangle^{2s}\,\dif\theta
					\lesssim\langle v\rangle^{-\kappa},
				\]
				{for \(a>d+2s\). This proves the estimates on \(|z|\leq1\), where \(\Omega_a(v-z,z)\lesssim\Omega_a(v,z)\) and \(\langle v-z\rangle\sim\langle v\rangle\). The same polar-coordinate argument, now applied to \(|z|^b\mathfrak r_{t,v}(z)^b\), proves \eqref{eq:global-Omega-capped-moment}; the condition \(b>s\) is precisely the integrability condition at the origin.}

				{On \(|z|>1\), both cutoff factors equal one. Splitting the radial integral at \(\rho=(v\cdot\theta)/2\) and \(\rho=2(v\cdot\theta)\), and using \(\Pi_\theta(v-\rho\theta)=\Pi_\theta v\), gives, for \(\beta>0\),}
				\begin{align*}
					\int_{|z|>1}
					\frac{\Omega_a(v,z)+\Omega_a(v-z,z)}{|z|^{d+2s}}\,\dif z
					&\lesssim\langle v\rangle^{-\kappa},\\
					\int_{|z|>1}
					\big(\langle v\rangle^{-\beta}+\langle v-z\rangle^{-\beta}\big)
					\Omega_a(v,z)\frac{\dif z}{|z|^{d+2s}}
					&\lesssim\langle v\rangle^{-\kappa-\min\{a,\beta\}}.
				\end{align*}
				{{For the tail integral containing the additional velocity weights,} further split into \(|z|\leq\langle v\rangle/2\), \(|v-z|\geq c\langle v\rangle\), and \(|v-z|<c\langle v\rangle\), with \(c<1/200\). In the last region, the indicator part of \(\Omega_a\) vanishes, while the remaining part is integrable with the stated decay. On the fixed annulus \(r_0\leq|z|\leq1\), the preceding angular bound without the factor \(\langle v\cdot\theta\rangle^{2s}\) is uniform. Together with the first tail bound, this proves \eqref{eq:finite-large-jump-rate}.}

				{The radial and tail bounds prove \eqref{eq:global-Omega-unweighted}. Applying them with \(b=\varepsilon_0\), summing over \(0\leq j\leq N_0\), and taking \((a,\beta)=(\varkappa_2-\eta_j,\varkappa)\) or \((\varkappa-\eta_j,\varkappa)\) proves \eqref{eq:global-Omega-weighted}.}
			\end{proof}
			\par\medskip\noindent
	{We now use Corollary~\ref{lemcoeff1} and Lemma~\ref{lem:global-Omega-integration} to prove the principal bilinear estimate for local well-posedness. Recall the operators defined in \eqref{defsv0}--\eqref{intg2}. The key term is}
\[
    \mathcal S_{v_0}(f_1,f_2)(t,\tau,x,v)
    =
    \bigl(H_\mu^{v_0}(t)\ast \mathcal Q_s(f_1,f_2)(\tau)\bigr)(x-tv,v),
\]
{The following bounds control this bilinear
Duhamel term in the critical norms.}
		\begin{lemma}\label{lemnew}
			{Let $T\in(0,1)$ and $t,\tau\in(0,T)$.
			{Here $t$ is the propagation time
and $\tau$ the source time.}
			For $m,k\in\mathbb N$ with $m\leq N_0$ and
			$k=0$ when $0<s<1$, one has}
			\begin{align}\label{Sh12a}
				&\|\mathcal{S}_{v_0}(h_1,h_2)(t,\tau)\|_{\tilde Y_m^{\lfloor s\rfloor+k}\cap\tilde Z_m^{\lfloor s\rfloor+k}}
				\lesssim \frac{\langle \tau/t\rangle^\frac{k}{2s}}{(t+\tau)^\frac{\varepsilon_0}{2s}\min\{t,\tau\}^{1-\frac{\varepsilon_0}{2s}}} 	\|h_1\|_T	\|h_2\|_{T},\\
				&\|\mathcal{S}_{v_0}(h,\mu)(t,\tau)\|_{\tilde Y_m^{\lfloor s\rfloor+k}\cap\tilde Z_m^{\lfloor s\rfloor+k}}
				+\|\mathcal{S}_{2,v_0}(\mu,h)(t,\tau)\|_{\tilde Y_m^{\lfloor s\rfloor+k}\cap\tilde Z_m^{\lfloor s\rfloor+k}}\notag\\
				\label{Sh12b}
				&\qquad\lesssim\langle\tau/t\rangle^\frac{k}{2s}
				\left[\frac{T^{\sigma_0}}{(t+\tau)^\frac{\varepsilon_0}{2s}\min\{t,\tau\}^{1-\frac{\varepsilon_0}{2s}}}
				+\mathbf1_{s<1}\tau^{-1+\sigma_0}\right]\|h\|_T,
			\end{align}
			{where the norms $\|\cdot \|_{\tilde Y_{m}^k}$ and $\|\cdot \|_{\tilde Z_{m}^k}$ are defined in \eqref{tilYntt}.}
		\end{lemma}
		        \begin{remark}
          {Since $\mu$ has no spatial decay,
\eqref{Sh12b} requires a separate Gaussian estimate rather than
substitution in \eqref{Sh12a}.}
        \end{remark}
		\begin{proof}[Proof of Lemma \ref{lemnew}]
			{\textbf{Landau case ($s=1$).}}
			{With $\tilde a_{v_0}$ as in Corollary~\ref{lemcoeff1}, we write}
			\begin{align*}
				\mathcal{Q}_{1,\mathrm{main}}(h_1,h_2)=\mathcal{D}_{v_0}^{(v)}(\tilde a_{v_0}\star h_1(v)\mathcal{D}_{v_0}^{(v)}h_2(v)), \ \ \ \ 	\mathcal{Q}_{1,\mathrm{rem}}(h_1,h_2)=-\mathcal{D}_{v_0}^{(v)}(\tilde a_{v_0}\star \mathcal{D}_{v_0}^{(v)}h_1(v)h_2(v)).
			\end{align*}
			{Set $G_{1,v_0}=((\tilde a_{v_0}\star h_1)\mathcal{D}_{v_0}^{(v)} h_2)(\tau)$ and $G_{2,v_0}=((\tilde a_{v_0}\star \mathcal{D}_{v_0}^{(v)}h_1) h_2)(\tau)$, and define}
			\begin{equation*}
				\begin{aligned}
					\|G_{i,v_0}\|_{\mathbf{X}_{N_0}}:=\sum_{m=0}^{N_0}&\sup_{\tau\in(0,T)}\left(\|\tilde \tau_{v_0}^\frac{1}{2}\langle v_0\rangle^{\kappa}G_{i,v_0}(\tau)\|_{\tilde Y_{m}}+	\|\tilde \tau_{v_0}^\frac{1}{2}\langle v_0\rangle^{\kappa}G_{i,v_0}(\tau)\|_{\tilde Z_{m}}\right.\\
					&\quad\quad\quad+\left.	\|\tilde \tau_{v_0}^\frac{1}{2}\langle v_0\rangle^{\kappa}G_{i,v_0}(\tau)\|_{\tilde Y_{m}^{\varepsilon_0}}+	\|\tilde \tau_{v_0}^\frac{1}{2}\langle v_0\rangle^{\kappa}G_{i,v_0}(\tau)\|_{\tilde Z_{m}^{\varepsilon_0}}\right).
				\end{aligned}
			\end{equation*}
			{We claim that}
			\begin{equation}\label{esGi}
				\begin{aligned}
					{\|G_{1,v_0}\|_{\mathbf{X}_{N_0}}+\|G_{2,v_0}\|_{\mathbf{X}_{N_0}}\lesssim\|h_1\|_T\|h_2\|_T.}
				\end{aligned}
			\end{equation}
			{{Fix $m,k\in\mathbb{N}$ with $m\leq N_0$,
and suppose first that $0<t<\tau<T$. Lemma~\ref{pwthm}
and integration by parts give}}
			\begin{equation}\label{dsf}
				\begin{aligned}
					&	|\nabla_{t,v_0}^{(x,v),m}\mathcal{D}_{t,v_0}^{(v),k}{\mathcal{S}}_{i,v_0} (h_1,h_2)|(t,\tau,x,v)\Big|_{v_0=v}\\
					&\lesssim \left| \iint_{\mathbb{R}^{2d}}\mathcal{D}_{t,v_0}^{(x,v),k}\mathcal{D}_{v_0}^{(v)}	H_\mu^{v_0}(t, y,u)\nabla_{t,v_0}^{(x,v),m} G_{i,v_0}(\tau,x-tv-y,v-u)\dif y \dif u\Big|_{v_0=v}\right|\\
					&\lesssim \tilde t_{v}^{-\frac{1}{2}} \iint_{\mathbb{R}^{2d}} \Phi_{t,v_0}(y,u)|\nabla_{t,v_0}^{(x,v),m} \delta^v_{u}G_{i,v_0}|(\tau,x-tv-y,v)\dif y \dif u\Big|_{v_0=v},
				\end{aligned}
			\end{equation} 
		{where the last inequality uses}
			\begin{align*}
				\int_{\mathbb{R}^{d}}\mathcal{D}_{t,v_0}^{(v),k}\mathcal{D}_{v_0}^{(v)}	H_\mu^{v_0}(t,y,u)\, \dif u=0.
			\end{align*}
			{If instead $0<\tau<t<T$, then \eqref{Hmues} gives}
			\begin{equation}\label{essf}
				\begin{aligned}
					&|\nabla_{t,v_0}^{(x,v),m}\mathcal{D}_{t,v_0}^{(v),k}	{\mathcal{S}}_{i,v_0} (h_1,h_2)|(t,\tau,x,v)\Big|_{v_0=v}\\
					&\quad\quad\lesssim  \iint_{\mathbb{R}^{2d}}|\nabla_{t,v_0}^{(x,v),m}\mathcal{D}_{t,v_0}^{(x,v), k} \mathcal{D}_{v_0}^{(v)} 	H_\mu^{v_0}|(t,x-tv-y,v-u)\,| G_{i,v_0}(\tau,y,u)|\dif y \dif u\Big|_{v_0=v}\\
					&\quad\quad\lesssim \tilde t_v^{-\frac{1}{2}}\iint_{\mathbb{R}^{2d}}	\Phi_{t,v_0}(y,u)|G_{i,v_0}(\tau,x-tv-y,v-u)|\dif y \dif u\Big|_{v_0=v}.
				\end{aligned}
			\end{equation}
{Applying Lemmas~\ref{le11} and
\ref{lem:recentered-convolution} to \eqref{dsf} and \eqref{essf}
(with a finite covering of $B_2$ when $t>1/2$), we obtain}
			\begin{align}\label{Si1}
				\|{\mathcal{S}}_{i,v_0} (h_1,h_2)(t,\tau)\|_{\tilde Y_{m}^k\cap\tilde Z_{m}^k}	\lesssim( t^{-1}(t/\tau)^{\frac{1+\varepsilon_0-k}{2}}\mathbf{1}_{t<\tau}+t^{-\frac{1}{2}}\tau^{-\frac{1}{2}}\mathbf{1}_{\tau<t})\|G_{i,v_0}\|_{\mathbf{X}_{N_0}}.
			\end{align}	
			{{Together with \eqref{esGi}, this proves \eqref{Sh12a}.}}
			\begingroup
			The corresponding velocity-difference estimate is
\begin{align}\label{eq:A26-Landau-terminal}
&\mathbf w(x,v)\Bigl[
{(\tilde t_v+\tilde \tau_v)}^{(d-\kappa)/2}
|\mathcal D_{t+\tau,v_0}^{(v)}\nabla_{t+\tau,v_0}^{(x,v),m}
\delta_h^v\mathcal S_{i,v_0}(t,\tau)|(x,v)
+\mathbf I\bigl(\mathcal D_{t+\tau,v_0}^{(v)}
\nabla_{t+\tau,v_0}^{(x,v),m}
\delta_h^v\mathcal S_{i,v_0}(t,\tau)\bigr)(x,v)
\Bigr]_{v_0=v}\notag\\
&\qquad\lesssim
\mathfrak r_{t,v}(h)\|G_{i,v_0}\|_{\mathbf X_{N_0}}
\begin{cases}
t^{-1+\varepsilon_0/2}\tau^{-\varepsilon_0/2},&t<\tau,\\
t^{-1/2}\tau^{-1/2},&\tau\leq t.
\end{cases}
\end{align}
			For $t<\tau$, take the terminal difference in \eqref{dsf}
			before absolute values: kernel cancellation and the source
			H\"older bound give $(t/\tau)^{\varepsilon_0/2}$, while
			Lemma~\ref{pwthm} supplies $\mathfrak r_{t,v}(h)$.
			For $\tau\leq t$, apply the same kernel-difference bound
			to \eqref{essf}, transferring the shifted profiles in both
			cases as in Lemma~\ref{le11}.
			\par\endgroup
			
\begingroup
It remains to prove \eqref{esGi}. For $m\leq N_0$, $i=1,2$,
and $0<[z]_v\leq1$, the product rule gives
\begin{align}\label{G2}
&\mathbf w(x,v)\langle v\rangle^\kappa\tilde\tau_v^{1/2}
\Bigl[
\tilde\tau_v^{(d-\kappa)/2}
\bigl(|\nabla_{\tau,v_0}^{(x,v),m}G_{i,v_0}|
+\mathfrak r_{\tau,v}(z)^{-\varepsilon_0}
|\delta_z^v\nabla_{\tau,v_0}^{(x,v),m}G_{i,v_0}|\bigr)\notag\\
&\qquad+\mathbf I(\nabla_{\tau,v_0}^{(x,v),m}G_{i,v_0})
+\mathfrak r_{\tau,v}(z)^{-\varepsilon_0}
\mathbf I(\delta_z^v\nabla_{\tau,v_0}^{(x,v),m}G_{i,v_0})
\Bigr](x,v)\Big|_{v_0=v}
\lesssim\|h_1\|_T\|h_2\|_T .
\end{align}
Indeed, apply the argument of Corollary~\ref{lemcoeff1}
to each differentiated coefficient and Lemma~\ref{lemi1i2} to each product.
For the difference terms use
$\delta_z^v(AB)=(\delta_z^vA)B+A(\cdot-z)\delta_z^vB$
before applying the same bounds.
The $Y_m^{1+\varepsilon_0}$- and $Z_m^{1+\varepsilon_0}$-bounds
control the differentiated factor, and interpolation gives the
lower orders. Each factor carries at most $m$ basic derivatives,
including when $m=N_0$. This proves \eqref{esGi}.

For $G_{1,v_0}$ with $(h_1,h_2)=(h,\mu)$ and $G_{2,v_0}$
with $(h_1,h_2)=(\mu,h)$, the same argument gives
\[
\sum_{i=1,2}\sum_{m=0}^{N_0}
\Bigl(
\|\tilde\tau_{v_0}^{1/2}\langle v_0\rangle^\kappa G_{i,v_0}(\tau)\|
_{\tilde Y_m\cap\tilde Z_m}
+\|\tilde\tau_{v_0}^{1/2}\langle v_0\rangle^\kappa G_{i,v_0}(\tau)\|
_{\tilde Y_m^{\varepsilon_0}\cap\tilde Z_m^{\varepsilon_0}}
\Bigr)
\lesssim\tau^{\sigma_0}\|h\|_T.
\]
Here derivatives of $\mu$ are bounded by a broader Gaussian;
the pointwise term gains $\tau^{(d-\kappa)/2}$ and the
differentiated Gaussian gains $\tau^{1/2}$.
Applying the source-time version of \eqref{Si1} therefore yields
\begin{equation}\label{s1hmu}
\|\mathcal S_{1,v_0}(h,\mu)(t,\tau)\|_{\tilde Y_m^k\cap\tilde Z_m^k}
+\|\mathcal S_{2,v_0}(\mu,h)(t,\tau)\|_{\tilde Y_m^k\cap\tilde Z_m^k}
\lesssim
\frac{T^{\sigma_0}\langle\tau/t\rangle^{(k-1)_+/2}}
{(t+\tau)^{\varepsilon_0/2}\min\{t,\tau\}^{1-\varepsilon_0/2}}\|h\|_T .
\end{equation}
\par\endgroup
			{It remains to estimate $\mathcal{S}_{2,v_0}(h,\mu)$. Observe that}
			\begin{align*}
				\mathcal{S}_{2,v_0}(h,\mu)(t,\tau,x,v)&=(H_\mu^{v_0}(t)\ast \mathcal{Q}_{s,\mathrm{rem}}(h,\mu)(\tau))(x-tv,v)\\*
				&=-(\mathcal{D}_{v_0}^{(v)}H_\mu^{v_0}(t)\ast G_{3,v_0}(\tau))(x-tv,v),
			\end{align*}
\begingroup
Here $G_{3,v_0}=(\tilde a_{v_0}\star\mathcal D_{v_0}^{(v)}h)(\tau)\mu$.
For $m\leq N_0$, the full Leibniz sum and Lemma~\ref{lemi1i2} give
\begin{align*}
\mathbf I(\nabla_{\tau,v_0}^{(x,v),m}G_{3,v_0})(x,v)
\lesssim
\sum_{m_1=0}^m&
\mathbf I\bigl(\nabla_{\tau,v_0}^{(x,v),m_1}
(\tilde a_{v_0}\star\mathcal D_{v_0}^{(v)}h)\bigr)(\tau,x,v)\\
&\times\sup_{|w|\leq1}
|\nabla_{\tau,v_0}^{(v),m-m_1}\mu|(v+w).
\end{align*}
Bound each Gaussian factor, including its local supremum,
by $C_m e^{-c|v|^2}$.
Apply the cutoff argument of Lemma~\ref{lemA1}, with radius
$\tilde\tau_v^{1/2}$, also after taking $\delta_z^v$.
The difference product rule and the
$Y_m^{1+\varepsilon_0}$- and $Z_m^{1+\varepsilon_0}$-bounds then give
\begin{align}\label{G31}
&\mathbf w(x,v)\langle v\rangle^\kappa\tilde\tau_v^{1/2}
\Bigl[
\tilde\tau_v^{(d-\kappa)/2}
\bigl(|\nabla_{\tau,v_0}^{(x,v),m}G_{3,v_0}|
+\mathfrak r_{\tau,v}(z)^{-\varepsilon_0}
|\delta_z^v\nabla_{\tau,v_0}^{(x,v),m}G_{3,v_0}|\bigr)\notag\\
&\qquad+\mathbf I(\nabla_{\tau,v_0}^{(x,v),m}G_{3,v_0})
+\mathfrak r_{\tau,v}(z)^{-\varepsilon_0}
\mathbf I(\delta_z^v\nabla_{\tau,v_0}^{(x,v),m}G_{3,v_0})
\Bigr](x,v)\Big|_{v_0=v}\notag\\
&\qquad\lesssim
\bigl(\tau^{(d-\kappa)/2}
+\mathbf1_{\kappa=d-1}\tau^{1/2}|\log\tau|
+\tau^{1/2}\bigr)\|h\|_T
\lesssim\tau^{\sigma_0}\|h\|_T .
\end{align}
The same source bounds also give \eqref{eq:A26-Landau-terminal} for
$\mathcal S_{2,v_0}(h,\mu)$, with the source factor
$\tau^{\sigma_0}\|h\|_T$.
\par\endgroup
			\begingroup
			All four source bounds in \eqref{G31} hold at each source time with
			the factor $\tau^{\sigma_0}$. Applying both branches
			of \eqref{Si1}, with $k=1$ for the Landau velocity
			derivative, gives
			\begin{align}\label{eq:A26-Landau-G3-both-orders}
			\|\mathcal S_{2,v_0}(h,\mu)(t,\tau)\|
			_{\tilde Y_m^k\cap\tilde Z_m^k}
			\lesssim
			\tau^{\sigma_0}
			\left[
			t^{-1+\varepsilon_0/2}\tau^{-\varepsilon_0/2}\mathbf1_{t<\tau}
			+t^{-1/2}\tau^{-1/2}\mathbf1_{\tau<t}
			\right]\langle\tau/t\rangle^{(k-1)_+/2}
			\|h\|_T .
			\end{align}
			For $\tau<t$, the ratio of the second term in
			\eqref{eq:A26-Landau-G3-both-orders} to
			$T^{\sigma_0}t^{-\varepsilon_0/2}\tau^{-1+\varepsilon_0/2}$ is at most
			$(\tau/T)^{\sigma_0}(t/\tau)^{(\varepsilon_0-1)/2}\leq1$.
			For $t<\tau$, the first term is bounded by the right-hand side
			of \eqref{Sh12b} because $\tau\leq T$.
			\par\endgroup
			{Together with \eqref{s1hmu}, this proves \eqref{Sh12b} and completes the Landau case.}\par\medskip
			{\bf Boltzmann Case: $s\in(0,1)$.} {As in the Landau case, the bilinear estimates require integration by parts. Moving a derivative from the nonlinear term to the kernel replaces the nonintegrable singularity $\tau^{-1}$ by $t^{-\varepsilon}\tau^{-1+\varepsilon}$ for some $\varepsilon\in(0,1)$. Redistributing derivatives between the quadratic factors also closes the estimates without derivative loss. In the Boltzmann case, both steps rely on the symmetry of the collision operator. A change of variables gives the following identities for $\mathcal{Q}_{s,\mathrm{main}}$ and $\mathcal{Q}_{s,\mathrm{rem}}$ defined in \eqref{qsmr}:}
			\begin{align}\nonumber
				&\int_{\mathbb{R}^d} g(v)\mathcal{Q}_{s,\mathrm{main}}(h_1,h_2)(v)\,\dif v\\
				&=-\frac{1}{2}\iint_{\mathbb{R}^{2d}}\delta_z^v(g(v)\mathbf{C}_{h_1}(v,z)) \delta_z^v h_2(v)		\mathbf{1}_{|z|\leq 1} \frac{\dif z \dif v}{|z|^{d+2s}} -\iint_{\mathbb{R}^{2d}}g(v)\mathbf{C}_{h_1}(v,z)\delta_z^v h_2(v)		\mathbf{1}_{|z|> 1} \frac{\dif z \dif v}{|z|^{d+2s}}\nonumber\\
				&=-\frac{1}{2}\iint_{\mathbb{R}^{2d}}\delta_z^vg(v)\delta_z^vh_2(v) \mathbf{C}_{h_1}(v,z) 	\mathbf{1}_{|z|\leq 1}\frac{\dif z \dif v}{|z|^{d+2s}} -\frac{1}{2}\iint_{\mathbb{R}^{2d}}g(v)\delta_{-z}^vh_2(v)\delta_{-z}^v \mathbf{C}_{h_1}(v,z) 	\mathbf{1}_{|z|\leq 1}\frac{\dif z \dif v}{|z|^{d+2s}}\nonumber\\
				&\quad\quad -\iint_{\mathbb{R}^{2d}}g(v)\delta_z^vh_2(v)\mathbf{C}_{h_1}(v,z)	\mathbf{1}_{|z|> 1} \frac{\dif z \dif v}{|z|^{d+2s}},\label{z18}
			\end{align}	
			and 
			\begin{align}
				&\int_{\mathbb{R}^d} g(v)\mathcal{Q}_{s,\mathrm{rem}}(h_1,h_2)(v)\,\dif v\nonumber\\
				&=\frac{1}{2}
				\iint_{\mathbb{R}^{2d}}\delta_z^v (g(v)h_2(v))\delta_z\mathbf{C}_{h_1}(v,z)\mathbf{1}_{|z|\leq 1}\frac{\dif z \dif v}{|z|^{d+2s}}+	\iint_{\mathbb{R}^{2d}}g(v)h_2(v)\delta_z^v\mathbf{C}_{h_1}(v,z)\mathbf{1}_{|z|> 1}\frac{\dif z \dif v}{|z|^{d+2s}}\nonumber\\
				&=\frac{1}{2}	\iint_{\mathbb{R}^{2d}} \delta_z^vg(v)h_2(v)\delta_z^v\mathbf{C}_{h_1}(v,z)\mathbf{1}_{|z|\leq 1}\frac{\dif z \dif v}{|z|^{d+2s}}+\frac{1}{2}	\iint_{\mathbb{R}^{2d}} g(v)\delta_{-z}^vh_2(v)\delta_{-z}^v\mathbf{C}_{h_1}(v,z)\mathbf{1}_{|z|\leq 1}\frac{\dif z \dif v}{|z|^{d+2s}}\nonumber\\
				&\quad\quad +\iint_{\mathbb{R}^{2d}}g(v)h_2(v)\delta_z^v\mathbf{C}_{h_1}(v,z)\mathbf{1}_{|z|> 1}\frac{\dif z \dif v}{|z|^{d+2s}}.\label{z19}
			\end{align}
			{Define}
			\begin{align*}
				&	\mathscr S_1(\tau,y,u;z)=\delta_z^v	h_2(\tau,y,u)	\mathbf{C}_{h_1(\tau,y)}(u,z),\quad\quad 	\mathscr S_2(\tau,y,u;z)= h_2(\tau,y,u)\delta_z^v\mathbf{C}_{h_1(\tau,y)}(u,z),	\\
				&\mathscr S_3(\tau,y,u;z)= \frac{1}{2}\delta^v_{-z} h_2(\tau,y,u)\delta_{-z}^v\mathbf{C}_{h_1(\tau,y)}(u,z).
			\end{align*}
			{{We first use \eqref{z18}--\eqref{z19} for smooth data.
In $\mathcal S_{v_0}=\mathcal S_{1,v_0}+\mathcal S_{2,v_0}$,
the $\mathcal Y_3$ terms cancel before estimating the jump integral.
Separate bounds for $\mathcal Y_3$ are needed only in the Maxwellian cross terms.} Equations \eqref{z18} and \eqref{z19} give}
			\begin{align*}
				&		\mathcal{S}_{1,v_0}(h_1,h_2) (t,\tau,x,v)=-\mathcal{Y}_{1,v_0}(h_1,h_2)(t,\tau,x,v)-\mathcal{Y}_{3,v_0}(h_1,h_2)(t,\tau,x,v),\\
				&		\mathcal{S}_{2,v_0}(h_1,h_2) (t,\tau,x,v)=\mathcal{Y}_{2,v_0}(h_1,h_2)(t,\tau,x,v)+\mathcal{Y}_{3,v_0}(h_1,h_2)(t,\tau,x,v),
			\end{align*}
			{where $\mathcal{Y}_{i,v_0}(h_1,h_2)$ are defined by}
			\begin{align*}
				\mathcal{Y}_{1,v_0}(t,\tau,x,v) &=\iint_{\mathbb{R}^{2d}}\int_{\mathbb{R}^d}\mathbf{H}^{v_0}(t,x-tv-y,v-u;z) \mathscr S_1(\tau,y,u;z)\frac{\dif z}{|z|^{d+2s}}\dif y\dif u,\\
				\mathcal{Y}_{2,v_0}(t,\tau,x,v)&=\iint_{\mathbb{R}^{2d}}\int_{\mathbb{R}^d}\mathbf{H}^{v_0}(t,x-tv-y,v-u;z)\mathscr S_2(\tau,y,u;z )\frac{\dif z}{|z|^{d+2s}}\dif y\dif u,\\
				\mathcal{Y}_{3,v_0}(t,\tau,x,v)&=\iint_{\mathbb{R}^{2d}}\int_{|z|\leq 1}H_\mu^{v_0}(t, x-tv-y,v-u)\mathscr S_3(\tau,y,u;z)\frac{\dif z}{|z|^{d+2s}}\dif y\dif u.
			\end{align*}
			{Here $\mathbf{H}^{v_0}(t,x,v;z)=\frac{1}{2}\mathbf{1}_{|z|\leq 1}(\delta_{-z}^vH_\mu^{v_0})(t,x,v)+\mathbf{1}_{|z|> 1}H_\mu^{v_0}(t,x,v)$. Fix $m,k\in\mathbb{N}$ with $m\leq N_0$. Then}
			\begin{align*}
				&|\nabla^{(x,v),m}_{t+\tau,v_0}\mathcal{D}_{t+\tau,v_0}^{(v),k}	\mathcal{Y}_{1,v_0}|(t,\tau,x,v)\\ &\lesssim\iint_{\mathbb{R}^{2d}}\int_{\mathbb{R}^d}\left(|\nabla_{t,v_0}^{(x,v),m-m_1}\mathcal{D}_{t+\tau,v_0}^{(x,v),k}\mathbf{H}^{v_0}|(t,x-tv-y,v-u;z)\, |\nabla_{\tau,v_0}^{(x,v),m_1}\mathscr S_1|(\tau,y,u;z)\right)\frac{\dif z \dif y\dif u}{|z|^{d+2s}},
			\end{align*}
			{where}
			\begin{align}\label{vt12}
			m_1=\begin{cases}
				m,\ \ \ \text{if } 0<t<\tau<T,\\
			0,\ \ ~\ \ \text{if } 0<\tau<t<T.
				\end{cases}
			\end{align}
				{Similarly,}
			\begin{align*}
				&	|\nabla^{(x,v),m}_{t+\tau,v_0}\mathcal{D}_{t+\tau,v_0}^{(v),k}	\mathcal{Y}_{2,v_0}(t,\tau,x,v)|\\
				&\quad\quad\lesssim  \iint_{\mathbb{R}^{2d}}\int_{\mathbb{R}^d}|\nabla_{t,v_0}^{(x,v),m-m_1}\mathcal{D}_{t+\tau,v_0}^{(x,v),k}\mathbf{H}^{v_0}|(t,x-tv-y,v-u;z)|\nabla_{\tau,v_0}^{(x,v),m_1}\mathscr S_2|(\tau,y,u;z)\frac{\dif z \dif y\dif u}{|z|^{d+2s}},\\
				&|\nabla^{(x,v),m}_{t+\tau,v_0}\mathcal{D}_{t+\tau,v_0}^{(v),k}				\mathcal{Y}_{3,v_0}(t,\tau,x,v)|\\
				&\quad\quad\lesssim\iint_{\mathbb{R}^{2d}}\int_{|z|\leq 1}|\nabla_{t,v_0}^{(x,v),m-m_1}\mathcal{D}_{t+\tau,v_0}^{(x,v),k}H_\mu^{v_0}|(t,x-tv-y,v-u) |\nabla_{\tau,v_0}^{(x,v),m_1}\mathscr S_3|(\tau,y,u;z)\frac{\dif z \dif y\dif u}{|z|^{d+2s}}.
			\end{align*}
			{In this Boltzmann case, Lemma \ref{Hmu} gives}
			\begin{equation}\label{delH}
				\begin{aligned}
					&|\nabla_{t,v_0}^{(x,v),l}\mathcal{D}_{t,v_0}^{(x,v),j} \mathbf{H}^{v_0}|(t,x,v;z)\lesssim {\mathfrak r_{t,v_0}(z)\bigl(\Phi_{t,v_0}(x,v)+\mathbf1_{|z|\leq1}\Phi_{t,v_0}(x,v+z)\bigr)},\\
					&|\nabla_{t,v_0}^{(x,v),l}\mathcal{D}_{t,v_0}^{(x,v),j} H_\mu^{v_0}|(t,x,v)\lesssim \Phi_{t,v_0}(x,v),\ \ \ \ \forall l, j\in\mathbb{N},
				\end{aligned}
			\end{equation}
			\begingroup
The first estimate follows from the mean value theorem for subunit
normalized increments and the triangle inequality otherwise,
using the local comparability of $\Phi_{t,v_0}$.
\par\endgroup
			\begingroup
			For $[h]_{v_0},[z]_{v_0}\leq1$ and fixed $l,j\geq0$, the mixed
kernel difference satisfies
\begin{align}\label{eq:A26-mixed-frozen-kernel}
&\bigl|\delta_{(th,-h)}^{(x,v)}\delta_{-z}^v
\nabla_{t,v_0}^{(x,v),l}\mathcal D_{t,v_0}^{(x,v),j}
H_\mu^{v_0}(t,x,v)\bigr|\notag\\
&\qquad\lesssim
\mathfrak r_{t,v_0}(h)\mathfrak r_{t,v_0}(z)
\sum_{i_1,i_2\in\{0,1\}}\Phi_{t,v_0}(x-i_1th,v+i_1h+i_2z).
\end{align}
			Apply the same argument successively in the commuting directions
			$(-t\nabla_x+\nabla_v)\cdot h$ and $\nabla_v\cdot z$.
			\par\endgroup
			\begingroup
			Consequently, for $0<[h]_v\leq1$ and $i=1,2$,
			\begin{align}\label{eq:A26-Y12-terminal-increment}
			&\mathbf w(x,v)
			\Bigl[
			(\tilde t_v^{\frac{d-\kappa}{2s}}+\tilde\tau_v^{\frac{d-\kappa}{2s}})
			|\nabla_{t+\tau,v_0}^{(x,v),m}
			\delta_h^v\mathcal Y_{i,v_0}(t,\tau,x,v)|
			+\mathbf I\bigl(
			\nabla_{t+\tau,v_0}^{(x,v),m}
			\delta_h^v\mathcal Y_{i,v_0}(t,\tau)
			\bigr)(x,v)
			\Bigr]_{v_0=v}\\
			&\qquad\lesssim
			\mathfrak r_{t,v}(h)t^{-1+\frac{\varepsilon_0}{2s}}\tau^{-\frac{\varepsilon_0}{2s}}
			\|h_1\|_T\|h_2\|_T,\qquad i=1,2.\notag
			\end{align}
			Apply \eqref{eq:A26-mixed-frozen-kernel} inside the convolution
			and shift $u\mapsto u-i_1h-i_2z$, $y\mapsto y+i_1th$.
			For $|z|\leq1$, use the bounded-translation comparisons in
			\eqref{eq:A25-envelope-shift}; for $|z|>1$, only the shift $i_1h$
			occurs, and \eqref{eq:global-Omega-unweighted}--\eqref{eq:global-Omega-weighted}
			apply at $v+i_1h$.
			The source bounds used in \eqref{Y12.1} and the finite covering
			\eqref{eq:potential-transport-cover} for $\mathbf I$ then yield
			\eqref{eq:A26-Y12-terminal-increment}.
			\par\endgroup
			{Recalling \eqref{defGv}, the kernel bound \eqref{delH} gives}
			\begin{align*}
				&	|\nabla^{(x,v),m}_{t+\tau,v_0}\mathcal{D}_{t, v_0}^{(v),k}	\mathcal{Y}_{i,v_0}|(t,\tau,x,v)
				\lesssim \mathcal{G}_{v_0}({\mathfrak{N}}_{i,v_0}^{m_1}(\tau))(t,x,v),\ \ \ i=1,2,3,
			\end{align*}
			{where, for fixed $t$,}
			\begin{align*}
					&	{{\mathfrak{N}}_{i,v_0}^{m_1}(\tau,x,v)=	\int_{\mathbb{R}^d}\mathfrak r_{t,v_0}(z)\bigl(|\nabla_{\tau,v_0}^{(x,v),m_1}\mathscr S_i(\tau,x,v;z)|+\mathbf1_{|z|\leq1}|\nabla_{\tau,v_0}^{(x,v),m_1}\mathscr S_i(\tau,x,v+z;z)|\bigr) \frac{\dif z}{|z|^{d+2s}},\ \ i=1,2,}\\
				&	{\mathfrak{N}}_{3,v_0}^{m_1}(\tau,x,v)=\int_{|z|\leq 1} |\nabla_{\tau,v_0}^{(x,v),m_1}\mathscr S_3(\tau,x,v;z)|\frac{\dif z}{|z|^{d+2s}}.
			\end{align*}
			\begingroup
In the shifted summand, $u\mapsto u-z$ restores the unshifted kernel
and evaluates the source at $u+z$. For $|z|\leq1$, the polynomial
weights are comparable at both endpoints and $\mathbf I$ commutes
with translation, while
$\langle v+z\rangle\asymp\langle v\rangle$,
$[z]_{v+z}\asymp[z]_v$.
The indicator in $\Omega_a(v,z)$ is then one;
$\Pi_{\widehat z}(v+z)=\Pi_{\widehat z}v$ and the parallel component
changes by at most one.  Hence, for every exponent used here,
\begin{equation}\label{eq:A25-envelope-shift}
\Omega_a(v+z,z)\asymp\Omega_a(v,z),\qquad
\Omega_{\mathbf w}(x,v+z,z)\asymp\Omega_{\mathbf w}(x,v,z),
\qquad |z|\leq1.
\end{equation}
The comparison for $\Omega_{\mathbf w}$ uses the minimum of the two
global coefficient bounds in \eqref{def:weighted-C-majorant}.
\par\endgroup
					{Using \eqref{Cfs1} and then \eqref{eq:global-Omega-weighted} in the final step, we obtain}
			\begin{align*}
				&\left(\tilde \tau_v^\frac{d-\kappa}{2s}{\mathfrak{N}}_{1,v_0}^{m_1}(\tau,x,v)+\mathbf{I}{\mathfrak{N}}_{1,v_0}^{m_1}(\tau,x,v)\right)\Big|_{v_0=v}\\
&{\lesssim\sum_{m_2=0}^{m_1}\sum_{\epsilon=0}^1
  \int_{\mathbb R^d}\bigl(\mathbf1_{\epsilon=0}+\mathbf1_{\epsilon=1}\mathbf1_{|z|\leq1}\bigr)\mathfrak r_{t,v_0}(z)}\\
&\quad{\times\Bigl(
  \tilde\tau_v^{\frac{d-\kappa}{2s}}
  |\delta_z^v\nabla^{(x,v),m_2}_{\tau,v_0}h_2|
  +|\mathbf I(\delta_z^v\nabla^{(x,v),m_2}_{\tau,v_0}h_2)|
  \Bigr)(\tau,x,v+\epsilon z)}\\
&\quad{\times
  |\nabla^{(x,v),m_1-m_2}_{\tau,v_0}\mathbf C_{h_1}|
  (\tau,x,v+\epsilon z,z)\frac{\dif z}{|z|^{d+2s}}
  \Big|_{v_0=v}}\\
&\lesssim \int_{\mathbb{R}^d} \mathfrak r_{t,v}(z)\mathfrak r_{\tau,v}(z)^{\varepsilon_0}(\langle v\rangle^{-\varkappa}+\langle v-z\rangle^{-\varkappa})\Omega_{\mathbf w}(x,v,z) \frac{\dif z}{|z|^{d+2s}}\|h_1\|_T {\|h_2\|_T}\\
				&\lesssim \tilde \tau_v^{-\frac{\varepsilon_0}{2s}}\tilde t_v^{-\frac{2s-\varepsilon_0}{2s}}\mathbf{w}^{-1}(x,v)\langle v\rangle^{-\kappa}\|h_1\|_T {\|h_2\|_T}.
			\end{align*}
			\begingroup
For $\mathfrak N_2$, use \eqref{delCfs2} on $[z]_v\leq1$
and the triangle bound from \eqref{Cfs1} otherwise.
The frozen-frame comparison costs at most the already available
velocity weights. The hypotheses of Theorem~\ref{Th1} give the needed
stronger inequality: since $|\gamma|<d+2s\leq d+2$ and
$\kappa\varepsilon_0/s<d$,
\[
 2\xi_{N_0}=\frac{|\gamma|N_0}{s}
 +\frac{\kappa\varepsilon_0}{s}
 <\frac{(d+2)N_0}{s}+d,
 \qquad
 \varkappa-2\xi_{N_0}>d+2s,
\]
because $\varkappa\geq\varkappa_2>10s^{-1}(d+2)N_0$ and
$N_0\geq10d$. This stronger condition is used here, rather than just
$\varkappa-\xi_{N_0}>d$ from Lemma~\ref{lemCf1}. Both bounds imply
\[
|\delta_z^v\nabla_{\tau,v_0}^{(x,v),j}\mathbf C_{h_1}|
(\tau,x,v,z)\Big|_{v_0=v}
\lesssim\mathfrak r_{\tau,v}(z)^{\varepsilon_0}
\bigl(\Omega_{\varkappa-2\xi_{N_0}}(v,z)
+\Omega_{\varkappa-2\xi_{N_0}}(v-z,z)\bigr)\|h_1\|_T,
\quad j\leq N_0.
\]
On the complementary region $\mathfrak r_{\tau,v}(z)=1$. Thus
\par\endgroup
            	\begin{align*}
			&\left(\tilde \tau_v^\frac{d-\kappa}{2s}{\mathfrak{N}}_{2,v_0}^{m_1}(\tau,x,v)+\mathbf{I}{\mathfrak{N}}_{2,v_0}^{m_1}(\tau,x,v)\right)\Big|_{v_0=v}\\
			&{\lesssim\sum_{m_2=0}^{m_1}
  \sum_{\epsilon=0}^1\int_{\mathbb R^d}
  \bigl(\mathbf1_{\epsilon=0}+\mathbf1_{\epsilon=1}\mathbf1_{|z|\leq1}\bigr)\mathfrak r_{t,v_0}(z)}\\
&\quad{\times\Bigl(
  \tilde\tau_v^{\frac{d-\kappa}{2s}}
  |\nabla^{(x,v),m_2}_{\tau,v_0}h_2|
  +\mathbf I(\nabla^{(x,v),m_2}_{\tau,v_0}h_2)
  \Bigr)(\tau,x,v+\epsilon z)}\\
&\quad{\times
  |\nabla^{(x,v),m_1-m_2}_{\tau,v_0}
  \delta_z^v\mathbf C_{h_1}|(\tau,x,v+\epsilon z,z)
  \frac{\dif z}{|z|^{d+2s}}\Big|_{v_0=v}}\\
		                &\lesssim \mathbf{w}^{-1}(x,v)\int_{\mathbb{R}^d}\mathfrak r_{t,v_0}(z) \mathfrak r_{\tau,v_0}(z) ^{\varepsilon_0}\big({\Omega_{\varkappa-2\xi_{N_0}}(v,z)+\Omega_{\varkappa-2\xi_{N_0}}(v-z,z)}\big)\frac{\dif z}{|z|^{d+2s}}\!\!\Big|_{v_0=v}{\|h_1\|_T} \|h_2\|_{T}\\
				&\lesssim\tilde \tau_v^{-\frac{\varepsilon_0}{2s}}\tilde t_v^{-\frac{2s-\varepsilon_0}{2s}} \mathbf{w}^{-1}(x,v)\langle v\rangle^{-\kappa}{\|h_1\|_T} \|h_2\|_{T}.
			\end{align*}
					{{The last inequality follows from
\eqref{eq:global-Omega-unweighted} with $a=\varkappa-2\xi_{N_0}$.}}
				
			{{For $i=1,2$,} the weight $\mathbf w(x,v)\langle v\rangle^\kappa$ is admissible in the sense of Definition~\ref{defad}. Lemma~\ref{le11} therefore gives}
			\begin{align*}
				&\sup_{x,v} \mathbf w(x,v)\langle v\rangle^\kappa\left((\tilde t_{v_0}^\frac{d-\kappa}{2s}+\tilde \tau_{v_0}^\frac{d-\kappa}{2s})\,\mathcal{G}_{v_0}({\mathfrak{N}}_{i,v_0}^{m_1}(\tau))+\mathbf{I}\mathcal{G}_{v_0}({\mathfrak{N}}_{i,v_0}^{m_1}(\tau))\right)(t,x,v)\Big|_{v_0=v}\\
				&\quad\quad	\lesssim 	\sup_{x,v} \mathbf w(x,v)\langle v\rangle^\kappa\left(\tilde  \tau_{v}^\frac{d-\kappa}{2s}|{\mathfrak{N}}_{i,v_0}^{m_1}(\tau,x,v)|+|\mathbf{I}{\mathfrak{N}}_{i,v_0}^{m_1}(\tau,x,v)|\right)\Big|_{v_0=v}.
			\end{align*}
			{Consequently,}
			\begingroup
			\begin{align}\label{Y12.1}
			&\sum_{i=1,2}\|\mathcal{Y}_{i,v_0}(h_1,h_2)(t,\tau)\|_{\tilde Y_{m}^k\cap\tilde Z_{m}^k}\notag\\*
			&\qquad\lesssim\langle\tau/t\rangle^{k/(2s)}t^{-1}(t/\tau)^{\varepsilon_0/(2s)}\|h_1\|_T\|h_2\|_T.
			\end{align}
			\endgroup
\begingroup
For the Maxwellian cross terms, one source difference is first order:
\[
|\delta_z^v\nabla_{\tau,v_0}^{(v),j}\mu|(v)
\lesssim_j |z|e^{-c|v|^2},\qquad |z|\leq1.
\]
When $\mu$ is in the coefficient, apply the same bound inside its
Carleman integral. The remaining difference is controlled by the
$\varepsilon_0$-seminorm of $h$.
The angular bounds used in Lemma~\ref{lem:global-Omega-integration}
and
$\int_0^1\rho^{\varepsilon_0-2s}\,\dif\rho
=(1+\varepsilon_0-2s)^{-1}$ therefore give
\[
\mathbf w(x,v)\sum_{(f_1,f_2)\in\{(h,\mu),(\mu,h)\}}
\Bigl[
\tilde\tau_v^{\frac{d-\kappa}{2s}}
\mathfrak N_{3,v_0}^{m_1}(f_1,f_2)(\tau,x,v)
+\mathbf I\mathfrak N_{3,v_0}^{m_1}(f_1,f_2)(\tau,x,v)
\Bigr]_{v_0=v}
\lesssim\tau^{-\frac{\varepsilon_0}{2s}}\|h\|_T
\lesssim\tau^{-1+\sigma_0}\|h\|_T .
\]
Applying Lemma~\ref{le11} with the weight $\mathbf w$ yields
\begin{equation}\label{Y3.1}
\|\mathcal Y_{3,v_0}(h,\mu)(t,\tau)\|_{\tilde Y_m^k\cap\tilde Z_m^k}
+\|\mathcal Y_{3,v_0}(\mu,h)(t,\tau)\|_{\tilde Y_m^k\cap\tilde Z_m^k}
\lesssim\langle\tau/t\rangle^{k/(2s)}
\tau^{-1+\sigma_0}\|h\|_T .
\end{equation}
For $\mathcal Y_1(h,\mu)$ and $\mathcal Y_2(\mu,h)$, the same
Gaussian difference is paired with the first-order kernel difference.
The resulting radial factor satisfies
\[
1+\int_0^1
\left(1\wedge\frac{\rho}{t^{1/(2s)}}\right)
\rho\,\frac{\dif\rho}{\rho^{1+2s}}
\lesssim
T^{\sigma_0}t^{-1+\frac{\varepsilon_0}{2s}}
\tau^{-\frac{\varepsilon_0}{2s}},
\qquad 0<t,\tau<T<1.
\]
Indeed, split at $t^{1/(2s)}$; the logarithmic case $2s=1$
is absorbed since $\sigma_0\leq1/10$.
Consequently,
\[
\|\mathcal Y_{1,v_0}(h,\mu)(t,\tau)\|_{\tilde Y_m^k\cap\tilde Z_m^k}
+\|\mathcal Y_{2,v_0}(\mu,h)(t,\tau)\|_{\tilde Y_m^k\cap\tilde Z_m^k}
\lesssim
\langle\tau/t\rangle^{k/(2s)}T^{\sigma_0}
t^{-1+\frac{\varepsilon_0}{2s}}
\tau^{-\frac{\varepsilon_0}{2s}}\|h\|_T .
\]
Taking a terminal kernel difference before these integrations adds
$\mathfrak r_{t,v}(\eta)$ to the last bound, by
\eqref{eq:A26-mixed-frozen-kernel}.
\par\endgroup
			\begingroup
			The separate Maxwellian $\mathcal Y_3$ terms satisfy, for $0<[\eta]_v\leq1$,
\begin{align}\label{eq:A26-Y3-terminal-increment}
&\mathbf w(x,v)\sum_{(f_1,f_2)\in\{(h,\mu),(\mu,h)\}}
\Bigl[
(\tilde t_v+\tilde\tau_v)^{\frac{d-\kappa}{2s}}
\bigl|\nabla_{t+\tau,v_0}^{(x,v),m}\delta_\eta^v
\mathcal Y_{3,v_0}(f_1,f_2)(t,\tau)\bigr|(x,v)\notag\\
&\qquad\qquad+\mathbf I\bigl(
\nabla_{t+\tau,v_0}^{(x,v),m}\delta_\eta^v
\mathcal Y_{3,v_0}(f_1,f_2)(t,\tau)
\bigr)(x,v)\Bigr]_{v_0=v}\notag\\
&\qquad\lesssim\mathfrak r_{t,v}(\eta)\tau^{-1+\sigma_0}\|h\|_T.
\end{align}
			Take the terminal kernel difference and apply Lemma~\ref{Hmu}.
			Transfer the two endpoint profiles by bounded velocity
			translation and \eqref{eq:potential-transport-cover}, then
			use the Maxwellian source bound preceding \eqref{Y3.1}.
			\par\endgroup
			\begingroup
			The time factors in
\eqref{Y12.1} and \eqref{S2v0} are integrable in source time:
for every $t>0$,
			\begin{align}\label{eq:A26-asymmetric-Duhamel}
			\int_0^t (t-\tau)^{-1+\frac{\varepsilon_0}{2s}}\tau^{-\frac{\varepsilon_0}{2s}}\,\dif\tau
			&=B\bigl(\tfrac{\varepsilon_0}{2s},1-\tfrac{\varepsilon_0}{2s}\bigr),\\
			\int_0^t\tau^{\sigma_0}
			\bigl((t-\tau)^{-1+\frac{\varepsilon_0}{2s}}\tau^{-\frac{\varepsilon_0}{2s}}+\tau^{-1}\bigr)
			\,\dif\tau
			&=B\bigl(\tfrac{\varepsilon_0}{2s},1-\tfrac{\varepsilon_0}{2s}+\sigma_0\bigr)t^{\sigma_0}
			+\frac{t^{\sigma_0}}{\sigma_0}.
			\notag
			\end{align}
			Retaining the terminal cutoff and splitting at
			$\tau=t/2$ and at $t-\tau=|h|^{2s}$ gives
			\begin{align}\label{eq:A26-terminal-endpoint-integrals}
			&\int_0^t(t-\tau)^{-1+\frac{\varepsilon_0}{2s}}\tau^{-\frac{\varepsilon_0}{2s}}
			\left(1\wedge\frac{|h|}{(t-\tau)^{1/(2s)}}\right)
			\dif\tau
			\lesssim 1\wedge\ell_t^v(h)^{\varepsilon_0},\notag\\
			&\int_0^t\tau^{-1+\sigma_0}
			\left(1\wedge\frac{|h|}{(t-\tau)^{1/(2s)}}\right)
			\dif\tau
			\lesssim t^{\sigma_0}(1\wedge\ell_t^v(h)^{\varepsilon_0}).
			\end{align}
			On $\tau<t/2$, bound the cutoff by $1\wedge\ell_t^v(h)$.
On $\tau>t/2$, the split at $t-\tau=|h|^{2s}$ proves the first
bound because $\varepsilon_0<1$; for the second, use
$1\wedge X\leq X^{\varepsilon_0}$ and $\varepsilon_0/(2s)<1$. The anisotropic application replaces $|h|$
			by $[h]_v$ and $t$ by $\tilde t_v$.
			\par\endgroup
			{{Canceling $\mathcal Y_3$ and applying
\eqref{Y12.1}, we obtain}}
			\begin{align*}
				\|\mathcal{S}_{v_0}(h_1,h_2)(t,\tau)\|_{\tilde Y^k_{m}\cap \tilde Z^k_{m}}&\lesssim \sum_{i=1,2}\|\mathcal{Y}_{i,v_0}(h_1,h_2)(t,\tau)\|_{\tilde Y_{m}^k\cap \tilde Z_{m}^k }\\
				&\lesssim \langle{\tau}/{t}\rangle^\frac{k}{2s}t^{-1}(t/\tau)^\frac{\varepsilon_0}{2s} \|h_1\|_{T}\|h_2\|_{T},
			\end{align*}
\begingroup
The preceding Maxwellian bounds give
\begin{align}\label{S2v0}
\|\mathcal S_{2,v_0}(\mu,h)(t,\tau)\|_{\tilde Y_m^k\cap\tilde Z_m^k}
&\lesssim\langle\tau/t\rangle^{k/(2s)}
\left(T^{\sigma_0}t^{-1+\frac{\varepsilon_0}{2s}}
\tau^{-\frac{\varepsilon_0}{2s}}+\tau^{-1+\sigma_0}\right)\|h\|_T,\\
\label{S1v0}
\|\mathcal S_{1,v_0}(h,\mu)(t,\tau)\|_{\tilde Y_m^k\cap\tilde Z_m^k}
&\lesssim\langle\tau/t\rangle^{k/(2s)}
\left(T^{\sigma_0}t^{-1+\frac{\varepsilon_0}{2s}}
\tau^{-\frac{\varepsilon_0}{2s}}+\tau^{-1+\sigma_0}\right)\|h\|_T .
\end{align}
\par\endgroup
			{It remains to estimate $\mathcal{S}_{2,v_0}(h,\mu)$. Observe that}
			\begin{align*}
				&	|\nabla_{t+\tau,v_0}^{(x,v),m} \mathcal{D}_{t+\tau,v_0}^{(v),k} \mathcal{S}_{2,v_0}(h,\mu)(t,\tau,x,v)|\\
				&\quad\quad\quad\quad\lesssim \iint_{\mathbb{R}^{2d}}|\nabla_{t,v_0}^{(x,v),m-m_1}\mathcal{D}_{t+\tau,v_0}^{(x,v),k}H_\mu^{v_0}(t,x-tv-y,v-u)\nabla_{\tau,v_0}^{(x,v),m_1}	\mathcal{Q}_{s,\mathrm{rem}}(h,\mu)(\tau,y,u)|\dif y\dif u\\
				&\quad\quad\quad\quad\lesssim   \mathcal{G}_{v_0}(\nabla_{\tau,v_0}^{(x,v),m_1}	\mathcal{Q}_{s,\mathrm{rem}}(h,\mu)(\tau))(t,x,v),
			\end{align*}
			{where $\mathcal{G}_{v_0}$ is defined in \eqref{defGv}, and $m_1$ is fixed by \eqref{vt12}. Lemma \ref{le11} gives}
			\begin{align*}
				&\sup_{x,v} \mathbf{w}(x,v)\left((\tilde  t_{v_0}^\frac{d-\kappa}{2s}+\tilde  \tau_{v_0}^\frac{d-\kappa}{2s})\mathcal{G}_{v_0}(\nabla_{\tau,v_0}^{(x,v),m_1}	\mathcal{Q}_{s,\mathrm{rem}}(h,\mu)(\tau))+\mathbf{I}\mathcal{G}_{v_0}(\nabla_{\tau,v_0}^{(x,v),m_1}	\mathcal{Q}_{s,\mathrm{rem}}(h,\mu)(\tau))\right)(t,x,v)\Big|_{v_0=v}\\
				&\quad\quad	\lesssim 	\sup_{x,v} \mathbf{w}(x,v)\left(\tilde \tau_{v}^\frac{d-\kappa}{2s}|(\nabla_{\tau,v_0}^{(x,v),m_1}	\mathcal{Q}_{s,\mathrm{rem}}(h,\mu)(\tau,x,v))|+|\mathbf{I}\nabla_{\tau,v_0}^{(x,v),m_1}	\mathcal{Q}_{s,\mathrm{rem}}(h,\mu)(\tau,x,v)|\right)\Big|_{v_0=v}.
			\end{align*}
			{Recall that}
			\begin{align*}
				\mathcal{Q}_{s,\mathrm{rem}}(f_1,f_2)(v)&=f_2(v)\int_{\mathbb{R}^d}\left( \mathbf{C}_{f_1}(v,z)-\mathbf{C}_{f_1}(v-z,z)\right)\frac{dz}{|z|^{d+2s}}:=f_2(v)\mathcal{K}(f_1)(v).
			\end{align*}
			{If $-d-\gamma\leq 0$, Lemma \ref{lemCfl} yields}
			\begin{align*}
				\sup_{x,v} \mathbf{w}(x,v)\left(\tilde \tau_{v}^\frac{d-\kappa}{2s}|\nabla_{\tau,v_0}^{(x,v),m_1}	\mathcal{Q}_{s,\mathrm{rem}}(h,\mu)|(\tau,x,v)+|\mathbf{I}\nabla_{\tau,v_0}^{(x,v),m_1}	\mathcal{Q}_{s,\mathrm{rem}}(h,\mu)|(\tau,x,v)\right)\Big|_{v_0=v}\lesssim \tau^{-1+\sigma_0}\|h\|_T.
			\end{align*}
			{Hence,}
			\begin{align*}
				\|\mathcal{S}_{2,v_0}(h,\mu)(t,\tau)\|_{\tilde Y^k_{m}}+	\|\mathcal{S}_{2,v_0}(h,\mu)(t,\tau)\|_{\tilde Z^k_{m}}\lesssim \tau^{-1+\sigma_0} \langle{\tau}/{t}\rangle^\frac{k}{2s}\|h\|_T.
			\end{align*}
			{If instead $-d-2s<\gamma<-d$, set
			\[
				\alpha:=-d-\gamma\in(0,2s).
			\]
			Lemma \ref{lemcanc} and fractional integration by parts give, for every smooth $g$,
			\[
				\int_{\mathbb R^d}g\,\mathcal Q_{s,\mathrm{rem}}(h,\mu)\,\dif v
				={\frac{c_{d,\gamma,s,b}C_{d,\alpha}}{2}}\iint_{\mathbb R^{2d}}
				\delta_z^v(g\mu)\,\delta_z^vh\,
				\frac{\dif z\dif v}{|z|^{d+\alpha}}.
			\]
{Here $C_{d,\alpha}$ is the standard singular-integral
normalization constant for $\Lambda^\alpha$.}
			Apply this identity with $g$ equal to each differentiated translate of $H_\mu^{v_0}(t)$ arising above. Expanding
			$\delta_z^v(g\mu)=\mu\,\delta_z^vg+g(\cdot-z)\delta_z^v\mu$, estimate the first term by \eqref{delH} and the $\varepsilon_0$-seminorm in \eqref{defyz}; the second term and the large-jump part are easier because $\mu$ is rapidly decaying. After the same weight and angular integrations used for $\mathcal Y_{1,v_0}$ and $\mathcal Y_{2,v_0}$, the only new time factor is bounded by
			\[
				1+\int_0^1
				\left(1\wedge\frac{\rho}{t^{1/(2s)}}\right)
				\left(1\wedge\frac{\rho}{\tau^{1/(2s)}}\right)^{\varepsilon_0}
				\frac{\dif\rho}{\rho^{1+\alpha}}
				\lesssim
				\frac{T^{\frac{d-\kappa}{2s}}}{(t+\tau)^{\frac{\varepsilon_0}{2s}}
				\min\{t,\tau\}^{1-\frac{\varepsilon_0}{2s}}}.
			\]
			Indeed, this follows by splitting at $t^{1/(2s)}$ and $\tau^{1/(2s)}$, since $\alpha<2s<1+\varepsilon_0$; the endpoint logarithms are harmless. For example, when $t\leq\tau$ and $\alpha>\varepsilon_0$, the leading term is
			$t^{-(\alpha-\varepsilon_0)/(2s)}\tau^{-\varepsilon_0/(2s)}
			=t^{-1}(t/\tau)^{\varepsilon_0/(2s)}t^{\frac{d-\kappa}{2s}}$, which displays the required gain.
			Consequently, Lemmas \ref{lemi1i2} and \ref{le11} yield
			\[
				\|\mathcal S_{2,v_0}(h,\mu)(t,\tau)\|_{\tilde Y_m^k}
				+\|\mathcal S_{2,v_0}(h,\mu)(t,\tau)\|_{\tilde Z_m^k}
				\lesssim
				\frac{T^{\frac{d-\kappa}{2s}}\langle\tau/t\rangle^{k/(2s)}}
				{(t+\tau)^{\frac{\varepsilon_0}{2s}}
				\min\{t,\tau\}^{1-\frac{\varepsilon_0}{2s}}}\|h\|_T.
			\]
			Finally, $\sigma_0\leq\frac{d-\kappa}{2s}$ by \eqref{def:sigma0}, so $T^{\frac{d-\kappa}{2s}}\leq T^{\sigma_0}$.}
			\begingroup
			A terminal kernel difference gives, for $0<[\eta]_v\leq1$,
			\begin{align}\label{eq:A26-hmu-rem-terminal}
			&\mathbf w(x,v)\Bigl[
(\tilde t_v+\tilde\tau_v)^{\frac{d-\kappa}{2s}}
\bigl|\nabla_{t+\tau,v_0}^{(x,v),m}\delta_\eta^v
\mathcal S_{2,v_0}(h,\mu)(t,\tau)\bigr|(x,v)\notag\\
&\qquad\qquad+\mathbf I\bigl(
\nabla_{t+\tau,v_0}^{(x,v),m}\delta_\eta^v
\mathcal S_{2,v_0}(h,\mu)(t,\tau)
\bigr)(x,v)\Bigr]_{v_0=v}\notag\\
&\qquad\lesssim
			\mathfrak r_{t,v}(\eta)\|h\|_T
			\begin{cases}
			\tau^{-1+\sigma_0},
			&\gamma\geq-d,\\
			\dfrac{T^{\frac{d-\kappa}{2s}}}{(t+\tau)^{\frac{\varepsilon_0}{2s}}\min\{t,\tau\}^{1-\frac{\varepsilon_0}{2s}}},
			&-d-2s<\gamma<-d.
			\end{cases}
			\end{align}
			For $\gamma<-d$, take $\delta_\eta^v$ in
			the fractional integration-by-parts identity
			before taking absolute values. In
			$\delta_z^v(g\mu)=\mu\,\delta_z^vg+
			g(\cdot-z)\delta_z^v\mu$, replace $g$ by
			its terminal difference. The first summand is
			bounded by \eqref{eq:A26-mixed-frozen-kernel};
			the second has the one-increment kernel factor
			$\mathfrak r_{t,v}(\eta)$ and the product
			$|\delta_z h|\,|\delta_z\mu|$ supplies
			$\mathfrak r_{\tau,v}(z)^{\varepsilon_0}
			(1\wedge|z|)$. This is locally integrable
			because $\alpha<2s<1+\varepsilon_0$.
			Transfer the four shifted profiles as in
			\eqref{eq:A26-Y12-terminal-increment}.
			For $\gamma\geq-d$ in \eqref{eq:A26-hmu-rem-terminal}, take the
			terminal difference of the frozen convolution
			with the direct $\tau^{-1+\sigma_0}$
			source bound.
			\par\endgroup
			{Since $0<\varepsilon_0<s$,
			\[
			 t^{-1+\frac{\varepsilon_0}{2s}}\tau^{-\frac{\varepsilon_0}{2s}}
			 \lesssim\frac{1}{(t+\tau)^{\frac{\varepsilon_0}{2s}}
			 \min\{t,\tau\}^{1-\frac{\varepsilon_0}{2s}}}.
			\]
			Together with $T^{\frac{d-\kappa}{2s}}\leq T^{\sigma_0}$ when $\gamma<-d$,
			this proves \eqref{Sh12a} and \eqref{Sh12b}.
The estimates extend by regularization, using the integrable
majorants for $\mathcal S=-\mathcal Y_1+\mathcal Y_2$.}
		\end{proof}
		
		{Lemma \ref{lemnew} yields the following estimate for $g_{2,v_0}$ defined in \eqref{intg2}.}
		
		\begin{proposition}
			\label{propg2}
			{For every $T\in(0,1)$,}
			\begin{align*}
				\lceil g_{2,v_0}\rfloor_T\lesssim \|h_1\|_T\|h_2\|_T+T^{\sigma_0}\sum_{i=3,4}\|h_i\|_T,
			\end{align*}
            {where $\sigma_0$ is defined in \eqref{def:sigma0}.}
		\end{proposition}
		\begin{proof}
			{By definition,}
			\begin{align*}
				\lceil g_{2,v_0}\rfloor_T\leq&\sum_{m\leq N_0}\sup_{t\in[0,T]}\Big( \|g_{2,v_0}(t)\|_{\tilde Y_{m}}+\|g_{2,v_0}(t)\|_{\tilde Z_{m}}+\|g_{2,v_0}(t)\|_{\tilde Y_{m}^{\lfloor s\rfloor+\varepsilon_0}}+\|g_{2,v_0}(t)\|_{\tilde Z_{m}^{\lfloor s\rfloor+\varepsilon_0}}\Big).
			\end{align*}
{We estimate the pointwise and recentered
potential norms using Lemmas~\ref{lemnew} and
\ref{lem:recentered-convolution}, taking terminal differences at
fixed $v_0$. Set}
			\begin{align*}
				\mathbf{G}(t,\tau,x,v)&=(H_\mu^{v_0}(t)\ast G(\tau))(x-tv,v)\\
				&= ({\mathcal{S}}_{v_0} (h_1,h_2)+\mathcal{S}_{v_0} (h_3,\mu)+\mathcal{S}_{2,v_0} (\mu,h_4))(t,\tau,x,v).
			\end{align*}
			{For brevity, set}
			\begin{align*}
				\mathcal{B}_T(h)=\|h_1\|_T\|h_2\|_T+T^{\sigma_0}\sum_{i=3,4}\|h_i\|_T.
			\end{align*}
			\begingroup
			First suppose $0<s<1$.
			Lemma~\ref{lemnew} gives, for every $m\leq N_0$,
			\begin{align}\label{eq:A26-Boltz-Duhamel-integral}
			\|g_{2,v_0}(t)\|_{\tilde Y_m\cap\tilde Z_m}
			&\lesssim
			\mathcal B_T(h)
			\int_0^t\frac{\dif\tau}
			{t^{\frac{\varepsilon_0}{2s}}\min\{t-\tau,\tau\}^{1-\frac{\varepsilon_0}{2s}}}
			\notag\\*
			&\quad+
			\sum_{j=3,4}\|h_j\|_T
			\int_0^t\tau^{-1+\sigma_0}\,\dif\tau\\*
			&\lesssim \mathcal B_T(h).\notag
			\end{align}
			The first integral is bounded uniformly in $t$;
			the second equals $t^{\sigma_0}/\sigma_0$.
			For the highest-order norm, fix $0<[\eta]_v\leq1$.
			Use \eqref{eq:A26-Y12-terminal-increment}
			for the full quadratic source,
			the Gaussian cross-term bound following \eqref{Y3.1},
			\eqref{eq:A26-Y3-terminal-increment}
			for the separate Maxwellian terms, and
			\eqref{eq:A26-hmu-rem-terminal}
			for $\mathcal S_{2,v_0}(h_3,\mu)$.
			These bounds add $\mathfrak r_{t-\tau,v}(\eta)$ to each
			integrand in \eqref{eq:A26-Boltz-Duhamel-integral}.
			Split the first integral at $\tau=t/2$. On the first half,
			$t-\tau\asymp t$; on the second, apply
			\eqref{eq:A26-terminal-endpoint-integrals}. Thus
			\[
			\int_0^t
			\frac{\mathfrak r_{t-\tau,v}(\eta)}
			{t^{\frac{\varepsilon_0}{2s}}\min\{t-\tau,\tau\}^{1-\frac{\varepsilon_0}{2s}}}\,\dif\tau
			\lesssim1\wedge\ell_{t,v}^v(\eta)^{\varepsilon_0}.
			\]
			The second estimate in \eqref{eq:A26-terminal-endpoint-integrals}
			controls the $\tau^{-1+\sigma_0}$ term.
			Fractional normalization gives the required
			$\tilde Y_m^{\varepsilon_0}$ and
			$\tilde Z_m^{\varepsilon_0}$ bounds.
			\par\endgroup
			{Now assume $s=1$.}
			{Then, by definition,}
			\begin{equation}\label{b1v}
				\begin{aligned}
					\|g_{2,v_0}(t)\|_{\tilde Y_{m}}
					&\lesssim \int_0^{t}\|\mathbf{G}(t-\tau,\tau)\|_{\tilde Y_{m}}\dif \tau .
				\end{aligned}
			\end{equation}
			{Lemma \ref{lemnew} yields}
			\begin{align*}
				&\|  \mathbf{G}(t-\tau,\tau)\|_{\tilde Y_{m}}\lesssim \frac{\mathcal B_T(h)}
			{t^{\varepsilon_0/(2s)}
				\min\{t-\tau,\tau\}^{1-\varepsilon_0/(2s)}}.
			\end{align*}
			{Substituting this estimate into \eqref{b1v}, we obtain}
			\begin{align*}
				&	\sup_{t\in[0,T]}\|g_{2,v_0}(t)\|_{\tilde Y_{m}}
				\lesssim \mathcal{B}_T(h).
			\end{align*}
			\begingroup
For the Landau highest-order norm, take the terminal
increment before integrating in source time. Write $r=t-\tau$.
The extra output derivative has scale $\tilde t_v^{1/2}$; on
$\tau\geq t/2$, transferring it to a propagation-time kernel costs
$(t/r)^{1/2}$. The source velocity increment contributes
$(r/\tau)^{(1+\varepsilon_0)/2}$, while the divergence contributes
$r^{-1}$, so their product is
\[
 r^{-1}(t/r)^{1/2}(r/\tau)^{(1+\varepsilon_0)/2}
 \lesssim r^{-1+\varepsilon_0/2}\tau^{-\varepsilon_0/2}.
\]
On $\tau\leq t/2$, $r\asymp t$ and the direct source estimate gives
$r^{-1/2}\tau^{-1/2}$. Thus the apparent singular factor is used
together with the source increment before integration.
Estimate \eqref{eq:A26-Landau-terminal}, with propagation time $t-\tau$, applies to
$G_1,G_2$ and, by \eqref{eq:A26-Landau-G3-both-orders},
to the Maxwellian source $G_3$ with the additional factor
$\tau^{\sigma_0}$. Splitting at $\tau=t/2$ and using
\eqref{eq:A26-terminal-endpoint-integrals} gives
\[
\int_0^t \mathfrak r_{t-\tau,v}(z)
\left[(t-\tau)^{-1+\varepsilon_0/2}\tau^{-\varepsilon_0/2}\mathbf1_{\tau>t/2}
+(t-\tau)^{-1/2}\tau^{-1/2}\mathbf1_{\tau<t/2}\right]
\dif\tau\lesssim1\wedge\ell_{t,v}^v(z)^{\varepsilon_0}.
\]
Inserting $\tau^{\sigma_0}$ gives
$t^{\sigma_0}(1\wedge\ell_{t,v}^v(z)^{\varepsilon_0})$.
Fractional normalization gives the
$\tilde Y_m^{1+\varepsilon_0}$ and
$\tilde Z_m^{1+\varepsilon_0}$ bounds, completing the proof.
\par\endgroup
		\end{proof}\par\vspace{0.4cm}
		{\bf 3. Estimate $g_{3,v_0}$.}\par\medskip
		{We now examine the remainder generated by freezing the coefficients.}
		\begin{align}\label{forre}
			g_{3,v_0}(t,x,v)=\int_0^t  \mathbf{R}_{v_0}(t-\tau,\tau,x,v)\ \dif \tau ,
		\end{align}
		{where}
		\begin{align*}
			\mathbf{R}_{v_0}(t,\tau,x,v)=(H_\mu^{v_0}(t)\ast  \mathcal{R}^{v_0}(g)(\tau))(x-tv,v),
		\end{align*}
		{and $\mathcal{R}^{v_0}(g)$ is defined in \eqref{defRe0}.}
		
		{We first use the following lemma, whose proof is analogous to that of Lemma \ref{lemcoeff}.}
		\begin{lemma}\label{lemdif}
			{For any $z_1,z_2\in\mathbb{S}^{d-1}$ and $v,v'\in\mathbb{R}^d$ satisfying $|v-v'|\leq 1$,}
			\begin{align*}
				|\left((a\star \mu(v)-a\star \mu(v'))z_1,z_2\right)|\lesssim |v-v'|  \langle v\rangle^{\gamma}(\langle v\rangle^2-(z_1\cdot v)^2)^\frac{1}{2}(\langle v\rangle^2-(z_2\cdot v)^2)^\frac{1}{2}.
			\end{align*}
		\end{lemma}

		\begingroup
		We shall use the following subcritical moment bound.
		For every admissible weight $\mathbf q$ and $0<\beta\le1$ with
		$\beta<2s$, choosing $N$ sufficiently large depending on
		$\mathbf q$ and $\beta$ gives
		\begin{align}\label{eq:A27-subcritical-frozen-moment}
		&\iint\Phi_{t,v_0}(y,u)(1\wedge|u|)
		\mathbf q^{-1}(x-tv-y,v-u)\,\dif y\dif u
		\big|_{v_0=v}\notag\\
		&\qquad\lesssim_{\mathbf q,\beta,N}
		\tilde t_v^{\beta/(2s)}\mathbf q^{-1}(x,v),
		\qquad 0<t\le1.
		\end{align}
		Indeed, $1\wedge|u|\le [u]_{v_0}^{\beta}$.
		Choose $\theta<1$ close enough to one that
		$\beta+2(d+2s)(1-\theta)<2s$. The proof of
		\eqref{intNthe} then gives a finite $\beta$ moment of
		$\mathcal N_{v_0}^{\theta}$. The remaining rapid factor in
		\eqref{defPh} transfers $\mathbf q^{-1}$ as in
		\eqref{weighttrans}; the substitution
		$u=\tilde t_{v_0}^{1/(2s)}\mathcal O_{v_0}w$
		extracts $\tilde t_v^{\beta/(2s)}$.
		
		\par\endgroup
	
\begingroup
\begin{proposition}\label{leR}
For $0<T\leq1$, with the auxiliary norm defined in \eqref{defnormT1},
\begin{align}\label{eq:recentered-freezing-remainder}
\lceil g_{3,v_0}\rfloor_T\lesssim T^{\sigma_0}\|g\|_T.
\end{align}
The constant depends only on the fixed structural parameters.
\end{proposition}
\begin{proof}
It suffices to consider smooth $g$. Write $r=t-\tau$ and choose
\[
s<\beta<\min\{1,2s-\varepsilon_0\}\quad(s<1),
\qquad \beta=1\quad(s=1).
\]
For $\tau\leq t/2$, keep the output derivatives on the kernel.
For $\tau\geq t/2$, transfer the $m$ basic derivatives to the
source using $\nabla_u-r\nabla_y$ for velocity derivatives;
their scales are controlled by the source scales since $r\leq\tau$.
The additional Landau velocity derivative stays on the kernel.

The velocity adjoints of the remainder are
\begin{align*}
(\mathcal R^{v_0})^*\varphi(u)
={}&\operatorname{p.v.}\int
\Big[(\mathbf C_\mu(v_0,h)-\mathbf C_\mu(u,h))\varphi(u)\\
&\hspace{17mm}-(\mathbf C_\mu(v_0,h)-\mathbf C_\mu(u+h,h))
\varphi(u+h)\Big]\frac{\dif h}{|h|^{d+2s}},\qquad s<1,\\
(\mathcal R^{v_0})^*\varphi(u)
={}&\operatorname{div}_u\!\left(
\big((a\star\mu)(u)-(a\star\mu)(v_0)\big)\nabla_u\varphi(u)
\right),\qquad s=1.
\end{align*}
Their velocity integrals vanish at each fixed spatial argument.
On the late interval we therefore subtract the constant velocity
Taylor term of the differentiated source. For Landau, subtract also
the linear term and retain its contribution:
\[
\int u_j(\mathcal R^{v_0})^*\varphi(u)\,\dif u
=\int(\operatorname{div}(a\star\mu))_j(u)\varphi(u)\,\dif u.
\]
The required first source derivative is bounded by the zeroth and
$1+\varepsilon_0$ norms, by averaging over an intrinsic velocity
segment of length $\tilde\tau_v^{1/2}$.

We now estimate the adjoint paired with
$H_\mu^{v_0}(r,x-rv-y,v-u)$, setting $v_0=v$ after taking
derivatives and differences.
Lemma~\ref{lemCf} bounds the Boltzmann coefficient differences by
the capped displacements $v-u$ and $v-u-h$ to the power $\beta$,
with their coefficient envelopes. For Landau use
Lemma~\ref{lemdif} and Gaussian moments.

For $s<1$, symmetrization in $h$ gives the following three jump
factors, multiplying positive kernel and coefficient envelopes:
\[
(1\wedge|v-u|)^\beta\mathfrak r_{r,v}(h)^2,
\qquad
(1\wedge|h|)^\beta\mathfrak r_{r,v}(h)^\beta,
\qquad
(1\wedge|h|)^{2\beta}.
\]
They correspond to a coefficient displacement, a first coefficient
increment, and a second coefficient increment, respectively.
Keep the envelopes at their source velocities and use
$[v-u]_v\lesssim[v-u-h]_v+[h]_v$ for shifted kernels.
On the late interval, the source increment supplies a further moment
of order $\varepsilon_0$. The kernel moment is finite since
$\beta+\varepsilon_0<2s$: in \eqref{intNthe}, choose $\theta$ with
$\beta+\varepsilon_0+2(d+2s)(1-\theta)<2s$.

Split the jump integral at $[h]_v=\tilde r_v^{1/(2s)}$ and $|h|=1$.
For $\alpha=0,\varepsilon_0$, the first two factors give
$r^{-1}\tilde r_v^{(\beta+\alpha)/(2s)}$ by the radial and angular
calculation in \eqref{eq:radial-Carleman-bound}: the inner powers
$2+\beta+\alpha$ and $2\beta+\alpha$ exceed $2s$, while
$\beta+\alpha<2s$. The second coefficient increment has finite
jump integral since $\beta>s$; large jumps are controlled after
source integration by Lemma~\ref{lem:global-Omega-integration}.
These terms are lower order.

\needspace{9\baselineskip}
Apply \eqref{G00} on the early interval, without subtracting a
source Taylor term, and \eqref{G11} on the late interval.
Lemma~\ref{lem:recentered-convolution} gives the potential bounds
from the same positive majorants. Let $k=\lfloor s\rfloor$ and
$0<[z]_v\leq1$. Include the output scales at $t$, but omit the
fractional normalization
$\tilde t_v^{\varepsilon_0/(2s)}[z]_v^{-\varepsilon_0}$.
The coefficient-displacement terms are bounded, up to a structural
constant, by
\begin{align}
r^{-1+\beta/(2s)}
\mathfrak r_{r,v}(z)\|g\|_T,
&\qquad \tau\leq t/2,\label{eq:recentered-remainder-early}\\
\left(\frac t r\right)^{k/(2s)}
r^{-1+(\beta+k+\varepsilon_0)/(2s)}
\tau^{-(k+\varepsilon_0)/(2s)}
\mathfrak r_{r,v}(z)\|g\|_T,
&\qquad \tau\geq t/2.\label{eq:recentered-remainder-late}
\end{align}
Here $\mathfrak r_{r,v}(z)=1\wedge([z]_v/\tilde r_v^{1/(2s)})$
comes from two consecutive kernel derivative bounds.
For Landau, the Taylor remainder has order $1+\varepsilon_0$;
the extra kernel derivative gives $(t/r)^{1/2}$ on the late interval.
The zeroth-order estimates use $k=0$ and omit the terminal cutoff.

On the late interval, a basic derivative on a coefficient is bounded
by \eqref{eq:recentered-remainder-late} with $\beta=0$, multiplied
by $\tau^{1/(2s)}$. Apply the coefficient bounds to derivatives of
$\mu$ and transfer the weights by \eqref{weighttrans}.
In the terminal difference, both endpoints keep $v_0=v$.
At $v-z$, the coefficient bounds applied to derivatives of $\mu$
give the additional displacement term
$|z|$ times \eqref{eq:recentered-remainder-early} or
\eqref{eq:recentered-remainder-late} with $\beta=0$ and without
the terminal cutoff. The Landau first-moment term has the late bound
\[
C t^{1/2}r^{-1/2}\tau^{-1/2}
\mathfrak r_{r,v}(z)\|g\|_T.
\]

To integrate the late bounds, use
\begin{align*}
\int_0^{t/2}r^{-1+\varepsilon_0/(2s)}
\mathfrak r_{r,v}(z)\,\dif r
&\lesssim\langle v\rangle^{\kappa\varepsilon_0/(2s)}[z]_v^{\varepsilon_0},\\
\int_0^{t/2}r^{-1+(\beta+\varepsilon_0)/(2s)}
\mathfrak r_{r,v}(z)\,\dif r
&\lesssim t^{\beta/(2s)}
\min\!\left\{t^{\varepsilon_0/(2s)},
\langle v\rangle^{\kappa\varepsilon_0/(2s)}[z]_v^{\varepsilon_0}\right\}.
\end{align*}
The first integral is estimated by splitting at $r=\langle v\rangle^\kappa[z]_v^{2s}$;
for the second integral use
$\mathfrak r_{r,v}(z)\leq\min\{1,([z]_v/\tilde r_v^{1/(2s)})^{\varepsilon_0}\}$.
Restoring the fractional normalization gives gains
$t^{\beta/(2s)}$ and $t^{1/(2s)}$ for the displacement and
coefficient-derivative terms. The extra endpoint term gives
$t^{\varepsilon_0/(2s)}$, since
\[
\tilde t_v^{\varepsilon_0/(2s)}[z]_v^{-\varepsilon_0}|z|
\leq t^{\varepsilon_0/(2s)}.
\]
For clarity, put $a=\varepsilon_0/(2s)$,
$q=\langle v\rangle^{\kappa/(2s)}[z]_v$ and $R=q^{2s}$.
If $R<t/2$, splitting the first integral above at $R$ gives
\[
 \int_0^R r^{-1+a}\,\dif r
 +q\int_R^{t/2}r^{-1+a-1/(2s)}\,\dif r
 \lesssim q^{\varepsilon_0};
\]
if $R\geq t/2$, the bound follows from
$t^a\lesssim q^{\varepsilon_0}$. The second integral is bounded both
by $Ct^{(\beta+\varepsilon_0)/(2s)}$ and, using
$1\wedge X\leq X^{\varepsilon_0}$, by
$Ct^{\beta/(2s)}q^{\varepsilon_0}$. This proves its stated minimum.
For the Landau first moment, $\tau\asymp t$ on the late interval and
\[
 \int_0^{t/2}r^{-1/2}\mathfrak r_{r,v}(z)\,\dif r
 \lesssim t^{1/2}\min\{1,(q/t^{1/2})^{\varepsilon_0}\},
\]
since $\int_0^{t/2}r^{-(1+\varepsilon_0)/2}\,\dif r<\infty$.
After normalization its gain is $t^{1/2}$.
On the early interval, $r\asymp t$; thus
\eqref{eq:recentered-remainder-early} integrates to the same gain
$t^{\beta/(2s)}$, also after fractional normalization.
For the zeroth-order norms, omit the fractional normalization and
the terminal cutoff. The gain exponents are
$\beta/(2s)$, $1/(2s)$, $\varepsilon_0/(2s)$ and, for Landau, $1/2$.
For $s<1$, $\beta/(2s)>1/2>\varepsilon_0/(2s)$; for $s=1$,
$\beta/(2s)=1/2>\varepsilon_0/2$. Also $1/(2s)\geq1/2$.
Hence every gain is at least $\sigma_0$ by \eqref{def:sigma0}.
Summing over $m\leq N_0$ proves
\eqref{eq:recentered-freezing-remainder}.

Finally, spatial and velocity regularization preserves these bounds.
Taking the cancellations before passing to the limit, dominated
convergence locally and Fatou's lemma for the potential terms
extend the estimate to $\|g\|_T<\infty$.
\end{proof}
\par\endgroup

		\begin{proof}[{Proof of Lemma \ref{lemloc}}]
{Propositions~\ref{Pro1}, \ref{propg2},
and \ref{leR}, together with \eqref{esSf1}, give}
			\begin{align*}
				\|g\|_T\lesssim \|f_0\|_{\mathrm{cr}}+\mathcal{B}_T(h)+T^{\sigma_0}\|g\|_{T},
			\end{align*}
			{where the implicit constant depends on $N_0$, $\gamma$, and $s${,
as well as the fixed structural parameters, including $c_b$ and $\delta_0$
in the Boltzmann case}, and $\mathcal{B}_T(h)=\|h_1\|_{T}\|h_2\|_{T}+T^{\sigma_0}\sum_{i=3,4}\|h_i\|_{T}$. Choosing $T_1>0$ sufficiently small, we obtain}
			\begin{align*}
				\|g\|_{T_1}\leq C\|f_0\|_{\mathrm{cr}}+C\mathcal{B}_{T_1}(h).
			\end{align*}
{On the interval where $g$ is defined, apply
the same local estimate after shifting the time origin by $T_1/2$.
For $g_1(t,x,v)=g(t+T_1/2,x,v)$,}
					\begin{align*}
				\|g_1\|_{T_1}&\leq C\|g(T_1/2)\|_{\mathrm{cr}}+C\mathcal{B}_{3T_1/2}(h)\leq C\|g\|_{T_1}+C\mathcal{B}_{3T_1/2}(h).
			\end{align*}
			\begingroup
			The scale comparison in this restart is
			\[
			\Xi:=\frac{d-\kappa}{2s}
			+N_0\left(1+\frac1{2s}\right)
			+\frac{\lfloor s\rfloor+\varepsilon_0}{2s}.
			\]
			For absolute time $t$ and restart time $r$ with
			$1\le t/r\le A$, the time weights, scaled derivatives
			and fractional normalization in \eqref{defyz} cost at
			most $A^\Xi$; the $Z$ terms omit the zeroth-order
			$Y$ time weight.
		For $g_1(t)=g(t+T_1/2)$, on the newly covered interval
		$T_1\le t\le3T_1/2$ one has
		$1\le t/(t-T_1/2)\le2$. Therefore
		\begin{align}\label{eq:A27-first-restart}
		\|g\|_{3T_1/2}
		&\le \|g\|_{T_1}+2^\Xi\|g_1\|_{T_1}\notag\\
		&\le C_\Xi\bigl(\|f_0\|_{\mathrm{cr}}
		+\mathcal B_{3T_1/2}(h)\bigr).
		\end{align}
		At the $j$th restart the ratio of absolute to elapsed time
		is at most $j+1$ on the newly covered interval. The same
		comparison and the local bound yield, with a structural
		constant $C_\Xi$,
		\begin{align}\label{eq:A27-restart-factorial}
		\|g\|_{(n+1)T_1/2}
		\le C_\Xi^n(n!)^\Xi
		\bigl(\|f_0\|_{\mathrm{cr}}
		+\mathcal B_{(n+1)T_1/2}(h)\bigr),
		\qquad n\ge1.
		\end{align}
		\par\endgroup
			
For $0<T\leq1$, the number of restarts satisfies
$n\leq\lceil2/T_1\rceil$. Since $T_1$ and $\Xi$ depend only on the
structural parameters, $C_\Xi^n(n!)^\Xi$ is bounded by a structural
constant. This proves \eqref{largt} with a single constant $C$.
For arbitrary finite $T$, the same argument gives the extension with
$C_T\leq C\exp(CT\log(T+2))$ in place of $C$.
		\end{proof}
		\subsection{Contraction and preservation of nonnegativity}\label{secproof}
\begingroup
For smooth data and source functions, a weighted energy approximation
solves \eqref{linlo}, keeping $\mathcal L_\mu$ unchanged.
The smooth Maxwellian coefficients, the even jump kernel and its finite
second moment give weighted Sobolev bounds at every finite order;
for $s=1$, use the corresponding differential energy estimates.
Uniformly local estimates in $x$ suffice when the data have no spatial
decay. These smooth solutions have finite $\|\cdot\|_T$-norm.

For the fixed-point source \eqref{GQ} with $h_i=h$, truncate $h$ near
$t=0$ and regularize and localize the data and $h$.
The approximations are uniformly bounded in their respective norms,
so Lemma~\ref{lemloc} gives uniform solution bounds.
The symmetrized weak form for $\mathcal Q_s(h,h)$ and the smooth
Maxwellian factors in the other terms give a uniform bound on each
compactly supported test pairing. Hence the weak equation preserves
the initial trace. Local compactness at positive times and passage to
the collision terms proceed as in Lemma~\ref{lem:weak-approximation};
lower semicontinuity gives \eqref{largt} for the limit.
The same estimate with zero data and source gives linear uniqueness.
\par\endgroup

		{{We now apply Lemma~\ref{lemloc}
in a fixed-point argument.}
For $\delta,T>0$, define}
		\begin{align*}
			\mathcal{V}_{T,\delta}:=\{f:\|f\|_T\leq2\delta,\ f(t)\to f_0\text{ in }\mathcal D'(\mathbb R^{2d})\text{ as }t\downarrow0\}.
		\end{align*}
\begingroup
By the preceding linear construction and Lemma~\ref{lemloc}, the
solution $g$ of \eqref{linlo} with $G=0$ and initial datum $f_0$
satisfies $\|g\|_T\le C\|f_0\|_{\rm cr}$ for $0<T\le1$.
Thus $\mathcal V_{T,\delta}$ is nonempty whenever
$2\delta\ge C\|f_0\|_{\rm cr}$, as ensured by the later choice of $C_1$.
\par\endgroup
{We next verify completeness of the solution space
and closure of the initial-trace condition.}
\begin{lemma}\label{lem:trace-closure}
Interpret the derivatives in \eqref{defnormT} distributionally and the
zeroth-order positive-time functions by their continuous phase-space representatives
(with essential suprema for the highest derivatives).
The space $\mathcal B_T=\{f:(0,T]\to\mathcal D'(\mathbb R^{2d}):\|f\|_T<\infty\}$
{of weakly measurable functions},
identifying functions equal almost everywhere in phase space at each time, is Banach. Its subset
$\mathcal V_{T,\delta}$ with a fixed distributional initial trace is closed.
\end{lemma}
\begin{proof}
Taking the anisotropy parameter equal to zero in \eqref{defI12}, and covering
a bounded velocity set by finitely many unit balls, gives, for every compact
$K\subset\mathbb R^{2d}$,
\[
\sup_{0<t\leq T}\|h(t)\|_{L^1(K)}\leq C_K\|h\|_T.
\]
The spatial essential supremum dominates $|h|$ almost everywhere, and
$|w|^{-\kappa}\geq1$ on the unit ball, which proves this inequality.
On $[\tau,T]\times K$, $\tau>0$, all the scale and weight factors in
\eqref{defyz} are bounded above and below, and the frozen derivative frames
have bounded inverses. A norm-Cauchy sequence therefore
converges locally in $L^\infty$ in phase space, uniformly for $t\in[{\tau},T]$,
as do its specified derivatives. Distributional testing identifies the
limiting derivatives; the zeroth-order limit has a continuous representative,
and the fractional differences pass to the limit almost everywhere. Fatou's lemma for the
potential terms and lower semicontinuity of the suprema show
$\|f_n-f\|_T\leq\liminf_m\|f_n-f_m\|_T$, proving completeness.
If each $f_n$ has trace $f_0$, then for $\varphi\in C_c^\infty$,
\[
\limsup_{t\downarrow0}|\langle f(t)-f_0,\varphi\rangle|
\leq C_\varphi\|f-f_n\|_T\longrightarrow0.
\]
The norm ball is closed by the same lower semicontinuity.
\end{proof}

{The following lemma gives weak-trace and compactness
estimates for bounded classical approximations.}
\begin{lemma}\label{lem:weak-approximation}
Let $0<s\leq1$ and let $F_j=\mu+f_j$ be
{classical positive-time solutions} of the full collision
equation on $(0,T]$, with distributional traces $F_{0,j}$.
Assume $\sup_j\|f_j\|_T\leq A$ and
$F_{0,j}\to F_0$ in $\mathcal D'$. Then a subsequence converges, for every
$\tau>0$ and integer $k<N_0/2$, in
$C([\tau,T];C^k_{\rm loc}(\mathbb R^{2d}))$ to a solution $F=\mu+f$.
It has trace $F_0$ and $\|f\|_T\leq A$. If $F_j\geq0$, then $F\geq0$.
Moreover, for every $\varphi\in C_c^\infty(\mathbb R^{2d})$,
\begin{equation}\label{eq:uniform-weak-trace}
 |\langle F_j(t)-F_{0,j},\varphi\rangle|
 \leq C_\varphi(A+1)^2t.
\end{equation}
\end{lemma}
\begin{proof}
Put $\kappa=-\gamma-2s\in[0,d)$. The $Z_0$ part of the norm, evaluated
with internal anisotropy parameter zero, implies uniformly in $j,t,x$
\begin{equation}\label{eq:uniform-moment-potential}
 \int\langle v\rangle^3|F_j(x,v)|\,\dif v
 +\sup_v\int_{|u|\leq1}|u|^{-\kappa}|F_j(x,v-u)|\,\dif u
 \leq C(A+1).
\end{equation}
For the first term, cover velocity space by unit balls and sum the bound
$C A\langle v\rangle^{-\varkappa_2}$; the stated weight satisfies
$\varkappa_2>d+3$. The Maxwellian contributes a fixed constant.

For $s<1$, use the symmetrized weak collision form, retaining azimuthal cancellation
before taking absolute values. Taylor's formula gives
\[
 |\langle {\mathcal Q_s}(F_j,F_j),\varphi\rangle|
 \leq C_\varphi\int_{x\in K_\varphi}\iint
 |F_j(x,v)F_j(x,v_*)|\,|v-v_*|^{\gamma+2}\,\dif v\dif v_*\dif x.
\]
The angular integral is finite because
$\int_0^{\pi/2}\theta^{1-2s}\,\dif\theta<\infty$.
For $|u|\leq1$,
$|u|^{\gamma+2}=|u|^{-\kappa+2-2s}\leq|u|^{-\kappa}$;
for $|u|>1$, it is bounded by $C\langle v\rangle^2\langle v_*\rangle^2$.
For $s=1$, use the symmetrized Landau weak form, with symmetric relative-velocity
truncation; $a(u)u=0$ removes derivative-of-cutoff terms. Its terms involve $a(u):D^2\varphi$ and
$(\operatorname{div}a)(u)\cdot(\nabla\varphi(v)-\nabla\varphi(v_*))$;
both are bounded by $C_\varphi|u|^{\gamma+2}$.
Equation \eqref{eq:uniform-moment-potential} thus bounds the collision pairing
by $C_\varphi(A+1)^2$. The transport pairing is bounded by the first moment.
Integrating the weak equation proves \eqref{eq:uniform-weak-trace}, and also
a uniform Lipschitz bound in time for each compactly supported smooth test.

On $[\tau,T]$, the $Y_m$ bounds give uniform local bounds for phase-space
derivatives through order $N_0$. Compact spatial embedding, together with
the preceding weak time equicontinuity, gives compactness in
$C([\tau,T];C^k(K))$, $k<N_0/2$, on each compact $K$.
For example, mollify on a slightly larger compact set: weak time
 equicontinuity controls the mollified functions, while the uniform spatial
 derivative bounds make the mollification error small in $C^k$.
A diagonal argument gives the asserted subsequence.

To pass the collision term to the limit in the weak equation, first restrict
velocities to a bounded set and exclude $|v-v_*|<r$ and $\theta<\epsilon$.
Local uniform convergence then suffices. The excluded relative-velocity part
is bounded uniformly by $C r^{2-2s}(A+1)^2$ when $s<1$, and the excluded
angular part by $C\epsilon^{2-2s}(A+1)^2$. For $s=1$, on a bounded velocity
set and $t\geq\tau>0$, the uniform local supremum bound instead gives
$C_{\tau,R}r^{d+\gamma+2}$; its exponent is positive. There is no angular
excision in this case.
Velocity tails tend to zero by the third moment in
\eqref{eq:uniform-moment-potential}; in the region $|v-v_*|<1$, combine
the potential bound with the tail mass of one factor.
These estimates justify passage on $[\tau,T]$. The uniform weak collision
bound makes the omitted interval $(0,\tau)$ contribute $O(\tau)$, so
$\tau\downarrow0$ gives the full weak equation. Equation \eqref{eq:uniform-weak-trace}
gives the initial trace. Weak-* lower semicontinuity handles the highest
derivatives; for the potential terms, first truncate the positive kernel and
use the dual characterization of $L^1_vL^\infty_x$, then remove the
truncations. Together with Fatou's lemma for lower derivatives, this gives
$\|f\|_T\leq A$.
Nonnegativity passes by local uniform convergence at positive times.
\end{proof}

		{For $h\in\mathcal{V}_{T,\delta}$, define $\mathfrak S h$ as the solution $g$ of \eqref{linlo} with $h_i=h$ for $1\leq i\leq4$, where $T$ and $\delta$ will be chosen below. Every fixed point of $\mathfrak S$ solves \eqref{eqpert2}.}
		\begin{theorem}\label{thmmap} 
			{For any $T,\delta\in(0,1)$ and $f,f_1,f_2\in\mathcal{V}_{T,\delta}$,}
			\begin{align}
				&\|\mathfrak{S}f\|_{T}\lesssim  \|f_0\|_{\mathrm{cr}}+\|f\|_{T}^2+T^{\sigma_0}\|f\|_{T},\label{esmap1}\\& \|\mathfrak{S}f_1-\mathfrak{S}f_2\|_{T}\lesssim  \|f_1-f_2\|_{T}	(\|f_1\|_{T}+\|f_2\|_{T}+T^{\sigma_0}).\label{esmap2}
			\end{align}
		\end{theorem}
		\begin{proof}
			{Estimate \eqref{esmap1} follows from Lemma~\ref{lemloc} with \(h_i=f\) for \(1\le i\le4\). To prove \eqref{esmap2}, use bilinearity to write the equation for \(\mathfrak S f_1-\mathfrak S f_2\) as the sum of the source quadruples \((f_1-f_2,f_1,f_1-f_2,f_1-f_2)\) and \((f_2,f_1-f_2,0,0)\), and apply Lemma~\ref{lemloc} to each.}
		\end{proof}\par\medskip
		{Fix $C_1\geq1$ large enough for the mapping
estimates and the linear iteration below}, and set
$\delta=2C_1\|f_0\|_{\mathrm{cr}}$.
Choose the data threshold and $T_0\in(0,1)$ so that
$4C_1\delta\leq1/4$ and $C_1T_0^{\sigma_0}\leq1/8$. Then
\[
\|\mathfrak S f\|_{T_0}\leq\delta,
\qquad
\|\mathfrak S f_1-\mathfrak S f_2\|_{T_0}
\leq\tfrac12\|f_1-f_2\|_{T_0}.
\]
{The contraction gives a unique fixed point
in $\mathcal V_{T_0,\delta}$, with}
$\|f\|_{T_0}\lesssim\|f_0\|_{\mathrm{cr}}$.

The fixed-point map does not itself preserve $\mu+f\geq0$.
We therefore use a separate approximation of the full distribution.
Let
\[
    \mu_\varepsilon^v(v)
    :=
    \varepsilon^{-d/2}\mu(v/\sqrt{\varepsilon}),
    \qquad
    \varrho_\varepsilon^x(x)
    :=
    \varepsilon^{-d/2}\mu(x/\sqrt{\varepsilon}),
    \qquad 0<\varepsilon<1.
\]
{By the semigroup property of the Gaussian heat kernel,}
\[
    \mu_\varepsilon^v\star\mu=\mu_{1+\varepsilon},
\]
{where \(\star\) denotes convolution in \(v\). Since \(f_0+\mu\geq0\),}
\[
    (f_0+\mu)*(\varrho_\varepsilon^x\mu_\varepsilon^v)
    =
    f_0*(\varrho_\varepsilon^x\mu_\varepsilon^v)
    +
    \mu_{1+\varepsilon}^v
    \geq0 ,
\]
{where \(*\) denotes convolution in \((x,v)\). Let \(\chi\in C_c^\infty(\mathbb R^d)\) satisfy \(0\leq\chi\leq1\) and \(\chi=1\) on \(B_1\), and define}
\[
    f_{0,\varepsilon}(x,v)
    :=
    \chi(\varepsilon v)\,
    \mu(v)\,\mu_{1+\varepsilon}(v)^{-1}\,
    \bigl(f_0*(\varrho_\varepsilon^x\mu_\varepsilon^v)\bigr)(x,v).
\]
{Then}
\[
\begin{aligned}
    f_{0,\varepsilon}(x,v)+\mu(v)
    &=
    \chi(\varepsilon v)\,
    \mu(v)\,\mu_{1+\varepsilon}(v)^{-1}
    \bigl[
        \bigl(f_0*(\varrho_\varepsilon^x\mu_\varepsilon^v)\bigr)(x,v)
        +
        \mu_{1+\varepsilon}(v)
    \bigr]  \\
    &\qquad
    +(1-\chi(\varepsilon v))\mu(v)
    \geq0 .
\end{aligned}
\]
The approximation is uniformly bounded in the data norm. Indeed, write
$\Phi_\varepsilon(y,z)=\varrho_\varepsilon^x(y)\mu_\varepsilon^v(z)$.
Sublinearity of $\mathbf I$ and the polynomial weight inequality give
\[
\begin{aligned}
\mathbf I(f_0*\Phi_\varepsilon)(x,v)
&\leq\iint\Phi_\varepsilon(y,z)\mathbf I f_0(x-y,v-z)\,\dif y\,\dif z,\\
\frac{\langle x,v\rangle^{\varkappa_1}\langle v\rangle^{\varkappa_2}}
 {\langle x-y,v-z\rangle^{\varkappa_1}\langle v-z\rangle^{\varkappa_2}}
&\lesssim\langle y,z\rangle^{\varkappa_1}\langle z\rangle^{\varkappa_2}.
\end{aligned}
\]
The Gaussian moment on the right is bounded uniformly for $0<\varepsilon<1$.
Since
\[
0\leq\chi(\varepsilon v)\frac{\mu(v)}{\mu_{1+\varepsilon}(v)}
\leq(1+\varepsilon)^{d/2}\leq2^{d/2},
\]
we obtain
\[
\|f_{0,\varepsilon}\|_{\mathrm{cr}}
\leq C_{\rm app}\|f_0\|_{\mathrm{cr}},
\]
where $C_{\rm app}$ depends on $d,\varkappa_1,\varkappa_2$ and is independent
of $\varepsilon$. Choose the initial smallness threshold so that this bound
lies below the contraction threshold. Each $f_{0,\varepsilon}$ is smooth and
compactly supported in $v$.
\begingroup
To construct nonnegative solutions for these data, start with the
solution $F_\varepsilon^0$ of
$\partial_t F_\varepsilon^0+v\cdot\nabla_xF_\varepsilon^0
=\mathcal Q_s(\mu,F_\varepsilon^0)$ with trace
$\mu+f_{0,\varepsilon}$, and solve successively
\[
 \partial_tF_\varepsilon^{n+1}+v\cdot\nabla_xF_\varepsilon^{n+1}
 =\mathcal Q_s(F_\varepsilon^n,F_\varepsilon^{n+1}),
 \qquad F_\varepsilon^{n+1}(0)=\mu+f_{0,\varepsilon}.
\]
At each step the small norm bound below supplies a positive lower
mass bound, so the smooth linear construction applies on $[0,T_0]$;
its higher-order energy constants may depend on $\varepsilon$ and $n$.
Nonnegativity follows from the maximum principle: for $s<1$,
$\mathbf C_{F_\varepsilon^n}\geq0$ and the smooth cancellation
coefficient is bounded; for $s=1$, use the Landau parabolic operator.
Writing $f_\varepsilon^n=F_\varepsilon^n-\mu$, Lemma~\ref{lemloc} gives
\[
 \|f_\varepsilon^{n+1}\|_{T_0}
 \leq C\|f_{0,\varepsilon}\|_{\mathrm{cr}}
 +C\bigl(\|f_\varepsilon^n\|_{T_0}+T_0^{\sigma_0}\bigr)
       \|f_\varepsilon^{n+1}\|_{T_0}
 +CT_0^{\sigma_0}\|f_\varepsilon^n\|_{T_0}.
\]
The estimate for $f_\varepsilon^0$ has only the data and linear terms.
For the same smallness choices, induction and absorption give
$\sup_n\|f_\varepsilon^n\|_{T_0}
\leq C\|f_{0,\varepsilon}\|_{\mathrm{cr}}$.
In particular, the mass lower bound persists at every step.
Applying the estimate to successive differences, which have zero trace,
and decreasing $T_0$ and the data threshold if necessary, yields
\[
 \|f_\varepsilon^{n+1}-f_\varepsilon^n\|_{T_0}
 \leq\tfrac12\|f_\varepsilon^n-f_\varepsilon^{n-1}\|_{T_0},
 \qquad n\geq1.
\]
Lemma~\ref{lem:trace-closure} therefore gives a solution
$f_\varepsilon$ on $[0,T_0]$ with $F_\varepsilon:=\mu+f_\varepsilon\geq0$.
Convergence at positive times passes the collision equation to the
limit, and uniqueness identifies this solution with the fixed point
for $f_{0,\varepsilon}$.
\par\endgroup

{The preceding estimates give the uniform bound}
\[
    \|f_\varepsilon\|_{T_0}
    \leq
    C\|f_{0,\varepsilon}\|_{\mathrm{cr}}
    \leq
    C\|f_0\|_{\mathrm{cr}} .
\]
\begingroup
Lemma~\ref{lem:weak-approximation} now gives a subsequence converging
locally at positive times to a solution with trace $f_0$ and the same
norm bound. The contraction estimate identifies the limit with $f$,
and nonnegativity passes to the limit:
\[
    f(t,x,v)+\mu(v)\geq0,
    \qquad
    0\leq t\leq T_0 .
\]
\par\endgroup
This completes the proof of Theorem~\ref{Th1}.
		\subsection{Global continuation and decay}\label{secglo}
We now discuss the global continuation and long-time estimates in
Corollary~\ref{thmglo}, starting from the regularized positive-time
data supplied by Theorem~\ref{Th1}.
To avoid repeating the background theory, we give a streamlined account
based on the polynomial/Maxwellian decomposition of \cite{Caf,CG24}
and the frequency-dependent hypocoercive framework of \cite{BCD},
together with the strong-singularity inputs of \cite{YZ}.
We retain the principal collision and energy estimates, the parameter
choices, and the positive-time restart, and refer to those works for
the development of the background framework.

The local estimates supply smooth, polynomially decaying data at positive
times. Global continuation then reduces to a weighted estimate for those data.
We use the decomposition of \cite{Caf,CG24}, in which one component carries
the polynomial tail and a second component is estimated in a
Maxwellian-weighted space. {The Maxwellian-weighted component uses the frequency-dependent energies of}
\cite{BCD}.

The Landau case $s=1$ is treated within the same two-component
continuation scheme. The nonlocal collision forms are replaced by the
corresponding differential forms, and the collision and commutator
estimates are obtained by integration by parts. The local structure
simplifies several of these estimates; see \cite[Remark~1.3]{BCD} for
the three-dimensional setting. To avoid repetition, we present the
 argument for $0<s<1$ and record only the modifications needed
for the Landau case.

{For $s=1$, the polynomial form corresponding to
\eqref{eq:raw-collision-form} is}
\[
 \mathfrak a_{\mu}[u]
 =\int(a\star\mu)(v)\nabla u\cdot\overline{\nabla u}\,\dif v
   +\|u\|_{L^2_{\gamma/2}}^2.
\]
{The corresponding parameters are}
\[
q_{\gamma,1}=\frac{\gamma}{2},\qquad
\lambda^L_\xi=\delta_\lambda|\xi|^2\langle\delta_{\mathrm{hyp}}^{-1}\xi\rangle^{-4/3}.
\]
\par

\begingroup
Restart the local solution at $t_*:=T_0/2$ by setting
\[
g_0(x,v):=f(t_*,x,v),\qquad
g(t,x,v):=f(t_*+t,x,v),\quad t\geq0.
\]
Here $t$ is elapsed time from $t_*$.
\par\endgroup
{Decompose}
\[
    g=g_1+\sqrt{\mu}\,g_2,
\]
{where \(g_1\) carries the polynomial velocity tail of the initial data, while \(g_2\) is the Maxwellian-weighted component with zero initial data. We consider the system}
\begin{equation}\label{eqdeco}
    \begin{aligned}
        &\partial_t g_1+v\cdot \nabla_x g_1+\mathbf{L}_s g_1
        +\mathbf{c}_1\chi_R(v)g_1
        =
        F(g_1,g_2),\\
        &\partial_t g_2+v\cdot \nabla_x g_2+\tilde{\mathcal{L}}_s g_2
        =
        G(g_1,g_2),\\
        &(g_1,g_2)|_{t=0}=(g_0,0).
    \end{aligned}
\end{equation}
{Here \(\mathbf{c}_1,R>0\) will be fixed later, and \(\chi_R\) is a smooth cutoff satisfying}
\[
    \mathbf{1}_{|v|\leq R}\leq \chi_R\leq \mathbf{1}_{|v|\leq 2R}.
\]
{The operator \(\mathbf{L}_s\) is defined in \eqref{defL}, and}
\[
    \tilde{\mathcal{L}}_s
    :=
    \mu^{-\frac12}\mathbf{L}_s\mu^{\frac12}.
\]
{The source terms are}
\begin{equation}\label{defFG}
    \begin{aligned}
        F(g_1,g_2)
        &=
        \mathcal{Q}_s(g,g)
        -
        \mathcal{Q}_s(\mu^{\frac12}g_2,\mu^{\frac12}g_2),\\
        G(g_1,g_2)
        &=
        \mu^{-\frac12}\mathbf{c}_1\chi_R(v)g_1
        +
        \mu^{-\frac12}
        \mathcal{Q}_s(\mu^{\frac12}g_2,\mu^{\frac12}g_2).
    \end{aligned}
\end{equation}
{If \((g_1,g_2)\) solves \eqref{eqdeco}, then}
\[
    g=g_1+\mu^{\frac12}g_2
\]
{solves the perturbation equation \eqref{eqperbo}. Thus it suffices to solve the decomposed system \eqref{eqdeco}.}

\begingroup
For $-d-2s<\gamma\leq-2s$, we combine the frequency-dependent scheme
of \cite{BCD} with the strong-singularity inputs of \cite{YZ}
for the Maxwellian component.
\par\endgroup

\Needspace{18\baselineskip}
{We record the global estimate when $0<s<1$.}
\begin{proposition}\label{propglo}
			{Let $\sigma,N\in\mathbb N$ satisfy $\sigma,N>d$, and let $M>d$ and $M^{\prime}>2s^{-1}(N+\sigma)$.}
{Fix energy indices satisfying
\eqref{eq:global-structural-budget} and
\eqref{eq:global-compact-source-budget}, with $N|\gamma-1|<M$.}
{There exists $\varepsilon>0$ such that, if the initial datum $f_0$ satisfies}
			\begin{align}\label{condata}
				\mathbf{Data}:=	\left\|\langle v\rangle^{M+M^{\prime}}\langle x\rangle^N \langle \nabla_{x,v}\rangle^{\sigma} f_0\right\|_{L_{x, v}^2}+\left\|\langle v\rangle^{M+M^{\prime}}\langle x\rangle^N \langle\nabla_{x,v}\rangle^\sigma f_0 \right\|_{L_v^2 L_x^1} \leq \varepsilon,
			\end{align}
			{then the corresponding solution $f$ of \eqref{eqperbo} has the following properties. For $\delta_{\mathrm{hyp}},\delta_{\lambda}>0$, chosen in the proof and depending on $\sigma$, $M$, and $M^{\prime}${,
as well as the fixed collision-kernel parameters, including $c_b$ and
$\delta_0$}, define}
			
			$$
			\lambda_{\xi}=\lambda(\xi)=\delta_{\lambda} |\xi|^2\langle \delta_{\mathrm{hyp}}^{-1} \xi\rangle^{-\frac{2(1+s)}{1+2s}}.
			$$
			\begingroup
{For these energy parameters, distinguish the output indices}
$J_{\rm out},\sigma_{\rm out}\in\mathbb N\cup\{0\}$ from $J,J'$ and
the initial regularity $\sigma$. With
$a_s=2s/(1+2s)$ and $L_0=\lceil(d+1)/a_s\rceil$, require
\[
 J_{\rm out}\leq J'-L_0,\qquad
 \sigma_{\rm out}\leq\min\{B/2,B'\},\qquad
 \sigma_{\rm out}+\lceil a_sJ_{\rm out}\rceil
       +\lfloor d/2\rfloor+1\leq\sigma.
\]
The required output estimate, including the differentiated and spatial
norm comparisons specified below, is
\begin{align*}
 \sup_{t>0}\Bigl(&\bigl\|\langle v\rangle^M
 \langle x/\langle t\rangle\rangle^{N/2}
 \langle\lambda(\nabla_x)t\rangle^{J_{\rm out}}
 \langle\nabla_{x,v}\rangle^{\sigma_{\rm out}}f(t)\bigr\|_{L^2_{x,v}}\\
 &+\langle t\rangle^{d/2}\bigl\|\langle v\rangle^M
 \langle\lambda(\nabla_x)t\rangle^{J_{\rm out}}
 \langle\nabla_{x,v}\rangle^{\sigma_{\rm out}}f(t)
 \bigr\|_{L_v^2L_x^\infty}\Bigr)\lesssim\mathbf{Data}.
\end{align*}
In particular, for $j,n\geq0$ with $j\leq J_{\rm out}$ and
$2j+n\leq\sigma_{\rm out}$,
\begin{align*}
 \sup_{t>0}\Bigl(&\langle t\rangle^j\bigl\|\langle v\rangle^M
 \langle x/\langle t\rangle\rangle^{N/2}
 \lambda(\nabla_x)^j\langle\nabla_{x,v}\rangle^nf(t)
 \bigr\|_{L^2_{x,v}}\\
 &+\langle t\rangle^{d/2+j}\bigl\|\langle v\rangle^M
 \lambda(\nabla_x)^j\langle\nabla_{x,v}\rangle^nf(t)
 \bigr\|_{L_v^2L_x^\infty}\Bigr)\lesssim\mathbf{Data}.
\end{align*}
\par\endgroup
            \end{proposition}

		\subsubsection{\texorpdfstring{{Linear estimates for $g_1$}}{Linear estimates for g1}}\label{secg1}
		{We begin the analysis of \eqref{eqdeco} with the linear problem}
		\begin{align}
			&	\partial_t f+v\cdot \nabla_x f+\mathbf{L}_s f+\mathbf{c}_1\chi_R(v)f=F,\label{lineq1}\\
			&	f|_{t=0}=f_0.\nonumber
		\end{align}
		{Following earlier work, we use the vector field}
		\begin{align*}
			Z=\nabla_v+t\nabla_x,
		\end{align*}
		{which satisfies}
		\begin{align*}
			[Z,\partial_t +v\cdot\nabla_x]=0.
		\end{align*}
		{We use the weighted Sobolev norms}
		\begin{align*}
			\|g\|_{L^2_{q}}=\|\langle v\rangle^qg\|_{L^2},\quad\quad	\|g\|_{H^s_{q}}=\|\langle v\rangle^qg\|_{H^s},
		\end{align*}
		{where $H^s$ is the standard Sobolev space on $\mathbb{R}^d_v$. We write $\langle \cdot,\cdot\rangle_{L^2}$ and $(\cdot,\cdot)_{L^2}$ for the standard $L^2$ inner products on $\mathbb{R}^d$ and $\mathbb{R}^d\times \mathbb{R}^d$, respectively. {The following lemma records the weighted coercivity estimate; see \cite{AMUXY,CG24}.}}\par\medskip
		\begin{lemma}\label{lemcoer1}
{Assume \eqref{assofb-lower}. The constants below
may depend on the fixed angular parameters, including $c_b$ and
$\delta_0$, in addition to $d,s,\gamma$ and the weight $k$.}
		There exist $c_0,\mathbf c_1,R>0$ such that
			\begin{align*}
				&\langle {\mathbf{L}}_s g+\mathbf{c}_1\chi_Rg,g\langle v\rangle^{2k} \rangle_{L^2}\geq c_0 \|g\|_{H^s_{k+\gamma/2}}^2.
			\end{align*}
		\end{lemma}
		Set
\[
 \vartheta:=\max\{2s,|\gamma|/2\},\qquad
 m_j(M,B):=M+\vartheta(B-j),\qquad j=|\alpha|+|\beta|.
\]
The polynomial energy and dissipation use this hierarchy:
\begin{align*}
 \mathbb E_{M,B}[f](t)&:=\sum_{|\alpha|+|\beta|\leq B}
 c_j\|Z^\beta\nabla_x^\alpha f(t)\|_{L_x^2L^2_{m_j(M,B)}}^2,\\
 \mathbb D_{M,B}[f](t)&:=\sum_{|\alpha|+|\beta|\leq B}
 c_j\|Z^\beta\nabla_x^\alpha f(t)\|_{L_x^2H^s_{m_j(M,B)+\gamma/2}}^2.
\end{align*}
Here the fixed positive coefficients $c_j$ are chosen successively
from the highest derivative level downwards; the coercivity constants
are chosen uniformly over the finite weight levels.
		{The commuted energy argument requires the following estimate.}
		\begin{lemma}
			\label{lemlines}
			{Suppose that \(f\) solves \eqref{lineq1}. Then, for
			{$\rho=0$ or $\rho\geq\frac12$},}
			\begin{align*}
				\frac{d}{dt}(\langle t\rangle^{2\rho}	\mathbb{E}_{M,B}[f](t))&+\frac{c_0}{2}\langle t\rangle^{2\rho}\mathbb{D}_{M,B}[f](t)\\&\lesssim \langle t\rangle^{2\rho}\sum_{|\alpha|+|\beta|\leq B}\left| \left(Z^\beta \nabla_x^\alpha F(t), \langle v\rangle^{2m_j(M,B)}Z^\beta \nabla_x^\alpha f(t)\right)_{L^2}\right|+\rho\mathbb{D}_{M-\gamma\rho,B}[f](t).
			\end{align*}
		\end{lemma}
		\begin{proof}
		{Fix $|\alpha|+|\beta|\leq B$ and write $m_j=m_j(M,B)$. Applying $Z^\beta \nabla_x^\alpha$ to \eqref{lineq1} gives}
			
			\begin{equation}\label{derif1}
				\begin{aligned}
					\partial_t \bigl(\langle t\rangle^\rho Z^\beta\nabla_x^\alpha f\bigr)+v\cdot \nabla_x \bigl(\langle t\rangle^\rho Z^\beta\nabla_x^\alpha f\bigr)+&\mathbf{L}_s \bigl(\langle t\rangle^\rho Z^\beta\nabla_x^\alpha f\bigr)+\mathbf{c}_1\chi_R(v)\bigl(\langle t\rangle^\rho Z^\beta\nabla_x^\alpha f\bigr)\\
					&=\langle t\rangle^\rho Z^\beta\nabla_x^\alpha F+\langle t\rangle^\rho \mathbf{Co}+\rho t\langle t\rangle^{\rho-2}Z^\beta\nabla_x^\alpha f,
				\end{aligned}
			\end{equation}
			{where $\mathbf{Co}$ denotes the commutator terms}
			\begin{align*}
				\mathbf{Co}=&\left(\mathbf{L}_s Z^\beta  \nabla_x^\alpha f- Z^\beta\mathbf{L}_s  \nabla_x^\alpha f\right)
				+\left(\mathbf{c}_1\chi_R(v)Z^\beta  \nabla_x^\alpha f-\mathbf{c}_1Z^\beta(\chi_R(v) \nabla_x^\alpha f)\right).
			\end{align*}
			{Taking the inner product of \eqref{derif1} with $\langle t\rangle^\rho\langle v\rangle^{2m_j}Z^\beta\nabla_x^\alpha f$ and applying Lemma \ref{lemcoer1}, we obtain}
			\begin{align*}
				&	\partial_t(\langle t\rangle^{2\rho} \|Z^\beta \nabla_x^\alpha f(t)\|_{L^2_{m_j}}^2)+c_0\langle t\rangle^{2\rho}\|Z^\beta \nabla_x^\alpha  f(t)\|_{H^s_{m_j+\gamma/2}}^2\\
				&\quad\lesssim \langle t\rangle^{2\rho}\left|\left(Z^\beta\nabla_x^\alpha F(t), \langle v\rangle^{2m_j}Z^\beta   \nabla_x^\alpha f(t)\right)_{L^2}\right|+\langle t\rangle^{2\rho}\left|\left(\mathbf{Co}(t), \langle v\rangle^{2m_j}Z^\beta  \nabla_x^\alpha  f(t)\right)_{L^2}\right|+\rho t\langle t\rangle^{2\rho-2}\|Z^\beta  \nabla_x^\alpha f(t)\|_{L^2_{m_j}}^2.
			\end{align*}
The time-weight term vanishes for $\rho=0$.
For $\rho\geq\frac12$, use $2\rho-1\geq0$ to split:
			\begin{align*}
				\langle t\rangle^{2\rho-1}	\|Z^\beta\nabla_x^\alpha f\|_{L^2_{m_j}}^2&\leq \langle t\rangle^{2\rho-1}\int_{\mathbb{R}^d}\mathbf{1}_{\langle t\rangle\geq \epsilon^{-1}\langle v\rangle^{-\gamma}}|Z^\beta\nabla_x^\alpha f|^2\langle v\rangle^{2m_j}\dif v\\
				&\quad+\langle t\rangle^{2\rho-1}\int_{\mathbb{R}^d}\mathbf{1}_{\langle t\rangle\leq \epsilon^{-1}\langle v\rangle^{-\gamma}}|Z^\beta\nabla_x^\alpha f|^2\langle v\rangle^{2m_j}\dif v\\
				&\leq \epsilon\langle t\rangle^{2\rho}\int_{\mathbb{R}^d}|Z^\beta\nabla_x^\alpha f|^2\langle v\rangle^{2m_j+\gamma}\dif v+C(\epsilon)\int_{\mathbb{R}^d}|Z^\beta\nabla_x^\alpha f|^2\langle v\rangle^{2m_j-\gamma(2\rho-1)}\dif v.
			\end{align*}		
			{Choosing $\epsilon>0$ so that $\epsilon\rho\ll c_0$, we obtain}
			\begin{align*}
				&	\partial_t(\langle t\rangle^{2\rho} \|Z^\beta \nabla_x^\alpha f(t)\|_{L^2L^2_{m_j}}^2)+\frac{c_0}{2}\langle t\rangle^{2\rho}\|Z^\beta  \nabla_x^\alpha f(t)\|_{L^2H^s_{m_j+\gamma/2}}^2\\
				&\quad\quad\lesssim \langle t\rangle^{2\rho}\left|\left(Z^\beta  \nabla_x^\alpha F(t), \langle v\rangle^{2m_j}Z^\beta  \nabla_x^\alpha  f(t)\right)_{L^2}\right|+\langle t\rangle^{2\rho}\sum_{ |\beta'|\leq |\beta|-1}\|Z^{\beta'} \nabla_x^\alpha f\|_{L^2L^2_{m_j}}^2+\rho\|Z^\beta \nabla_x^\alpha  f\|_{L^2L^2_{m_j-\gamma(2\rho-1)/2}}^2.
			\end{align*}
			The last term is controlled by $\rho\mathbb D_{M-\gamma\rho,B}[f]$.
For $U_\ell=Z^{\beta'}\nabla_x^\alpha f$ at a lower derivative level
$\ell=|\alpha|+|\beta'|<j$,
\[
 m_\ell+\gamma/2-m_j
 =\vartheta(j-\ell)-|\gamma|/2\geq0,
\]
so $\|U_\ell\|_{L^2_{m_j}}\leq
\|U_\ell\|_{H^s_{m_\ell+\gamma/2}}$.
Choosing the $c_j$ successively absorbs the lower-order sum.
The same summation applies to nonincreasing time exponents $\rho_j$,
with higher-moment weights $m_j-\gamma\rho_j$.
		\end{proof}\par\medskip
		{The preceding lemma yields the following weighted estimate, which propagates polynomial spatial decay.}
		\begin{corollary}\label{corof1}
			{Let $N\in \mathbb{N}$ and assume $N|\gamma-1|<M$. For $0\leq n\leq N$, set $M_n=M-n|\gamma-1|$, $m_{n,j}=m_j(M_n,B)$, and $f_n=\langle x\rangle^n f$. If \(f\) solves \eqref{lineq1}, then there are fixed $a_n>0$ such that, for
			{$\rho=0$ or $\rho\geq\frac12$},}
			\begin{align*}
				\sum_{n\leq N}a_n\frac{d}{dt}(\langle t\rangle^{2\rho}	\mathbb{E}_{M_n,B}[f_n](t))&+	\frac{c_0}{2}	\sum_{n\leq N}a_n(\langle t\rangle^{2\rho}\mathbb{D}_{M_n,B}[f_n](t))\\
		\quad\quad\quad\quad\quad\quad\lesssim\langle t\rangle^{2\rho}	\sum_{n\leq N}\sum_{|\alpha|+|\beta|\leq B}&\left| \left(Z^\beta \nabla_x^{\alpha}F_n(t), \langle v\rangle^{2m_{n,j}}Z^\beta\nabla_x^{\alpha} f_n(t)\right)_{L^2}\right|+\rho\sum_{n\leq N}\mathbb{D}_{M_n-\gamma\rho,B}[f_n](t).
			\end{align*}
		\end{corollary}
		\begin{proof}
			{For $0\leq n\leq N$, observe that}
			\begin{align*}
				\partial_t f_n+v\cdot \nabla_x f_n+\mathbf{L}_sf_n+\mathbf{c}_1 \chi_R(v )f_n=F_n+R_n,
			\end{align*}
			{where $R_n=v\cdot \nabla_x(\langle x\rangle^n) f$.}
			{Lemma \ref{lemlines} yields}
			\begin{align*}
				&\frac{d}{dt}(\langle t\rangle^{2\rho}	\mathbb{E}_{M_n,B}[f_n](t))+\langle t\rangle^{2\rho}\mathbb{D}_{M_n,B}[f_n](t)\\
				&\leq 	C\langle t\rangle^{2\rho}\sum_{|\alpha|+|\beta|\leq B}\left(\left| \left(Z^\beta \nabla_x^\alpha F_n(t), \langle v\rangle^{2m_{n,j}}Z^\beta\nabla_x^\alpha f_n(t)\right)_{L^2}\right|+\left| \left(Z^\beta \nabla_x^\alpha R_n(t), \langle v\rangle^{2m_{n,j}}Z^\beta \nabla_x^\alpha f_n(t)\right)_{L^2}\right|\right)\\
				&\quad\quad\quad+ C\rho \mathbb{D}_{M_n-\gamma\rho,B}[f_n](t).
			\end{align*}
			{By H\"older's inequality,}
			\begin{align*}
			&\langle t\rangle^{2\rho}
			\sum_{|\alpha|+|\beta|\leq B}
			\left|\left(
			Z^\beta\nabla_x^\alpha R_n(t),
			\langle v\rangle^{2m_{n,j}}Z^\beta\nabla_x^\alpha f_n(t)
			\right)_{L^2}\right|\\
				&\lesssim \langle t\rangle^{2\rho}\sum_{|\alpha|+|\beta|\leq B}
				\|Z^\beta\nabla_x^\alpha f_n\|_{H^s_{m_{n,j}+\gamma/2}}
				\|Z^\beta\nabla_x^\alpha R_n\|_{H^{-s}_{m_{n,j}-\gamma/2}}\\
				&\lesssim \varepsilon\langle t\rangle^{2\rho}
				\mathbb{D}_{M_n,B}[f_n](t)+C_{\varepsilon,N}\langle t\rangle^{2\rho}
				\sum_{0\leq\ell<n}\mathbb{D}_{M_\ell,B}[f_\ell](t),
				\qquad \varepsilon\in(0,1),
			\end{align*}
		{where the last inequality follows from the Leibniz rule applied to \(R_n=v\cdot\nabla_x\langle x\rangle^n f\): each derivative of the polynomial factor lowers the spatial level. For \(n=0\), \(R_0=0\). Consequently,}
			\begin{align*}
				&	\frac{d}{dt}(\langle t\rangle^{2\rho}	\mathbb{E}_{M_n,B}[f_n](t))+\langle t\rangle^{2\rho}\mathbb{D}_{M_n,B}[f_n](t)\\
				&\leq 	C_N\langle t\rangle^{2\rho}\sum_{0\leq\ell<n}\mathbb{D}_{M_\ell,B}[f_\ell](t)+C\langle t\rangle^{2\rho}\sum_{|\alpha|+|\beta|\leq B}\left| \left(Z^\beta \nabla_x^\alpha F_n(t), \langle v\rangle^{2m_{n,j}}Z^\beta\nabla_x^\alpha  f_n(t)\right)_{L^2}\right|+ C\rho \mathbb{D}_{M_n-\gamma\rho,B}[f_n](t).
			\end{align*}
			{Multiplying the inequalities by fixed positive coefficients, chosen successively from level \(N\) down to level \(0\), absorbs the triangular lower-level sum and completes the proof.}
		\end{proof}

		\subsubsection{\texorpdfstring{{Linear estimates for $g_2$}}{Linear estimates for g2}}\label{secg2}
		{We now consider the linear equation}
		\begin{equation}\label{lineq2}
			\begin{aligned}
				&\partial_t f +v\cdot \nabla_x f+\tilde {\mathcal{L}}_s f=G,\\
				&f|_{t=0}=0.
			\end{aligned}
		\end{equation}
		{As proved in \cite{Guo}, $\tilde {\mathcal{L}}_s$ is an unbounded symmetric operator on $L^2_v$ satisfying}
		\begin{align*}
			\langle \tilde {\mathcal{L}}_sf,f\rangle_{L^2}\geq 0.
		\end{align*}
		{Moreover, $\langle \tilde {\mathcal{L}}_sf,f\rangle=0$ if and only if $f\in \operatorname{ker} \tilde{ \mathcal{L}}_s$, the $(d+2)$-dimensional space spanned by the collision invariants}
		\begin{align*}
			\operatorname{ker}\tilde{ \mathcal{L}}_s=\operatorname{span}\{\sqrt{\mu}, v\sqrt{\mu}, (2d)^{-\frac{1}{2}}(|v|^2-d)\sqrt{\mu}\}.
		\end{align*}
		{Let $\mathbf P$ be the orthogonal projection from $L^2_v(\mathbb{R}^d)$ onto $\operatorname{ker}\tilde{\mathcal L}_s$. Define $(\rho,\mathsf{m},\mathsf{e})$ by}
		\begin{equation}\label{defabc}
			\begin{aligned}
				&\rho(t,x)=\int_{\mathbb{R}^d}\sqrt{\mu}f(t,x,v)\dif v,\\&\mathsf{m}(t,x)=\int_{\mathbb{R}^d}v\sqrt{\mu}f(t,x,v)\dif v,\\& \mathsf{e}(t,x)=\frac{1}{\sqrt{2d}}\int_{\mathbb{R}^d}(|v|^2-d)\sqrt{\mu}f(t,x,v)\dif v,
			\end{aligned}
		\end{equation}
		{These functions represent the mass, momentum, and internal-energy components of $\mathbf Pf$, respectively, and}
		\begin{align*}
			\mathbf{P}f=(\rho+\mathsf{m}\cdot v+(2d)^{-\frac{1}{2}}(|v|^2-d)\mathsf{e})\sqrt{\mu}.
		\end{align*}
		{Following
\cite{AMUXY2011-CMP,AMUXY2012JFA}, define the principal dissipative form}
        \begin{align*}
        \mathcal{A}[g]=\frac{1}{2}\int_{\mathbb{R}^{2d}}\int_{\mathbb{S}^{d-1}}B(|v_*-v|,\cos\theta)(\mu_*(g-g')^2+g_*^2(\sqrt{\mu'}-\sqrt{\mu})^2)\dif \sigma \dif v_* \dif v.
        \end{align*}
		{Under \eqref{assofb-lower}, the lower-bound
constant in the following norm comparison depends on $c_b$ and
$\delta_0$:}
		\begin{align*}
			\|g\|^2_{H^s_{\gamma/2}}+\|g\|^2_{L^2_{s+\gamma/2}}\lesssim 	\mathcal{A}[g]\lesssim   \|g\|^2_{H^s_{s+\gamma/2}}.
		\end{align*}
		{The difference
$\langle\tilde{\mathcal L}_s g,g\rangle_{L^2}-\mathcal A[g]$
is a compact perturbation.}
The following lemma summarizes the coercivity bounds; see \cite{DL,Guo,Mou06,MS07,YZ}.
		\begin{lemma}\label{lemcoer}
{Assume \eqref{assofb-lower}; the structural dependence
of the constants includes $c_b$ and $\delta_0$.}
			{Suppose that $-d-2s<\gamma\leq-2s$. Then there exist $c_0,\mathbf c_1,R>0$ such that}
			\begin{align*}
				&\langle \tilde {\mathcal{L}}_s g,g \rangle_{L^2}\geq c_0\mathcal{A}[(\mathrm{Id}-\mathbf{P})g].
			\end{align*}
		\end{lemma}
		{Lemma \ref{lemcoer} controls the microscopic component $(\mathrm{Id}-\mathbf P)g$. Dissipation of the degenerate component $\mathbf Pg$ is recovered from a fluid-type moment system, following the mechanism introduced by Kawashima \cite{Kawa} and later refined in \cite{Guo,Guo-0,DuanHyp2011,Duan11,Guo06}. Let $f$ solve \eqref{lineq2} with prescribed source $G$. Taking the velocity moments}
		\begin{align*}
			&	\sqrt{\mu(v)},\ \  (v_i	\sqrt{\mu(v)})_{i=1}^d,\ \ (2d)^{-\frac{1}{2}}(|v|^2-d)	\sqrt{\mu(v)},\\  &\left[(v_iv_j-\delta_{ij})	\sqrt{\mu(v)}\right]_{d\times d},\ \  \left((|v|^2-d-2)v_i	\sqrt{\mu(v)}\right)_{i=1}^d.
		\end{align*}
		{yields the following system for $(\rho,\mathsf{m},\mathsf{e})$ defined in \eqref{defabc}:}
		\begin{equation*}
			\begin{cases}
				&\partial_t \rho+\nabla_x \cdot \mathsf{m}=\langle\sqrt{\mu},G\rangle,\\
				&\partial_t \mathsf{m}+\nabla_x (\rho+\sqrt{\frac{2}{d}}\mathsf{e})+\nabla_x\cdot \Gamma((\mathrm{Id}-\mathbf{P})f)=\langle v\sqrt{\mu},G\rangle,\\
				&\partial_t \mathsf{e}+\sqrt{\frac{2}{d}}\nabla_x\cdot \mathsf{m}+\frac{1}{\sqrt{2d}}\nabla_x\cdot \Lambda((\mathrm{Id}-\mathbf{P})f)=\langle (2d)^{-\frac{1}{2}}(|v|^2-d)	\sqrt{\mu}, G\rangle ,\\
				&\partial_t [\Gamma_{ij}((\mathrm{Id}-\mathbf{P})f)+\sqrt{\frac{2}{d}}\mathsf{e}\delta_{ij}]+\partial_i \mathsf{m}_j+\partial_j \mathsf{m}_i=\Gamma_{ij}(\mathbf{r}+G),\\
				&\partial_t \Lambda_i((\mathrm{Id}-\mathbf{P})f)+\frac{2(d+2)}{\sqrt{2d}}\partial_i \mathsf{e}=\Lambda_i(\mathbf{r}+G),
			\end{cases}
		\end{equation*}
		{where $\Lambda=(\Lambda_i)_{i=1}^d$ and $\Gamma=(\Gamma_{ij})_{d\times d}$ are defined by}
		\begin{align*}
			&\Lambda_i (f)=\langle    (|v|^2-d-2)v_i\sqrt{\mu},f\rangle_{L^2},\quad\quad\quad	\Gamma_{ij}(f)
			:=
			\left\langle
			(v_iv_j-\delta_{ij})\sqrt\mu,f
			\right\rangle_{L_v^2},
		\end{align*}
		{and}
		\begin{align*}
			\mathbf r=-v\cdot \nabla_x (\mathrm{Id}-\mathbf{P})f-\tilde{ \mathcal{L}}_s(\mathrm{Id}-\mathbf{P})f.
		\end{align*}
		The coefficient in the last moment equation follows from
\[
\int_{\mathbb R^d}(|v|^2-d-2)v_i^2(|v|^2-d)\mu(v)\,\dif v=2(d+2).
\]
{The following lemma controls the macroscopic
component at low frequencies.}
\begin{lemma}\label{lem:lowfreq-repair}
Fix $K<\infty$, $B\in\mathbb N$, $0<s<1$, and
$-d-2s<\gamma\leq-2s$. For smooth solutions of \eqref{lineq2}, there is a
real quadratic interaction $\mathcal M$ such that, for $|\xi|\leq K$,
\begin{equation}\label{eq:lowfreq-repaired}
 \partial_t\mathcal M+c_{B,K}\sum_{|\beta|\leq B}
 \langle t\rangle^{-2|\beta|}|\xi|^2
 \|Z^\beta\mathbf P\widehat f\|_2^2
 \leq C_{B,K}\bigl(\|\mu^{1/4}({\mathrm{Id}}-\mathbf P)\widehat f\|_2^2
                  +|\Omega\widehat G|^2\bigr),
\end{equation}
where $\Omega u=(\langle v^\alpha\sqrt\mu,u\rangle)_{|\alpha|\leq4}$,
and $|\mathcal M|\leq C_K\|\widehat f\|_2^2$.
\end{lemma}
\begin{proof}
{We suppress Fourier hats on the macroscopic moments.}
Set $h=({\mathrm{Id}}-\mathbf P)\widehat f$,
$S=\Gamma(h)$, $q=\Lambda(h)$, and
\[
 a=\sqrt{2/d},\quad b=(2d)^{-1/2},\quad c=2(d+2)(2d)^{-1/2},
 \qquad T_\xi {\mathsf{m}}=\xi\otimes {\mathsf{m}}+{\mathsf{m}}\otimes\xi-\tfrac2d(\xi\cdot {\mathsf{m}}){\mathrm{Id}}.
\]
Since $\operatorname{tr}S=0$, taking the trace-free part of the stress equation
gives the closed moment identities
\begin{align*}
 \rho_t&=-i\xi\cdot {\mathsf{m}}+G_\rho,&
 {\mathsf{m}}_t&=-i\xi(\rho+{a\mathsf e})-iS\xi+G_m,\\
 {\mathsf{e}}_t&=-ia\xi\cdot {\mathsf{m}}-ib\xi\cdot q+G_e,&
 S_t&=-iT_\xi {\mathsf{m}}+R_S,& q_t&=-ic\xi {\mathsf{e}}+R_q,
\end{align*}
where $R_S$ is the trace-free part of $\Gamma(\widehat{\mathbf r}+\widehat G)$
and $R_q=\Lambda(\widehat{\mathbf r}+\widehat G)$.
Writing $H=\|\mu^{1/4}h\|_2$ and $W=|\Omega\widehat G|$,
\begin{equation}\label{eq:moment-remainder-bound}
 |S|+|q|\leq CH,\qquad |R_S|+|R_q|\leq C_K(H+W).
\end{equation}
Here the collision contribution follows directly from
$|\widetilde{\mathcal L}_s(P\sqrt\mu)(v)|
\leq C_P\sqrt\mu(v)\langle v\rangle^{N_P}$ for polynomials $P$ of degree
at most three. Indeed, detailed balance writes this expression as
$\sqrt\mu(v)\int B\mu_*[P+P_*-P'-P_*']$.
After azimuthal averaging the polynomial difference is bounded by
$C|v-v_*|^2\theta^2(1+|v|+|v_*|)^{N_P}$.
Its integrability follows from $\gamma+2>-d$ and $s<1$.
Testing against $h$ proves \eqref{eq:moment-remainder-bound}; transport
moments have degree at most four.

Define, using complex conjugates in all pairings,
\[
 I_e=\operatorname{Re}(i\xi {\mathsf{e}}\cdot\overline q),\qquad
 I_m=\operatorname{Re}(iT_\xi {\mathsf{m}}:\overline S),\qquad
 I_\rho=\operatorname{Re}(i\rho\xi\cdot\overline{{\mathsf{m}}}).
\]
Differentiating and using \eqref{eq:moment-remainder-bound} gives, for every
$\eta>0$, with fixed positive $c_e,c_m,c_\rho$,
\begin{align*}
 \partial_t I_e+c_e|\xi|^2|{\mathsf{e}}|^2
 &\leq\eta|\xi|^2|{\mathsf{m}}|^2+C_{\eta,K}(H^2+W^2),\\
 \partial_t I_m+c_m|\xi|^2|{\mathsf{m}}|^2
 &\leq\eta|\xi|^2(|\rho|^2+|{\mathsf{e}}|^2)+C_{\eta,K}(H^2+W^2),\\
 \partial_t I_\rho+c_\rho|\xi|^2|\rho|^2
 &\leq C|\xi|^2(|{\mathsf{m}}|^2+|{\mathsf{e}}|^2)+C_K(H^2+W^2).
\end{align*}
For clarity, the leading negative terms are respectively
$-c|\xi|^2|{\mathsf{e}}|^2$, $-|T_\xi {\mathsf{m}}|^2$, and $-|\xi|^2|\rho|^2$.
Also,
$|T_\xi {\mathsf{m}}|^2=2|\xi|^2|{\mathsf{m}}|^2+(2-4/d)|\xi\cdot {\mathsf{m}}|^2
\geq2|\xi|^2|{\mathsf{m}}|^2$ for $d\geq2$.
Every remaining product is bounded by Young's inequality; factors
$|\xi|^2H^2$ are bounded by $K^2H^2$.
Choose $\lambda>0$ so small that $\lambda C<\min(c_e,c_m)/4$, and then
$\eta<\min(c_e,c_m,\lambda c_\rho)/4$.
The sum $\mathcal M=I_e+I_m+\lambda I_\rho$ controls
$|\xi|^2(|\rho|^2+|{\mathsf{m}}|^2+|{\mathsf{e}}|^2)$ as asserted.
Finally $Z=\nabla_v+it\xi$ in Fourier variables, and the binomial formula
on the finite-dimensional range of $\mathbf P$ gives
\[
 \sum_{|\beta|\leq B}\langle t\rangle^{-2|\beta|}
 \|Z^\beta\mathbf P\widehat f\|_2^2
 \leq C_{B,K}\|\mathbf P\widehat f\|_2^2.
\]
This proves \eqref{eq:lowfreq-repaired} and the quadratic bound on
$\mathcal M$.
\end{proof}
{We use Lemma~\ref{lem:lowfreq-repair}
only in the low-frequency energy; the high-frequency estimate uses
the energy $E^{\mathrm h}_{M,B}$ defined below.}

	{Taking the Fourier transform in $x$ gives}
		\begin{equation*}
			\begin{aligned}
				&	\partial_t \hat f+i\xi\cdot v\hat f+\tilde {\mathcal{L}}_s \hat f=\hat G,\\
				&\hat f|_{t=0}=\hat f_0.
			\end{aligned}
		\end{equation*}
		{We recall the frequency-dependent hypocoercive energies introduced in \cite{BCD}, defined separately at high and low spatial frequencies. The high-frequency energy captures the enhanced dissipation generated by transport and velocity diffusion, whereas the low-frequency energy incorporates the micro--macro structure and Kawashima coupling needed to control the hydrodynamic modes. The high-frequency energy and dissipation are}\par\vspace{0.1cm}
        \begin{align*}
			E_{M, B}^{\text{h}}(\hat f)(t, \xi)&:=\frac{1}{2} \sum_{\alpha+|\beta| \leq B} \frac{2^{-\mathrm{c} |\beta|}\langle \xi\rangle^\alpha}{\langle t\rangle^{2 |\beta|}}\left(\left\|\langle v\rangle^M Z^\beta \hat{f}\right\|_{L_v^2}^2+a_{ \xi}\left\|\langle v\rangle^{M+q_{\gamma, s}} Z^\beta \nabla_v \hat{f} \right\|_{L_v^2}^2\right. \\
			& \left.\quad \quad \quad \quad \quad \quad \quad \quad \quad \quad+2 b_{\xi} \operatorname{Re}\left\langle\langle v\rangle^{M+q_{\gamma, s}} Z^\beta i \xi \hat{f},\langle v\rangle^{M+q_{\gamma, s}} Z^\beta \nabla_v \hat{f}\right\rangle_{L_v^2}\right),\\
			D_{M, B}^{\text{h}}(\hat f)(t, \xi)&:=\sum_{\alpha+|\beta| \leq B} \frac{2^{-c |\beta|}\langle \xi\rangle^\alpha}{\langle t\rangle^{2 |\beta|}}\left( \mathcal{A}
			\left[\langle v\rangle^M Z^\beta \hat{f}\right]+ a_{ \xi} \mathcal{A}\left[\langle v\rangle^{M+q_{\gamma, s}} \nabla_v Z^\beta \hat{f}\right] \right.\\
			&\left. \quad \quad \quad \quad \quad \quad \quad \quad \quad\quad\quad+  b_{ \xi}|\xi|^2\left\|\langle v\rangle^{M+q_{\gamma, s}} Z^\beta \hat{f}\right\|_{L_v^2}^2\right).
		\end{align*}
{The low-frequency energy and dissipation are}

		\begin{align*}
			E_{M, B}^{\text{l}}(\hat f)(t, \xi)&:=\frac{1}{2}\|\hat f\|_{L^2_v}^2+\frac{1}{2} \sum_{|\beta| \leq B} \sum_{j=0}^1 \frac{2^{-c_j |\beta|}}{\langle t\rangle^{2 |\beta|}}\left({|\xi|}\left\|Z^\beta \hat{f}\right\|_{L_v^2}^2+c_1\left\|Z^\beta \nabla_v^j(\mathrm{Id}-\mathbf{P}) \hat{f}\right\|_{L_v^2}^2\right. \\
			&	\left.\quad \quad \quad \quad \quad \quad \quad \quad \quad \quad \quad \quad \quad \quad+c_2\left\|\langle v\rangle^{M+j q_{\gamma, s}} Z^\beta \nabla_v^j(\mathrm{Id}-\mathbf{P}) \hat{f}\right\|_{L_v^2}^2\right) \\
			&\quad	+{c_0} \mathcal{M}+{c_3} \sum_{|\beta| \leq B} \frac{2^{-C_0 |\beta|}}{\langle t\rangle^{2 |\beta|}} \operatorname{Re}\left\langle\langle v\rangle^{M+q_{\gamma, s}} Z^\beta(i \xi(\mathrm{Id}-\mathbf{P}) \hat{f}),\langle v\rangle^{M+q_{\gamma, s}} Z^\beta \nabla_v(\mathrm{Id}-\mathbf{P}) \hat{f}\right\rangle_{L_v^2},\\
			D_{M, B}^{\text{l}}(\hat f)(t,\xi)&:=\mathcal{A}[(\mathrm{Id}-\mathbf{P}) \hat{f}] +\sum_{|\beta| \leq B} \frac{2^{-\mathrm{C}_0 |\beta|}}{\langle t\rangle^{2 |\beta|}} |\xi|^2\left(\left\|Z^\beta \mathbf{P} \hat{f}\right\|_{L_v^2}^2+\left\|\langle v\rangle^{M+q_{\gamma, s}} Z^\beta(\mathrm{Id}-\mathbf{P}) \hat{f}\right\|_{L_v^2}^2\right) \\
			&\quad +\sum_{|\beta| \leq B} \sum_{j=0}^1 \frac{2^{-\mathrm{c}_j |\beta|}}{\langle t\rangle^{2 |\beta|}}\left(|\xi| \mathcal{A}\left[Z^\beta(\mathrm{Id}-\mathbf{P}) \hat{f}\right]+ \mathcal{A}\left[Z^\beta\left(\nabla_v\right)^j(\mathrm{Id}-\mathbf{P}) \hat{f}\right]\right. \\
			& \quad  \quad \quad \quad \quad \quad\quad \quad \quad\left.+ \mathcal{A}\left[\langle v\rangle^{M+j q_{\gamma, s}} Z^\beta\left(\nabla_v\right)^j(\mathrm{Id}-\mathbf{P}) \hat{f}\right]\right).
		\end{align*}
		{Here $\alpha\in\mathbb{N}$, $\beta$ is a multi-index, and $c\geq1$ is sufficiently large. Moreover,}
		\begin{align*}
			{\frac{\gamma}{2s}-\kappa_0<q_{\gamma,s}<\frac{\gamma}{2s}
\quad(0<s<1),\qquad q_{\gamma,1}=\frac{\gamma}{2},\qquad 0<\kappa_0\ll1,}
		\end{align*} 
		{and}
		\begin{align*}
			a_\xi=a_{\rm h}|\xi|^{-\frac{2}{2s+1}},\qquad b_\xi=b_{\rm h}|\xi|^{-1-\frac{1}{2s+1}},
		\end{align*}
		where $a_{\rm h},b_{\rm h}>0$ satisfy
$b_{\rm h}^2/a_{\rm h}\leq1/16$. In particular,
\[
\frac{b_\xi^2|\xi|^2}{a_\xi}=\frac{b_{\rm h}^2}{a_{\rm h}}.
\]
Since $q_{\gamma,s}<0$, Cauchy--Schwarz and Young's inequality show that the
mixed term is bounded by a fixed fraction of the two positive diagonal terms.
{Thus the high-frequency energy is positive.}
		
		{Set}
		\begin{equation*}
			\begin{aligned}
				& E_{M, B}(\hat f)(t, \xi):=\mathbf{1}_{|\xi| \geq \delta_{\mathrm{hyp}}^{-1}} E_{M, B}^{\mathrm{h}}(\hat f)(t, \xi)+\mathbf{1}_{|\xi|<\delta_{\mathrm{hyp}}^{-1}} E_{M, B}^{\mathrm{l}}(\hat f)(t, \xi), \\
				&D_{M, B}(\hat f)(t, \xi):=\mathbf{1}_{|\xi| \geq \delta_{\mathrm{hyp}}^{-1}} D_{M, B}^{\mathrm{h}}(\hat f)(t, \xi)+\mathbf{1}_{|\xi|<\delta_{\mathrm{hyp}}^{-1}} D_{M, B}^{\mathrm{l}}(\hat f)(t, \xi).
			\end{aligned}
		\end{equation*}
		The following lemma is an extension of the frequency-dependent estimate
in \cite{BCD} to the present weights and collision range.
{For the linear and velocity-weight commutator inputs,
we use the estimates of \cite[Theorem~2.2 and Corollary~2.2]{YZ},
with the conventions described above.}
		\begin{lemma}
			\label{lemlines1}
			{Suppose that $f$ solves \eqref{lineq2}. Then}
			\begin{equation*}
				\begin{aligned}
					\frac{\partial}{\partial t}E_{M, B}(\hat f)(t, \xi)+\delta_{\mathrm{hyp}} D_{M, B}(\hat f)(t, \xi)\leq \mathcal{N}_{M,B}(\hat G,\hat f)(t,\xi),
				\end{aligned}
			\end{equation*}
			{where $\mathcal{N}_{M, B}(\hat G,\hat f)(t, \xi):=\mathbf{1}_{|\xi| \geq \delta_{\mathrm{hyp}}^{-1}} NL_{M, B}^{\mathrm{h}}(t, \xi)+\mathbf{1}_{|\xi|<\delta_{\mathrm{hyp}}^{-1}} NL_{M, {B}}^{\mathrm{l}}(t, \xi)$, with}
			\begin{align*}
				N L_{M, B}^{\mathrm{h}}=&\sum_{\alpha+|\beta| \leq B}  \frac{2^{-c |\beta|}}{\langle t\rangle^{2 |\beta|}}\langle \xi\rangle^{ \alpha}\left( \left|\left\langle \langle v\rangle^M Z^\beta \hat G,\langle v\rangle^M Z^\beta \hat{f} \right\rangle_{L_v^2}\right|\right.  +a_{ \xi}\left|\left\langle\langle v\rangle^{M+q_{\gamma, s}} \nabla_v Z^\beta \hat G,\langle v\rangle^{M+q_{\gamma, s}} \nabla_v Z^\beta \hat{f}\right\rangle_{L_v^2}\right| \\
				& +b_{ \xi}\left|\left\langle\langle v\rangle^{M+q_{\gamma, s}} Z^\beta i \xi\hat G,\langle v\rangle^{M+q_{\gamma, s}} \nabla_v Z^\beta \hat{f}\right\rangle_{L_v^2} \right| \left.+b_{ \xi}\left|\left\langle\langle v\rangle^{M+q_{\gamma, s}} \nabla_v Z^\beta \hat G,\langle v\rangle^{M+q_{\gamma, s}} Z^\beta i \xi \hat{f}\right\rangle_{L_v^2}\right|\right),
			\end{align*}
			and 
			\begin{align*}
				N L_{M, B}^{\mathrm{l}}=&\left|\left\langle \hat G, \hat f\right\rangle_{L^2_v}\right|+\sum_{|\beta| \leq B} \sum_{j=0}^1 \frac{2^{-C_j |\beta|}}{\langle t\rangle^{2 |\beta|}}\left(
				\left| {|\xi|}\left\langle Z^\beta \hat G, Z^\beta \hat{f}\right\rangle_{L_v^2}\right| \right.\\
				& \quad+c_1\left|\left\langle Z^\beta\left(\nabla_v\right)^j(\mathrm{Id}-\mathbf{P})\hat G, Z^\beta\left(\nabla_v\right)^j(\mathrm{Id}-\mathbf{P}) \hat{f}\right\rangle_{L_v^2} \right|\\
				& \left.\quad+c_2\left|\left\langle\langle v\rangle^{M+j q_{\gamma, s}} Z^\beta\left(\nabla_v\right)^j (\mathrm{Id}-\mathbf{P})\hat G,\langle v\rangle^{M+j q_{\gamma, s}} Z^\beta\left(\nabla_v\right)^j(\mathrm{Id}-\mathbf{P}) \hat{f}\right\rangle_{L_v^2}\right| \right)\\
				& +{c_0} |\Omega\hat G|^2+\sum_{|\beta| \leq B} \frac{2^{-c |\beta|}}{\langle t\rangle^{2 |\beta|}} c_3\left(\left|\left\langle\langle v\rangle^{M+q_{\gamma, s}} Z^\beta i \xi(\mathrm{Id}-\mathbf{P}) \hat G,\langle v\rangle^{M+q_{\gamma, s}} \nabla_v Z^\beta(\mathrm{Id}-\mathbf{P}) \hat{f}\right\rangle_{L_v^2}\right|\right. \\
				& \left.\quad+\left|\left\langle\langle v\rangle^{M+q_{\gamma, s}} \nabla_v Z^\beta (\mathrm{Id}-\mathbf{P})\hat G,\langle v\rangle^{M+q_{\gamma, s}} Z^\beta(\mathrm{Id}-\mathbf{P}) i \xi \hat{f}\right\rangle_{L_v^2}\right|\right).
			\end{align*}
            {Here $\Omega f=\langle (v^\alpha \sqrt{\mu})_{0\leq|\alpha|\leq 4}, f\rangle$.}
		\end{lemma}
	\begin{remark}[Velocity weights and decay]\label{reaa1}
For strictly soft potentials, $\gamma+2s<0$, one cannot infer
$D_{M,B}\gtrsim\lambda(\xi)E_{M,B}$ with the same velocity weight.
At a fixed nonzero low frequency, a microscopic packet translated to large
velocity has dissipation-to-energy ratio tending to zero. Thus the displayed
energy inequality does not give uniform exponential decay of each mode in
that norm. The higher-moment energies below are needed to recover decay by
interpolation between velocity weights. The boundary case
$\gamma+2s=0$ is not included in this obstruction.
\end{remark}

\begingroup
{The following lemma quantifies the decay obtained
from a moment reserve.}
\begin{lemma}\label{lem:finite-decay-budget}
Let $J$ be a nonnegative integer, $\lambda,H\geq0$, and let $E\geq0$
be locally absolutely continuous with $E(0)\leq H$. Suppose that, for
every $R\geq1$ and almost every $t\geq0$,
\[
 E'+c\lambda R^{-\varpi}E\leq C\lambda R^{-A}H,
 \qquad A,\varpi>0,
\]
where $c,C>0$ are independent of $R,t$. If $\varpi(2J+1)<A$, then
$\sup_{t\geq0}\langle\lambda t\rangle^{2J}E(t)\leq C_JH$.
\end{lemma}
\begin{proof}
The case $\lambda=0$ is immediate. For $\lambda>0$, write
$u=\langle\lambda t\rangle$ and fix a sufficiently small $\varepsilon>0$.
The choice $R=1$ bounds $E$ by $CH$. When $u\geq\varepsilon^{-1}$,
choose $R=(\varepsilon u)^{1/\varpi}$ and multiply by $u^{2J}$.
Since $u'\leq\lambda$, the derivative of this weight is absorbed by
$c\varepsilon^{-1}\lambda u^{2J-1}E$. The remaining right-hand side is
$C_\varepsilon\lambda u^{2J-A/\varpi}H$, which is integrable because
$2J-A/\varpi<-1$.
\end{proof}

For the reconstruction inequality in \cite[Lemma~4.8]{BCD}, the exponents are
\begin{equation}\label{eq:finite-linear-budget}
 \varpi=\frac{|\gamma|(2-s)+2s|q_{\gamma,s}|}{1+s},\qquad
 A=|\gamma+2s|+2m_{\rm r},\qquad
 \varpi(2J+1)<|\gamma+2s|+2m_{\rm r},
\end{equation}
where $m_{\rm r}$ is the additional velocity weight.
At $s=1$, these become $\varpi=|\gamma|$ and
$|\gamma|(2J+1)<|\gamma+2|+2m_{\rm r}$.
These choices require the corresponding reconstruction inequality
for each component.
\par\endgroup

\subsubsection{Global estimates}
		{We now combine the linear estimates from Subsections \ref{secg1} and \ref{secg2} to derive global bounds for the coupled nonlinear system \eqref{eqdeco}. Set $r_d=\lfloor d/2\rfloor+1$, $K=B_{\rm p}\leq\sigma$, and choose $\rho>1$ with}\par\medskip
\[
 \vartheta K+|\gamma|(\rho+2K)<M'.
\]
		Use running energy suprema and accumulated dissipations:
\begin{align*}
\mathcal E_N[f](t)
&:=\sup_{0\leq r\leq t}\sum_{n\leq N}\langle r\rangle^{2\rho}
 \mathbb E_{M_n,B}[\langle x\rangle^nf](r),\\
\mathcal D_N[f](t)
&:=\sum_{n\leq N}\int_0^t\langle r\rangle^{2\rho}
 \mathbb D_{M_n,B}[\langle x\rangle^nf](r)\,\dif r.
\end{align*}
{A superscript $[L]$ on $\mathcal E_N$ or $\mathcal D_N$
means that the polynomial derivative index $B$ is replaced by $L$.}
For the Maxwellian-weighted component, define
\begin{align*}
\tilde{\mathcal E}[f](t)
&:=\sup_{0\leq r\leq t}\int_{\mathbb R^d}
 \langle\lambda_\xi r\rangle^{2J}E_{M,B}(\widehat f)(r,\xi)\,\dif\xi,\\
\tilde{\mathcal E}_{mom}[f](t)
&:=\sup_{0\leq r\leq t}\int_{\mathbb R^d}
 E_{M+M',B}(\widehat f)(r,\xi)\,\dif\xi,\\
\tilde{\mathcal E}_{LF}[f](t)
&:=\sup_{0\leq r\leq t}\sup_\xi
 \langle\lambda_\xi r\rangle^{2J'}E_{M',B'}(\widehat f)(r,\xi),\\
\tilde{\mathcal E}_{LF,mom}[f](t)
&:=\sup_{0\leq r\leq t}\sup_\xi
 E_{M'+M_{J'},B'}(\widehat f)(r,\xi).
\end{align*}
The structural conditions corresponding to \cite[(2.19)]{BCD} are
\begin{equation}\label{eq:global-structural-budget}
\begin{gathered}
 H_*:=\max\{\varpi,B+|\gamma|+2s\},\qquad M_J:=M',\\
 M,M'>2H_*,\qquad M_J,M_{J'}>2H_*(J+1),\qquad
 M'+M_{J'}<M-1,\\
 2d<B',\qquad B>B'+d,\qquad d<J'<J-1.
\end{gathered}
\end{equation}
The corresponding dissipations are
\begin{align*}
\tilde{\mathcal D}[f](t)
&:=\int_0^t\int_{\mathbb R^d}
 \langle\lambda_\xi r\rangle^{2J}D_{M,B}(\widehat f)(r,\xi)\,\dif\xi\,\dif r,\\
\tilde{\mathcal D}_{mom}[f](t)
&:=\int_0^t\int_{\mathbb R^d}D_{M+M',B}(\widehat f)(r,\xi)\,\dif\xi\,\dif r,\\
\tilde{\mathcal D}_{LF}[f](t)
&:={\sup_\xi\int_0^t
 \bigl[\langle\lambda_\xi r\rangle^{2J'}D_{M',B'}(\widehat f)(r,\xi)\bigr]\,\dif r},\\
\tilde{\mathcal D}_{LF,mom}[f](t)
&:={\sup_\xi\int_0^t D_{M'+M_{J'},B'}(\widehat f)(r,\xi)\,\dif r}.
\end{align*}
{For the frequency-supremum norms, integrate
the energy inequality at each fixed frequency before taking the supremum.}
\begingroup
To retain the higher-weight forcing, set, for $\eta\geq0$ and $L\in\mathbb N$,
$F_n=\langle x\rangle^nF$ and $p_n=\langle x\rangle^np$. Define
\begin{align*}
 \mathscr F_{m,L}^{(\eta)}[F,p](T)
 := {}&\sum_{n\leq N}\sum_{|\alpha|+|\beta|\leq L}
 \int_0^T\!\langle r\rangle^{2\eta}\\
 &\quad\times\Bigl|\Bigl(Z^\beta\nabla_x^\alpha F_n,
 \langle v\rangle^{2m_j(m-n|\gamma-1|,L)}
 Z^\beta\nabla_x^\alpha p_n\Bigr)_{L^2}\Bigr|\,\dif r.
\end{align*}
{Under the commutator bounds specified above,
apply Corollary~\ref{corof1}}
at time weight $\rho$ and then at time weight zero with base velocity
weight $M-\gamma\rho$. This gives
\begin{equation}\label{firg1}
\begin{aligned}
 \mathcal E_N^{[L]}[g_1](T)+c\mathcal D_N^{[L]}[g_1](T)
 &\lesssim\sum_{n\leq N}\mathbb E_{M_n-\gamma\rho,L}
       [\langle x\rangle^ng_0](0)\\
 &\quad+\mathscr F_{M,L}^{(\rho)}[F(g_1,g_2),g_1](T)\\
 &\quad+\mathscr F_{M-\gamma\rho,L}^{(0)}[F(g_1,g_2),g_1](T).
\end{aligned}
\end{equation}
\par\endgroup
\begingroup
For the strengthened polynomial hierarchy below, the nonlinear input is
\begin{equation}\label{energ1}
 \mathscr F_K^{\rm dec}[F(p,q),p](T)
 +\mathscr F_K^{\rm mom}[F(p,q),p](T)\lesssim X_T^{3/2}.
\end{equation}
Here $\mathscr F_K^{\rm dec}$ is the sum of the commuted source pairings
with weights $\langle t\rangle^{2\rho_j}\langle v\rangle^{2m_{n,j}}$;
$\mathscr F_K^{\rm mom}$ uses $\langle v\rangle^{2h_{n,j}}$.
This input includes $\mathcal Q_s(p,p)$ as well as the mixed terms.
\par\endgroup
{{For the Maxwellian component, applying}
Lemma \ref{lemlines1} to the $g_2$-equation in \eqref{eqdeco}, we obtain}
		\begin{equation}\label{firg2}
			\begin{aligned}
				\tilde {\mathcal{E}}[g_2](t)&+\frac{\delta_{\mathrm{hyp}}}{10}\tilde {\mathcal{D}}[g_2](t)\\
				&\leq C\int_0^t \int_{\mathbb{R}^d} L_{J,M,B}(\hat g_2)(\tau,\xi)\dif \xi \dif \tau +C \int_0^t \int_{\mathbb{R}^d}\langle \lambda_\xi \tau\rangle^{2J}\mathcal{N}_{M,B}(\widehat {G(g_1,g_2)},\hat  g_2)(\tau,\xi) \dif \xi \dif \tau ,
			\end{aligned}
		\end{equation}
		{where $G(g_1,g_2)$ is defined in \eqref{defFG}, and $L_{J,M,B}(\hat g_2)$ is the linear error term}
		\begin{align*}
			L_{J,M,B}(\hat g_2)(t,\xi):=2J\lambda_\xi\langle\lambda_\xi t\rangle^{2J-1}E_{M,B}(\hat g_2)(t,\xi).
		\end{align*}
{The moment estimate is obtained in the same way.
For the two frequency-supremum estimates, take the supremum of the
time-integrated right-hand side, consistently with the definitions above.}
		
{The nonlinear estimates use distinct
dissipative forms for the two components. For the polynomial
component, define}
\begin{align}
 \mathfrak a_\mu[u]
 &:=\frac12\iint \mathbf C_\mu(v,z)
       |u(v)-u(v-z)|^2\frac{\dif z\dif v}{|z|^{d+2s}}
       +\|u\|_{L^2_{\gamma/2}}^2,\label{eq:raw-collision-form}\\
 \mathbb D^{\sharp}_{M,B}[u]
 &:=\sum_{|\alpha|+|\beta|\leq B}c_j\int
 \mathfrak a_\mu[\langle v\rangle^{m_j(M,B)}Z^\beta\partial_x^\alpha u](x)\,\dif x.
 \label{eq:full-dissipation-candidate}
\end{align}
The nonlinear closure requires the full-form coercivity 
\begin{equation}\label{eq:global-full-form-input}
 \langle\mathbf L_su+\mathbf c_1\chi_Ru,
              \langle v\rangle^{2m}u\rangle
 \geq c\,\mathfrak a_\mu[\langle v\rangle^mu]
\end{equation}
through the displayed polynomial weights, together with its commuted
version. The Maxwellian component uses $\mathcal A$.

{We first formulate the nonlinear estimate
for ordinary derivatives. For $\partial^\nu=\partial_{x,v}^\nu$, write}
\[
 E_B(u)=\sum_{|\nu|\leq B}c_{|\nu|}
       \|\langle v\rangle^{m_{|\nu|}(M,B)}\partial^\nu u\|_2^2,
 \qquad
 D_B^\sharp(u)=\sum_{|\nu|\leq B}c_{|\nu|}\int
       \mathfrak a_\mu[\langle v\rangle^{m_{|\nu|}(M,B)}\partial^\nu u]\,\dif x.
\]
Let $D_B^{\mathcal A}$ denote the same sum with $\mathcal A$ in place of
$\mathfrak a_\mu$.
{Set $p=g_1$ and $q=g_2$, and use
$F(p,q)$ from \eqref{defFG}.}
\begingroup
{We use derivative-dependent velocity weights
in the following mixed collision estimate.}
\begin{lemma}[Mixed collision estimate in dimension $d$]
\label{lem:mixed-collision-sketch}
Let $d\geq3$, $0<s<1$, and $-d-2s<\gamma\leq-2s$, and take
\[
 b(\cos\theta)=c_{d,s}\bigl(\sin(\theta/2)\bigr)^{-(d-1)-2s},
 \qquad 0<\theta\leq\pi/2,\qquad c_{d,s}>0.
\]
Set $r_d=\lfloor d/2\rfloor+1$. There is a sufficiently large finite
$M_0=M_0(d,s,\gamma)$ such that the following holds for every integer
$K\geq4r_d$, $\vartheta\geq2s$, and $M>\vartheta K+M_0$.
Set $m_j=M-\vartheta j$, $w_j=\langle v\rangle^{m_j}$, and define
\begin{align*}
 \mathcal E_K^{\rm mix}(u)&:=\sum_{|\alpha|\leq K}
       \|w_{|\alpha|}\partial^\alpha u\|_2^2,\\
 \mathcal D_K^{\rm p}(u)&:=\sum_{|\alpha|\leq K}\int
       \mathfrak a_\mu[w_{|\alpha|}\partial^\alpha u]\,\dif x,\qquad
 \mathcal D_K^{\rm q}(u):=\sum_{|\alpha|\leq K}\int
       \mathcal A[w_{|\alpha|}\partial^\alpha u]\,\dif x,
 \qquad \partial^\alpha=\partial_{x,v}^\alpha.
\end{align*}
Write $\mathcal Q_s=\mathcal Q_{\rm near}+\mathcal Q_{\rm far}$ using a
fixed smooth radial cutoff, equal to one near $v-v_*=0$ and supported
in $|v-v_*|<1$. For smooth rapidly decreasing $p,q$ and
$\diamond\in\{\mathrm{near},\mathrm{far}\}$,
\begin{align}
 \mathcal T_\diamond
 &:=\sum_{|\alpha|\leq K}\Bigl(
 |\langle\partial^\alpha\mathcal Q_\diamond(\sqrt\mu q,p),
                  w_{|\alpha|}^2\partial^\alpha p\rangle|
 +|\langle\partial^\alpha\mathcal Q_\diamond(p,\sqrt\mu q),
                  w_{|\alpha|}^2\partial^\alpha p\rangle|\Bigr)\nonumber\\
 &\leq C\mathcal E_K^{\rm mix}(q)^{1/2}\mathcal D_K^{\rm p}(p)
 +C\mathcal E_K^{\rm mix}(p)^{1/2}
       \mathcal D_K^{\rm p}(p)^{1/2}\mathcal D_K^{\rm q}(q)^{1/2}.
 \label{eq:mixed-tame-hierarchy}
\end{align}
Here $C$ depends on $d,s,\gamma,K,M,\vartheta,c_{d,s}$ and the fixed cutoff.
The same bound holds for the full mixed collision terms.
\end{lemma}
\begin{proof}
For velocity functions, write
\[
 L_m(u)=\|\langle v\rangle^m u\|_2,\qquad
 A_m^{\rm p}(u)=\mathfrak a_\mu[\langle v\rangle^m u]^{1/2},\qquad
 A_m^{\rm q}(u)=\mathcal A[\langle v\rangle^m u]^{1/2},
\]
and set $R_m(u)=\sum_{|\eta|\leq r_d}L_m(\partial_v^\eta u)$ and
$S_m(u)=\sum_{|\eta|\leq r_d}A_m^{\rm p}(\partial_v^\eta u)$.
The forms control $H^s_{m+\gamma/2}$, and
$A_a^\bullet(u)\lesssim A_b^\bullet(u)$ for $0\leq a\leq b$.
For the polynomial form, the comparison with the regularized kinetic
norm is particularly direct: if $U=\langle v\rangle^m u$, then
\[
 \|U\|_{L^2_{\gamma/2}}^2+
 \frac12\iiint\langle v-v_*\rangle^\gamma b(\cos\theta)\mu_*
                    |U'-U|^2\,\dif\sigma\,\dif v_*\,\dif v
 \leq \bigl(A_m^{\rm p}(u)\bigr)^2.
\]
This follows from the Carleman formula and
$\langle v-v_*\rangle^\gamma\leq|v-v_*|^\gamma$.
The Sobolev comparison and the weight monotonicity follow from the
same form and the radial-weight commutator; its two angular differences
are integrable since $s<1$.

Let $\psi_\zeta=\partial_v^\zeta\sqrt\mu$, $|\zeta|\leq K$, and let
$T_{\diamond,m}^\zeta(f,g,h)$ be the sum of the absolute velocity
pairings of $\mathcal Q_\diamond(\psi_\zeta f,g)$ and
$\mathcal Q_\diamond(g,\psi_\zeta f)$ with $\langle v\rangle^{2m}h$.
For a sufficiently large coefficient weight $\ell=\ell(d,s,\gamma)$,
the velocity estimates used below are
\begin{align}
 T_{\diamond,m}^\zeta(f,g,h)
 &\lesssim R_\ell(f)A_{m+2s}^{\rm p}(g)A_m^{\rm p}(h),
 \label{eq:mixed-velocity-low}\\
 T_{\diamond,m}^\zeta(f,g,h)
 &\lesssim R_{m+2s}(g)A_\ell^{\rm q}(f)A_m^{\rm p}(h)
 +S_{m+2s}(g)\bigl[L_\ell(f)A_m^{\rm p}(h)
                         +A_\ell^{\rm q}(f)L_m(h)\bigr],
 \label{eq:mixed-velocity-high}\\
 T_{\diamond,m}^{0}(f,g,g)
 &\lesssim R_\ell(f)\bigl(A_m^{\rm p}(g)\bigr)^2,
 \label{eq:mixed-velocity-diagonal}
\end{align}
uniformly for $m_K\leq m\leq M$, after increasing $M_0$ if necessary.

We briefly explain the dimension dependence of these bounds.
In the near-field argument of \cite[Proposition~2.1]{YZ}, replace
the three-dimensional Bernstein factor by $2^{d\nu/2}$ and use
\[
 b(\cos\theta)(\sin\theta)^{d-2}\asymp\theta^{-1-2s},\qquad
 \int_0^1 a^{d-1+\gamma+2s}\,\dif a=(d+\gamma+2s)^{-1}.
\]
The dyadic summation costs $d/2+\delta\leq r_d$ derivatives on
one coefficient and $s$ on each other factor. Local transfer of
polynomial weights and Gaussian multiplication give the first two
bounds, even without the $2s$ weight loss, and also the diagonal bound.
For the far field, repeat the annular and weight-commutator arguments
of \cite[Section~2.3]{He2018} and
\cite[Lemma~2.10 and Proposition~2.8]{CG24} with the relative-velocity
cutoff. The cutoff removes the kinetic singularity at zero, while
the angular cancellation is unchanged. The remaining coefficient
moments are finite for a sufficiently large $\ell$: in dimension $d$
the tail criterion is $\int\langle v\rangle^{-a}\dif v<\infty$ for
$a>d$. The regularized norm above is controlled by the present forms,
and each weighted $H_v^s$ norm of $\psi_\zeta f$ is
controlled by either $R_\ell(f)$ or $A_\ell^{\rm q}(f)$.
These observations give \eqref{eq:mixed-velocity-low} and
\eqref{eq:mixed-velocity-high}. In the diagonal case, pre/post-collisional
symmetrization retains the principal Dirichlet form; its absolute
value and the weight commutator are bounded by
$R_\ell(f)(A_m^{\rm p}(g))^2$, giving
\eqref{eq:mixed-velocity-diagonal} without an additional $2s$ weight.
Thus the cited three-dimensional proofs are used with the dimension
changes just described.

Expand the differentiated terms with
$f=\partial^\beta q$, $g=\partial^\eta p$, $h=\partial^\alpha p$ and
$|\beta|+|\eta|+|\zeta|=|\alpha|=:j$.
The embedding $H_x^{r_d}\hookrightarrow L_x^\infty$, together with
the velocity derivatives in $R$ and $S$, costs $2r_d$ derivatives.
If $|\beta|\leq\lfloor K/2\rfloor$, then $|\beta|+2r_d\leq K$,
so $R_\ell(f)$ is controlled in $L_x^\infty$ by
$\mathcal E_K^{\rm mix}(q)^{1/2}$.
When $|\eta|\leq j-1$, $m_{|\eta|}\geq m_j+2s$, and
\eqref{eq:mixed-velocity-low} gives the first term of
\eqref{eq:mixed-tame-hierarchy}. If $|\eta|=j$, then $\beta=\zeta=0$
and $g=h$; apply \eqref{eq:mixed-velocity-diagonal} instead.
If $|\beta|>\lfloor K/2\rfloor$, then
\[
 |\eta|+2r_d\leq j-1,\qquad m_{|\eta|+2r_d}\geq m_j+2s.
\]
Hence $R_{m_j+2s}(g)$ and $S_{m_j+2s}(g)$ are controlled in
$L_x^\infty$ by $\mathcal E_K^{\rm mix}(p)^{1/2}$ and
$\mathcal D_K^{\rm p}(p)^{1/2}$, respectively.
Applying \eqref{eq:mixed-velocity-high} with the other two factors in
$L_x^2$ gives the two products in \eqref{eq:mixed-tame-hierarchy},
including the term with $L_{m_j}(h)$.
Choose $M_0\geq\ell+|\gamma|/2$ so that every coefficient moment is
available at the lowest weight, and sum the finitely many terms.
\end{proof}
\par\endgroup

For smooth rapidly decreasing $p,q$, the nonlinear bound is
\begin{equation}\label{eq:anisotropic-core-target}
\begin{aligned}
 \sum_{|\nu|\leq B}\bigl(&|\langle\partial^\nu{F(p,q)},
                 \langle v\rangle^{2m_{|\nu|}(M,B)}\partial^\nu p\rangle|
 +|\langle\partial^\nu{\bigl(\mu^{-1/2}
\mathcal Q_s(\sqrt\mu q,\sqrt\mu q)\bigr)},
                 \langle v\rangle^{2m_{|\nu|}(M,B)}\partial^\nu q\rangle|\bigr)\\
 &\leq C\bigl(E_B(p)+E_B(q)\bigr)^{1/2}
              \bigl(D_B^\sharp(p)+D_B^{\mathcal A}(q)\bigr).
\end{aligned}
\end{equation}
{The $E^{1/2}D$ structure comes from expanding
derivatives, using Sobolev embedding for the lower-order factor, and
controlling the other two factors by the collision dissipations.
Cauchy--Schwarz and $2ab\leq a^2+b^2$ combine the mixed dissipative
products. The collision pairings and the norms on the right-hand side
must be estimated in compatible velocity-weight hierarchies.}

For detailed computations underlying the polynomially weighted
collision and commutator estimates, see
\cite[Proposition~2.8, Lemma~2.10, and Propositions~4.1--4.2]{CG24}.
The treatment of the strong kinetic singularity follows the methods of
\cite[Section~2.2, Theorem~2.1, and Corollary~2.1]{YZ},
while the frequency-dependent energy argument and higher-moment
estimates follow the framework of \cite[Sections~4--5]{BCD}.
To avoid repetition, we retain the weight allocation, the derivative
bookkeeping, and the estimates needed to close the coupled bootstrap,
and omit calculations that follow the same arguments.

With $K=B_{\rm p}$, rebase the mixed estimate at $M+\vartheta K$.
At spatial level $n$, use
\[
 m_{n,j}=M_n+\vartheta(K-j),\qquad
 \rho_j=\rho+2(K-j),\qquad h_{n,j}=m_{n,j}-\gamma\rho_j.
\]
These are the weights used in the polynomial bootstrap below.
Write $\beta_0=|\beta|$, $\eta_0=|\eta|$, $\zeta_0=|\zeta|$,
and $j=\beta_0+\eta_0+\zeta_0$ in the mixed expansion. The time factors are
\[
 \beta_0+\rho_j-\rho_{\eta_0}=-\beta_0-2\zeta_0\leq0,\qquad
 \beta_0+r_d+\rho_j-\rho_{\eta_0+2r_d}
 =-\beta_0-2\zeta_0+5r_d<0
\]
in the low- and high-order cases, respectively; the latter uses
$\beta_0>K/2$ and $K>10r_d$. Thus conversion of $Z$ to ordinary
velocity derivatives is compatible with the auxiliary Gaussian energy.
At the higher-moment level, interpolate between $m_{n,\ell}$ and
$h_{n,\ell}$ before applying the same comparisons.
Since $m_{n,j}\geq M_n$, the output velocity weights are retained.

{The bootstrap uses
\eqref{eq:anisotropic-core-target} and its commuted and weighted versions,
with the full polynomial dissipation in
\eqref{eq:full-dissipation-candidate}. The compactly supported linear
forcing is treated separately. For $\gamma<-d$, the cancellation term
$c f\Lambda^\alpha g$, $\alpha=-d-\gamma\in(0,2s)$, from
Lemma~\ref{lemcanc} is kept in the weak pairing. For
$H=f\langle v\rangle^{2m}h$, fractional integration by parts gives
\[
 \langle\Lambda_v^\alpha g,H\rangle
 =\frac{C_{d,\alpha}}2\iiint
 \frac{\delta_z^v g(x,v)\,\delta_z^v H(x,v)}{|z|^{d+\alpha}}
 \,\dif z\,\dif v\,\dif x.
\]
Splitting at $|z|=\delta\in(0,1)$ and applying Cauchy--Schwarz yields
\[
 |\langle\Lambda_v^\alpha g,H\rangle|
 \lesssim\delta^{2s-\alpha}
 \|g\|_{L_x^2\dot H_v^s}\|H\|_{L_x^2\dot H_v^s}
 +\delta^{-\alpha}\|g\|_2\|H\|_2.
\]
The factor $\delta^{2s-\alpha}$ expresses the lower velocity-derivative
order of the cancellation term.}

{We now turn to the Maxwellian component
and estimate the right-hand side of \eqref{firg2}.}
By \cite[Proposition~5.2]{BCD}, the linear error satisfies
		\begin{align*}
			\int_0^t \int_{\mathbb{R}^d} L_{J,M,B}(\hat f)(\tau,\xi)\,\dif \xi \dif \tau \leq \frac{\delta_{\mathrm{hyp}}}{10} \tilde {\mathcal{D}}[f](t)+J \sup_{\tau\in[0,t]} \tilde {\mathcal{E}}_{mom}[f](\tau).
		\end{align*}
		{For the bilinear term $G_1= \mu^{-\frac{1}{2}}{\mathcal{Q}_s}(\mu^\frac{1}{2} g_2,\mu^\frac{1}{2}g_2 )$,}
{the collision inputs are the strong-singularity
bounds of \cite[Theorem~2.1 and Corollary~2.1]{YZ}:}
		\begin{align}\label{Non1}
			&\int_0^t \int_{\mathbb{R}^d}\langle \lambda_\xi \tau\rangle^{2J}\mathcal{N}_{M,B}(\widehat {G_1},\widehat{ g_2})(\tau,\xi) \dif \xi \dif \tau  \lesssim \tilde {\mathcal{E}}^\frac{1}{2}[ g_2](t)\tilde {\mathcal{D}}^\frac{1}{2}[ g_2](t)(\tilde {\mathcal{D}}^\frac{1}{2}[ g_2](t)+\tilde {\mathcal{E}}_{LF}^\frac{1}{2}[ g_2](t)).
		\end{align}
		The analogous bounds for the remaining energies require the corresponding
extensions of \cite[Propositions~5.3--5.7]{BCD}.
{In the frequency-supremum estimates, pairings
with two frequency-integrable factors are bounded using
$L_\xi^2*L_\xi^2\subset L_\xi^\infty$ and Cauchy--Schwarz in time
at fixed $\xi$, before taking the supremum.}
\begingroup
For the compact forcing $G_2=\mu^{-1/2}\mathbf c_1\chi_Rg_1$, retain the
source-square term by writing, with $K=\delta_{\mathrm{hyp}}^{-1}$,
\[
\mathcal N_{M,B}(\widehat G,\widehat f)
 =\mathcal N^{\rm bil}_{M,B}(\widehat G,\widehat f)
  +c_0\mathbf1_{|\xi|<K}|\Omega\widehat G|^2,
\]
where $\mathcal N^{\rm bil}_{M,B}$ denotes the remaining terms.

{We first estimate the macroscopic moments
of the compact source.}
\begin{lemma}\label{lem:compact-source-moments}
Let $p=g_1$ be smooth. For every integer $0\leq k\leq\min\{B,J\}$,
\begin{equation}\label{eq:compact-source-moments}
 S_J(t):=\int_{|\xi|<K}\langle\lambda_\xi t\rangle^{2J}
 |\Omega\widehat G_2(t,\xi)|^2\,\dif\xi
 \leq C\langle t\rangle^{2(J-k)-2\rho}\mathcal E_N[p](t).
\end{equation}
If $\rho+k>J+\tfrac12$, then, uniformly in $T$,
\begin{equation}\label{eq:compact-source-moments-integrated}
 \int_0^T S_J(t)\,\dif t\leq C\mathcal E_N[p](T).
\end{equation}
The constants depend on the fixed parameters and cutoff, independently
of $p,t,T$.
\end{lemma}
\begin{proof}
Each component of $\Omega G_2$ is $m_a=\int a(v)p\,\dif v$, with
$a(v)=\mathbf c_1v^\alpha\chi_R(v)$, $|\alpha|\leq4$.
For $Z_\xi=\nabla_v+it\xi$, integration by parts gives
\[
 it\xi_j\widehat m_a
 =\int aZ_{\xi,j}\widehat p\,\dif v
  +\int(\partial_ja)\widehat p\,\dif v.
\]
Iteration and Cauchy--Schwarz imply
\[
 \langle t\xi\rangle^{2k}|\widehat m_a|^2
 \leq C\sum_{|\beta|\leq k}
 \|Z_\xi^\beta\widehat p\|_{L_v^2(B_{2R})}^2.
\]
On $|\xi|<K$ we have
$\langle\lambda_\xi t\rangle^J
 \leq C\langle t\rangle^{J-k}\langle t\xi\rangle^k$.
Plancherel and the $n=0$ term of $\mathcal E_N[p]$ give
\eqref{eq:compact-source-moments}. Its time integral is bounded by
\eqref{eq:compact-source-moments-integrated}, since $\mathcal E_N[p]$
is nondecreasing and $2(J-k)-2\rho<-1$.
\end{proof}

{We next bound the remaining compact-source pairings.}
\begin{lemma}\label{lem:compact-source-forcing}
Fix coefficients for which $E_{M,B}$ is equivalent to its positive diagonal
part. Set $a_s=2s/(1+2s)$, $L=B+\lceil a_sJ\rceil+1$, and
$P_L(t)=\mathcal E_N^{[L]}[p](t)$. Then
\begin{equation}\label{Non2}
 \int\langle\lambda_\xi t\rangle^{2J}
 \mathcal N^{\rm bil}_{M,B}(\widehat G_2,\widehat q)(t,\xi)\,\dif\xi
 \leq C\langle t\rangle^{J+1-\rho}
 P_L(t)^{1/2}\tilde{\mathcal E}[q](t)^{1/2}.
\end{equation}
Consequently, if $\rho>J+2$, then
\begin{equation}\label{eq:compact-source-full-integral}
 \int_0^T\!\int\langle\lambda_\xi t\rangle^{2J}
 \mathcal N_{M,B}(\widehat G_2,\widehat q)\,\dif\xi\,\dif t
 \leq C\bigl(P_L(T)^{1/2}\tilde{\mathcal E}[q](T)^{1/2}+P_L(T)\bigr).
\end{equation}
The same bounds hold with a frequency supremum in place of the frequency
integral when $N>d/2$, with the corresponding supremum energy of $q$.
\end{lemma}
\begin{proof}
Let $\mathscr H_{M,B}$ be the positive diagonal part of $E_{M,B}$.
Cauchy--Schwarz in each source pairing gives
$\mathcal N^{\rm bil}_{M,B}(\widehat G,\widehat q)
 \leq C\mathscr H_{M,B}(\widehat G)^{1/2}
          \mathscr H_{M,B}(\widehat q)^{1/2}$.
For $G_2=\mathbf c_1\mu^{-1/2}\chi_Rp$, use the Leibniz rule,
$\nabla_v=Z-t\nabla_x$, and
\[
 \langle\lambda_\xi t\rangle^J
 \lesssim\langle t\rangle^J\langle\xi\rangle^{a_sJ}.
\]
The cutoff coefficients and their derivatives are bounded on $B_{2R}$;
the projected terms have Gaussian velocity profiles. Plancherel therefore gives
\[
 \int\langle\lambda_\xi t\rangle^{2J}
 \mathscr H_{M,B}(\widehat G_2)\,\dif\xi
 \leq C\langle t\rangle^{2J+2-2\rho}P_L(t).
\]
Here $\lceil a_sJ\rceil$ derivatives control the spatial multiplier,
while the extra derivative and factor $\langle t\rangle$ control
$\nabla_vZ^\beta p=Z^{\beta+e_j}p-t\partial_{x_j}Z^\beta p$
componentwise. Cauchy--Schwarz in $\xi$ proves \eqref{Non2}.
For the frequency supremum, use $L_x^1L_v^2$ in place of Plancherel.
Cauchy--Schwarz in $x$ and $N>d/2$ reduce this norm to the $n=N$ energy:
the Leibniz expansion of $p=\langle x\rangle^{-N}(\langle x\rangle^Np)$
produces at most $\langle t\rangle^{|\beta|}$, canceled by the
$\langle t\rangle^{-|\beta|}$ factor in $\mathscr H_{M,B}^{1/2}$.
Finally integrate \eqref{Non2} and apply
Lemma~\ref{lem:compact-source-moments} with $k=0$ to obtain
\eqref{eq:compact-source-full-integral}. In the frequency-supremum case,
the same moment bound follows from the preceding $L_x^1L_v^2$ estimate.
\end{proof}

For the other Maxwellian-component energies, apply this lemma with
$(B,J)$ replaced by $(B,0)$, $(B',J')$, or $(B',0)$ and with their
respective velocity weights. It suffices to take
\begingroup
\begin{equation}\label{eq:global-compact-source-budget}
\begin{gathered}
 \max\{B+\lceil a_sJ\rceil+1,B'+\lceil a_sJ'\rceil+1\}
       \leq K:=B_{\rm p}\leq\sigma,\qquad K>10r_d,\\
 \rho>\max\{J,J'\}+K+2,\qquad
 \vartheta K+|\gamma|(\rho+2K)<M'.
\end{gathered}
\end{equation}
\endgroup
 Young's inequality bounds each integrated forcing term by
$\varepsilon$ times the corresponding Maxwellian energy plus
$C_\varepsilon\mathcal E_N^{[B_{\rm p}]}[p]$.
\par\endgroup
\begingroup
Write $p_n=\langle x\rangle^np$ and
$U_{n,\alpha,\beta}=Z^\beta\partial_x^\alpha p_n$, with
$j=|\alpha|+|\beta|$. Weighted $L^2$-norms in the auxiliary energies
below integrate over both $x$ and $v$. The polynomial bootstrap uses
\begin{align*}
 P_T^\sharp[p]:=\sum_{n\leq N,\,j\leq K}c_{n,j}\sum_{|\alpha|+|\beta|=j}
 \Bigl[&\sup_{0\leq r\leq T}\langle r\rangle^{2\rho_j}
                   \|U_{n,\alpha,\beta}(r)\|_{L^2_{m_{n,j}}}^2
       +\sup_{0\leq r\leq T}\|U_{n,\alpha,\beta}(r)\|_{L^2_{h_{n,j}}}^2\\
 &+\int_0^T\!\int\Bigl(
       \langle r\rangle^{2\rho_j}
          \mathfrak a_\mu[\langle v\rangle^{m_{n,j}}U_{n,\alpha,\beta}]
       +\mathfrak a_\mu[\langle v\rangle^{h_{n,j}}U_{n,\alpha,\beta}]
       \Bigr)\,\dif x\,\dif r\Bigr].
\end{align*}
The positive coefficients $c_{n,j}$ absorb the triangular commutators.
Keep the original frequency-dependent energies of $q$ and add a
same-order Sobolev energy, following the method of
\cite[Section~4.1 and Proposition~4.1]{YZ}:
\begin{align*}
 \ell_\beta&:=M+(\vartheta+|\gamma|)(K-|\beta|),\\
 \mathscr E_K^q[q]&:=\sum_{|\alpha|+|\beta|\leq K}
      \|\partial_x^\alpha\partial_v^\beta q\|_{L^2_{\ell_\beta}}^2,\\
 \mathscr D_K^q[q]&:=\sum_{|\alpha|+|\beta|\leq K}\int
      \mathcal A[\langle v\rangle^{\ell_\beta}
      \partial_x^\alpha\partial_v^\beta(\mathrm{Id}-\mathbf P)q]\,\dif x
      +\|\mathbf Pq\|_{H_x^KL_v^2}^2.
\end{align*}
The macroscopic zeroth-order time integral uses the low-frequency
bounds; higher orders use the moment interaction.
For $G_2=a(v)p$, $a=\mathbf c_1\mu^{-1/2}\chi_R\in C_c^\infty$,
\[
 \mathscr E_K^q[G_2](r)^{1/2}
 \lesssim\langle r\rangle^{K-\rho}P_T^\sharp[p]^{1/2}.
\]
Thus the auxiliary source uses $K$ derivatives of $p$ and is integrable
when $\rho>J+K+2$.
Define
\begin{equation}\label{eq:global-bootstrap-functionals}
\begin{aligned}
 Q_T[q]&:=\bigl(\tilde{\mathcal E}+\tilde{\mathcal E}_{mom}
       +\tilde{\mathcal E}_{LF}+\tilde{\mathcal E}_{LF,mom}\bigr)[q](T)\\
 &\quad+\bigl(\tilde{\mathcal D}+\tilde{\mathcal D}_{mom}
       +\tilde{\mathcal D}_{LF}+\tilde{\mathcal D}_{LF,mom}\bigr)[q](T),\\
 R_T[q]&:=\sup_{0\leq r\leq T}\mathscr E_K^q[q](r)
       +\int_0^T\mathscr D_K^q[q](r)\,\dif r,\\
 X_T&:=P_T^\sharp[p]+Q_T[q]+R_T[q].
\end{aligned}
\end{equation}
The compact-source contributions to $Q_T$ satisfy
\[
 \mathcal S_{G_2}(T)\leq\eta Q_T[q]+C_\eta P_T^\sharp[p],\qquad\eta>0.
\]
From \eqref{eq:global-full-form-input},
\eqref{energ1}, and the same-order Gaussian estimate,
\begin{equation}\label{eq:global-bootstrap-target}
 P_T^\sharp[p]\leq C_0\mathbf{Data}^2+C_1X_T^{3/2},\qquad
 Q_T[q]+R_T[q]\leq C_2P_T^\sharp[p]+C_3X_T^{3/2}.
\end{equation}
Under these bounds, continuity of the running norms gives
\begin{equation}\label{re11}
 P_T^\sharp[p]+Q_T[q]+R_T[q]\leq C\mathbf{Data}^2.
\end{equation}
\par\endgroup
\begingroup
We next estimate the spatial weight. Choose $\chi\in C_c^\infty([0,\infty))$
with $0\leq\chi\leq1$, $\chi=1$ on $[0,1/10]$ and $\chi=0$ on
$[10,\infty)$. Set $\phi_0=1$, $\mathbf g_0=g_2$, and, for
$1\leq k\leq\lceil N/2\rceil$,
\[
 \phi_k=\langle x\rangle^k(1-\chi(|x|/\langle t\rangle)),\qquad
 \mathbf g_k=\phi_kg_2.
\]
Then
\[
 (\partial_t+v\cdot\nabla_x+\tilde{\mathcal L}_s)\mathbf g_k=G_k+R_k,
 \qquad G_k=\phi_kG(g_1,g_2),\quad
 R_k=g_2(\partial_t+v\cdot\nabla_x)\phi_k.
\]
To retain the output velocity weight, start the spatial hierarchy at
$M+\lceil N/2\rceil$, which is covered by the moment energy since
$\lceil N/2\rceil\leq M'$. Define
\begin{align*}
 \mathbf E_k(t)&:=\int E_{M+\lceil N/2\rceil-k,B}
       (\widehat{\mathbf g}_k)(t,\xi)\,\dif\xi,
 &\mathbf E_k^*(t)&:=\sup_{0\leq r\leq t}\mathbf E_k(r),\\
 \mathbf D_k(t)&:=\int_0^t\!\int D_{M+\lceil N/2\rceil-k,B}
       (\widehat{\mathbf g}_k)(r,\xi)\,\dif\xi\,\dif r.
\end{align*}
The collision bound needed for this hierarchy is
\begin{align*}
 \int_0^t\!\int\mathcal N_{M+\lceil N/2\rceil-k,B}
       (\widehat G_k,\widehat{\mathbf g}_k)\,\dif\xi\,\dif\tau
 \lesssim {}&\mathbf{Data}\,\mathbf D_k
       +\mathbf{Data}(\mathbf E_k^*)^{1/2}\mathbf D_k^{1/2}
       +\mathbf{Data}(\mathbf E_k^*)^{1/2}+\mathbf{Data}^2.
\end{align*}
The last two terms follow for $\phi_kG_2$ from
Lemma~\ref{lem:compact-source-forcing}, since $k\leq N$ and
$|x|\gtrsim\langle t\rangle$ on the support of $\phi_k$.

For the transport remainder,
\[
 |(\partial_t+v\cdot\nabla_x)\phi_k|
 \lesssim_k\langle v\rangle\bigl(\phi_{k-1}
 +\langle x\rangle^{k-1}
   \mathbf1_{\langle t\rangle/10\leq|x|\leq10\langle t\rangle}\bigr).
\]
The Leibniz rule, including the factors $\langle t\rangle^{-|\beta|}Z^\beta$
in the energy, therefore gives
\[
 \int E_{M+\lceil N/2\rceil-k,B}(\widehat R_k)\,\dif\xi
 \lesssim\mathbf E_{k-1}(t)+\langle t\rangle^{2k-2}\mathbf{Data}^2.
\]
Cauchy--Schwarz in the bilinear pairings and the Gaussian moment bound give
\[
\begin{aligned}
 \int\mathcal N_{M+\lceil N/2\rceil-k,B}
       (\widehat R_k,\widehat{\mathbf g}_k)\,\dif\xi
 \lesssim {}&\bigl(\mathbf E_{k-1}^{1/2}
       +\langle t\rangle^{k-1}\mathbf{Data}\bigr)\mathbf E_k^{1/2}\\
 &+\mathbf E_{k-1}+\langle t\rangle^{2k-2}\mathbf{Data}^2.
\end{aligned}
\]
Thus, after absorbing the small collision contribution,
Lemma~\ref{lemlines1} gives
\begin{align*}
 \mathbf E_k^*(t)+c\mathbf D_k(t)
 \lesssim {}&\mathbf{Data}^2+\int_0^t
 \Bigl[\bigl(\mathbf E_{k-1}^{1/2}(\tau)
       +\langle\tau\rangle^{k-1}\mathbf{Data}\bigr)
       \mathbf E_k^{1/2}(\tau)\\
 &\hspace{65pt}+\mathbf E_{k-1}(\tau)
       +\langle\tau\rangle^{2k-2}\mathbf{Data}^2\Bigr]\,\dif\tau.
\end{align*}
Since $\mathbf E_0^*\lesssim\mathbf{Data}^2$ and $\mathbf g_k(0)=0$,
induction and Young's inequality give
$\mathbf E_k^*(t)\lesssim\langle t\rangle^{2k}\mathbf{Data}^2$.
Inside $|x|\leq10\langle t\rangle$ the normalized spatial weight is
bounded; outside, $\phi_k=\langle x\rangle^k$. Hence
\[
 \sup_{t>0}\|\langle v\rangle^M
       \langle x/\langle t\rangle\rangle^{N/2}g_2(t)\|_2
 \lesssim\mathbf{Data}.
\]
\par\endgroup
\par\medskip
		\begin{proof}[Reduction of Corollary \ref{thmglo} to Proposition \ref{propglo}]
			{
			Let $\delta_{\mathrm{loc}}$ be the smallness threshold in Theorem \ref{Th1}. If $\mu+f_0\geq0$ and $\|f_0\|_{\mathrm{cr}}\leq\delta_{\mathrm{loc}}$, then Theorem \ref{Th1} yields a unique solution of \eqref{eqperbo} on $[0,T_0]$.
{Set $t_*:=T_0/2$ and restart the global argument from
$f(t_*)$. Local uniqueness in the common solution class identifies the
solutions on $[t_*,T_0]$; use local estimates through $T_0$ and
global estimates thereafter.}
For $h$ defined on $[0,\infty)$, set
			}
			\[
				\|h\|_{\mathcal X}
				:=\|h\|_{T_0}+\|h({t_*}+\cdot)\|_{\mathcal G},
			\]
\begingroup
Fix structural parameters for which the quadratic energies are positive.
Assume \eqref{eq:global-bootstrap-target}. For the solution class,
let $P_T^0$ be the fixed-weight version of $P_T^\sharp$ with
$\rho_j=J+3$, $m_{n,j}=M_n$, and $h_{n,j}=M_n-\gamma(J+3)$.
Then $P_T^0\lesssim P_T^\sharp$; the strengthened hierarchy and $R_T$
are auxiliary construction norms. Define
\begin{equation}\label{eq:global-class-definition}
\begin{aligned}
 \|h\|_{\mathcal G}:={}&\inf_{h=p+\sqrt\mu q}\sup_{T>0}
       \bigl(P_T^0[p]+Q_T[q]\bigr)^{1/2}\\
 &+\sup_{t>0}\Bigl(
 \bigl\|\langle v\rangle^M\langle x/\langle t\rangle\rangle^{N/2}
 \langle\lambda(\nabla_x)t\rangle^{J_{\rm out}}
 \langle\nabla_{x,v}\rangle^{\sigma_{\rm out}}h(t)\bigr\|_{L^2_{x,v}}\\
 &\hspace{35pt}+\langle t\rangle^{d/2}
 \bigl\|\langle v\rangle^M\langle\lambda(\nabla_x)t\rangle^{J_{\rm out}}
 \langle\nabla_{x,v}\rangle^{\sigma_{\rm out}}h(t)
 \bigr\|_{L_v^2L_x^\infty}\Bigr),\\
 \mathcal G:={}&\{h:\|h\|_{\mathcal G}<\infty\}.
\end{aligned}
\end{equation}
\par\endgroup
The infimum is over decompositions with finite displayed quantities, not
only solutions of the decomposed equations. Thus this definition specifies a
class independently of a particular construction.

{The weight assumption also permits applying
Theorem~\ref{Th1} with derivative index $10N_0$, giving}
			\begin{align*}
				\sup_{0<t\leq T_0} \sum_{m+n\leq 10N_0} t^{n+\frac{d-\kappa+m+n}{2s}}\langle x\rangle^{\varkappa_1}\langle v\rangle^{\frac{5\varkappa_2}{6}} |\langle\nabla_v\rangle^m\langle\nabla_x\rangle^nf(t,x,v)| \leq C\|f_0\|_{\mathrm{cr}}.
			\end{align*}
			{Interpolating the preceding estimate yields}
			\begin{align}\label{shor}
				\sup_{0<t\leq T_0} \sum_{m+n\leq N_0} t^{n+\frac{d-\kappa+m}{2s}}\langle x\rangle^{\varkappa_1}\langle v\rangle^{\frac{5\varkappa_2}{6}} |\langle\nabla_v\rangle^m\langle\lambda(\nabla_x)\rangle^nf(t,x,v)|\lesssim\|f_0\|_{\mathrm{cr}}.
			\end{align}
			{At $t=t_*=T_0/2$, the local estimate gives}
			\begin{align*}
				\|\langle x\rangle^{\frac{4\varkappa_1}{5}}\langle v\rangle^{\frac{4\varkappa_2}{5}} \langle\nabla_{x,v}\rangle^{10N_0}f({t_*})\|_{L^2_{x,v}}+\|\langle x\rangle^{\frac{4\varkappa_1}{5}}\langle v\rangle^{\frac{4\varkappa_2}{5}} \langle\nabla_{x,v}\rangle^{10N_0}f({t_*})\|_{L^2_{v}L^1_x}\leq C{t_*^{-\frac{100(N_0+d)}{s}}}\|f_0\|_{\mathrm{cr}}.
			\end{align*}
			{
			Choose the global smallness threshold $\varepsilon_{\mathrm{data}}$ in Corollary \ref{thmglo} so that $\varepsilon_{\mathrm{data}}\leq\delta_{\mathrm{loc}}$ and
			$C\varepsilon_{\mathrm{data}}{t_*^{-\frac{100(N_0+d)}{s}}}
			\leq\varepsilon_{\mathrm{glo}}$, where
			$\varepsilon_{\mathrm{glo}}$ is the constant $\varepsilon$ from Proposition \ref{propglo}. For the restarted datum $f({t_*})$, set
			}
			\[
			\sigma=10N_0,\qquad
			N=\left\lfloor\frac{3\varkappa_1}{4}\right\rfloor,\qquad
			M=\frac{\varkappa_2}{2},\qquad
			M'=\frac{\varkappa_2}{4}.
			\]
			{Then $\sigma,N\in\mathbb N$, $N,M>d$, and the main weight assumption implies $M'>2s^{-1}(N+\sigma)$. Indeed, if $A=N_0+\varkappa_1$, then $N+\sigma\leq10A$, whereas $M'>(5/2)s^{-1}A^2>20s^{-1}A$. Moreover, $N\leq4\varkappa_1/5$ and $M+M'=3\varkappa_2/4<4\varkappa_2/5$, so the weights in \eqref{condata} are controlled by those in the preceding estimate at time ${t_*}$. {For the source and structural moment conditions, take
$0<s\leq1$, choose $0<\kappa_0\leq1/2$, and set}}
\begingroup
\begin{equation}\label{eq:global-parameter-allocation}
\begin{aligned}
 a_s&=\frac{2s}{1+2s},\quad
 L_0=\left\lceil\frac{d+1}{a_s}\right\rceil,
 &J'&=N_0+L_0,\quad J=J'+2,\\
 B&=4N_0,\quad B'=2N_0,
 &B_{\rm p}&=B+\lceil a_sJ\rceil+1,\\
 \rho&=J+B_{\rm p}+3,\quad k=0,
 &J_{\rm out}&=N_0,\quad\sigma_{\rm out}=2N_0,\\
 R_{\rm mom}&=\left\lceil2H_*(J+1)\right\rceil+2,
 &M_J&=M',\quad M_{J'}=R_{\rm mom}.
\end{aligned}
\end{equation}
Here $R_{\rm mom}$ is a velocity-moment reserve, distinct from the cutoff
radius $R$. These choices give
\[
 d<J'<J-1,\qquad B>B'+d,\qquad
 a_s(J'-J_{\rm out})\geq d+1>d,\qquad
 B_{\rm p}\leq\tfrac{14}{3}N_0+d+5<5N_0<\sigma.
\]
Indeed, $N_0\geq10d$ and $d\geq3$ imply
$N_0>3d+15$ and $N_0>d/2$, which give respectively
$\frac{14}{3}N_0+d+5<5N_0$ and $4N_0>2N_0+d$.
They also satisfy the output derivative budget in Proposition~\ref{propglo}.
Writing $A=N_0+\varkappa_1$, we have
$\varpi\leq(7/2)|\gamma|$, $H_*\leq4A$ and $J+1\leq A/s$.
Consequently,
\[
 R_{\rm mom}\leq8s^{-1}A^2+3,\qquad
 4(R_{\rm mom}+1)<40s^{-1}A^2<\varkappa_2.
\]
With $K=B_{\rm p}\leq B+J+2$, $B\geq40d$, and $J\geq10d$,
the enlarged derivative and time hierarchy also fits the moment reserve:
\[
 \vartheta K+|\gamma|(\rho+2K)<2H_*(J+1)<R_{\rm mom}.
\]
Thus \eqref{eq:global-structural-budget} holds; in particular,
\[
 M_J,M_{J'}>2H_*(J+1),\qquad
 M'+M_{J'}<M-1,\qquad
 \vartheta K+|\gamma|(\rho+2K)<R_{\rm mom}<M'.
\]
{The compact-source conditions hold with $k=0$;
also $N|\gamma-1|<M$ and $\lceil N/2\rceil<M'$.} The largest polynomial weight satisfies
$h_{0,0}<M+M'=3\varkappa_2/4<4\varkappa_2/5$.
The linear moment condition \eqref{eq:finite-linear-budget} follows for
both $J$ and $J'$, including at $s=1$.
\par\endgroup
				{For $d\geq2$ and $\varkappa_1\geq10d$, this choice also gives $N/2\geq\varkappa_1/(2d)$. Hence the spatial weight supplied by Proposition \ref{propglo} may be weakened to the one used below.}
			{
			With these choices, \eqref{re11} and the subsequent spatially weighted estimates yield
			}
			\[
				\|f({t_*}+\cdot)\|_{\mathcal G}
				\lesssim\|f_0\|_{\mathrm{cr}}.
			\]
			{
			{Applying Proposition~\ref{propglo} to $f(t_*)$ and restricting to original times $t\geq T_0$ gives}
			}
			\begin{align*}
				& \sup _{{t\geq T_0}}\left(\left\|\langle x/\langle t\rangle\rangle^{\frac{\varkappa_1}{2d}}\langle v\rangle^{\frac{\varkappa_2}{2}}\left\langle\lambda\left( \nabla_x\right) \langle t\rangle \right\rangle^{N_0}\langle \nabla_{x,v}\rangle^{2N_0}f(t)\right\|_{L_{x, v}^2} \right.\\
				&  \quad\quad\quad\quad\quad\quad\quad\quad\quad\quad\quad\quad\quad\quad\left.+ \langle t\rangle^{d / 2}	\left\|\langle v\rangle^{\frac{\varkappa_2}{2}}\left\langle\lambda\left( \nabla_x\right) \langle t\rangle\right\rangle^{N_0}\langle \nabla_{x,v}\rangle^{2N_0}f(t)\right\|_{L_v^2 L_x^{\infty}} \right)\lesssim \|f_0\|_{\mathrm{cr}}.
			\end{align*}
\begingroup
To return to the original time variable, set $\tau=t-t_*$.
For $t\geq T_0=2t_*$,
\[
\frac{t}{2}\leq\tau\leq t,\qquad \tau\geq T_0/2,
\qquad \langle t\rangle\asymp\langle\tau\rangle.
\]
Consequently,
\[
\big\langle\lambda(\xi)\langle t\rangle\big\rangle^{N_0}
\lesssim_{T_0,N_0}
\big\langle\lambda(\xi)\tau\big\rangle^{N_0}.
\]
The spatial weights and decay factors are likewise comparable,
giving the displayed estimate in the original time variable.
\par\endgroup
{Sobolev embedding now gives}
			\begin{align*}
				\left\|\langle x/\langle t\rangle\rangle^\frac{\varkappa_1}{2d}\langle v\rangle^\frac{\varkappa_2}{2}\left\langle\lambda\left( \nabla_x\right) \langle t\rangle\right\rangle^{N_0}\langle \nabla_v\rangle^{N_0}f(t)\right\|_{L_{x, v}^\infty}
				&\lesssim \left\|\langle x/\langle t\rangle\rangle^\frac{\varkappa_1}{2d}\langle v\rangle^\frac{\varkappa_2}{2}\left\langle\lambda\left( \nabla_x\right) \langle t\rangle\right\rangle^{N_0}\langle \nabla_{x,v}\rangle^{2N_0}f(t)\right\|_{L^2_{x,v}}\\
				&\lesssim \|f_0\|_{\mathrm{cr}},
			\end{align*}
			and 
			\begin{align*}
				\left\|\langle v\rangle^\frac{\varkappa_2}{2}\left\langle\lambda\left( \nabla_x\right) \langle t\rangle\right\rangle^{N_0}\langle \nabla_v\rangle^{N_0}f(t)\right\|_{L_{x, v}^\infty}
				&\lesssim \left\|\langle v\rangle^\frac{\varkappa_2}{2}\left\langle\lambda\left( \nabla_x\right) \langle t\rangle\right\rangle^{N_0}\langle \nabla_{x,v}\rangle^{2N_0}f(t)\right\|_{L^2_{v}L^\infty_x}\\
				&\lesssim \langle t\rangle^{-\frac{d}{2}}\|f_0\|_{\mathrm{cr}}.
			\end{align*}
			{
			{Combining the local estimate on $[0,T_0]$ with the restarted global bound yields}
			}
			\[
				\|f\|_{\mathcal X}\lesssim\|f_0\|_{\mathrm{cr}}.
			\]
			{
			Together with \eqref{shor}, the preceding Sobolev estimates bound the left-hand side of \eqref{main} by $C\|f\|_{\mathcal X}$. 

			}
\begingroup
Finally, let $h$ have the same initial trace and
$\|h\|_{\mathcal X}\leq2C\|f_0\|_{\mathrm{cr}}$.
The local difference estimate gives $h=f$ on $[0,T_0]$.
For later times, the output part of $\mathcal G$, with
$\sigma_{\rm out}=2N_0$, and weighted Sobolev embedding give
\[
 \sup_{t\geq T_0}\sum_{|\alpha|\leq N_0+2}
 \Bigl(\|\langle v\rangle^{\varkappa_2/4}
       \partial_{x,v}^\alpha h(t)\|_\infty
       +\|\langle v\rangle^{\varkappa_2/4}
       \partial_{x,v}^\alpha f(t)\|_\infty\Bigr)
 \lesssim\|h\|_{\mathcal X}+\|f\|_{\mathcal X}.
\]
On each interval of length at most one, this controls the local norm
with spatial exponent zero and velocity exponent $\varkappa_2/4$;
the normalized derivative factors are bounded by one, the fractional
differences by the next derivative, and the kernel in \eqref{defI12}
has uniformly bounded mass. This velocity exponent satisfies
the hypothesis of Theorem~\ref{Th1}.
Choose a fixed sufficiently short interval and decrease $\varepsilon_{\mathrm{data}}$
if necessary. The zero-trace difference estimate \eqref{esmap2}
then absorbs its right-hand side. Positive-time traces follow from
the weak equation. Restarting at successive endpoints
gives $h=f$ for all times. Applying Theorem~\ref{Th1} at the same
restart times also propagates $\mu+f\geq0$.
\par\endgroup
		\end{proof}

			\section{Short-time Green functions for prescribed backgrounds}		\label{secGreen}

The coefficient $\mathbf C_g(v-\alpha,\alpha)$ need not be comparable to
$\mathbf C_g(v,\alpha)$ for large $\alpha$. We therefore freeze only the
small-jump operator along kinetic characteristics. Its Green function becomes
the reference kernel for a second Duhamel expansion containing the large
jumps and the cancellation remainder. This follows the parametrix principle
of \cite{Friedman1964}, with kinetic transport and velocity-dependent scales.

The freezing error is controlled by the H\"older regularity of $g$.
In the cancellation term, a weak integration-by-parts identity puts a
difference on the coefficient and another on the product of kernels.
The inequality $2b_0>2s$ makes the jump integral locally integrable.
The estimates below track both the time gain and the loss of velocity weights
at each step.

\subsection{Green function for small jumps}
{We begin with the small-jump operator}
\begin{align}\label{eq:def-small-jump-operator}
	\mathcal L_g^{<}f(t,x,v)
	:=\operatorname{p.v.}\int_{|\alpha|<r_0}
	\mathbf C_g(t,x,v,\alpha)
	\bigl(f(t,x,v)-f(t,x,v-\alpha)\bigr)
	\frac{\dif\alpha}{|\alpha|^{d+2s}},
\end{align}
{where the parameter} $r_0\in(0,1)$ {will be fixed below.}

{Let} $\G^{(0)}(t,x,v;\tau,y,w)$ {denote the Green
function associated with} \eqref{eq:def-small-jump-operator}:
\begin{align}\label{eq:Green-small}
	\left\{
	\begin{aligned}
		&\partial_t\G^{(0)}+v\cdot\nabla_x\G^{(0)}
		+\mathcal L_g^{<}\G^{(0)}=0,
		&&t>\tau,\\
		&\G^{(0)}(\tau,x,v;\tau,y,w)
		=\delta(x-y)\delta(v-w).
	\end{aligned}
	\right.
\end{align}

	{Fix a terminal point}
			\(\mathfrak q=(t_0,x_0,v_0)\), {where} \(\tau<t_0\le T\),
			{and set}
			\[
			Z_{\mathfrak q}(r)
			:=
			\bigl(x_0-(t_0-r)v_0,v_0\bigr),
			\qquad \tau\le r\le t_0.
			\]
{Define the frozen small-jump operator by}
\begin{align}\label{eq:def-frozen-small-operator}
	\mathcal L_g^{<,\mathfrak q}(r)f(u)
	:=\operatorname{p.v.}\int_{|\alpha|<r_0}
	\mathbf C_g(r,Z_{\mathfrak q}(r),\alpha)
	\bigl(f(u)-f(u-\alpha)\bigr)
	\frac{\dif\alpha}{|\alpha|^{d+2s}}.
\end{align}
	 	{The corresponding kernel} \(\mathcal G_{\mathfrak q}^{(0)}\)
	 	{is defined by}
			\[
			\left\{
			\begin{aligned}
				&\partial_t\mathcal G_{\mathfrak q}^{(0)}
				+v\cdot\nabla_x\mathcal G_{\mathfrak q}^{(0)}
				+\mathcal L_g^{<,\mathfrak q}(t)
				\mathcal G_{\mathfrak q}^{(0)}=0,
				&&t>r,\\
				&\mathcal G_{\mathfrak q}^{(0)}
				(r,x,v;r,\zeta,u)
				=\delta(x-\zeta)\delta(v-u).
			\end{aligned}
			\right.
			\]

{To absorb the frozen-coefficient error while preserving the kinetic
geometry, we introduce a transport-adapted cutoff. For} $c_0\in(0,1)$,
{define}
	\begin{align}
		\chi_{\mathfrak q}(t,x,v)
		:=
		\chi\left(
		\frac{x_0-(t_0-t)v-x}{c_0},
		\frac{v_0-v}{c_0}
		\right),
		\label{eq:transport-localization-cutoff}
	\end{align}
	    {where} $\chi\in C_c^\infty(\mathbb{R}^{2d})$ {satisfies}
	    $ \mathbf{1}_{|z|<1}\leq \chi(z)\leq \mathbf{1}_{|z|<2}$.
	{Then}
	\begin{align*}
		(\partial_t+v\cdot\nabla_x)
		\chi_{\mathfrak q}=0.
	\end{align*}
	    {Equation \eqref{eq:Green-small} can now be written as}
    \begin{align*}
        &{\partial_t(\G^{(0)}\chi_{\mathfrak{q}})}+v\cdot\nabla_x(\G^{(0)}\chi_{\mathfrak{q}})
		+\mathcal L_g^{<,\mathfrak{q}}(\G^{(0)}\chi_{\mathfrak{q}})=\chi_{\mathfrak q}
		(\mathcal L_g^{<,\mathfrak q}-\mathcal L_g^{<})\G^{(0)}
		+\mathcal C_{\mathfrak q}[\G^{(0)}],
		\quad\quad t>\tau,\\
		&\G^{(0)}(\tau,x,v;\tau,y,w) \chi_{\mathfrak{q}}
		=\delta(x-y)\delta(v-w) \chi_{\mathfrak{q}}(\tau,y,w),
    \end{align*}
	    {where}
	    $\mathcal C_{\mathfrak q}[\G^{(0)}]=\mathcal L_g^{<,\mathfrak{q}}(\G^{(0)}\chi_{\mathfrak{q}})-(\mathcal L_g^{<,\mathfrak{q}}\G^{(0)})\chi_{\mathfrak{q}}$
	    {is the commutator. This identity yields the Duhamel formula}
\begin{align}\label{eq:global-small-duhamel}
	\G^{(0)}\chi_{\mathfrak{q}}(t,x,v)
	=\mathcal G_{\mathfrak q}^{(0)}\chi_{\mathfrak{q}}(\tau,y,w)
	+\mathcal G_{\mathfrak q}^{(0)}
	\oast\left\{ \chi_{\mathfrak{q}}
	\bigl(\mathcal L_g^{<,\mathfrak q}
	-\mathcal L_g^{<}\bigr)\G^{(0)}+\mathcal C_{\mathfrak q}[\G^{(0)}]\right\}.
\end{align}
{The free frozen kernel is multiplied by the initial factor}
$\chi_{\mathfrak q}(\tau,y,w)$, {as required by the initial
condition, whereas the factor on the left is evaluated at the terminal point}
$(t,x,v)$. {The operator difference is given explicitly by}
\begin{align}\label{eq:small-coefficient-error}
	&\bigl((\mathcal L_g^{<,\mathfrak q}
	-\mathcal L_g^{<})F\bigr)
	(r,\zeta,u;\tau,y,w)\notag\\
	&\quad=
	\operatorname{p.v.}\int_{|\alpha|<r_0}
	\Bigl(
	\mathbf C_g(r,Z_{\mathfrak q}(r),\alpha)
	-\mathbf C_g(r,\zeta,u,\alpha)
	\Bigr)
	\bigl(\delta_{\alpha}^{(u)}F\bigr)
	(r,\zeta,u;\tau,y,w)
	\frac{\dif\alpha}{|\alpha|^{d+2s}}.
\end{align}

\subsubsection{\texorpdfstring{{Estimates for the frozen kernel}}{Estimates for the frozen kernel}}
For $t>0$ and $(X,V)\in\mathbb R^{2d}$, set
\begin{align*}
	R_t(X,V):=\frac{|X|}{t}+|V|,
	\qquad
	\mathfrak m_N(t,X,V):=\langle R_t(X,V)\rangle^{-N}.
\end{align*}
{Set}
\begin{align}\label{eq:def-weighted-small-envelope}
	\bT_{v_0}(t,X,V)
	&:=
	t^{-d}\tilde t_{v_0}^{-\frac ds}
	\langle v_0\rangle^2
	\mathcal N_{v_0}\left(
	\frac{X}{t\tilde t_{v_0}^{1/(2s)}},
	\frac{V}{\tilde t_{v_0}^{1/(2s)}}
	\right),\\
	\bK_{v_0,N}(t,X,V)
	&:=
	\bT_{v_0}(t,X,V)\mathfrak m_N(t,X,V).
\end{align}
{The additional factor} $\mathfrak m_N$ {is comparable
to one in each fixed one-jump region and contributes rapid decay only when}
$R_t(X,V)$ {is genuinely large.}

{Let $p_{\mathfrak q}^{(0)}$ denote the
physical density of the frozen small-jump process associated with
$\mathcal L_g^{<,\mathfrak q}$:}
\begin{align}
		{p_{\mathfrak q}^{(0)}}(t,\tau,x,v)
		=(2\pi)^{-2d}
		\iint_{\mathbb R^{2d}}
		&e^{-\mathrm i(x\cdot\xi+v\cdot\eta)}
		\exp\left[
		-\int_\tau^t
		\psi_{r,\mathfrak q}^{(0)}
		\bigl(\eta+(t-r)\xi\bigr)
		\dif r
		\right]\dif\xi\dif\eta.
		\label{eq:frozen-small-fourier}
	\end{align}

	    {Here the frozen symbol} $\psi_{r,\mathfrak q}^{(0)}(z)$
	    {is}
\begin{align*}
		\psi_{r,\mathfrak q}^{(0)}(z)
		:=
		\int_{|\alpha|<r_0}
		\bigl(1-\cos(z\cdot\alpha)\bigr)
		\nu_{r,\mathfrak q}(\dif\alpha), \quad\quad\quad\tau\leq r\leq t.
	\end{align*}
	    {This symbol is associated with the measure}
   \begin{align*}
		\nu_{r,\mathfrak q}(\dif\alpha)
		:=
		\mathbf 1_{\{|\alpha|<r_0\}}
		\mathbf C_g\bigl(r,Z_{\mathfrak q}(r),\alpha\bigr)
		\frac{\dif\alpha}{|\alpha|^{d+2s}}.
	\end{align*}
{For every admissible background} $g$, {the Carleman
coefficient is even in the jump variable:}
$\mathbf C_g(r,z,\alpha)=\mathbf C_g(r,z,-\alpha)$. {This is the
symmetry used in \eqref{frozen-cancellation}. Consequently,}
$\nu_{r,\mathfrak q}$ {is even, the sine part of the Fourier symbol
vanishes, and the symbol is exactly the real cosine expression above.
Moreover,}
    \begin{align}\label{equivGH}
    \mathcal{G}_{\mathfrak{q}}^{(0)}(t,x,v;\tau,y,w)=	{p_{\mathfrak q}^{(0)}}(t,\tau,x-y-(t-\tau)w,v-w).
    \end{align}
{The following lemma gives a weighted bound
for the frozen small-jump kernel.}
\begin{lemma}
	\label{lem:frozen-small-jump-weighted}
{Let} $m\in{\mathbb N}$, $N\geq0$, {and}
$0<\tau<t<T$, {with} $\sigma:=t-\tau\leq1$. {Then}
	\begin{align}
		\left|
		\mathcal D_{\sigma,v_0}^{(x,v),m}
		{p_{\mathfrak q}^{(0)}}(t,\tau,x,v)
		\right|
		\lesssim_{m,N}
		\bK_{v_0,N}(\sigma,x,v).
		\label{eq:frozen-polynomial-product}
	\end{align}
\end{lemma}
\begin{proof}
	{We first prove the estimate for} $m=0$.
	{The case} $m\geq1$ {follows from the same contour
	shift and the normalized Fourier-multiplier bounds for}
	$\mathcal D_{\sigma,v_0}^{(x,v),m}$.
	
	Let
	\begin{align*}
		\theta=(\theta_x,\theta_v)\in\mathbb R^{2d},
		\qquad
		L:=|\theta_x|+|\theta_v|,
	\end{align*}
	{and shift the Fourier contours according to}
	\begin{align*}
		\xi\longmapsto\xi-\frac{\mathrm i\theta_x}{\sigma},
		\qquad
		\eta\longmapsto\eta-\mathrm i\theta_v.
	\end{align*}
	Define
	\begin{align*}
		\Theta_r
		:=
		\theta_v+\frac{t-r}{\sigma}\theta_x.
	\end{align*}
	{The contour shift gives}
	\begin{align}
		{p_{\mathfrak q}^{(0)}}(t,\tau,x,v)
		={}&
		e^{-\theta_x\cdot x/\sigma-\theta_v\cdot v}
		{(2\pi)^{-2d}}
		\iint_{\mathbb R^{2d}}
		e^{-\mathrm i(x\cdot\xi+v\cdot\eta)}
		\notag\\
		&\times
		\exp\left[
		-\int_\tau^t
		\psi_{r,\mathfrak q}^{(0)}
		\left(
		\eta+(t-r)\xi-\mathrm i\Theta_r
		\right)
		\dif r
		\right]\dif\xi\dif\eta.
		\label{eq:frozen-shifted-fourier}
	\end{align}
	{For} $q,\Theta\in\mathbb R^d$, {symmetry in}
	$\alpha$ {yields}
	\begin{align*}
		\Re\psi_{r,\mathfrak q}^{(0)}(q-\mathrm i\Theta)
		&=
		\int_{|\alpha|<r_0}
		\left[
		1-\cos(q\cdot\alpha)\cosh(\Theta\cdot\alpha)
		\right]
		\nu_{r,\mathfrak q}(\dif\alpha)
		\\
		&=
		\psi_{r,\mathfrak q}^{(0)}(q)
		-
		\int_{|\alpha|<r_0}
		\cos(q\cdot\alpha)
		\bigl(\cosh(\Theta\cdot\alpha)-1\bigr)
		\nu_{r,\mathfrak q}(\dif\alpha)
		\\
		&\geq
		\psi_{r,\mathfrak q}^{(0)}(q)
		-M_{r,\mathfrak q}(\Theta),
	\end{align*}
	where
	\begin{align*}
		M_{r,\mathfrak q}(\Theta)
		:=
		\int_{|\alpha|<r_0}
		\bigl(\cosh(\Theta\cdot\alpha)-1\bigr)
		\nu_{r,\mathfrak q}(\dif\alpha).
	\end{align*}
	Since
	\begin{align*}
		\cosh z-1\leq \frac12z^2e^{|z|},
	\end{align*}
	{the frozen second-moment estimate yields}
	\begin{align}
		M_{r,\mathfrak q}(\Theta)
		&\leq
		C|\Theta|^2e^{C|\Theta|}
		\int_{|\alpha|<r_0}
		|\alpha|^2\nu_{r,\mathfrak q}(\dif\alpha)
		\notag\\
		&\leq
		C\langle v_0\rangle^{-\kappa}
		|\Theta|^2e^{C|\Theta|}.
		\label{eq:cosh-small-jump}
	\end{align}
	{Consequently, the change of variables}
	$\varrho=(t-r)/\sigma$ {gives}
	\begin{align}
		\int_\tau^t
		M_{r,\mathfrak q}(\Theta_r)\dif r
		&\leq
		C\tilde{\sigma}_{v_0}
		\int_0^1
		|\theta_v+\varrho\theta_x|^2
		e^{C|\theta_v+\varrho\theta_x|}
		\dif\varrho
		\notag\\
		&\leq
		C\tilde{\sigma}_{v_0}L^2e^{CL}.
		\label{eq:integrated-cosh-bound}
	\end{align}
	{We next justify the real-axis majorant and the contour shift.
	The measures} $\nu_{r,\mathfrak q}$ {are even, and for every}
	$A>0$, {their compact support and the uniform second-moment
	estimate imply}
	\begin{align}
		\sup_{\tau\leq r\leq t}
		\int_{|\alpha|<r_0}|\alpha|^2e^{A|\alpha|}
		\nu_{r,\mathfrak q}(\dif\alpha)
		\lesssim_A\langle v_0\rangle^{-\kappa}.
		\label{eq:frozen-compensated-exponential-moment}
	\end{align}
	{This is the compensated exponential moment. The factor}
	$|\alpha|^2$ {is essential because} $\nu_{r,\mathfrak q}$
	{may have infinite mass at the origin. Because}
	$1-\cos(z\cdot\alpha)=O_K(|\alpha|^2)$ {locally uniformly for}
	$z\in K\Subset\mathbb C^d$, $\psi_{r,\mathfrak q}^{(0)}$
	{is entire, uniformly in} $r$ {on compact sets.
	Furthermore,}
	$\{[\alpha]_{v_0}\leq r_0\}\subset\{|\alpha|\leq r_0\}$,
	{so Remark~\ref{rem:small-jump-coercivity}, with}
	$\delta=r_0${, using \eqref{assofb-lower},} {gives}
	\begin{align}
		\psi_{r,\mathfrak q}^{(0)}(z)
		\gtrsim
		\langle v_0\rangle^{-\kappa}[[z]]_{v_0}^{2s}
		\min\{1,r_0[[z]]_{v_0}\}^{2-2s}.
		\label{eq:uniform-truncated-coercivity}
	\end{align}
{The implicit lower-bound constant includes dependence
on $c_b$ and $\delta_0$, fixed throughout the following estimates.}
	Set
	\begin{align*}
		p:=\sigma\tilde\sigma_{v_0}^{1/(2s)}\mathcal O_{v_0}\xi,
		\qquad
		q:=\tilde\sigma_{v_0}^{1/(2s)}\mathcal O_{v_0}\eta.
	\end{align*}
	{Using} $[[z]]_{v_0}=|\mathcal O_{v_0}z|$,
	$\tilde\sigma_{v_0}\leq1$, {and} $\varrho=(t-r)/\sigma$,
	{we obtain}
	\begin{align}
		\int_\tau^t\psi_{r,\mathfrak q}^{(0)}
		\bigl(\eta+(t-r)\xi\bigr)\dif r
		&\gtrsim_{r_0}\int_0^1
		\bigl(|q+\varrho p|^2\wedge|q+\varrho p|^{2s}\bigr)\dif\varrho
		\notag\\
		&\gtrsim_{r_0}
		(|p|^2+|q|^2)\wedge(|p|^2+|q|^2)^s.
		\label{eq:integrated-truncated-coercivity}
	\end{align}
	{The first inequality follows from}
	\(
	u^{2s}\min\{1,r_0\tilde\sigma_{v_0}^{-1/(2s)}u\}^{2-2s}
	\geq r_0^{2-2s}(u^2\wedge u^{2s})
	\).
	{The second follows by scaling separately in the regions}
	$|p|+|q|\leq1$ {and} $|p|+|q|>1$; {on the unit
	sphere, the integral vanishes only if} $p=q=0$.
	{Set}
	\begin{align*}
		\Lambda_\sigma(\xi,\eta)
		&:=1+\tilde\sigma_{v_0}^{1/(2s)}
		\bigl(\sigma|\mathcal O_{v_0}\xi|+|\mathcal O_{v_0}\eta|\bigr),\\
		\mathfrak P_{v_0}(\sigma)
		&:=\sigma^{-d}\tilde\sigma_{v_0}^{-d/s}\langle v_0\rangle^2.
	\end{align*}
	{Because} $\det\mathcal O_{v_0}=\langle v_0\rangle^{-1}$,
	{the change of variables has Jacobian}
	\begin{align*}
		\dif\xi\dif\eta=\mathfrak P_{v_0}(\sigma)\,\dif p\dif q.
	\end{align*}
	Consequently, for every $m\in{\mathbb N}$,
	\begin{align}
		\iint_{\mathbb R^{2d}}\Lambda_\sigma(\xi,\eta)^m
		\exp\left[-c\int_\tau^t\psi_{r,\mathfrak q}^{(0)}
		\bigl(\eta+(t-r)\xi\bigr)\dif r\right]\dif\xi\dif\eta
		\lesssim_m\mathfrak P_{v_0}(\sigma).
		\label{eq:real-axis-fourier-integral}
	\end{align}
	Let
	\(
	\Phi(\xi,\eta):=\int_\tau^t\psi_{r,\mathfrak q}^{(0)}
	(\eta+(t-r)\xi)\dif r
	\).
	{For any coordinatewise intermediate shift}
	$b=(b_x,b_v)$ {between} $0$ {and}
	$(\theta_x/\sigma,\theta_v)$, {we have}
	$\sigma|b_x|+|b_v|\leq L$.  The real-part estimate and
	\eqref{eq:cosh-small-jump} imply
	\begin{align*}
		\left|e^{-\Phi(\xi-\mathrm i b_x,\eta-\mathrm i b_v)}\right|
		\leq e^{C\tilde\sigma_{v_0}L^2e^{CL}}e^{-\Phi(\xi,\eta)}.
	\end{align*}
	{On these contours, the normalized multiplier of degree} $m$
	{is bounded by}
	\begin{align*}
		C_m\bigl(1+\tilde\sigma_{v_0}^{1/(2s)}L\bigr)^m
		\Lambda_\sigma(\xi,\eta)^m.
	\end{align*}
	{Multiply the Fourier integrand by the entire regulator}
	$e^{-\varepsilon(\xi\cdot\xi+\eta\cdot\eta)}$ {and apply
	Cauchy's theorem successively in the} $2d$ {variables over
	bounded boxes. For fixed} $\varepsilon>0$, {all integrals over
	the connecting faces vanish as the boxes expand. The preceding estimates
	provide an integrable majorant independent of} $\varepsilon$,
	{so dominated convergence as} $\varepsilon\downarrow0$
	{proves \eqref{eq:frozen-shifted-fourier}, including after any
	number of normalized derivatives.}

	{Equations}
	\eqref{eq:frozen-shifted-fourier}--\eqref{eq:real-axis-fourier-integral}
	{therefore imply}
	\begin{align}
		\left|{p_{\mathfrak q}^{(0)}}(t,\tau,x,v)\right|
		\lesssim
		\mathfrak P_{v_0}(\sigma)
		\exp\left[
		-\theta_x\cdot\frac{x}{\sigma}
		-\theta_v\cdot v
		+C\tilde{\sigma}_{v_0}L^2e^{CL}
		\right].
		\label{eq:raw-frozen-exponential-pointwise}
	\end{align}
	{For} $m\geq1$, {differentiating under the Fourier
	integral introduces the normalized multipliers associated with}
	$\mathcal D_{\sigma,v_0}^{(x,v),m}$. {After the contour shift,
	their} $\theta${-dependent part is bounded by}
	\begin{align*}
		C_m
		\left(
		1+\tilde{\sigma}_{v_0}^{1/(2s)}L
		\right)^m.
	\end{align*}
	{This factor is bounded if}
	$\tilde{\sigma}_{v_0}^{1/(2s)}L\leq1$ {and is otherwise
	absorbed into} $C\tilde{\sigma}_{v_0}(1+L)^2e^{CL}$. {Thus,
	after enlarging the constant, for every} $m\in{\mathbb N}$,
	\begin{align}
		\left|
		\mathcal D_{\sigma,v_0}^{(x,v),m}
		{p_{\mathfrak q}^{(0)}}(t,\tau,x,v)
		\right|
		\lesssim_m
		\mathfrak P_{v_0}(\sigma)
		\exp\left[
		-\theta_x\cdot\frac{x}{\sigma}
		-\theta_v\cdot v
		+C\tilde{\sigma}_{v_0}(1+L)^2e^{CL}
		\right].
		\label{eq:raw-frozen-exponential}
	\end{align}
	
\begingroup
To retain the anisotropic profile after the contour shift, introduce the
tilted measures and density
\begin{align*}
\nu_{r,\mathfrak q}^{\theta}(\dif\alpha)
&:=e^{\Theta_r\cdot\alpha}\nu_{r,\mathfrak q}(\dif\alpha),\\
{p_{\mathfrak q}^{\theta}}(t,\tau,x,v)
&:=\exp\!\left[\theta_x\cdot x/\sigma+\theta_v\cdot v
-\int_\tau^t M_{r,\mathfrak q}(\Theta_r)\,\dif r\right]
{p_{\mathfrak q}^{(0)}}(t,\tau,x,v).
\end{align*}
The compensated exponential moment above makes this a probability density.
Normalize space, velocity and time as in the proof of
Proposition~\ref{prop-frozen-anisotropic-upper}, and denote the normalized
measures by $\widetilde\nu_\rho$ and
$\widetilde\nu_\rho^\theta$, $0\leq\rho\leq1$. Since
$|\alpha|<r_0$, \eqref{eq:normalized-pointwise-upper} gives
\[
\widetilde\nu_\rho^\theta(\dif u)
\leq C e^{r_0L}|u|^{-d-2s}\,\dif u.
\]
Evenness of $\widetilde\nu_\rho$ also gives, for every $\zeta\in\mathbb R^d$,
\begin{align*}
\int_{|u|\leq1}(1-\cos(\zeta\cdot u))
\widetilde\nu_\rho^\theta(\dif u)
&\geq\int_{|u|\leq1}(1-\cos(\zeta\cdot u))
\widetilde\nu_\rho(\dif u)
\gtrsim_{r_0}|\zeta|^2\wedge|\zeta|^{2s},\\
\int_{|u|\leq1}|u|^2\widetilde\nu_\rho^\theta(\dif u)
&\lesssim e^{r_0L}.
\end{align*}
The lower bound follows from Remark~\ref{rem:small-jump-coercivity}
by restricting to $|u|\leq r_0$, whose physical jumps satisfy
$|\alpha|\leq r_0$.
The normalized drift created by the tilt is
\[
\int_{|u|\leq1}u\left(
e^{\tilde\sigma_{v_0}^{1/(2s)}
\Theta_{\tau+\sigma\rho}\cdot\mathcal O_{v_0}u}-1\right)
\widetilde\nu_\rho(\dif u),
\]
whose magnitude is at most $CL e^{CL}$. Retaining this drift, the
Fourier integration-by-parts argument in
Proposition~\ref{prop-frozen-anisotropic-upper} gives a normalized
small-jump density $q_1^\theta$ satisfying
\[
|\partial^k q_1^\theta(z)|
\leq C_{k,A}e^{C_{k,A}L}\langle z\rangle^{-A},
\qquad A\geq0.
\]
For the normalized jumps $|u|>1$, use the isotropic reference measure
$C e^{r_0L}\mathbf1_{|u|>1}|u|^{-d-2s}\dif u$.
Termwise Poisson comparison and \cite[Lemma~2.3]{HouZhang} bound its
unit-box probabilities by
$C\exp(Ce^{CL})\mathcal N(z)$; the polynomial dependence on the
reference intensity is absorbed into this exponential.
The unit-box convolution and profile moderation
\eqref{eq:A10-profile-moderation} therefore yield
\begin{align}\label{eq:frozen-tilted-profile}
\left|\mathcal D_{\sigma,v_0}^{(x,v),m}
{p_{\mathfrak q}^{\theta}}(t,\tau,x,v)\right|
\leq C_m\exp(C_m e^{C_mL})\bT_{v_0}(\sigma,x,v).
\end{align}
All constants are independent of $\sigma$ and $v_0$.

Differentiating the exponential factor defining
${p_{\mathfrak q}^\theta}$ costs at most $(1+L)^m$, because
$\tilde\sigma_{v_0}^{1/(2s)}\leq1$ and
$\|\mathcal O_{v_0}\|\leq1$. Combining
\eqref{eq:frozen-tilted-profile} with
\eqref{eq:integrated-cosh-bound} gives
\[
\left|\mathcal D_{\sigma,v_0}^{(x,v),m}
{p_{\mathfrak q}^{(0)}}(t,\tau,x,v)\right|
\leq C_m\bT_{v_0}(\sigma,x,v)
\exp\!\left[-\theta_x\cdot x/\sigma-\theta_v\cdot v
+C_m e^{C_mL}\right].
\]
For $(x,v)\ne(0,0)$, choose
\[
(\theta_x,\theta_v)
=\lambda\frac{(x/\sigma,v)}{|(x/\sigma,v)|},
\qquad
\lambda=c_m\log(2+R_\sigma(x,v)),
\]
with $c_m>0$ sufficiently small. Then
$L\leq\sqrt2\lambda$ and
$\theta_x\cdot x/\sigma+\theta_v\cdot v
\geq\lambda R_\sigma(x,v)/\sqrt2$.
Choose $c_m$ so that $\sqrt2 C_mc_m<1$; the negative term absorbs
$C_m e^{C_mL}$ for sufficiently large $R_\sigma(x,v)$.
For bounded $R_\sigma(x,v)$, take $\theta=0$. Enlarging the constant
and decreasing $c_m$ gives
\begin{align}\label{eq:frozen-relative-exponential}
\left|\mathcal D_{\sigma,v_0}^{(x,v),m}
{p_{\mathfrak q}^{(0)}}(t,\tau,x,v)\right|
\lesssim_m\bT_{v_0}(\sigma,x,v)
e^{-c_mR_\sigma(x,v)\log(2+R_\sigma(x,v))}.
\end{align}
This implies \eqref{eq:frozen-polynomial-product} for every $N\geq0$.
\par\endgroup
\end{proof}

\begingroup
We also need small-time decay away from the jump range.
\begin{lemma}
\label{lem:separated-frozen-kernel}
Let $\delta\in(0,1)$ and $k\geq1$. There is a structural constant $C_*>0$
such that, if $r_0\leq\delta/(C_*(k+1))$, then for every $m,N\geq0$
and all sufficiently small $\sigma=t-\tau$,
\begin{equation}\label{eq:separated-frozen-kernel}
\left|\mathcal D_{\sigma,v_0}^{(x,v),m}
{p_{\mathfrak q}^{(0)}}(t,\tau,x,v)\right|
\lesssim_{m,N,\delta,k,r_0}
\sigma^k\bK_{v_0,N}(\sigma,x,v),
\qquad R_\sigma(x,v)\geq\delta.
\end{equation}
The time threshold and constants are uniform in $v_0$.
\end{lemma}
\begin{proof}
Take $\sigma\leq r_0^{2s}$ and use the tilt and normalization from
the preceding proof. If
$\tilde\sigma_{v_0}^{1/(2s)}L\leq1$, the normalized small-jump
coercivity, moments and drift are uniformly bounded; hence the Schwartz
bounds for $q_1^\theta$ are uniform in $L$.
For the normalized jumps $|u|>1$, the density bound is
$C e^{r_0L}|u|^{-d-2s}$, whereas their total rate is at most
\[
C+C\tilde\sigma_{v_0}L^2e^{r_0L}.
\]
Indeed, evenness replaces the exponential by its hyperbolic cosine in
the rate integral, and the excess rate is bounded by the compensated
exponential moment.

We retain the actual rate in the Poisson estimate. The upper-bound
calculation in \cite[Lemma~2.3, equation~(2.11)]{HouZhang} gives
\[
\mathbb P(Z_1^{(1),\theta}\in Q_1(z))
\leq C B^2(1+\Lambda)^A\mathcal N(z)
\]
when the jump density is bounded by $B|u|^{-d-2s}$ and its rate
is bounded by $\Lambda\geq1$, for a structural $A$.
The calculation also applies to time-dependent measures: pad the rate
to $\Lambda$ by zero marks, bound the two selected nonzero jump
densities, and integrate the remaining marks by their total mass.
Thus, if $\tilde\sigma_{v_0}L^2e^{r_0L}\leq1$, unit-box convolution
with $q_1^\theta$ and differentiation of the tilt give
\[
\left|\mathcal D_{\sigma,v_0}^{(x,v),m}
{p_{\mathfrak q}^{(0)}}\right|
\lesssim_m\bT_{v_0}(\sigma,x,v)
\exp\!\left[-\theta_x\cdot x/\sigma-\theta_v\cdot v+2r_0L\right].
\]
Choose the tilt in the direction $(x/\sigma,v)$ with magnitude
$\lambda=4\sqrt2(k+1)\delta^{-1}\log(1/\sigma)$.
Taking $C_*$ large ensures both
$\sqrt2 r_0\lambda\leq\frac12\log(1/\sigma)$ and
$2\sqrt2r_0\leq\delta/(2\sqrt2)$.
Since $\tilde\sigma_{v_0}\leq\sigma$, the two restrictions on $L$
hold for all sufficiently small $\sigma$, uniformly in $v_0$.
For $R_\sigma\geq\delta$, the resulting exponential is bounded by
\[
\exp\!\left[-\lambda R_\sigma/(2\sqrt2)\right]
\lesssim_{N,\delta,k}\sigma^k\langle R_\sigma\rangle^{-N}.
\]
This proves \eqref{eq:separated-frozen-kernel}.
\end{proof}
\par\endgroup

The kinetic composition law is
\begin{align}\label{eq:kinetic-composition}
	X=X_1+X_2+a_1V_2,
	\qquad
	V=V_1+V_2,
	\qquad
	t=a_1+a_2.
\end{align}
{This implies}
\begin{align}\label{eq:kinetic-weight-triangle}
	R_t(X,V)
	\lesssim R_{a_1}(X_1,V_1)+R_{a_2}(X_2,V_2),
\end{align}
and therefore
\begin{align}\label{eq:weight-submultiplicative}
	\mathfrak m_N(a_1,X_1,V_1)
	\mathfrak m_N(a_2,X_2,V_2)
	\lesssim_N
	\mathfrak m_N(t,X,V).
\end{align}
{Combining \eqref{eq:weight-submultiplicative} with the
subconvolution, finite-difference, and tail estimates for} $\bT$
{yields the corresponding weighted estimates for} $\bK$.
{A shift with} $|\alpha|<1$ {changes the unscaled
weight} $\mathfrak{m}_N$ {by at most a bounded factor. We also use}
\begin{align}\label{eq:local-profile-comparison}
	\bK_{v_0,N}\sim_N\bK_{v_1,N}
	\qquad\text{if}\qquad |v_0-v_1|\le C.
\end{align}
{{For noncomparable reference velocities, the required profile comparison
uses a structural exponent}} $M_{\mathrm{pc}}>0$ {such
that}
\begin{align}\label{eq:global-profile-comparison}
	\bT_{v_1}(t,X,V)
	\lesssim
	\langle v_1-v_0\rangle^{M_{\mathrm{pc}}}
	\bT_{v_0}(t,X,V).
\end{align}

\subsubsection{\texorpdfstring{{Estimates for the error terms}}{Estimates for the error terms}}
	{Define}	\[
\mathfrak W_T(f)(x,v)
:=
\int_{|\alpha|\ge T^{1/(2s)}}
f(x,v-\alpha)
\frac{\dif \alpha}{|\alpha|^{d+2s}} .
\]

	{Let} \(p_t=p_t(x,v)\) {denote the fractional
Kolmogorov kernel defined by}
\begin{align}
\label{defp}
\widehat p_t(\xi,\eta)
=
\exp\left\{
-\int_0^t |\eta+r\xi|^{2s}\,\dif r
\right\},
\qquad t>0 .
\end{align} 
{We first establish a kinetic composition estimate for the
large-jump error terms.}
\begin{lemma}\label{lemconvopp}
	For any $T>0$, one has
	\begin{equation}
		\label{eq:large-jump-bridge-assumption}
		\begin{aligned}
			&\int_0^T
			\iint_{\mathbb R^{2d}}
			p_{T-\tau}
			\bigl(x-y-(T-\tau)u,v-u\bigr)
			\mathfrak W_T(p_\tau)(y,u)
			\,\dif y\,\dif u\,\dif \tau       \lesssim
			 p_T(x,v).
		\end{aligned}
	\end{equation}
\end{lemma}
\begin{proof}
	Let \(t_1,t_2>0\) with \(t_1+t_2=T\), and set
	\[
	I_{t_1,t_2}(x,v)
	:=
	\iint_{\mathbb R^{2d}}
	p_{t_1}(x-y-t_1u,v-u)\,
	\mathfrak W_T(p_{t_2})(y,u)\,\dif y\,\dif u .
	\]
	Writing
	\[
	\nu_T(\dif \alpha)
	=
	\mathbf 1_{\{|\alpha|\ge T^{1/(2s)}\}}
	\frac{\dif \alpha}{|\alpha|^{d+2s}},
	\]
	we have
	\[
	\mathfrak W_T(p_{t_2})(y,u)
	=
	\int_{\mathbb R^d}
	p_{t_2}(y,u-\alpha)\,\nu_T(\dif \alpha).
	\]
	Taking the Fourier transform in \((x,v)\) and using \eqref{defp} gives
	\[
	\begin{aligned}
		\widehat I_{t_1,t_2}(\xi,\eta)
		&=
		\widehat p_{t_1}(\xi,\eta)\,
		\widehat p_{t_2}(\xi,\eta+t_1\xi)\,
		\widehat\nu_T(\eta+t_1\xi)  \\
		&=
		\widehat p_T(\xi,\eta)\,
		\widehat\nu_T(\eta+t_1\xi).
	\end{aligned}
	\]
	Hence, by taking the inverse Fourier transform,
	\[
	I_{t_1,t_2}(x,v)
	=
	\int_{|\alpha|\ge T^{1/(2s)}}
	p_T(x-t_1\alpha,v-\alpha)
	\frac{\dif\alpha}{|\alpha|^{d+2s}} .
	\]
	Therefore,
	\[
	\begin{aligned}
		&\int_0^T
		\iint_{\mathbb R^{2d}}
		p_{T-\tau}
		\bigl(x-y-(T-\tau)u,v-u\bigr)
		\mathfrak W_T(p_\tau)(y,u)
		\,\dif y\,\dif u\,\dif\tau  \\
		&\qquad =
		\int_0^T
		\int_{|\alpha|\ge T^{1/(2s)}}
		p_T(x-\tau\alpha,v-\alpha)
		\frac{\dif \alpha\,d\tau}{|\alpha|^{d+2s}} .
	\end{aligned}
	\]
		{It remains to bound the last term by} \(C p_T(x,v)\).
	
		{The time-one estimate \eqref{eq:kolmogorov-unit-profile} and
		the scaling induced by \eqref{defp}, namely,}
	\[
	x=T^{1+\frac1{2s}}X,\qquad
	v=T^{\frac1{2s}}V,\qquad
	\alpha=T^{\frac1{2s}}h,\qquad
	\tau=T\sigma,
	\]
		{reduce the desired estimate to}
	\[
	\int_0^1
	\int_{|h|\ge1}
	p_1(X-\sigma h,V-h)
	\frac{\dif h\,\dif\sigma}{|h|^{d+2s}}
	\lesssim p_1(X,V).
	\]
	Since
	\[
	|X|+|V|
	\le
	|X-\sigma h|+|V-h|+C|h|,
	\qquad 0\le \sigma\le1,
	\]
		{we obtain}
	\[
	\begin{aligned}
		&\langle |X-\sigma h|+|V-h|\rangle^{-(d+2s)}
		\langle h\rangle^{-(d+2s)}  \\
		&\qquad\lesssim 
		\langle |X|+|V|\rangle^{-(d+2s)}
		\bigl(
		\langle |X-\sigma h|+|V-h|\rangle^{-(d+2s)}
		+
		\langle h\rangle^{-(d+2s)}
		\bigr).
	\end{aligned}
	\]
	Since \(|h|^{-(d+2s)}\lesssim \langle h\rangle^{-(d+2s)}\) on
	\(\{|h|\ge1\}\), it follows that
	\[
	\begin{aligned}
		&\int_0^1
		\int_{|h|\ge1}
		p_1(X-\sigma h,V-h)
		\frac{\dif h\,\dif \sigma}{|h|^{d+2s}} \lesssim 
		\langle |X|+|V|\rangle^{-(d+2s)}
		(J_1+J_2),
	\end{aligned}
	\]
	where
	\[
	\begin{aligned}
		J_1
		&:=
		\int_0^1\int_{\mathbb R^d}
		\langle |X-\sigma h|+|V-h|\rangle^{-(d+2s)}
		\int_0^1
		\langle X-\sigma h-\theta(V-h)\rangle^{-(d+2s)}
		\,\dif \theta\,\dif h\,\dif \sigma,\\
		J_2
		&:=
		\int_0^1\int_{\mathbb R^d}
		\langle h\rangle^{-(d+2s)}
		\int_0^1
		\langle X-\sigma h-\theta(V-h)\rangle^{-(d+2s)}
		\,\dif \theta\,\dif h\,\dif \sigma.
	\end{aligned}
	\]
	For \(J_1\), let \(B=V-h\). Then
	\[
	X-\sigma h=X-\sigma V+\sigma B,
	\qquad
	X-\sigma h-\theta(V-h)=X-\sigma V+(\sigma-\theta)B.
	\]
	Moreover,
	\[
	\langle |X-\sigma h|+|V-h|\rangle^{-(d+2s)}
	\le
	\langle B\rangle^{-(d+2s)}.
	\]
	Applying the elementary bound
	\[
	\int_{\mathbb R^d}
	\langle B\rangle^{-(d+2s)}
	\langle Y+\lambda B\rangle^{-(d+2s)}
	\,\dif B
	\lesssim \langle Y\rangle^{-(d+2s)},
	\qquad |\lambda|\le1,
	\]
		{we obtain}
	\[
	J_1
	\lesssim \int_0^1
	\langle X-\sigma V\rangle^{-(d+2s)}
	\,\dif \sigma .
	\]
	For \(J_2\), we use
	\[
	X-\sigma h-\theta(V-h)
	=
	X-\theta V-(\sigma-\theta)h.
	\]
	The same convolution bound gives
	\[
	J_2
	\lesssim \int_0^1
	\langle X-\theta V\rangle^{-(d+2s)}
	\,\dif \theta .
	\]
		{Combining the two estimates yields}
	\[
	\int_0^1
	\int_{|h|\ge1}
	p_1(X-\sigma h,V-h)
	\frac{\dif h\,\dif \sigma}{|h|^{d+2s}}
	\lesssim p_1(X,V).
	\]
		{Scaling back gives}
	\[
	\int_0^T
	\int_{|\alpha|\ge T^{1/(2s)}}
	p_T(x-\tau \alpha,v-\alpha)
	\frac{\dif \alpha\,\dif \tau}{|\alpha|^{d+2s}}
	\lesssim  p_T(x,v).
	\]
		{This proves the lemma.}
\end{proof}

{We extend the weighted convolution bounds
to finite differences in the following lemma.}
\begin{lemma}
	\label{lem:decorated-weighted-convolution}
Let $s<b<\min\{1,2s\}$, and define
\begin{equation}\label{notaxv}
	\begin{aligned}
	&	\sigma:=t-\tau,
		\qquad
		a_1:=t-r,
		\qquad
		a_2:=r-\tau,\\
&	(X_1,V_1):=(x-\zeta-a_1u,v-u),\quad\ 
		(X_2,V_2):=(\zeta-y-a_2w,u-w),\quad\ 
		(X,V):=(x-y-\sigma w, v-w).
	\end{aligned}
\end{equation}
		{We use the capped increment} \(\mathfrak r_{\sigma,q}\)
		{defined in \eqref{def:truncated-increment}.}
	Write
	\begin{align*}
		K_1
		&:=
		\bK_{v,N}(a_1,X_1,V_1),\quad\quad 
		K_2
		:=
		\bK_{u,N}(a_2,X_2,V_2),\quad\quad
		K_0
		:=
		\bK_{v,N}(\sigma,X,V).
	\end{align*}
{Define the shifted-profile sums by}
	\begin{align*}
	K_1^\sharp(\alpha)
		&:=
		\bK_{v,N}(a_1,X_1,V_1)
		+
		\bK_{v,N}(a_1,X_1+a_1\alpha,V_1+\alpha),\\
	K_2^\sharp(\alpha)
		&:=
		\sum_{\varepsilon_1,\varepsilon_2\in\{0,1\}}
		\bK_{u-\varepsilon_1\alpha,N}
		(a_2,X_2,V_2-\varepsilon_2\alpha).
	\end{align*}
Then
	\begin{align}
		&\int_\tau^t
		\iint_{\mathbb R^{2d}}
		\int_{[\alpha]_v\le \tilde \sigma_v^{1/(2s)}}
		\mathfrak r_{a_1,v}(\alpha)
		\mathfrak r_{a_2,u}(\alpha)^b
		K_1^\sharp(\alpha)K_2^\sharp(\alpha)
		{\Omega_{\varkappa}}(u,\alpha) \mathbf{1}_{|u-v|\leq 1}
		\frac{\dif\alpha}{|\alpha|^{d+2s}}
		\dif\zeta\dif u\dif r
	\lesssim K_0,
		\label{eq:decorated-convolution-leading}
	\end{align}
	and
	\begin{align}
		&\int_\tau^t
		\iint_{\mathbb R^{2d}}
		\int_{[\alpha]_v\le \tilde \sigma_v^{1/(2s)}}
		|\alpha|^b
		\mathfrak r_{a_2,u}(\alpha)^b
		K_1^\sharp(\alpha)K_2^\sharp(\alpha)
		{\Omega_{\varkappa}}(u,\alpha) \mathbf{1}_{|u-v|\leq 1}
		\frac{\dif\alpha}{|\alpha|^{d+2s}}
		\dif\zeta\dif u\dif r
	\lesssim
		\tilde \sigma_v^\frac{b}{2s} K_0.
		\label{eq:decorated-convolution-coefficient}
	\end{align}
	Moreover,
	\begin{align}
		&\int_\tau^t
		\iint_{\mathbb R^{2d}}
		K_1
		\int_{\tilde\sigma_v^{1/(2s)}<[\alpha]_v<1}
		K_2^\sharp(\alpha) 	{\Omega_{\varkappa}}(u,\alpha) \mathbf{1}_{|u-v|\leq 1}
		\frac{\dif\alpha}{|\alpha|^{d+2s}}
		\dif\zeta\dif u\dif r
		\lesssim K_0.
		\label{eq:decorated-convolution-tail}
	\end{align}
\end{lemma}
\begin{proof}
	\begingroup
	On $|u-v|\leq1$ and $[\alpha]_v\leq\tilde\sigma_v^{1/(2s)}\leq1$,
the reference velocities $u,u-\alpha$ are within distance two of $v$.
Apply \eqref{eq:local-profile-comparison} to each summand of
$K_2^\sharp(\alpha)$ on this localized set.
	The kinetic composition law sends each pair of shifted endpoints
	to a total argument differing from $(X,V)$ by at most
	$C(\sigma\alpha,\alpha)$. Its intrinsic normalized size is
	bounded in this range. The fixed-reference subconvolution and
	weight submultiplicativity therefore give
	\begin{align}
		\iint_{\mathbb R^{2d}}
		\mathbf1_{|u-v|\leq1}
		K_1^\sharp(\alpha)K_2^\sharp(\alpha)
		\,\dif\zeta\dif u
		\lesssim_N K_0,
		\qquad [\alpha]_v\leq\tilde\sigma_v^{1/(2s)}.
		\label{eq:shifted-weighted-subconvolution}
	\end{align}
	On the same localized set, $|\alpha|\leq1$ makes the
	indicator in \eqref{defom} equal to one, and bounded changes
	of the reference velocity give
	\[
	{\Omega_{\varkappa}}(u,\alpha)\lesssim
	{\Omega_{\varkappa}}(v,\alpha)+{\Omega_{\varkappa}}(v-\alpha,\alpha),
	\qquad
	\mathfrak r_{a_2,u}(\alpha)\asymp
	\mathfrak r_{a_2,v}(\alpha).
	\]
	Hence
	\eqref{eq:global-Omega-unweighted}, with $a=\varkappa$
	and exponent $b$, gives
	\begin{align*}
		\int_{|\alpha|\leq1}
		\mathfrak r_{a_1,v}(\alpha)\mathfrak r_{a_2,u}(\alpha)^b
		{\Omega_{\varkappa}}(u,\alpha)\frac{\dif\alpha}{|\alpha|^{d+2s}}
		\lesssim
		a_1^{-\frac{2s-b}{2s}}a_2^{-\frac b{2s}},
		\qquad |u-v|\leq1.
	\end{align*}
	Combining this fixed-reference coefficient bound with
	\eqref{eq:shifted-weighted-subconvolution} and the finite
	$r$-integral proves \eqref{eq:decorated-convolution-leading}.
	\par\endgroup
	
	{With the same localization retained
	through the profile comparison,}
	\eqref{eq:global-Omega-capped-moment} yields
	\begin{align*}
		\int_{|\alpha|\leq1}|\alpha|^b
		\mathfrak r_{a_2,u}(\alpha)^b{\Omega_{\varkappa}}(u,\alpha)
		\frac{\dif\alpha}{|\alpha|^{d+2s}}
		\lesssim
		\langle u\rangle^{-\frac{\kappa b}{2s}}
		a_2^{-1+\frac b{2s}}.
	\end{align*}
	{Integrating in} \(r\) {and using}
	\(\langle u\rangle\sim\langle v\rangle\) {produces the factor}
	\(\tilde\sigma_v^{b/(2s)}\), {which proves
	\eqref{eq:decorated-convolution-coefficient}.}

	{Finally, \eqref{eq:global-Omega-unweighted}, with both time
	parameters equal to} \(\sigma\), {gives}
	\begin{align*}
		\int_{\tilde \sigma_v^{1/(2s)}<[\alpha]_v<1}
		{\Omega_{\varkappa}}(u,\alpha)
		\frac{\dif\alpha}{|\alpha|^{d+2s}}
		\lesssim \sigma^{-1}.
	\end{align*}
\begingroup
On $|u-v|\leq1$, $|\alpha|<1$, replace the reference velocities
in $K_2^\sharp$ by $v$ using \eqref{eq:local-profile-comparison}.
The unshifted term follows from subconvolution and the preceding
tail coefficient bound.
By \eqref{eq:weight-submultiplicative} and
$\mathfrak m_N(t,X,V-\alpha)\sim\mathfrak m_N(t,X,V)$
for $|\alpha|<1$, it suffices to prove the localized bridge
\begin{align}\label{convott}
	\int_{\tau}^t \int_{\mathbb{R}^{2d}}
	\mathbf1_{|u-v|\leq1}\bT_v(a_1,X_1,V_1)
	\int_{\tilde\sigma_v^{1/(2s)}<[\alpha]_v<1}
	\bT_v(a_2,X_2,V_2-\alpha)
	{\Omega_{\varkappa}}(u,\alpha)\frac{\dif\alpha}{|\alpha|^{d+2s}}
	\dif\zeta\dif u\dif r
	\lesssim\bT_v(\sigma,X,V).
\end{align}
\par\endgroup
\begingroup
The kinetic scaling of \eqref{defp}, together with
\eqref{eq:kolmogorov-unit-profile} and \eqref{anisotropic-HZ-profile}, gives
\[
\bT_{v_0}(t,X,V)\asymp
\langle v_0\rangle^{2-d\kappa}
p_{\tilde t_{v_0}}\bigl(
\langle v_0\rangle^{-\kappa}\mathcal O_{v_0}^{-1}X,
\mathcal O_{v_0}^{-1}V\bigr).
\]
The changes $\zeta\mapsto\langle v_0\rangle^{-\kappa}
\mathcal O_{v_0}^{-1}\zeta$ and $u\mapsto\mathcal O_{v_0}^{-1}u$
have joint Jacobian $\langle v_0\rangle^{2-d\kappa}$, matching
the density prefactor.
\endgroup
{On the localized integration region,}
\begin{align}
	\label{fff}
\frac{{\Omega_{\varkappa}}(u,\alpha)}{|\alpha|^{d+2s}}\lesssim \frac{\langle v\rangle^{1-\kappa}}{[\alpha]_v^{d+2s}},
\qquad {|u-v|\leq1,\quad |\alpha|<1}.
\end{align}
{After \eqref{fff}, positivity allows us
to remove $\mathbf1_{|u-v|\leq1}$ and enlarge the jump range to
$[\alpha]_v\geq\tilde\sigma_v^{1/(2s)}$.
With $\alpha=\mathcal O_v\beta$ and effective time $\tilde r_v$,
Lemma~\ref{lemconvopp} applies with a comparable cutoff and gives
\eqref{convott}.}
{Moreover,}
\[
\det\mathcal O_v=\langle v\rangle^{-1},
			\qquad
			\dif r
			=\langle v\rangle^\kappa\dif\tilde r_v.
\]
\begingroup
The jump and time Jacobians cancel the coefficient in \eqref{fff}:
$\langle v\rangle^{1-\kappa}\dif\alpha\,\dif r
=\dif\beta\,\dif\tilde r_v$ under $\alpha=\mathcal O_v\beta$.
This proves the shifted tail case of
\eqref{eq:decorated-convolution-tail}.
\endgroup
\end{proof}

{We next estimate the error terms generated by} $
	\bigl(\mathcal L_g^{<,\mathfrak q}
	-\mathcal L_g^{<}\bigr)$ {and} $\mathcal C_{\mathfrak q}$.

{Using the H\"older bound \eqref{z27} when the distance is at
most one and applying the pointwise bound \eqref{z26} at the two endpoints
otherwise, we obtain}
\begin{align}\label{eq:globalized-coefficient-bound}
	&\left|
	\mathbf C_g(r,Z_{\mathfrak q}(r),\alpha)
	-\mathbf C_g(r,\zeta,u,\alpha)
	\right|\notag\\
	&\quad\lesssim
	\left(
	\big|Z_{\mathfrak q}(r)-(\zeta,u)\big|^{b_0}
	\wedge 1\right)
	\bigl({\Omega_{\varkappa}}(v_0,\alpha)+{\Omega_{\varkappa}}(u,\alpha)\bigr).
\end{align}
{Define}
\[
\mathscr A_{\mathfrak q}F=
		\chi_{\mathfrak q}
		(\mathcal L_g^{<,\mathfrak q}-\mathcal L_g^{<})F
		+\mathcal C_{\mathfrak q}[F].
\]
	        {Throughout this subsection, assume that}
$
	s<b_0<\min\{1,2s\},
$
{and set}
\begin{align*}
	\beta_0:=\frac{b_0}{1+2s}.
\end{align*}
{The following lemma controls the localized
freezing and cutoff errors.}
\begin{lemma}
	\label{lem:localized-weighted-parametrix-error}
{Choose $0<c_0<1$ and
$0<r_0\leq c_0/(48C_*)$, where $C_*$ is the constant in
Lemma~\ref{lem:separated-frozen-kernel}. Take $T$ within the time range
of that lemma for $(\delta,k)=(c_0/16,2)$, with
$T\leq r_0^{2s}$ and $T^{1/(2s)}\leq c_0/8$.
The estimate below is restricted to terminal points satisfying
$|(x_0-(t_0-t)v-x,v_0-v)|\leq c_0/2$.}
		{Suppose that}
	\begin{align}
		|F(t,x,v;\tau,y,w)|+{[F]^{\mathrm{int}}_{v,b_0}}(t,x,v;\tau,y,w)
		\lesssim
		\bK_{v,N}
		(t-\tau,x-y-(t-\tau)w,v-w).
		\label{eq:localized-bootstrap-assumption}
	\end{align}
 {Then}
	\begin{align}
		&(|\mathcal{G}_{\mathfrak{q}}^{(0)}\oast \mathscr A_{\mathfrak q}F|+{[\mathcal{G}_{\mathfrak{q}}^{(0)}\oast \mathscr A_{\mathfrak q}F]^{\mathrm{int}}_{v,b_0}}+{[\mathcal{G}_{\mathfrak{q}}^{(0)}\oast \mathscr A_{\mathfrak q}F]^{\mathrm{int}}_{x,\beta_0}})(t,x,v;\tau,y,w)
		\notag\\
		&\quad\qquad\lesssim_N
		\left(
		c_0^{b_0}
		+c_0^{-1}\sigma^\frac{b_0}{2s}
		+{C_{N,c_0,r_0}\sigma^2}
		\right)
		\bK_{v,N}(\sigma,x-y-\sigma w,v-w),
		\label{eq:localized-weighted-error}
	\end{align}
    where $\sigma=t-\tau$.
\end{lemma}
	\begin{proof}
			{We use the notation of
			Lemma~\ref{lem:decorated-weighted-convolution} and may assume}
		{$0<T\leq\min\{1,r_0^{2s}\}$ and
$0<r_0\leq c_0/(48C_*)$. The coefficient-error constants are uniform
in $r_0$ on this time range: the normalized unit jumps lie within the
physical cutoff and retain uniform coercivity.}
			{The stronger relation between} \(T\) {and}
			\(c_0\) {needed to identify the localized and unlocalized
			terminal seminorms will be imposed explicitly in the bootstrap argument.
			Define}
	\begin{align*}
		\mathcal B(r,\zeta,u,\alpha)
		:=
		\chi_{\mathfrak q}(r,\zeta,u)
		(\mathbf C_g(r,Z_{\mathfrak q}(r),\alpha)
		-\mathbf C_g(r,\zeta,u,\alpha)).
	\end{align*}
	{On the support of the cutoff,}
	\begin{align*}
		|u-v_0|+\left|
		\zeta-\bigl(x_0-(t_0-r)u\bigr)
		\right|\lesssim c_0.
	\end{align*}
	{Because} \(\sigma\le1\), {\eqref{z27} and the local
	comparison of} \({\Omega_{\varkappa}}\) {give}
	\begin{align}
		|	\mathcal B(r,\zeta,u,\alpha)|
		\lesssim
		c_0^{b_0}{\Omega_{\varkappa}}(u,\alpha).
		\label{eq:localized-coefficient-smallness}
	\end{align}
	{Moreover,}
	\begin{align*}
		\delta_\alpha^u\mathcal B
		&=
		(\mathbf C_g(r,Z_{\mathfrak q}(r),\alpha)
		-\mathbf C_g(r,\zeta,u,\alpha))
		\delta_\alpha^u\chi_{\mathfrak q}
	-
		\chi_{\mathfrak q}(r,\zeta,u-\alpha)
		\delta_\alpha^u\mathbf C_g(r,\zeta,u,\alpha).
	\end{align*}
		{Equations \eqref{z26}--\eqref{z27} then imply}
	\begin{align}
		|\delta_\alpha^u	\mathcal B(r,\zeta,u,\alpha)|
		\lesssim
		\left(
		c_0^{b_0-1}|\alpha|+|\alpha|^{b_0}
		\right)
		\bigl(
		{\Omega_{\varkappa}}(u,\alpha)+{\Omega_{\varkappa}}(u-\alpha,\alpha)
		\bigr).
		\label{eq:localized-coefficient-difference}
	\end{align}
	{The coefficient} \(\mathcal B(u,\alpha)\) {is even
	in} \(\alpha\). {Applying the discrete integration-by-parts
	identity from \eqref{z18} first on symmetric truncations therefore gives}
	\begin{align}
		&\int_{\mathbb R^d}
		\int_{|\alpha|<r_0}
		\mathcal G_{\mathfrak q}^{(0)}(u)\mathcal B(u,\alpha)\delta_\alpha^uF(u)
		\frac{\dif\alpha}{|\alpha|^{d+2s}}\dif u
		\notag\\
		&\qquad=
		\frac12
		\int_{\mathbb R^d}
		\int_{|\alpha|<r_0}
		\delta_\alpha^u(\mathcal G_{\mathfrak q}^{(0)}\mathcal B)(u,\alpha)
		\delta_\alpha^uF(u)
		\frac{\dif\alpha}{|\alpha|^{d+2s}}\dif u.
		\label{eq:exact-nonlocal-ibp}
	\end{align}
	{The resulting product difference is}
	\begin{align}
		\delta_\alpha^u(\mathcal G_{\mathfrak q}^{(0)}\mathcal B)
		={}&
		\mathcal B(u,\alpha)\delta_\alpha^u		\mathcal G_{\mathfrak q}^{(0)}(u)
		+
		\mathcal G_{\mathfrak q}^{(0)}(u-\alpha)\delta_\alpha^u	\mathcal B(u,\alpha).
		\label{eq:expanded-nonlocal-ibp}
	\end{align}
	
	{For} \([\alpha]_v\le \tilde \sigma_v^{1/(2s)}\),
	{the normalized finite-difference estimates yield}
	\begin{align*}
		|\delta_\alpha^u		\mathcal G_{\mathfrak q}^{(0)}|
		&\lesssim
		\mathfrak r_{a_1,v}(\alpha)K_1^\sharp(\alpha),\\
		|\delta_\alpha^uF|
		&\lesssim
		\mathfrak r_{a_2,u}(\alpha)^{b_0}K_2^\sharp(\alpha).
	\end{align*}
	{By \eqref{eq:localized-coefficient-smallness} and
	\eqref{eq:decorated-convolution-leading}, the first term in
	\eqref{eq:expanded-nonlocal-ibp} is bounded by}
	\begin{align*}
		C_Nc_0^{b_0}K_0.
	\end{align*}
	{This leading coefficient error has no small-time factor; its
	smallness instead comes from the localization radius} \(c_0\).
	
	{For the second term in \eqref{eq:expanded-nonlocal-ibp}, use}
	$|\alpha|\leq|\alpha|^{b_0}$ {and}
	$c_0^{b_0-1}\leq c_0^{-1}$. {Equations
	\eqref{eq:localized-coefficient-difference} and
	\eqref{eq:decorated-convolution-coefficient} then give}
	\begin{align*}
		C_Nc_0^{-1}\sigma^\frac{b_0}{2s}K_0.
	\end{align*}
	{For} \(\tilde \sigma_v^{1/(2s)}<[\alpha]_v\),
	{no integration by parts is required. Using
	\eqref{eq:localized-coefficient-smallness},}
	\begin{align*}
		|\delta_\alpha^uF|
		\le
		|F(u)|+|F(u-\alpha)|,
	\end{align*}
{and \eqref{eq:decorated-convolution-tail}, we obtain}
	\begin{align*}
		C_Nc_0^{b_0}K_0.
	\end{align*}
	{To estimate the terminal finite differences, apply the
	difference to the left frozen kernel. When} $a_1$
	{exceeds the increment scale, use the first normalized derivative
	from Lemma~\ref{lem:frozen-small-jump-weighted}; otherwise, use the sum of
	the two undifferenced kernels. After the scaling} $a_1=\sigma\rho$,
	{the endpoint estimates become}
	\begin{equation}
			\begin{aligned}
			&\int_0^1
			\rho^{-1+\frac{b_0}{2s}}
			(1-\rho)^{-\frac{b_0}{2s}}
			\left(1\wedge
			\ell_{\sigma,v}^{v}(h)\rho^{-\frac1{2s}}\right)
			\dif\rho
			\lesssim
			\ell_{\sigma,v}^{v}(h)^{b_0},\\
			&\int_0^1
			\rho^{-1+\frac{b_0}{2s}}
			(1-\rho)^{-\frac{b_0}{2s}}
			\left(1\wedge
			\ell_{\sigma,v}^{x}(h)\rho^{-1-\frac1{2s}}\right)
			\dif\rho
			\lesssim
			\ell_{\sigma,v}^{x}(h)^{\frac{b_0}{1+2s}}.
		\end{aligned}
	\end{equation}
	{Split the first integral at}
	$\rho=\ell_{\sigma,v}^{v}(h)^{2s}$ {and the second at}
	$\rho=\ell_{\sigma,v}^{x}(h)^{2s/(1+2s)}$. {The same argument
	applies to the shifted frozen profiles in
	Lemma~\ref{lem:decorated-weighted-convolution}.}
	{For $\ell_{\sigma,v}^{\diamond}(h)\le1$, the normalized endpoint
shift is bounded. Local reference-velocity comparison and
\eqref{eq:A10-profile-moderation} therefore compare the shifted and
unshifted $\bK$ profiles uniformly within this intrinsic range.
Taking the increment suprema therefore gives}
	\begin{align}
		&\left|\mathcal G_{\mathfrak q}^{(0)}
		\oast
		\chi_{\mathfrak q}
		(\mathcal L_g^{<,\mathfrak q}-\mathcal L_g^{<})F\right|
		+
		\left[\mathcal G_{\mathfrak q}^{(0)}
		\oast
		\chi_{\mathfrak q}
		(\mathcal L_g^{<,\mathfrak q}-\mathcal L_g^{<})F
		\right]^{{\mathrm{int}}}_{v,b_0}
		\notag\\
		&\quad+
		\left[\mathcal G_{\mathfrak q}^{(0)}
		\oast
		\chi_{\mathfrak q}
		(\mathcal L_g^{<,\mathfrak q}-\mathcal L_g^{<})F
		\right]^{{\mathrm{int}}}_{x,\beta_0}
		\notag\\
		&\qquad
		\lesssim_N
		\left(
		c_0^{b_0}+c_0^{-1}\sigma^\frac{b_0}{2s}
		\right)K_0.
		\label{eq:localized-coefficient-error-final}
	\end{align}
	
	{It remains to estimate the cutoff commutator. Recall that}
	\begin{align}
		\mathcal C_{\mathfrak q}[F]
		&=
		\int_{|\alpha|<r_0}
		\mathbf C_g(r,Z_{\mathfrak q}(r),\alpha)
		\delta_\alpha^u\chi_{\mathfrak q}(r,\zeta,u)
		F(r,\zeta,u-\alpha)
		\frac{\dif\alpha}{|\alpha|^{d+2s}}\notag\\
		&=
		F\,\mathcal L_g^{<,\mathfrak q}
		\chi_{\mathfrak q}
		-
		\int_{|\alpha|<r_0}
		\mathbf C_g(r,Z_{\mathfrak q}(r),\alpha)
		\delta_\alpha^u\chi_{\mathfrak q}
		\delta_\alpha^uF
		\frac{\dif\alpha}{|\alpha|^{d+2s}}.
		\label{eq:commutator-decomposition}
	\end{align}
	Symmetry implies
	\begin{align*}
		\left|
		\mathcal L_g^{<,\mathfrak q}
		\chi_{\mathfrak q}
		\right|
		\lesssim c_0^{-2s}.
	\end{align*}
\begingroup
If a cutoff difference is nonzero, at least one of its endpoints lies
outside the core of $\chi_{\mathfrak q}$. Since $|\alpha|<r_0\ll c_0$
and the terminal point is in the inner region, both arguments of the
left frozen kernel satisfy $R_{a_1}(X_1,V_1)\geq c_0/16$.
Indeed, the difference of the two cutoff coordinates is
$(X_1+(t_0-t)V_1,V_1)$; a jump changes those coordinates by at most
$2r_0$. The same separation holds at the admissible shifted terminal
points by the imposed geometric scale condition.

Lemma~\ref{lem:separated-frozen-kernel} therefore supplies $a_1^2$
for the left kernel and its normalized derivatives. For the order-zero
commutator, use $|\delta_\alpha\chi_{\mathfrak q}|
\lesssim c_0^{-1}|\alpha|$ in the small intrinsic range and the
one-tail bridge in the complementary range. Together with the
unshifted convolution for $F\mathcal L_g^{<,\mathfrak q}\chi_{\mathfrak q}$,
these give $C_{N,c_0,r_0}\sigma^2K_0$.
For either terminal increment, put $b_v=b_0$, $b_x=\beta_0$ and use
\[
a_1^2\bigl(1\wedge\ell_{a_1,v}^{\diamond}(h)\bigr)
\leq\sigma^2\ell_{\sigma,v}^{\diamond}(h)^{b_\diamond},
\qquad \diamond\in\{v,x\},
\]
because $b_v/(2s)=b_x(1+1/(2s))=b_0/(2s)<1$.
The same positive convolution estimates then control both seminorms.
Combining the commutator bound with
\eqref{eq:localized-coefficient-error-final} proves
\eqref{eq:localized-weighted-error}.
\par\endgroup
\end{proof}
\subsubsection{Small-jump bounds and source-variable regularity}

\begingroup
Throughout this subsection, H\"older estimates use the intrinsic
seminorms \eqref{eq:def-intrinsic-terminal-holder}. Their conversion to
the physical seminorms is given by Lemma~\ref{lem:A34-physical-transfer}.
\par\endgroup
{Fix} $N\ge0$ {and define}
\begin{align*}
	\mathfrak M_N(T)
	:=
	\sup_{\substack{0<\tau<t\le T\\x,v,y,w\in\mathbb R^d}}
	\frac{
	(|\G^{(0)}|+{[\G^{(0)}]^{\mathrm{int}}_{v,b_0}}+{[\G^{(0)}]^{\mathrm{int}}_{x,\beta_0}})(t,x,v;\tau,y,w)
	}{
		\bK_{v,N}
		(t-\tau,x-y-(t-\tau)w,v-w)
	}.
\end{align*}
{For a fixed reference point}
$\mathfrak q=(t_0,x_0,v_0)$, {define the inner kinetic neighborhood}
\begin{align*}
	\mathscr B_{\mathfrak q}^{\mathrm{in}}
	:=
	\left\{
	(t,x,v):
	\left|
	\left(
	x_0-(t_0-t)v-x,\,
	v_0-v
	\right)
	\right|
	\leq \frac{c_0}{2}
	\right\}.
\end{align*}
{We impose the following geometric scale condition in the
bootstrap:}
\begin{equation}\label{eq:bootstrap-cutoff-scale}
	0<T\le1,
	\qquad
	T^{1/(2s)}\le\frac{c_0}{8}.
\end{equation}
{If} $(t,x,v)\in\mathscr B_{\mathfrak q}^{\mathrm{in}}$,
{then} $\chi_{\mathfrak q}(t,x,v)=1$ {and}
$|v-v_0|\leq c_0/2$. {The local profile-comparison estimate
therefore gives}
\begin{align*}
	\bK_{v_0,N}(\sigma,X,V)
	\sim
	\bK_{v,N}(\sigma,X,V),
	\qquad
	\sigma=t-\tau.
\end{align*}
{All constants in the preceding localized estimates are uniform in}
$\mathfrak q$, {and the family}
$\{\mathscr B_{\mathfrak q}^{\mathrm{in}}\}_{\mathfrak q}$
{covers the entire phase space. Consequently,}
\begin{align*}
	\mathfrak M_N(T)
	\sim 
	\sup_{\mathfrak q}
	\sup_{\substack{
			0<\tau<t\leq t_0\leq T\\
			(t,x,v)\in\mathscr B_{\mathfrak q}^{\mathrm{in}}\\
			y,w\in\mathbb R^d}}
	\frac{
	(|\G^{(0)}|+{[\G^{(0)}]^{\mathrm{int}}_{v,b_0}}+{[\G^{(0)}]^{\mathrm{int}}_{x,\beta_0}})(t,x,v;\tau,y,w)
	}{
		\bK_{v,N}
		(t-\tau,x-y-(t-\tau)w,v-w)
	}.
\end{align*}
{Recall the Duhamel formula \eqref{eq:global-small-duhamel}:}
\begin{align*}
	\G^{(0)}\chi_{\mathfrak{q}}
	=\mathcal G_{\mathfrak q}^{(0)}\chi_{\mathfrak{q}}
	+\mathcal G_{\mathfrak q}^{(0)}
	\oast\left\{ \chi_{\mathfrak{q}}
	\bigl(\mathcal L_g^{<,\mathfrak q}
	-\mathcal L_g^{<}\bigr)\G^{(0)}+\mathcal C_{\mathfrak q}[\G^{(0)}]\right\}.
\end{align*}

{We next verify that the localized identity controls both
unlocalized terminal seminorms. Set}
\[
	\Xi_{\mathfrak q}(t,x,v)
	:=
	\bigl(x_0-(t_0-t)v-x,\,v_0-v\bigr).
\]
If \((t,x,v)\in\mathscr B_{\mathfrak q}^{\mathrm{in}}\), then
\(|\Xi_{\mathfrak q}(t,x,v)|\le c_0/2\).  An admissible increment in
{the velocity seminorm in
\eqref{eq:def-intrinsic-terminal-holder}} satisfies
\[
	{|h|\le[h]_v
	\le\tilde\sigma_v^{1/(2s)}\le T^{1/(2s)}.}
\]
Since \(0\le t_0-t\le T\le1\), \eqref{eq:bootstrap-cutoff-scale} gives
\begin{align*}
	\left|\Xi_{\mathfrak q}(t,x,v-h)\right|
	&\le
	\left|\Xi_{\mathfrak q}(t,x,v)\right|
	+\left|\bigl((t_0-t)h,h\bigr)\right|\\
	&\le \frac{c_0}{2}+2T^{1/(2s)}
	\le\frac{3c_0}{4}<c_0.
\end{align*}
Likewise, an admissible increment in
{the spatial seminorm in
\eqref{eq:def-intrinsic-terminal-holder}} satisfies
\[
	{|h|\le[h]_v
	\le\sigma\tilde\sigma_v^{1/(2s)}\le T^{1/(2s)},}
\]
and hence
\[
	\left|\Xi_{\mathfrak q}(t,x-h,v)\right|
	\le\frac{c_0}{2}+T^{1/(2s)}
	\le\frac{5c_0}{8}<c_0.
\]
{Thus the base point and both shifted terminal points lie in the
region where} \(\chi_{\mathfrak q}=1\). {Consequently, every
increment appearing in the two normalized seminorms satisfies}
\[
	\delta_h^v\bigl(\G^{(0)}\chi_{\mathfrak q}\bigr)
	=\delta_h^v\G^{(0)},
	\qquad
	\delta_h^x\bigl(\G^{(0)}\chi_{\mathfrak q}\bigr)
	=\delta_h^x\G^{(0)}.
\]
{The restriction} \(t\le t_0\) {in the localized
supremum causes no loss because, for each terminal point, we may take}
\(\mathfrak q=(t,x,v)\). {These neighborhoods therefore cover the
entire terminal phase space. By \eqref{equivGH} and
Lemma~\ref{lem:frozen-small-jump-weighted}, the frozen term has the required
bound, while Lemma~\ref{lem:localized-weighted-parametrix-error} controls the
remainder. Taking the supremum first over each}
$\mathscr B_{\mathfrak q}^{\mathrm{in}}$ {and then over}
$\mathfrak q$ {gives}
\begin{align*}
	\mathfrak M_N(T)
	\le C_N+C_N\left(
		c_0^{b_0}
		+c_0^{-1}T^\frac{b_0}{2s}
		+{C_{N,c_0,r_0}T^2}
		\right)\mathfrak M_N(T).
\end{align*}
{Let} \(C_N\ge1\) {be the constant in the preceding
inequality, {enlarged to cover the auxiliary
terminal estimate in Proposition~\ref{propsmalljump}.} First choose}
\begin{equation}\label{eq:bootstrap-choice-c0}
	0<c_0<
	{\min\left\{1,\,(6C_N)^{-1/b_0}\right\}}.
\end{equation}
{Next fix $0<r_0\leq c_0/(48C_*)$.
Choose $T_0$ within the time range of
Lemma~\ref{lem:separated-frozen-kernel} for $(\delta,k)=(c_0/16,2)$,
and require}
\begin{equation}\label{eq:bootstrap-choice-T0}
	0<T_0\le
	\min\left\{
	1,\ r_0^{2s},\
	\left(\frac{c_0}{8}\right)^{2s},\
	\left(\frac{c_0}{6C_N}\right)^{2s/b_0},\
	{(6C_NC_{N,c_0,r_0})^{-1/2}}
	\right\}.
\end{equation}
{This choice includes the geometric condition
\eqref{eq:bootstrap-cutoff-scale} and ensures that}
\[
	C_Nc_0^{b_0}\le\frac16,\qquad
	C_Nc_0^{-1}T_0^{b_0/(2s)}\le\frac16,\qquad
	{C_NC_{N,c_0,r_0}T_0^2\le\frac16}.
\]
{Thus the bootstrap coefficient is at most} \(1/2\).
{Whenever the bootstrap quantity is finite,
absorption gives $\mathfrak M_N(T_0)\lesssim_N1$.}
{We regularize the background while preserving
its uniform bounds.}
\begin{lemma}\label{lem:background-regularization}
Let $g\geq0$ satisfy \eqref{lbmass} and \eqref{Ig}, with
$\varkappa>d+2$. There are smooth
$g_\epsilon\geq0$ such that
\[
 \int g_\epsilon(t,x,v)\,\dif v\geq\mathsf m_0,
 \qquad
 \|g_\epsilon\|_{L^\infty_T\mathbf C^{b_0}}
 \leq C\|g\|_{L^\infty_T\mathbf C^{b_0}},
\]
with $C$ independent of $\epsilon\in(0,1)$.
The mass and energy upper bounds and the entropy bound in
\eqref{eq:macro-main} hold uniformly for $g_\epsilon$, with possibly
enlarged structural constants. Moreover, $g_\epsilon\to g$ in $L^p_{\rm loc}$ for every finite $p$.
{The approximations may also be chosen so that
$|\partial_{x,v}^{k}g_\epsilon|\leq C_{k,\epsilon}g_\epsilon$
for every $k\geq0$.}
\end{lemma}
\begin{proof}
\begingroup
Extend $g$ by repeated reflection in time and use a nonnegative product
mollifier of integral one. Choose its spatial and velocity factors
proportional to $e^{-\sqrt{1+|z|^2}}$ and rescale them by $\epsilon$.
Their derivatives are bounded by multiples of the factors themselves.
Integration in $v$ preserves the mass lower bound; the inequality
$\langle v\rangle^\varkappa\lesssim
\langle v-u\rangle^\varkappa\langle u\rangle^\varkappa$
and the finite moments of the mollifier preserve the weighted
H\"older bound. Peetre's inequality, or
$|v+u|^2\leq2|v|^2+2|u|^2$, gives
$\mathsf E_\epsilon\leq2\mathsf E_0+C\mathsf M_0$ uniformly.
Jensen's inequality for the convex function
$\Phi(q)=q\log(1+q)$ and integration in $v$ give
$\int\Phi(g_\epsilon)\,\dif v\leq\mathsf H_0$; the same argument
applies to $q\log q$ if that entropy convention is used.
The derivative estimate and local convergence follow
directly from convolution.
\par\endgroup
\end{proof}

{We now establish the small-jump kernel bounds.}
\begin{proposition}\label{propsmalljump}
{Fix a finite collection of rapid exponents
$N\geq0$, and make the choices of $c_0,r_0$ above for its largest exponent.
For each exponent in this collection, there exists a uniform
elapsed-time threshold} \(T_{\mathrm{sj}}\in(0,1]\) {such that,
whenever} \(0\le\tau<t\le T\) {and}
\(0<t-\tau\le T_{\mathrm{sj}}\),
	    \begin{align}\label{eq:weighted-small-green-final}
(|\G^{(0)}|+{[\G^{(0)}]^{\mathrm{int}}_{v,b_0}}+{[\G^{(0)}]^{\mathrm{int}}_{x,\beta_0}})(t,x,v;\tau,y,w)
	\lesssim_N
	\bT_v(\sigma,X,V)
	\left\langle\frac{|X|}{\sigma}+|V|\right\rangle^{-N},
\end{align}
{where}
\begin{align*}
	\sigma:=t-\tau,
	\qquad
	X:=x-y-\sigma w,
	\qquad
	V:=v-w.
\end{align*}
{For a fixed
$b_0<b<\min\{1,2s,(1+2s)b_0\}$, the same bound holds with the
terminal exponents $b$ and $b/(1+2s)$, after making the common
cutoff and time choices for both pairs of exponents.}
\end{proposition}
\begin{proof}
\begingroup
For $g_\epsilon$ from Lemma~\ref{lem:background-regularization},
the positive Carleman representation gives
$|\partial_{x,v}^{k}\mathbf C_{g_\epsilon}|
\leq C_{k,\epsilon}\mathbf C_{g_\epsilon}$.
Retaining all jumps $|\alpha|<r_0$, weighted energy estimates and
a Galerkin limit solve the smooth Cauchy problem with rapidly
decreasing smooth data. Evenness and the second jump moment control
the coefficient and weight commutators, giving finite polynomially
weighted Sobolev norms at every order. The maximum principle preserves
nonnegativity.

Fix $(\tau,y,w)$ and initialize this problem at $\tau$ with
the frozen kernel
$\mathcal G^{(0)}_{(\tau,y+\eta w,w)}(\tau,x,v;\tau-\eta,y,w)$
for $g_\epsilon$, where $\eta>0$ and the background is extended
in time as above.
In the bootstrap use the elapsed time $t-\tau+\eta$ and displacement
$x-y-(t-\tau+\eta)w$, with $t-\tau+\eta\leq T_0$.
Write $F_{\epsilon,\eta}$ for this smooth solution and
$a=t-\tau+\eta\in[\eta,T_0]$, $X_a=x-y-aw$, $V=v-w$.
We justify finiteness before taking any supremum over the source.
For fixed $\epsilon,\eta,y,w$, the weighted Sobolev estimates and
Sobolev embedding in $2d$ variables imply, for every $M$,
\[
 \sup_{\tau\leq t\leq\tau+T_0-\eta}
 \|\langle x,v\rangle^M F_{\epsilon,\eta}(t)\|_{C^1_{x,v}}
 <\infty.
\]
Indeed, take more than $d+1$ derivatives in the energy estimate; the
initial frozen density is smooth and has every polynomial moment
by Lemma~\ref{lem:frozen-small-jump-weighted} at elapsed time $\eta$.
The coefficient derivative bounds above allow the energy estimate
with each such finite weight and derivative order.

The denominator also has an explicit polynomial lower bound. Since
$\mathcal N_v(A,B)\gtrsim
\langle[A]_v+[B]_v\rangle^{-2(d+2s)}$,
$[z]_v\leq2\langle v\rangle|z|$, and $a\geq\eta$, definition
\eqref{eq:def-weighted-small-envelope} gives
\begin{align}\label{eq:bootstrap-denominator-lower}
 \bK_{v,N}(a,X_a,V)&\geq c_{\eta,y,w,N}
 \langle x,v\rangle^{-Q_N},\\
 Q_N&:=N+2(d+2s)\left(2+\frac{\kappa}{2s}\right).\notag
\end{align}
Here the prefactor $a^{-d-d/s}\langle v\rangle^{\kappa d/s+2}$
is at least one, and both normalized profile arguments are bounded
by $C_{\eta,y,w}\langle x,v\rangle^{2+\kappa/(2s)}$.
Both intrinsic increment lengths are at most one. The fundamental
theorem of calculus and comparability of weights on unit segments
therefore bound the numerator, including its two intrinsic seminorms,
by $C_{\epsilon,\eta,y,w,M}\langle x,v\rangle^{-M}$.
Take $M>Q_N$. The quotient supremum over $(t,x,v)$ is now finite;
its preliminary bound may depend on $\epsilon,\eta,y,w$.
Frozen subconvolution bounds the initial term uniformly, placing
differences on the factor with larger elapsed time.
The error integrals are restrictions of those already estimated.
For each fixed source, the resulting finite quotient $A_{\epsilon,\eta,y,w}$
satisfies
\[
 A_{\epsilon,\eta,y,w}\leq C_N+
 C_N\left(c_0^{b_0}+c_0^{-1}T_0^{b_0/(2s)}
 +C_{N,c_0,r_0}T_0^2\right)A_{\epsilon,\eta,y,w}.
\]
The coefficient is at most $1/2$ by
\eqref{eq:bootstrap-choice-c0}--\eqref{eq:bootstrap-choice-T0}.
Absorb first to obtain $A_{\epsilon,\eta,y,w}\leq2C_N$, and only
then take the supremum over the source and approximation parameters.
Thus no uniform a priori finiteness assumption is used.

The same argument yields the auxiliary terminal exponents:
use $b$ in the leading difference product, while the coefficient
difference retains exponent $b_0$. Its capped moment is bounded by
\eqref{eq:global-Omega-capped-moment}, since
$\mathfrak r^b\leq\mathfrak r^{b_0}$.
In the coefficient-increment term the jump difference falls on the
input, so the frozen terminal endpoint has the stronger gains
$\min\{1,2s\}$ and $2s/(1+2s)$, strictly above the requested exponents.
Choose $c_0,r_0$ and the time threshold once for both estimates.

First let $\eta\downarrow0$, and then $\epsilon\downarrow0$.
Local compactness and the weak equation give \eqref{eq:Green-small};
the time-integrable first-moment bound in the weak formulation
preserves the delta trace. The estimates pass to the limit by Fatou's
lemma. The localized Duhamel estimate, applied to a difference with
zero initial data, gives uniqueness in this kernel class and hence
the evolution identity.
\par\endgroup
\end{proof}

\begingroup
The next lemma converts the intrinsic bounds to physical seminorms.
\par\endgroup
\begin{lemma}
\label{lem:A34-physical-transfer}
\begingroup
Set $\sigma=t-\tau$ and
$\Lambda_*=(2(d+2s)+1)(1+\kappa/(2s))$.
Suppose a
nonnegative envelope $E$ obeys
\[
 E(t,x^{\diamond,h},v^{\diamond,h};\tau,y,w)
 \le C\langle v\rangle^{2(d+2s)(1+\kappa/(2s))}E(t,x,v;\tau,y,w)
 \quad\text{when }\ell_{\sigma}^{\diamond}(h)\le1,
 \qquad\diamond\in\{x,v\},
\]
where $(x^{x,h},v^{x,h})=(x-h,v)$ and
$(x^{v,h},v^{v,h})=(x,v-h)$.
If $|F|+[F]^{\mathrm{int}}_{v,b_0}
+[F]^{\mathrm{int}}_{x,\beta_0}\le A E$ at every terminal point,
then
\begin{equation}\label{eq:A34-physical-transfer}
 |F|+[F]_{v,b_0}+[F]_{x,\beta_0}
 \le C A\langle v\rangle^{\Lambda_*}E.
\end{equation}
The envelope $\bK_{v,N}(\sigma,X,V)$ has this shift property
for fixed $N$ and $0<\sigma\le1$.
\par\endgroup
\end{lemma}
\begin{proof}
\begingroup
Since $|h|\le[h]_v\lesssim\langle v\rangle|h|$ and
$\tilde\sigma_v=\sigma\langle v\rangle^{-\kappa}$,
\[
 \ell_{\sigma,v}^{\diamond}(h)
 \lesssim\langle v\rangle^{1+\kappa/(2s)}\ell_\sigma^{\diamond}(h).
\]
For $\ell_{\sigma,v}^{\diamond}(h)\le1$, the intrinsic seminorm
therefore bounds the physical quotient by
$CA\langle v\rangle^{b_\diamond(1+\kappa/(2s))}E$,
where $b_v=b_0$ and $b_x=\beta_0$.
For $\ell_{\sigma,v}^{\diamond}(h)>1$ but
$\ell_\sigma^{\diamond}(h)\le1$, one has
$\ell_\sigma^{\diamond}(h)^{-b_\diamond}
\lesssim\langle v\rangle^{b_\diamond(1+\kappa/(2s))}$.
Use the two pointwise bounds for $F$ and the assumed envelope
comparison at the shifted endpoint.  This gives
$CA\langle v\rangle^{(2(d+2s)+b_\diamond)(1+\kappa/(2s))}E$.
Taking the supremum proves \eqref{eq:A34-physical-transfer}.

For $\bK_{v,N}$, $|h|\le1$ makes $v$ and $v-h$ comparable,
and the unscaled rapid factor changes by a bounded factor.
In the scaled $\mathcal N$ profile, the displacement is at most
$C\langle v\rangle^{1+\kappa/(2s)}$; apply
\eqref{eq:A10-profile-moderation}.
\par\endgroup
\end{proof}

The next step concerns the source variables $(y,w)$. For fixed $(t,x,v)$,
the backward equation is
\begin{align}
	-\partial_\tau \G^{(0)}
	-w\cdot\nabla_y\G^{(0)}
	+\mathcal L_g^{<,*}\G^{(0)}
	=0,
	\qquad \tau<t,
	\label{eq:backward-adjoint-small-green}
\end{align}
{where, by the evenness of} $\mathbf C_g(\tau,y,w,\alpha)$
{in} $\alpha$,
\begin{align}
	\mathcal L_g^{<,*}\phi(\tau,y,w)
	:=
	\operatorname{p.v.}\int_{|\alpha|\leq r_0}
	\Bigl[
	\mathbf C_g(\tau,y,w,\alpha)\phi(y,w)
	-
	\mathbf C_g(\tau,y,w-\alpha,\alpha)
	\phi(y,w-\alpha)
	\Bigr]
	\frac{\dif\alpha}{|\alpha|^{d+2s}}.
	\label{eq:adjoint-full-small-jump}
\end{align}
The placement of the coefficient in \eqref{eq:adjoint-full-small-jump}
differs from that in the forward operator. Frozen symmetry therefore does not
by itself transfer the terminal-variable H\"older estimate to the source
variables.

For a precise weak formulation, fix $(\tau,y)$ and write
$C(w,\alpha)=\mathbf C_g(\tau,y,w,\alpha)$.
Let $\overline C(\alpha)$ be its coefficient frozen along the chosen
characteristic, and truncate both jump integrals to
$\varepsilon<|\alpha|\leq r_0$. Denote the resulting operators by
$L_\varepsilon$ and $\overline L_\varepsilon$. For smooth compactly supported
$\phi,\psi$, a change of variables gives
\begin{equation}\label{eq:weak-adjoint-freezing-error}
\begin{aligned}
\langle(L_\varepsilon^*-\overline L_\varepsilon^*)\phi,\psi\rangle
=\int_{\mathbb R^d}\int_{\varepsilon<|\alpha|\leq r_0}
 &(C(w,\alpha)-\overline C(\alpha))\phi(w)\\[-3pt]
 &\times(\psi(w)-\psi(w+\alpha))
 \frac{\dif\alpha\,\dif w}{|\alpha|^{d+2s}}.
\end{aligned}
\end{equation}
\begingroup
Symmetrizing in $\alpha$ before passing to the limit gives
\begin{equation}\label{eq:adjoint-symmetrized-freezing}
\begin{aligned}
\langle(L_\varepsilon^*-\overline L_\varepsilon^*)\phi,\psi\rangle
&=\frac12\iint_{\varepsilon<|\alpha|\leq r_0}
\delta_\alpha^w\bigl((C-\overline C)\phi\bigr)
\delta_\alpha^w\psi\,
\frac{\dif\alpha\,\dif w}{|\alpha|^{d+2s}},\\
\delta_\alpha^w\bigl((C-\overline C)\phi\bigr)
&=(C-\overline C)\delta_\alpha^w\phi
+\phi(w-\alpha)\delta_\alpha^wC.
\end{aligned}
\end{equation}
Localization as in Lemma~\ref{lem:localized-weighted-parametrix-error}
gives bootstrap coefficient $Cc_0^{b_0}$ for the first product in
\eqref{eq:adjoint-symmetrized-freezing}. In the second product,
the frozen test kernel supplies a
first-order difference; $1+b_0>2s$ makes the jump integral convergent,
and its time integral is bounded by $C c_0^{-1}\sigma^{b_0/(2s)}$.
The complementary intrinsic tail contributes $Cr_0^{b_0}$,
which is absorbed by $Cc_0^{b_0}$, while separated cutoff terms
are bounded by $C_{N,c_0,r_0}\sigma^2$.
Terminal differences of the frozen test kernel and the same endpoint
integrals control the two backward seminorms.

Reverse time and reflect velocity. The displacement becomes
$(-X+\sigma V,V)$, whose reference profile is comparable to
$\bK_{w,N}(\sigma,X,V)$.
More explicitly, for fixed $(t,x,v)$ define
\[
 A_N^*(h;t,x,v):=\sup_{0<t-\tau\leq h}\sup_{y,w}
 \frac{|\G^{(0)}|+[\G^{(0)}]^{\mathrm{int}}_{w,b_0}
 +[\G^{(0)}]^{\mathrm{int}}_{y,\beta_0}}
 {\bK_{w,N}(t-\tau,x-y-(t-\tau)w,v-w)}.
\]
For the regularized adjoint equation, start from a frozen density
at a positive lag $\eta$ at terminal time and replace $t-\tau$ by
$t-\tau+\eta$. Its quotient is finite for each fixed
$\epsilon,\eta,t,x,v$: weighted adjoint energy estimates give the
same $C^1$ decay as above, and
\eqref{eq:bootstrap-denominator-lower} applies with source and terminal
variables interchanged after velocity reflection. The adjoint energy
commutators are bounded by
$|\partial_{x,v}^k\mathbf C_{g_\epsilon}|
\leq C_{k,\epsilon}\mathbf C_{g_\epsilon}$ and the second jump moment.

Apply the localized backward Duhamel identity with a characteristic
cutoff. In \eqref{eq:adjoint-symmetrized-freezing}, the product
$(C-\overline C)\delta_\alpha\phi\,\delta_\alpha\psi$
gives $Cc_0^{b_0}A_N^*$; the product
$\phi(w-\alpha)\delta_\alpha C\,\delta_\alpha\psi$ gives
$Cc_0^{-1}h^{b_0/(2s)}A_N^*$ by the radial integral above and the
same terminal endpoint integrals. The tail and separated cutoff terms
give respectively $Cr_0^{b_0}A_N^*$ and
$C_{N,c_0,r_0}h^2A_N^*$. Frozen subconvolution bounds the initial
term by $C_N$, uniformly in the lag. Thus the finite regularized
quotient satisfies
\[
 A_N^*\leq C_N+
 C_N\left(c_0^{b_0}+r_0^{b_0}
 +c_0^{-1}h^{b_0/(2s)}+C_{N,c_0,r_0}h^2\right)A_N^*.
\]
Enlarge the common constant before choosing $c_0,r_0$, and reduce
the common elapsed-time threshold if necessary, so the parenthesis
multiplied by $C_N$ is at most $1/2$. Absorption is now legitimate.
First let $\eta\downarrow0$, then $\epsilon\downarrow0$.
Take the symmetrized identity before either limit; its jump majorant
is integrable because $1+b_0>2s$, and its time majorants are the
integrable endpoint bounds above. Local convergence and Fatou's lemma
pass the pointwise and seminorm bounds to the limit. The limit is the
backward version of the already constructed kernel: duality for smooth
forward and backward solutions identifies their propagators, and
uniqueness in Proposition~\ref{propsmalljump} passes this identity to
the kernel limit. Finally take the supremum over the fixed terminal
point. This proves
\par\endgroup
\begin{align}
&\bigl(|\G^{(0)}|+{[\G^{(0)}]^{\mathrm{int}}_{w,b_0}}+{[\G^{(0)}]^{\mathrm{int}}_{y,\beta_0}}\bigr)
 (t,x,v;\tau,y,w)\notag\\
&\qquad\lesssim_N\bT_w(\sigma,X,V)
 \left\langle |X|/\sigma+|V|\right\rangle^{-N},
\label{eq:weighted-small-green-source-natural}
\end{align}
{The source seminorms are those defined in
\eqref{eq:def-intrinsic-terminal-holder}.}

{The evolution identity gives the following
mixed-difference estimate.}
\begin{lemma}\label{lem:mixed-composition}
Suppose a kernel $K$ has the evolution property
\[
 K(t,z;\tau,z')=\int K(t,z;m,\zeta)K(m,\zeta;\tau,z')\,\dif\zeta,
 \qquad \tau<m<t,
\]
with absolutely convergent products at the four endpoints under consideration.
Then
\begin{equation}\label{eq:mixed-composition}
 \delta_h^z\delta_k^{z'}K(t,z;\tau,z')
 =\int\delta_h^zK(t,z;m,\zeta)
        \delta_k^{z'}K(m,\zeta;\tau,z')\,\dif\zeta.
\end{equation}
In particular, for nonnegative majorants and factors, if the two differences
are bounded by $a_h A_h(\zeta)$ and
$b_k B_k(\zeta)$ and
$\int A_hB_k\,\dif\zeta\leq C E(t,z;\tau,z')$, then the mixed difference
is bounded by $C a_hb_k E(t,z;\tau,z')$.
\end{lemma}
\begin{proof}
Apply the evolution identity at each of the four pairs of shifted endpoints
and subtract. The asserted estimate follows by taking absolute values inside
the integral. Neither factor requires a mixed derivative.
\end{proof}
\par\smallskip
Apply this lemma to the small-jump kernel with $m=(t+\tau)/2$.
A velocity increment with exponent $b_0$, or a spatial increment with exponent
$\beta_0=b_0/(1+2s)$, changes its normalized factor by
$2^{b_0/(2s)}$ when the elapsed time is halved. Thus the product loses only
$2^{b_0/s}$, in addition to the explicitly tracked envelope weights.
The one-sided terminal and adjoint estimates and a weighted convolution bound
for their shifted envelopes imply the corresponding mixed estimate via
\eqref{eq:mixed-composition}. For increments exceeding a half-interval scale,
use the pointwise bound for that factor and retain the other difference;
the associated normalized factor is bounded below by a structural constant.
{The evolution property used here follows
from the uniqueness argument in Proposition~\ref{propsmalljump}.}

{To express this estimate in the envelope from
\eqref{eq:weighted-small-green-final}, we use the global profile comparison.
For some structural constant} $M_{\mathrm{pc}}\geq0$,
\begin{align}
	\bT_w(\sigma,X,V)
	\lesssim
	\langle V\rangle^{M_{\mathrm{pc}}}
	\bT_v(\sigma,X,V).
	\label{eq:source-target-profile-comparison}
\end{align}

{Apply \eqref{eq:weighted-small-green-source-natural} with}
$N+M_{\mathrm{pc}}$ {and use}
\begin{align*}
	\langle V\rangle
	\leq
	\left\langle
	\frac{|X|}{\sigma}+|V|
	\right\rangle
\end{align*}

{to obtain the following source-variable estimate.}

	\begin{proposition}
		\label{prop:small-jump-source-holder}
			{For each requested rapid exponent
			$N$, include $N+M_{\mathrm{pc}}$ in the common cutoff choices
			of Proposition~\ref{propsmalljump}. There exists}
			\(T_{\mathrm{src}}\in(0,1]\) {such that, whenever}
			\(0\le\tau<t\le T\) {and}
			\(0<t-\tau\le T_{\mathrm{src}}\),
	\begin{align}
		&\bigl(
		|\G^{(0)}|
		+{[\G^{(0)}]^{\mathrm{int}}_{w,b_0}}
		+{[\G^{(0)}]^{\mathrm{int}}_{y,\beta_0}}
		\bigr)(t,x,v;\tau,y,w)
	\lesssim_N
		\bT_v(\sigma,X,V)
		\left\langle
		\frac{|X|}{\sigma}+|V|
		\right\rangle^{-N},
		\label{eq:weighted-small-green-source}
	\end{align}
		{where}
	\begin{align*}
		\sigma:=t-\tau,
		\qquad
		X:=x-y-\sigma w,
		\qquad
		V:=v-w.
	\end{align*}
\end{proposition}

\subsection{Parametrix expansion}
{{We now construct the Green function of
$\partial_t+v\cdot\nabla_x-\mathcal Q_s(g,\cdot)$ from the small-jump kernel $\G^{(0)}$.} Define}
\begin{align}\label{eq:def-full-perturbation}
	\mathcal B_gF
	:=
	\mathcal L_g^{>}F+\mathcal Q_{s,\mathrm{rem}}(g,F).
\end{align}
{With the sign convention in \eqref{eq:full-green-via-small}, set}
\begin{align}\label{eq:def-large-jump-operator}
	\mathcal L_g^{>}F(t,x,v)
	:=
	\int_{|\alpha|\ge r_0}
	\mathbf C_g(t,x,v,\alpha)
	\bigl(F(t,x,v-\alpha)-F(t,x,v)\bigr)
	\frac{\dif\alpha}{|\alpha|^{d+2s}}.
\end{align}
{On the time interval under consideration, a mild Green kernel}
\(H\) {for the full operator satisfies}
\begin{align}\label{eq:full-green-via-small}
	\partial_tH+v\cdot\nabla_xH+\mathcal L_g^{<}H
	=\mathcal B_gH
\end{align}
{in the distributional sense, with initial trace}
\(\delta(x-y)\delta(v-w)\), {and obeying the mild identity
\eqref{eq:full-duhamel-small} below. We construct this kernel directly.}

{All operators in \eqref{eq:def-full-perturbation} act on the}
$(t,x,v)$ {variables of the kernel, with} $(\tau,y,w)$
{held fixed.}

{Define the Duhamel operator} $\mathscr K$ {by applying}
$\mathcal B_g$ {to the first set of space--velocity variables of} $F$:
\begin{align}\label{eq:def-full-duhamel-operator}
	(\mathscr K F)(t,x,v;\tau,y,w)
	&:=
	\int_\tau^t\iint_{\mathbb R^{2d}}
	\G^{(0)}(t,x,v;r,\zeta,u)
	\Bigl[\mathcal B_g
	\bigl(F(\,\cdot\,;\tau,y,w)\bigr)\Bigr](r,\zeta,u)
	\,\dif\zeta\dif u\dif r.
\end{align}

The mild equation associated with \eqref{eq:full-green-via-small} is
\begin{align}\label{eq:full-duhamel-small}
	H
	=
	\G^{(0)}+\mathscr KH.
\end{align}
{At the end of this subsection, we prove that this equation has a
unique solution in the specified affine parametrix class.}

{Define recursively}
\begin{align}\label{eq:def-full-parametrix-terms}
	\G^{(j+1)}
	&:=
	\mathscr K\G^{(j)},
	\qquad j\ge0.
\end{align}
{Every mild kernel} \(H\) {satisfying
\eqref{eq:full-duhamel-small} obeys, for each integer} \(J\ge0\)
{for which the displayed convolutions are integrable,}
\begin{align}\label{eq:finite-full-expansion}
	H
	=
	\sum_{j=0}^{J}\G^{(j)}+\mathfrak R_H^{(J+1)},
\end{align}
where
\begin{align}\label{eq:exact-full-remainder}
	\mathfrak R_H^{(J+1)}
	=
	\mathscr K^{J+1}H.
\end{align}

		\subsubsection{Velocity-transfer weights and the large-jump tail}

For $N\ge0$, define
\begin{align}\label{eq:def-ratio-weight}
	\varpi_N (\sigma,x,v,w)
	&:=
	\left\langle
	\frac{\langle |x|/\sigma+|v|\rangle}{\langle w\rangle}
	\right\rangle^{-N}.
\end{align}
{The triangle inequality gives}
\begin{align}\label{eq:ratio-weight-characteristic-form}
	\varpi_N(\sigma,x,v,w)
	\sim_N
	\left(
	1+
	\frac{|x-\sigma w|/\sigma+|v-w|}{\langle w\rangle}
	\right)^{-N}.
\end{align}
{We also define the velocity ratio weight}
\begin{align}\label{eq:def-velocity-ratio-weight}
	\varpi_N^{\mathrm v}(v,w)
	:=\varpi_N(1,0,v,w)
	=
	\left\langle
	\frac{\langle v\rangle}{\langle w\rangle}
		\right\rangle^{-N}.
	\end{align}
{We use this weight only at an instantaneous large jump. If}
$a_0=a_1+a_2$,
{then}
\begin{align}\label{eq:ratio-submultiplicative}
	&\varpi_N (a_1,x-\zeta,v,u)
	\varpi_N (a_2,\zeta-y,u,w)
	\lesssim_N
	\varpi_N (a_0,x-y,v,w).
\end{align}
{More generally, for every} $q,u\in\mathbb R^d$,
\begin{align}\label{eq:ratio-submultiplicative-jump}
	&\varpi_N (a_1,x-\zeta,v,u)
	\varpi_N^{\mathrm v}(u,q)
	\varpi_N (a_2,\zeta-y,q,w)
	\lesssim_N
	\varpi_N (a_0,x-y,v,w).
\end{align}
{The same direct comparison yields}
\begin{align}\label{eq:ratio-jump-transfer}
	\varpi_N^{\mathrm v}(u,q)
	\varpi_N(a,x,q,w)
	\lesssim_N
	\varpi_N(a,x,u,w).
\end{align}
Moreover, for $0\le m\le N$,
\begin{align}\label{eq:weight-transfer-ratio}
	\langle v\rangle^m
	\varpi_N (a,x,v,w)
	\lesssim_N
	\langle w\rangle^m
	\varpi_{N-m} (a,x,v,w).
\end{align}
{Indeed, up to constants depending only on} $N$,
\begin{align*}
	\varpi_N(a,x,v,w)
	\sim_N
	\left(
	1+\frac{|x|/a+|v|}{\langle w\rangle}
	\right)^{-N}.
\end{align*}
{Set}
\begin{align*}
	A_0:=1+\frac{|x-y|/a_0+|v|}{\langle w\rangle},\quad\quad
	A_1:=1+\frac{|x-\zeta|/a_1+|v|}{\langle u\rangle},\quad\quad
	A_2:=1+\frac{|\zeta-y|/a_2+|u|}{\langle w\rangle}.
\end{align*}
{The identity}
\begin{align*}
	\frac{x-y}{a_0}
	=
	\frac{a_1}{a_0}\frac{x-\zeta}{a_1}
	+
		\frac{a_2}{a_0}\frac{\zeta-y}{a_2}.
	\end{align*}
{Together with}
$\langle u\rangle/\langle w\rangle\lesssim A_2$,
{this gives} $A_0\lesssim A_1A_2$, {which proves
\eqref{eq:ratio-submultiplicative}. To prove
\eqref{eq:ratio-submultiplicative-jump}, replace} $A_2$ {by}
\begin{align*}
	\widetilde A_2
		:=1+\frac{|\zeta-y|/a_2+|q|}{\langle w\rangle}.
	\end{align*}
{Because}
$\langle u\rangle/\langle w\rangle\lesssim (1+\langle u\rangle/\langle q\rangle)\widetilde A_2$,
{the same calculation gives}
$A_0\lesssim A_1 (1+\langle u\rangle/\langle q\rangle)\widetilde A_2$.
{This proves \eqref{eq:ratio-submultiplicative-jump}; omitting the
first time interval in the same comparison yields
\eqref{eq:ratio-jump-transfer}. Finally, \eqref{eq:weight-transfer-ratio}
follows from}
\begin{align*}
	\frac{\langle v\rangle}{\langle w\rangle}
	\lesssim
	\left\langle
	\frac{\langle |x|/a+|v|\rangle}{\langle w\rangle}
	\right\rangle.
\end{align*}

Let
\begin{align*}
	\nu_v(\alpha)
	:=
	\frac{\langle v\rangle^{1-\kappa}}
	{[\alpha]_v^{d+2s}},
\end{align*}
{where} $\nu_v$ {is the anisotropic reference tail in
the definition of} $\bT_v$. {Recall that}
\({\Omega_{\varkappa}}\) {is defined in \eqref{defom}. We
will use the following tail-transfer estimate.}
\begin{lemma}
	\label{lem:large-jump-tail-transfer}
	Set
	\begin{align}\label{eq:def-N-tail}
		N_{\mathrm{tail}}
		:=
		\varkappa-2d-2s+\min\{\kappa,1\}.
	\end{align}
	For $m,N\ge0$ satisfying
	$m+N\le N_{\mathrm{tail}}$, one has
	\begin{align}\label{eq:large-jump-tail-transfer}
		|\alpha|^m
		\frac{{\Omega_{\varkappa}}(v,\alpha)}{|\alpha|^{d+2s}}
		\lesssim
		\langle v-\alpha\rangle^{m+(\kappa-1)_+}
		\varpi_N^{\mathrm v} (v,v-\alpha)
		\nu_v(\alpha).
	\end{align}
\end{lemma}
\begin{proof}
	{Set} $q:=v-\alpha$ {and decompose}
	${\Omega_{\varkappa}}=\Omega_1+\Omega_2$ {according to the two terms in
	\eqref{defom}.}
	
	On the support of $\Omega_1$,
	\begin{align*}
		|\alpha|
		\lesssim
		\langle\Pi_{\widehat\alpha}v\rangle
		=
		\langle\Pi_{\widehat\alpha}q\rangle.
	\end{align*}
	{Consequently,} $\langle v\rangle\lesssim\langle q\rangle$,
	{so the ratio weight is harmless and}
	$|\alpha|^m\lesssim\langle q\rangle^m$. {Moreover, the
	definition of} $[\alpha]_v$ {gives}
	\begin{align*}
		\frac{\Omega_1(v,\alpha)|\alpha|^{-(d+2s)}}
		{\nu_v(\alpha)}
		\lesssim
		\left(
		\frac{\langle\Pi_{\widehat\alpha}v\rangle}
		{\langle v\rangle}
		\right)^{1-\kappa}
		\langle\mathrm P_{\widehat\alpha}v\rangle^{
			2d+2s-1-\varkappa
		}
		\lesssim
		\langle q\rangle^{(\kappa-1)_+}.
	\end{align*}
	{This proves \eqref{eq:large-jump-tail-transfer} for} $\Omega_1$.
	
	{For} $\Omega_2$, {division by} $\nu_v$
	{leaves the decay factor}
	\begin{align*}
		\frac{\Omega_2(v,\alpha)|\alpha|^{-(d+2s)}}
		{\nu_v(\alpha)}
		\lesssim
		\langle |v|+|\alpha|\rangle^{-N_{\mathrm{tail}}}.
	\end{align*}
	{If} $\langle v\rangle\lesssim\langle q\rangle$,
	{then} $|\alpha|^m\lesssim\langle q\rangle^m$
	{and the claim follows immediately. If}
	$\langle v\rangle\gg\langle q\rangle$, {then}
	$|\alpha|\sim\langle v\rangle$ {and}
	\begin{align*}
		|\alpha|^m
		\langle |v|+|\alpha|\rangle^{-N_{\mathrm{tail}}}
		\lesssim
		\langle q\rangle^m
		\left(
		\frac{\langle q\rangle}{\langle v\rangle}
		\right)^N
	\end{align*}
		{whenever} $m+N\le N_{\mathrm{tail}}$.
		{This proves the lemma.}
		\vspace{0.3cm}
\end{proof}	

\subsubsection{\texorpdfstring{{Estimates for the remainder terms}}{Estimates for the remainder terms}}
For $j\ge0$, define
\begin{align}\label{eq:def-regularity-envelope}
	\mathfrak T_v^{ (j)}(\sigma,X,V)
	:=
	\bT_v(\sigma,X,V)
	\wedge
	\bigl(\sigma^{\frac{jb_0}{2s}}\mathfrak P_v(\sigma)\bigr).
\end{align}
{Here}
$\mathfrak P_v(\sigma)=\sigma^{-d}\tilde \sigma_v^{-\frac{d}{s}}\langle v\rangle^2$
{is the peak amplitude of} $\G^{(0)}$.

{Recall the profile-comparison parameter} $M_{\mathrm{pc}}$
{from \eqref{eq:global-profile-comparison} and set}
$$n_0=\max\{M_{\mathrm{pc}}+10d,\frac{\kappa d}{s}+10\}.$$
\begingroup
For the intermediate mass induction we use the integrated
fixed-increment modulus
\begin{equation}\label{eq:def-fixed-increment-mass}
 \mathfrak B_P(F;\sigma,y,w)
 :=\sup_{0<|h|\le r_0}\iint_{\mathbb R^{2d}}
 \varpi_P(\sigma,x-y,v,w)^{-1}
 \frac{|\delta_h^vF(t,x,v;\tau,y,w)|}
 {\mathfrak r_{\sigma,v}(h)^{b_0}}\,\dif x\dif v.
\end{equation}
The supremum is taken after integration. Outside the intrinsic range
the denominator is one; for $|h|\le r_0<1$, the ratio weights at
$v$ and $v-h$ are comparable.
\par\endgroup
{The next lemma estimates the large-jump error.}
\begin{lemma}
	\label{lem:one-step-large-jump}
	Let \(j\in{\mathbb N}\), \(M,N\ge0\) with
$
		N\ge  n_0,
		N+M_{\mathrm{pc}}\le N_{\mathrm{tail}}.
$
{Suppose that} $\sigma=t-\tau$ {satisfies}
$0<\sigma\le1$ {and that} \(F\) {satisfies}
	\begin{align}
		|F(t,x,v;\tau,y,w)|
		\lesssim
		\langle w\rangle^M
		\varpi_N (\sigma,x-y,v,w)
		\mathfrak T_v^{ (j)}(\sigma,x-y-\sigma w,v-w),
		\label{eq:large-jump-input-pointwise}
	\end{align}
		{together with the weighted mass estimate}
	\begin{align}
		\iint_{\mathbb R^{2d}}
		\langle v\rangle^{ n_0}
		\varpi_{N-n_0} (\sigma,x-y,v,w)^{-1}
		|F(t,x,v;\tau,y,w)|
		\dif x\dif v
		\lesssim \sigma^{\frac{jb_0}{2s}}
		\langle w\rangle^{M+ n_0}.
		\label{eq:large-jump-input-mass}
	\end{align}
	Then one has
	\begin{align}
		&\left(
		\left|\G^{(0)}\oast\mathcal L_{g}^{>}F\right|
		\right)
		(t,x,v;\tau,y,w)\\
	&\quad\quad	\lesssim{}
\langle w\rangle^{M+ n_0}
		\varpi_{N- n_0} (\sigma,x-y,v,w)
		\mathfrak T_v^{ (j+1)}(\sigma,x-y-\sigma w,v-w),
		\label{eq:large-jump-gain-one-step}
	\end{align}
\begingroup
The terminal intrinsic seminorms have the same pointwise bound:
\begin{align}\label{eq:A25-one-step-intrinsic-holder}
&\bigl[\G^{(0)}\oast\mathcal L_g^{>}F\bigr]^{\mathrm{int}}_{v,b_0}
+\bigl[\G^{(0)}\oast\mathcal L_g^{>}F\bigr]^{\mathrm{int}}_{x,\beta_0}
\lesssim \langle w\rangle^{M+n_0}
\varpi_{N-n_0}(\sigma,x-y,v,w)
\mathfrak T_v^{(j+1)}(\sigma,x-y-\sigma w,v-w).
\end{align}
\par\endgroup
and
	\begin{align}
		&\iint_{\mathbb R^{2d}}
		\varpi_{N-n_0}(\sigma,x-y,v,w)^{-1}
		\bigl(
		|\G^{(0)}\oast\mathcal L_g^{>}F|
		\bigr)
		(t,x,v;\tau,y,w)
		\dif x\dif v
	\lesssim
		\sigma^{\frac{(j+1)b_0}{2s}}
		\langle w\rangle^{M+n_0}.
		\label{eq:large-jump-gain-weighted-mass}
	\end{align}
\begingroup
The same time gain holds for the integrated fixed-increment modulus:
\begin{equation}\label{eq:A25-one-step-fixed-mass}
 \mathfrak B_{N-n_0}
 (\G^{(0)}\oast\mathcal L_g^{> }F;\sigma,y,w)
 \lesssim \sigma^{\frac{(j+1)b_0}{2s}}
 \langle w\rangle^{M+n_0}.
\end{equation}
\par\endgroup
\end{lemma}
\begin{proof}
{Decompose}
$\mathcal{L}_{g}^>=\mathcal{L}_{g,+}^>+\mathcal{L}_{g,-}^>$,
{where}
	\begin{align*}
&	\mathcal L_{g,+}^{>}F(r,\zeta,u)
	:=
	\int_{|\alpha|\ge r_0}
	\mathbf C_g(r,\zeta,u,\alpha)
	F(r,\zeta,u-\alpha)
	\frac{\dif\alpha}{|\alpha|^{d+2s}},\\
&	\mathcal L_{g,-}^{>}F(r,\zeta,u)
	:=
	-F(r,\zeta,u)
	\int_{|\alpha|\ge r_0}
	\mathbf C_g(r,\zeta,u,\alpha)
	\frac{\dif\alpha}{|\alpha|^{d+2s}}.
\end{align*}
	{We prove the gain, peak, loss, and mass estimates in turn.}
	
	\medskip
	\noindent
	\textit{Step 1: expansion of the gain term.}
	By definition,
	\begin{align}
		&\left(
		\G^{(0)}\oast\mathcal L_{g,+}^{>}F
		\right)(t,x,v;\tau,y,w)
		\notag\\
		&\quad=
		\int_\tau^t
		\iint_{\mathbb R^{2d}}
		\G^{(0)}(t,x,v;r,\zeta,u)
		\int_{|\alpha|\ge r_0}
		\mathbf C_g(r,\zeta,u,\alpha)
		F(r,\zeta,u-\alpha;\tau,y,w)
		\frac{\dif\alpha}{|\alpha|^{d+2s}}
		\dif\zeta\dif u\dif r.
		\label{eq:large-jump-gain-expanded}
	\end{align}
{Set}
	\begin{align*}
		\sigma:=t-\tau,
		\qquad
		a_1:=t-r,
		\qquad
		a_2:=r-\tau.
	\end{align*}
	{For the intermediate post-jump velocity} \(u\),
	{also set} $q:=u-\alpha$ {and}
	\begin{align*}
		(X_1,V_1):=(x-\zeta-a_1u,v-u),\quad\quad 
		(X_2,V_2):=(\zeta-y-a_2w, q-w),\quad\quad (X,V)=(x-y-\sigma w,v-w).
	\end{align*}
	{The collision coefficient satisfies}
	\begin{align*}
		\mathbf C_g(r,\zeta,u,\alpha)
		\lesssim {\Omega_{\varkappa}}(u,\alpha).
	\end{align*}
{Moving from the profile based at} \(q=u-\alpha\)
{to the post-jump profile based at} \(u\), {and then
to the terminal profile based at} \(v\), {costs at most}
	\(\langle\alpha\rangle^{M_{\mathrm{pc}}}+\langle v-u\rangle^{M_{\mathrm{pc}}}\).
{Apply the frozen-kernel estimate with rapid exponent}
	$N+M_{\mathrm{pc}}$. {Since}
	$|v-u|\leq |x-\zeta-a_1u|/a_1+|v-u|$, one has
	\begin{align*}
	\langle v-u\rangle^{M_{\mathrm{pc}}}
	\left\langle\frac{|x-\zeta-a_1u|}{a_1}+|v-u|\right\rangle^{-N-M_{\mathrm{pc}}}
	\leq
	\left\langle\frac{|x-\zeta-a_1u|}{a_1}+|v-u|\right\rangle^{-N}.
	\end{align*}
	{Thus, the second summand is absorbed without reducing the required
	exponent} $N$. {The remaining coefficient is bounded by}
	\begin{align*}
	\langle\alpha\rangle^{M_{\mathrm{pc}}}
	\frac{{\Omega_{\varkappa}}(u,\alpha)}{|\alpha|^{d+2s}}.
\end{align*}
{On} $|\alpha|\geq r_0$,
	$\langle\alpha\rangle^{M_{\mathrm{pc}}}
	\leq C(r_0,M_{\mathrm{pc}})|\alpha|^{M_{\mathrm{pc}}}$. {Therefore, we may apply Lemma~\ref{lem:large-jump-tail-transfer}, with
		a constant depending on the fixed cutoff} $r_0$. {Taking}
$
		m=M_{\mathrm{pc}}
$
		 {in that lemma gives}
	\begin{align}
		|\alpha|^{M_{\mathrm{pc}}}
		\frac{{\Omega_{\varkappa}}(u,\alpha)}{|\alpha|^{d+2s}}
		\lesssim
		\langle q\rangle^{ n_0}
		\varpi_N^{\mathrm v} (u,q)
		\nu_u(\alpha).
		\label{eq:one-step-tail-applied}
	\end{align}
	
	{The rapid off-diagonal decay of the outer small-jump kernel
	allows us to reserve the factor}
$
		\varpi_N (a_1,x-\zeta,v,u).
$
	{Indeed, set}
	\begin{align*}
		R_1
		:=
		\frac{|x-\zeta-a_1u|}{a_1}+|v-u|.
	\end{align*}
	{Then}
	\begin{align*}
		\frac{|x-\zeta|/a_1+|v|}{\langle u\rangle}
		\lesssim1+R_1.
	\end{align*}
	{The rapid factor of} $\G^{(0)}$ {with exponent at
	least} $N$ {absorbs the inverse weight. The input estimate
	\eqref{eq:large-jump-input-pointwise} supplies}
	\(\varpi_N (a_2,\zeta-y,q,w)\). {Moreover,
	\eqref{eq:weight-transfer-ratio} gives}
	\begin{align}
		\langle q\rangle^{ n_0}
		\varpi_N (a_2,\zeta-y,q,w)
		\lesssim_N
		\langle w\rangle^{ n_0}
		\varpi_{N- n_0} (a_2,\zeta-y,q,w).
		\label{eq:one-step-endpoint-transfer}
	\end{align}
	{Since} \(N\ge n_0\), {we may reduce the exponents
	of the first two ratio weights from} \(N\) {to} \(N-n_0\).
	{The kinetic composition estimate
	\eqref{eq:ratio-submultiplicative-jump} then yields}
	\begin{align}
		&\varpi_N (a_1,x-\zeta,v,u)
		\varpi_N^{\mathrm v} (u,q)
		\varpi_{N- n_0} (a_2,\zeta-y,q,w)
\lesssim_N
		\varpi_{N- n_0} (\sigma,x-y,v,w).
		\label{eq:one-step-three-ratio-weights}
	\end{align}
	{Combining}
	\eqref{eq:large-jump-input-pointwise},
	\eqref{eq:one-step-tail-applied},
	\eqref{eq:one-step-endpoint-transfer}, and
	\eqref{eq:one-step-three-ratio-weights}, and using the
	{profile-to-reference-tail comparison, whose polynomial cost is
	already included in} $M_{\mathrm{pc}}$, {reduces the tail part
	of \eqref{eq:large-jump-gain-expanded} to}
		\begin{align*}
		\langle w\rangle^{M+ n_0}
		\varpi_{N- n_0} (\sigma,x-y,v,w)
		\int_\tau^t
		\iint_{\mathbb R^{2d}}
		\int_{|\alpha|\ge r_0}
		\bT_v(a_1,X_1,V_1)
		\nu_v(\alpha)
		\bT_v(a_2,X_2,V_2)
		\dif\alpha\dif\zeta\dif u\dif r.
	\end{align*}
	{The one-tail convolution estimate \eqref{convott} therefore
	yields}
	\begin{align}
		\left|
		\G^{(0)}\oast\mathcal L_{g,+}^{>}F
		\right|
		\lesssim
		\langle w\rangle^{M+ n_0}
		\varpi_{N- n_0} (\sigma,x-y,v,w)
		\bT_v(\sigma,X,V).
	\label{eq:one-step-tail-bound}
\end{align}
	\medskip
	\noindent
	\textit{Step 2: improved peak estimate.}
	{We next prove the sharper peak estimate}
	\begin{align}
		\left|
		\G^{(0)}\oast\mathcal L_{g,+}^{>}F
		\right|
		\lesssim
		\langle w\rangle^{M+ n_0}
		\varpi_{N- n_0} (\sigma,x-y,v,w)
		\sigma^{\frac{(j+1)b_0}{2s}}\mathfrak P_v(\sigma).
		\label{eq:one-step-peak-bound}
	\end{align}
	Split the time integral into
	\begin{align*}
		I_{1}
		:=
		\int_\tau^{(t+\tau)/2}\cdots\dif r,
		\qquad
		I_{2}
		:=
		\int_{(t+\tau)/2}^t\cdots\dif r.
	\end{align*}
	
	{On the first interval,} \(a_1\ge\sigma/2\).
	{The normalized} \(L^1\)-\(L^\infty\) {estimate for
	the outer small-jump kernel gives}
	\begin{align}
		\G^{(0)}(t,x,v;r,\zeta,u)
		\lesssim_N
		\varpi_N (a_1,x-\zeta,v,u)
		\mathfrak P_v(a_1)
		\lesssim_N
		\varpi_N (a_1,x-\zeta,v,u)
		\mathfrak P_v(\sigma).
		\label{eq:outer-early-peak}
	\end{align}
	{To retain the spatial component of the new ratio weight, write}
	\begin{align*}
		|F(r,\zeta,q;\tau,y,w)|
		={}&
		\varpi_{N-n_0}(a_2,\zeta-y,q,w)\\
		&\times
		\left[
		\varpi_{N-n_0}(a_2,\zeta-y,q,w)^{-1}
		|F(r,\zeta,q;\tau,y,w)|
		\right].
	\end{align*}
	{Set} $\ell:=N-n_0$. {Reduce the outer weight to
	exponent} $\ell$ {and use}
	\begin{align*}
		\varpi_N^{\mathrm v}(u,q)
		=
		\varpi_\ell^{\mathrm v}(u,q)
		\varpi_{n_0}^{\mathrm v}(u,q).
	\end{align*}
	{By \eqref{eq:ratio-submultiplicative-jump}, the product of the
	outer weight with exponent} $\ell$, {the factor}
	$\varpi_\ell^{\mathrm v}(u,q)$, {and the first factor above is
	bounded by} $\varpi_\ell(\sigma,x-y,v,w)$. {The incoming
	weighted large-jump rate controls the remaining jump factor:}
	\begin{align}\label{eq:incoming-weighted-large-jump-rate}
		\sup_{q\in\mathbb R^d}
		\int_{|\alpha|\ge r_0}
		\varpi_{n_0}^{\mathrm v}(q+\alpha,q)
		\nu_{q+\alpha}(\alpha)\dif\alpha
		\lesssim1.
	\end{align}
{This is the incoming analogue of the finite weighted large-jump
rate and follows from the two-region decomposition used in
Lemma~\ref{lem:large-jump-tail-transfer}.}
	Consequently, the weighted mass assumption
	\eqref{eq:large-jump-input-mass} gives
	\begin{align*}
		|I_{1}|
		&\lesssim
		\langle w\rangle^{M+ n_0}
		\varpi_{N- n_0} (\sigma,x-y,v,w)
		\mathfrak P_v(\sigma)
		\int_\tau^{(t+\tau)/2}
		a_2^{\frac{jb_0}{2s}}\dif r\\
		&\lesssim
		\langle w\rangle^{M+ n_0}
		\varpi_{N- n_0} (\sigma,x-y,v,w)
	\sigma^{1+\frac{jb_0}{2s}}\mathfrak P_v(\sigma).
	\end{align*}
	
	{On the late interval,} \(a_2\ge\sigma/2\).
	{By}
	\eqref{eq:large-jump-input-pointwise} and the definition of
	\(\mathfrak T_q^{ (j)}\),
	\begin{align}
		|F(r,\zeta,q;\tau,y,w)|
		\lesssim 
		\langle w\rangle^M
		\varpi_N (a_2,\zeta-y,q,w)
		a_2^{\frac{jb_0}{2s}}\mathfrak P_q(a_2).
		\label{eq:late-input-peak}
	\end{align}
	{The profile transition from} \(q\) {to} \(u\),
	{together with the tail-transfer estimate,}
	\eqref{eq:weight-transfer-ratio},
	\eqref{eq:ratio-jump-transfer}, and
	\(\int_{|\alpha|\ge r_0}\nu_u(\alpha)\dif\alpha\lesssim1\),
	{implies}
	\begin{align}
		|\mathcal L_{g,+}^{>}F(r,\zeta,u)|
		\lesssim
		\langle w\rangle^{M+ n_0}
		\varpi_{N- n_0} (a_2,\zeta-y,u,w)
		a_2^{\frac{jb_0}{2s}}\mathfrak P_u(a_2).
		\label{eq:late-large-jump-peak}
	\end{align}
	{Using the weighted} \(L^1\) {estimate for the outer
	small-jump kernel, \eqref{eq:ratio-submultiplicative}, and}
	\(a_2\sim\sigma\), {we obtain}
	\begin{align*}
		|I_{2}|
		&\lesssim
		\langle w\rangle^{M+ n_0}
		\varpi_{N- n_0} (\sigma,x-y,v,w)
		\sigma^{\frac{jb_0}{2s}}\mathfrak P_v(\sigma)
		\int_{(t+\tau)/2}^t\dif r\\
		&\lesssim
		\langle w\rangle^{M+ n_0}
		\varpi_{N- n_0} (\sigma,x-y,v,w)
		\sigma^{1+\frac{jb_0}{2s}}\mathfrak P_v(\sigma).
	\end{align*}
	Because \(0<\sigma\le1\),
	\begin{align*}
		\sigma^{1+\frac{jb_0}{2s}}
		\le
		\sigma^{\frac{(j+1)b_0}{2s}}.
	\end{align*}
	{This proves \eqref{eq:one-step-peak-bound}.}
	
	{Taking the minimum of}
	\eqref{eq:one-step-tail-bound} and
	\eqref{eq:one-step-peak-bound}, we obtain 
    \begin{align}
		\left|
		\G^{(0)}\oast\mathcal L_{g,+}^{>}F
		\right|
		(t,x,v;\tau,y,w)
		\lesssim{}&
\langle w\rangle^{M+ n_0}
		\varpi_{N- n_0} (\sigma,x-y,v,w)
		\mathfrak T_v^{ (j+1)}(\sigma,x-y-\sigma w,v-w),
		\label{eq:large-jump-gain-one-step+}
	\end{align}
			\medskip
	\noindent
	\textit{Step 3: loss term.}
	{We first note that}
	\begin{align*}
		\int_{|\alpha|\ge r_0}
		\mathbf C_g(r,\zeta,u,\alpha)
		\frac{\dif\alpha}{|\alpha|^{d+2s}}\lesssim 1.
	\end{align*}
	 {Hence}
	\begin{align*}
		|\mathcal L_{g,-}^{>}F(r,\zeta,u)|
		\lesssim
		\langle w\rangle^{M+ n_0}
		\varpi_{N- n_0} (a_2,\zeta-y,u,w)
		\mathfrak T_u^{ (j)}(a_2,X_2,u-w).
	\end{align*}
	{For each} \(r\), {the standard weighted
	subconvolution estimate and \eqref{eq:ratio-submultiplicative} give}
	\begin{align*}
		&\iint_{\mathbb R^{2d}}
		\G^{(0)}(t,x,v;r,\zeta,u)
		|\mathcal L_{g,-}^{>}F(r,\zeta,u)|
		\dif\zeta\dif u\\
		&\qquad\lesssim
		\langle w\rangle^{M+ n_0}
		\varpi_{N- n_0} (\sigma,x-y,v,w)
		\mathfrak T_v^{ (j)}(\sigma,X,V).
	\end{align*}
	{Integration in} \(r\) {yields}
	\begin{align*}
		\left|
		\G^{(0)}\oast\mathcal L_{g,-}^{>}F
		\right|
		\lesssim
		\langle w\rangle^{M+ n_0}
		\varpi_{N- n_0} (\sigma,x-y,v,w)
		\sigma\,\mathfrak T_v^{ (j)}(\sigma,X,V).
	\end{align*}
	Since $0<\sigma\le1$ and $b_0/(2s)\le1$, the definition of the
	envelope gives
	\begin{align*}
		\sigma\,\mathfrak T_v^{(j)}(\sigma,X,V)
		\le
		\mathfrak T_v^{(j+1)}(\sigma,X,V).
	\end{align*}
	    {Combining this estimate with
	    \eqref{eq:large-jump-gain-one-step+} proves the pointwise bound in
	    \eqref{eq:large-jump-gain-one-step}. For the terminal differences,
	    apply the difference to the outer small-jump kernel.}
	    {Set $(X_1^{x,h},V_1^{x,h})=(X_1-h,V_1)$,
$(X_1^{v,h},V_1^{v,h})=(X_1,V_1-h)$, with
$v^{x,h}=v$ and $v^{v,h}=v-h$.
Proposition~\ref{propsmalljump} bounds subunit intrinsic increments;
for larger increments use both endpoint bounds. Thus}
\begin{equation}\label{eq:A25-two-point-small-jump}
{
|\delta_h^\diamond\G^{(0)}|
\lesssim
\bigl(1\wedge[\ell_{a_1,v}^\diamond(h)]^{b_\diamond}\bigr)
\Bigl[\bK_{v,L}(a_1,X_1,V_1)
+\bK_{v^{\diamond,h},L}(a_1,X_1^{\diamond,h},V_1^{\diamond,h})\Bigr],
\qquad\diamond\in\{x,v\}.}
\end{equation}
{where} \(b_v=b_0\) {and} \(b_x=\beta_0\).
\begingroup
Set $x^{x,h}=x-h$, $v^{x,h}=v$ and
$x^{v,h}=x$, $v^{v,h}=v-h$; write
$X^{\diamond,h}=x^{\diamond,h}-y-\sigma w$ and
$V^{\diamond,h}=v^{\diamond,h}-w$.
Repeating the early/late time splitting of Steps~1--3 for each term,
and setting $\rho=a_1/\sigma$, gives the endpoint integrals
\begin{align*}
 &\int_0^1\!\left(1\wedge
 [\ell_{\sigma,v}^{v}(h)\rho^{-1/(2s)}]^{b_0}\right)\dif\rho
 \lesssim [\ell_{\sigma,v}^{v}(h)]^{b_0},\\
 &\int_0^1\!\left(1\wedge
 [\ell_{\sigma,v}^{x}(h)\rho^{-1-1/(2s)}]^{\beta_0}\right)\dif\rho
 \lesssim [\ell_{\sigma,v}^{x}(h)]^{\beta_0}.
\end{align*}
Both singular powers equal $b_0/(2s)<1$.
For increments in the physical seminorm range, the resulting bound is
\begin{align}\label{eq:A25-two-point-one-step}
|\delta_h^\diamond(\G^{(0)}\oast\mathcal L_g^{>}F)|
\lesssim{}&
[\ell_{\sigma,v}^\diamond(h)]^{b_\diamond}
\langle w\rangle^{M+n_0}
\sum_{\epsilon=0}^1
\varpi_{N-n_0}(\sigma,x^{\diamond,\epsilon h}-y,
 v^{\diamond,\epsilon h},w)\\[-2pt]
&\qquad\times
\mathfrak T_{v^{\diamond,\epsilon h}}^{(j+1)}
(\sigma,X^{\diamond,\epsilon h},V^{\diamond,\epsilon h}),
\qquad \diamond\in\{x,v\},\notag
\end{align}
where the zero increment means the unshifted endpoint.
\begingroup
For $\ell_{\sigma,v}^{\diamond}(h)\le1$, local reference-velocity
comparison and \eqref{eq:A10-profile-moderation} bound both endpoint
envelopes in \eqref{eq:A25-two-point-one-step} by the unshifted one.
Taking the supremum gives \eqref{eq:A25-one-step-intrinsic-holder}.
The velocity loss in \eqref{eq:A25-profile-shift-cost} arises only
in the larger physical increment range.
\par\endgroup
For each fixed nonzero increment $h$ in the physical seminorm range, the inverse ratio
weights at the two endpoints are comparable when
$|h|\le\sigma^{1/(2s)}$ in velocity or
$|h|\le\sigma^{1+1/(2s)}$ in space.  After changing terminal
variables in the shifted term, the Step~4 weighted mass argument gives
\begin{align}\label{eq:A25-two-point-mass}
\iint\varpi_{N-n_0}(\sigma,x-y,v,w)^{-1}
\frac{|\delta_h^\diamond(\G^{(0)}\oast\mathcal L_g^{>}F)|}
{[\ell_{\sigma,v}^\diamond(h)]^{b_\diamond}}\,
\dif x\dif v
\lesssim{}&
\sigma^{(j+1)b_0/(2s)}\langle w\rangle^{M+n_0}.
\end{align}
The denominator uses the intrinsic, not the physical, increment scale.
\begingroup
Equation~\eqref{eq:A25-two-point-mass} controls the modulus
\eqref{eq:def-fixed-increment-mass}, with the supremum outside the
integral; no interchange of the two is used.
\par\endgroup

For an increment in the physical seminorm range, \eqref{eq:A10-profile-moderation}
and local reference-velocity comparison give
\begin{align}\label{eq:A25-profile-shift-cost}
\mathfrak T_{v^{\diamond,h}}^{(j+1)}
(\sigma,X^{\diamond,h},V^{\diamond,h})
&\lesssim
\langle v\rangle^{2(d+2s)(1+\kappa/(2s))}
\mathfrak T_v^{(j+1)}(\sigma,X,V),\\
\ell_{\sigma,v}^\diamond(h)^{b_\diamond}
&\lesssim
\langle v\rangle^{b_0(1+\kappa/(2s))}
\ell_\sigma^\diamond(h)^{b_\diamond}.
\end{align}
Indeed, $|\mathcal O_v^{-1}h|/(\tilde\sigma_v^{1/(2s)})$
for velocity shifts, and
$|\mathcal O_v^{-1}h|/(\sigma\tilde\sigma_v^{1/(2s)})$
for spatial shifts, are at most
$\langle v\rangle^{1+\kappa/(2s)}$ in the stated range;
the unscaled rapid factor and $\varpi_N$ change by bounded factors.
The profile comparison in \eqref{eq:A25-profile-shift-cost} applies to the
minimum defining $\mathfrak T$ because the peak term is independent of $(X,V)$.
{Thus conversion of
\eqref{eq:A25-two-point-one-step} to physical seminorms costs at most
$\Lambda_*$ velocity powers, since}
\[
{(2(d+2s)+b_0)(1+\kappa/(2s))\leq\Lambda_*}
\]
{By \eqref{eq:weight-transfer-ratio},
the ratio weight incurs the same loss.} Lemma~\ref{lem:A34-physical-transfer}
implements this conversion after the intrinsic induction.
\par\endgroup

	\medskip
	\noindent
	\textit{Step 4: mass estimate.}
	{Again set} $\ell=N-n_0$. {Inverting
	\eqref{eq:ratio-submultiplicative-jump} gives}
	\begin{align*}
		&\varpi_\ell(\sigma,x-y,v,w)^{-1}\lesssim_\ell
		\varpi_\ell(a_1,x-\zeta,v,u)^{-1}
		\bigl(\varpi_\ell^{\mathrm v}(u,q)\bigr)^{-1}
		\varpi_\ell(a_2,\zeta-y,q,w)^{-1}.
	\end{align*}
	{The rapid decay of the outer small-jump kernel yields the
	weighted} $L^1$ {estimate}
	\begin{align*}
		\iint_{\mathbb R^{2d}}
		\varpi_\ell(a_1,x-\zeta,v,u)^{-1}
		\G^{(0)}(t,x,v;r,\zeta,u)
		\dif x\dif v
		\lesssim_\ell 1.
	\end{align*}
	{Moreover, the factor} $\varpi_N^{\mathrm v}(u,q)$
	{in \eqref{eq:one-step-tail-applied} cancels the inverse jump
	weight, leaving}
	\begin{align*}
		\varpi_N^{\mathrm v}(u,q)
		\bigl(\varpi_\ell^{\mathrm v}(u,q)\bigr)^{-1}
		=
		\varpi_{n_0}^{\mathrm v}(u,q).
	\end{align*}
	Thus \eqref{eq:incoming-weighted-large-jump-rate} and
	\eqref{eq:large-jump-input-mass} imply
	\begin{align*}
		&\iint_{\mathbb R^{2d}}
		\varpi_\ell(\sigma,x-y,v,w)^{-1}
		\left|
		\G^{(0)}\oast\mathcal L_{g,+}^{>}F
		\right|
		\dif x\dif v
		\lesssim
		\langle w\rangle^{M+ n_0}
		\int_\tau^t a_2^{\frac{jb_0}{2s}}\dif r.
	\end{align*}
	{For the loss term, invert
	\eqref{eq:ratio-submultiplicative} and combine the same weighted}
	$L^1$ {estimate for the outer kernel with
	\eqref{eq:large-jump-input-mass}. This gives the same bound, and hence}
	\begin{align*}
		&\iint_{\mathbb R^{2d}}
		\varpi_\ell(\sigma,x-y,v,w)^{-1}
		|\G^{(0)}\oast\mathcal L_g^{>}F(t,x,v;\tau,y,w)|
		\dif x\dif v\\
		&\qquad\lesssim
		\langle w\rangle^{M+ n_0}
		\sigma^{1+\frac{jb_0}{2s}}
		\le
		\langle w\rangle^{M+ n_0}
		\sigma^{\frac{(j+1)b_0}{2s}}.
	\end{align*}
\begingroup
For \eqref{eq:A25-one-step-fixed-mass}, fix $0<|h|\le r_0$.
On $\ell_{\sigma,v}^{v}(h)\le1$, use \eqref{eq:A25-two-point-mass};
this set is empty unless $|h|\le\sigma^{1/(2s)}$.
On the complement the denominator is one, so the order-zero mass
bound controls both endpoints after translation. Their inverse ratio
weights are comparable for $|h|\le r_0<1$. Taking the supremum after
integration gives the assertion.
\par\endgroup
\end{proof}

{We next estimate the contribution of the remainder}
$\mathcal{Q}_{s,\mathrm{rem}}(g,\G)$.
\begin{lemma}
	\label{lem:one-step-cancellation-remainder}
		{Let} \(j\in{\mathbb N}\), \(M\geq0\),
		\(N\geq n_0\), {and assume the standing tail conditions}
		$\kappa\geq0$ {and} $\varkappa>d$. {Set}
	$0<\sigma=t-\tau\le T\le\min\{1,r_0^{2s}\}$, where $T$ is chosen
	within the time range of Propositions~\ref{propsmalljump} and
	\ref{prop:small-jump-source-holder} for the rapid exponents used below.
		{Suppose that the following two estimates hold:}
	\begin{align}
		\bigl(|F|+{[F]^{\mathrm{int}}_{v,b_0}}\bigr)(t,x,v;\tau,y,w)
		&\lesssim
		\langle w\rangle^M
		\varpi_N(\sigma,x-y,v,w)
		\mathfrak T_v^{(j)}(\sigma,x-y-\sigma w,v-w),
		\label{eq:qrem-F-pointwise}
	\end{align}
	and
	\begingroup
	\begin{align}
		&\iint_{\mathbb R^{2d}}
		\varpi_N(\sigma,x-y,v,w)^{-1}|F(t,x,v;\tau,y,w)|
		\dif x\dif v
		+\mathfrak B_N(F;\sigma,y,w)
		\lesssim
		\langle w\rangle^{M}
		\sigma^{\frac{jb_0}{2s}}.
		\label{eq:qrem-F-weighted-mass}
	\end{align}
	\par\endgroup
	Then, for $\mathscr{R}F=\G^{(0)}
	\oast
	\mathcal Q_{s,\mathrm{rem}}(g,F)$, one has
	\begin{align}
		&\bigl(
		|\mathscr{R}F|+{[\mathscr{R}F]^{\mathrm{int}}_{v,b_0}}
		+{[\mathscr{R}F]^{\mathrm{int}}_{x,\beta_0}}
		\bigr)(t,x,v;\tau,y,w)\notag\\
		&\qquad\lesssim
		\langle w\rangle^{M+n_0}
		\varpi_{N-n_0}(\sigma,x-y,v,w)
		\mathfrak T_v^{(j+1)}(\sigma,x-y-\sigma w,v-w).
		\label{eq:qrem-one-step}
	\end{align}
	Moreover,
	\begingroup
	\begin{align}
		&\iint_{\mathbb R^{2d}}
		\varpi_{N-n_0}(\sigma,x-y,v,w)^{-1}
		|\mathscr{R}F(t,x,v;\tau,y,w)|\dif x\dif v
		+\mathfrak B_{N-n_0}(\mathscr{R}F;\sigma,y,w)
		\lesssim \langle w\rangle^{M+n_0}
		\sigma^{\frac{(j+1)b_0}{2s}}.
		\label{eq:qrem-one-step-mass}
	\end{align}
	\par\endgroup
\end{lemma}
\begin{proof}
	{Use the notation in \eqref{notaxv} and}
	$K_0,K_1^\sharp,K_2^\sharp$ {from
	Lemma~\ref{lem:decorated-weighted-convolution}. Set}
	$X:=x-y-\sigma w$, $V:=v-w$, {and}
	\begin{align*}
		\vartheta:=\frac{b_0}{2s}\in(0,1).
	\end{align*}
	{By \eqref{qsmr} and \eqref{defK},}
	\(\mathcal Q_{s,\mathrm{rem}}(g,F)=F\mathcal K(g)\)
	{distributionally. Split the defining integral into}
	$|\alpha|<r_0$ {and} $|\alpha|\ge r_0$.
	{Denote the restricted operators by}
	$\mathcal Q_{s,\mathrm{rem}}^{<}$ {and}
	$\mathcal Q_{s,\mathrm{rem}}^{>}$, {and their Duhamel terms by}
	$\mathscr R_{<}F$ {and} $\mathscr R_{>}F$.
	
	\medskip
	\noindent
	\textit{The range $|\alpha|<r_0$.}
	{We do not estimate this part as a pointwise multiplier. For}
	$\varepsilon>0$, {truncate it further to}
	$\varepsilon<|\alpha|<r_0$. {For fixed} $(r,\zeta)$,
	{set}
	\begin{align*}
		H(u):={}&
		\G^{(0)}(t,x,v;r,\zeta,u)
		F(r,\zeta,u;\tau,y,w).
	\end{align*}
	{The coefficient is even in} \(\alpha\).
	{Applying the discrete integration-by-parts identity
	\eqref{z19} on the symmetric truncation, with the product replaced by}
	\(H\), {gives}
	\begin{align}
		&\int_{\mathbb R^d}
		\int_{\varepsilon<|\alpha|<r_0}
		H(u)\delta_\alpha^u\mathbf C_g(r,\zeta,u,\alpha)
		\frac{\dif\alpha}{|\alpha|^{d+2s}}
		\dif u\notag\\
		&\qquad=
		\frac12\int_{\mathbb R^d}
		\int_{\varepsilon<|\alpha|<r_0}
		\delta_\alpha^uH(u)
		\delta_\alpha^u\mathbf C_g(r,\zeta,u,\alpha)
		\frac{\dif\alpha}{|\alpha|^{d+2s}}
		\dif u.
		\label{eq:qrem-small-discrete-ibp}
	\end{align}
	{Moreover,}
	\begin{align}
		\delta_\alpha^uH(u)
		={}&F(r,\zeta,u)\,
		\delta_\alpha^u\G^{(0)}(t,x,v;r,\zeta,u)+\G^{(0)}(t,x,v;r,\zeta,u-\alpha)
		\delta_\alpha^uF(r,\zeta,u).
		\label{eq:qrem-small-product-difference}
	\end{align}
	{Applying \eqref{z27} at both endpoints yields}
	\begin{align}
		\left|
		\delta_\alpha^u\mathbf C_g(r,\zeta,u,\alpha)
		\right|
		\lesssim
		|\alpha|^{b_0}
		\bigl(
		{\Omega_{\varkappa}}(u,\alpha)+{\Omega_{\varkappa}}(u-\alpha,\alpha)
		\bigr),
		\qquad |\alpha|<r_0.
		\label{eq:qrem-small-coefficient-difference}
	\end{align}
	{Each summand in \eqref{eq:qrem-small-product-difference} contains
	a difference of order} $b_0$ {of either} $\G^{(0)}$
	{or} $F$, {multiplied by a coefficient difference of
	the same order. Thus the integrand on the right-hand side of
	\eqref{eq:qrem-small-discrete-ibp} is} $O(|\alpha|^{2b_0})$
	{at the origin. Because} $b_0>s$, {it is absolutely
	integrable against} $|\alpha|^{-d-2s}\dif\alpha$. {We may
	therefore let} $\varepsilon\downarrow0$ {in
	\eqref{eq:qrem-small-discrete-ibp}.}
	
	{We now combine the normalized source-variable estimate for}
	$\G^{(0)}$, {the velocity estimate for} $F$,
	{and \eqref{eq:qrem-small-coefficient-difference}. On}
	$|u-v|\le1$, {the two relevant moment bounds, up to local
	profile comparison, are}
	\begin{align}
		&\int_{|\alpha|<r_0}
		|\alpha|^{b_0}
		\mathfrak r_{a_i,q}(\alpha)^{b_0}
		\bigl(
		{\Omega_{\varkappa}}(q,\alpha)+{\Omega_{\varkappa}}(q-\alpha,\alpha)
		\bigr)
		\frac{\dif\alpha}{|\alpha|^{d+2s}}
		\lesssim
		\langle q\rangle^{-\kappa\vartheta}
		a_i^{-1+\vartheta},
		\qquad i=1,2.
		\label{eq:qrem-small-moments}
	\end{align}
	{For} $i=1$, {take} $(a_i,q)=(a_1,v)$;
	{for} $i=2$, {take} $(a_i,q)=(a_2,u)$.
	{This is \eqref{eq:global-Omega-capped-moment} with}
	\(a=\varkappa\) {and} \(b=b_0\); {in particular,}
	\(2b_0>2s\) {ensures integrability at the origin.}
	\begingroup
	The second product term in \eqref{eq:qrem-small-product-difference}
	uses the weaker mass norm \eqref{eq:def-fixed-increment-mass}.
	Fix $0<|\alpha|<r_0$. Since $\varkappa>d$,
	\[
	{\Omega_{\varkappa}}(u,\alpha)+{\Omega_{\varkappa}}(u-\alpha,\alpha)
	\lesssim\langle u\rangle,
	\qquad
	\mathfrak r_{a_2,u}(\alpha)^{b_0}
	\lesssim\langle u\rangle^{b_0(1+\kappa/(2s))}
	\Bigl(1\wedge(|\alpha|/a_2^{1/(2s)})^{b_0}\Bigr).
	\]
	The ratio-weight transfer \eqref{eq:weight-transfer-ratio},
	with $1+b_0(1+\kappa/(2s))\leq n_0$, gives
	\begin{align}\label{eq:qrem-fixed-alpha-mass}
	&\iint\varpi_{N-n_0}(a_2,\zeta-y,u,w)^{-1}
	|\delta_\alpha^uF(r,\zeta,u;\tau,y,w)|
	\bigl({\Omega_{\varkappa}}(u,\alpha)+{\Omega_{\varkappa}}(u-\alpha,\alpha)\bigr)
	\,\dif\zeta\dif u\notag\\
	&\qquad\lesssim
	\langle w\rangle^{n_0}\mathfrak B_N(F;a_2,y,w)
	\Bigl(1\wedge(|\alpha|/a_2^{1/(2s)})^{b_0}\Bigr).
	\end{align}
	Here $\alpha$ is fixed during the $(\zeta,u)$ integration.
	Integrating \eqref{eq:qrem-fixed-alpha-mass} against
	$|\alpha|^{b_0-d-2s}\dif\alpha$ then costs at most
	\[
	\int_{|\alpha|<r_0}|\alpha|^{b_0}
	\Bigl(1\wedge(|\alpha|/a_2^{1/(2s)})^{b_0}\Bigr)
	\frac{\dif\alpha}{|\alpha|^{d+2s}}
	\lesssim a_2^{-1+b_0/(2s)}.
	\]
	This proves the peak estimate using the fixed-increment mass norm.
	\par\endgroup
	{The precise two-sided decorated estimate is}
	\begin{align}
		&\int_\tau^t\iint_{\mathbb R^{2d}}
		\int_{\substack{
				[\alpha]_v\le\tilde\sigma_v^{1/(2s)}\\
				|\alpha|<r_0}}
		|\alpha|^{b_0}
		\left(
		\mathfrak r_{a_1,v}(\alpha)^{b_0}
		+\mathfrak r_{a_2,u}(\alpha)^{b_0}
		\right)
		K_1^\sharp(\alpha)K_2^\sharp(\alpha)\notag\\
		&\qquad\times
		\bigl(
		{\Omega_{\varkappa}}(u,\alpha)+{\Omega_{\varkappa}}(u-\alpha,\alpha)
		\bigr)
		\mathbf{1}_{|u-v|\le1}
		\frac{\dif\alpha}{|\alpha|^{d+2s}}
		\dif\zeta\dif u\dif r
		\lesssim
		\sigma^{\vartheta}K_0.
		\label{eq:qrem-two-sided-decorated}
	\end{align}
	{To prove \eqref{eq:qrem-two-sided-decorated}, combine
	\eqref{eq:qrem-small-moments} with the shifted subconvolution bound. The
	two summands contribute} $a_1^{-1+\vartheta}$ {and}
	$a_2^{-1+\vartheta}$, {respectively. This repeats the leading
	argument of Lemma~\ref{lem:decorated-weighted-convolution}, retaining the
	source-side difference in the first summand.}
\begingroup
On $[\alpha]_v>\tilde\sigma_v^{1/(2s)}$, use
$|\delta_\alpha\mathbf C_g|\lesssim r_0^{b_0}
({\Omega_{\varkappa}}(u,\alpha)+{\Omega_{\varkappa}}(u-\alpha,\alpha))$ and bound the kernel
and input differences by their two endpoints.
The unshifted convolution and the one-tail bridge
\eqref{eq:decorated-convolution-tail} give
$Cr_0^{b_0}K_0$, without a relative small-time gain.
\par\endgroup
	{The region} $|u-v|>1$
	{is absorbed by reserving finitely many powers of the rapid
	decay of the outer kernel. Consequently,}
	\begin{align}
		|\mathscr R_{<}F|
		\lesssim{}&
		\langle w\rangle^{M+n_0}
		\varpi_{N-n_0}(\sigma,x-y,v,w)
		{(\sigma^{\vartheta}+r_0^{b_0})}
\bT_v(\sigma,X,V).
		\label{eq:qrem-small-tail-bound}
	\end{align}
	
	{To estimate the regularized peak, split the time interval at
	its midpoint. For the first term in
	\eqref{eq:qrem-small-product-difference}, combine the outer peak with the
	weighted mass of} $F$ {on the first half, and the weighted}
	$L^1$ {bound for the outer kernel with the input peak on the
	second half.}{ For the second term, use
	the integrated fixed-increment mass
	\eqref{eq:qrem-fixed-alpha-mass} and the pointwise H\"older bound
	for} $F$. {The resulting time
	integrals are}
	\begin{align*}
		\int_\tau^t
		a_1^{-1+\vartheta}a_2^{j\vartheta}\dif r
		+\int_\tau^t
		a_2^{-1+(j+1)\vartheta}\dif r
		\lesssim
		\sigma^{(j+1)\vartheta}.
	\end{align*}
	{Therefore,}
	\begin{align}
		|\mathscr R_{<}F|
		\lesssim{}&
		\langle w\rangle^{M+n_0}
		\varpi_{N-n_0}(\sigma,x-y,v,w)
		\sigma^{(j+1)\vartheta}\mathfrak P_v(\sigma).
		\label{eq:qrem-small-peak-bound}
	\end{align}
	{Taking the minimum of}
	\eqref{eq:qrem-small-tail-bound} and
	\eqref{eq:qrem-small-peak-bound} gives the required pointwise bound
	{for} $\mathscr R_{<}F$.
	
\begingroup
For terminal differences use the auxiliary exponent $b>b_0$ from
Proposition~\ref{propsmalljump}, the source estimate of
Proposition~\ref{prop:small-jump-source-holder}, and
Lemma~\ref{lem:mixed-composition}.
The mixed difference of the outer kernel has the terminal factor
$(1\wedge\ell_{a_1,v}^{v}(h))^b$, or
$(1\wedge\ell_{a_1,v}^{x}(h))^{b/(1+2s)}$, and the source factor
$\mathfrak r_{a_1,u}(\alpha)^{b_0}$.
All shifted endpoint profiles are retained.
Together with \eqref{eq:qrem-small-moments}, the time integrals reduce to
\begin{align}
&\int_0^1
\bigl(\rho^{-1+\vartheta}+(1-\rho)^{-1+\vartheta}\bigr)
\left(1\wedge\ell_{\sigma,v}^{v}(h)\rho^{-1/(2s)}\right)^b
\,\dif\rho
\lesssim \ell_{\sigma,v}^{v}(h)^{b_0},\notag\\
&\int_0^1
\bigl(\rho^{-1+\vartheta}+(1-\rho)^{-1+\vartheta}\bigr)
\left(1\wedge\ell_{\sigma,v}^{x}(h)\rho^{-1-1/(2s)}\right)^{b/(1+2s)}
\,\dif\rho
\lesssim \ell_{\sigma,v}^{x}(h)^{\beta_0}.
\label{eq:qrem-terminal-endpoints}
\end{align}
Splitting at the increment scale gives these bounds without a
logarithmic loss, since $b>b_0$.
The midpoint argument in \eqref{eq:qrem-small-peak-bound} retains
the peak gain, while the weighted bridge in
\eqref{eq:qrem-small-tail-bound} controls the tail.
Consequently,
\begin{align}
[\mathscr R_{<}F]^{\mathrm{int}}_{v,b_0}
+[\mathscr R_{<}F]^{\mathrm{int}}_{x,\beta_0}
&\lesssim
\langle w\rangle^{M+n_0}
\varpi_{N-n_0}(\sigma,x-y,v,w)
\mathfrak T_v^{(j+1)}(\sigma,X,V).
\label{eq:qrem-small-holder-bound}
\end{align}

For the mass bounds, invert \eqref{eq:ratio-submultiplicative}
and integrate the outer kernel. Its source difference gives the first
time integral below; a difference on $F$ gives the second by
\eqref{eq:qrem-fixed-alpha-mass}, with $\alpha$ fixed during spatial
integration. For the modulus, keep $h$ fixed and apply the mixed
estimate and \eqref{eq:qrem-terminal-endpoints} before taking the
supremum. Rapid decay of the outer kernel absorbs reference-velocity
and scale comparisons. The transfer to $w$ costs only
$1+b_0(1+\kappa/(2s))\leq n_0$.
Outside the intrinsic range, the denominator in
\eqref{eq:def-fixed-increment-mass} is one and the order-zero bound
controls both terminal endpoints. Thus
\begin{align}
&\iint_{\mathbb R^{2d}}
\varpi_{N-n_0}(\sigma,x-y,v,w)^{-1}
|\mathscr R_{<}F|\,\dif x\,\dif v
+\mathfrak B_{N-n_0}(\mathscr R_{<}F;\sigma,y,w)\notag\\
&\qquad\lesssim
\langle w\rangle^{M+n_0}
\left(\int_\tau^t a_1^{-1+\vartheta}a_2^{j\vartheta}\,\dif r
+\int_\tau^t a_2^{-1+(j+1)\vartheta}\,\dif r\right)
\lesssim \langle w\rangle^{M+n_0}\sigma^{(j+1)\vartheta}.
\label{eq:qrem-small-weighted-mass}
\end{align}
\par\endgroup

	\medskip
	\noindent
	\textit{The range $|\alpha|\ge r_0$.}
	Set
	\begin{align*}
		m_{>}(r,\zeta,u)
		:={}&
		\int_{|\alpha|\ge r_0}
		\delta_\alpha^u\mathbf C_g(r,\zeta,u,\alpha)
		\frac{\dif\alpha}{|\alpha|^{d+2s}}.
	\end{align*}
	{This integral converges absolutely. More precisely,}
	\begin{align}
		|m_{>}(r,\zeta,u)|
		\lesssim{}&
		\int_{|\alpha|\ge r_0}
		\bigl(
		{\Omega_{\varkappa}}(u,\alpha)+{\Omega_{\varkappa}}(u-\alpha,\alpha)
		\bigr)
		\frac{\dif\alpha}{|\alpha|^{d+2s}}
		\lesssim1.
		\label{eq:qrem-large-finite-rate}
	\end{align}
	{The first inequality follows from \eqref{z26} at both
	endpoints, and the second from \eqref{eq:finite-large-jump-rate} with}
	\(a=\varkappa\).
	
	{Thus} $\mathscr R_{>}F=\G^{(0)}\oast(m_{>}F)$
	{is a finite-rate loss term. The weighted subconvolution estimate
	and the midpoint splitting from Step~3 of
	Lemma~\ref{lem:one-step-large-jump} give, respectively,}
	\begin{align*}
		|\mathscr R_{>}F|
		&\lesssim
		\langle w\rangle^{M+n_0}
		\varpi_{N-n_0}(\sigma,x-y,v,w)
		\sigma\bT_v(\sigma,X,V),\\
		|\mathscr R_{>}F|
		&\lesssim
		\langle w\rangle^{M+n_0}
		\varpi_{N-n_0}(\sigma,x-y,v,w)
		\sigma^{1+j\vartheta}\mathfrak P_v(\sigma).
	\end{align*}
	{Since} $\vartheta<1$ {and} $\sigma\le1$,
	{these estimates yield the required bound in}
	$\mathfrak T_v^{(j+1)}$. {For a terminal difference, the
	short- and long-time decomposition gives the natural velocity gain}
	$\min\{1,2s\}$ {and spatial gain} $2s/(1+2s)$.
	{These dominate} $b_0$ {and} $\beta_0$,
	{respectively, so both terminal seminorms satisfy the same bound.}
	
	\begingroup
	Finally, invert \eqref{eq:ratio-submultiplicative}, use
	\eqref{eq:qrem-large-finite-rate}, and repeat the fixed terminal
	increment argument from Step~4 of
	Lemma~\ref{lem:one-step-large-jump}. Since the large-$\alpha$
	term is a finite-rate loss, its input is only the order-zero
	weighted mass in \eqref{eq:qrem-F-weighted-mass}. Uniformly for
	each fixed $0<|h|\le r_0$, this yields
	\begin{align*}
		\iint_{\mathbb R^{2d}}
		\varpi_{N-n_0}(\sigma,x-y,v,w)^{-1}
		|\mathscr R_{>}F|\dif x\dif v
		+\mathfrak B_{N-n_0}(\mathscr R_{>}F;\sigma,y,w)\\
		&\qquad\lesssim
		\langle w\rangle^{M+n_0}
		\sigma^{1+j\vartheta}
		\lesssim
		\langle w\rangle^{M+n_0}
		\sigma^{(j+1)\vartheta}.
	\end{align*}
	\par\endgroup
	{Combining the small- and large-}$\alpha$
	{estimates proves \eqref{eq:qrem-one-step} and
	\eqref{eq:qrem-one-step-mass}.}
\end{proof}

Define
\begin{align}\label{eq:def-Nj-Mj}
	{N^{(j)}}:=N_{\rm tail}-(j+1)n_0,
	\qquad
	M_j:=jn_0.
\end{align}
{The following proposition estimates the parametrix terms in
\eqref{eq:finite-full-expansion}.}
\begin{proposition}
	\label{prop:finite-parametrix-bounds}
		{Assume} ${N^{(j)}}>0$ {for} $0\le j\le J+1$.
		{For each fixed} $J$, {there exists} $T_J>0$
		{such that, whenever} $0<\sigma=t-\tau\le T_J$,
		{the following estimates hold for every}
		$0\leq j\leq J$:
	\begin{align}\label{eq:parametrix-term-pointwise}
		&\bigl(
		|\G^{(j)}|+{[\G^{(j)}]^{\mathrm{int}}_{v,b_0}}
		+{[\G^{(j)}]^{\mathrm{int}}_{x,\beta_0}}
		\bigr)(t,x,v;\tau,y,w)\notag\\
		&\qquad\lesssim
		\langle w\rangle^{M_j}
		\varpi_{{N^{(j)}}}(\sigma,x-y,v,w)
		\mathfrak T_v^{(j)}(\sigma,x-y-\sigma w,v-w).
	\end{align}
		{Moreover,}
	\begingroup
	\begin{align}\label{eq:parametrix-term-mass}
		&\iint_{\mathbb R^{2d}}
		\varpi_{{N^{(j)}}}(\sigma,x-y,v,w)^{-1}
		|\G^{(j)}(t,x,v;\tau,y,w)|\dif x\dif v
		+\mathfrak B_{{N^{(j)}}}(\G^{(j)};\sigma,y,w)
		\lesssim
		\langle w\rangle^{M_j}
		\sigma^{\frac{jb_0}{2s}}.
	\end{align}
	\par\endgroup
		{For each fixed truncation order} $J$,
		{the constants satisfy}
	$C_j\le C_J$ for $0\le j\le J$.
\end{proposition}
\begin{proof}
	{The case} $j=0$ {follows from
	Proposition~\ref{propsmalljump}: its rapid factor absorbs the inverse ratio
	weight in \eqref{eq:parametrix-term-mass}, and}
	$\mathfrak T_v^{(0)}=\bT_v$.
	\begingroup
	For $\mathfrak B_{{N^{(0)}}}(\G^{(0)})$, the intrinsic pointwise
	velocity estimate controls increments with
	$\ell_{\sigma,v}^v(h)\le1$ after integrating its two endpoint
	profiles. On the complement the denominator is one, and the
	order-zero weighted mass controls both endpoints. This is uniform
	in $0<|h|\le r_0$.
	\par\endgroup
	
	{Assume both estimates hold at level} $j$.
	{Since} ${N^{(j+1)}}={N^{(j)}}-n_0>0$,
	{\eqref{eq:weight-transfer-ratio} gives}
	\begin{align*}
		\langle v\rangle^{n_0}\varpi_{{N^{(j+1)}}}^{-1}
		\lesssim
		\langle w\rangle^{n_0}\varpi_{{N^{(j)}}}^{-1}.
	\end{align*}
	{Thus \eqref{eq:parametrix-term-mass} implies
	\eqref{eq:large-jump-input-mass} for} $(M,N)=(M_j,{N^{(j)}})$
	{and directly yields \eqref{eq:qrem-F-weighted-mass}. Moreover,}
	${N^{(j)}}+M_{\rm pc}\le N_{\rm tail}$ {because}
	$n_0\ge M_{\rm pc}$. {The two one-step lemmas therefore apply;
	together with} $M_{j+1}=M_j+n_0$, {their order-zero conclusions give
	the order-zero estimates at level} $j+1$.
	{The intrinsic pointwise and fixed-increment mass bounds in
	Lemmas~\ref{lem:one-step-large-jump} and \ref{lem:one-step-cancellation-remainder}
	give both assertions at level $j+1$.}
\end{proof}

{In the finite parametrix expansion, each perturbation improves
the regularized peak by a factor} $\sigma^{b_0/(2s)}$ {at the cost
of finitely many powers of} $\langle w\rangle$ {and a finite loss
in the ratio-weight exponent. After finitely many iterations, the resulting
positive time power dominates the original peak singularity. To sum the
remaining series, we first show that the large-jump operator preserves the
ratio weight.}
\begin{lemma}
	\label{lem:large-jump-ratio-preservation}
	Let $0\leq P\leq N_{\mathrm{tail}}$.  For $a\in(0,1]$ and
	$z,u,w\in\mathbb R^d$, one has
	\begin{align}
		&\int_{|\alpha|\geq r_0}
		\frac{{\Omega_{\varkappa}}(u,\alpha)}{|\alpha|^{d+2s}}
		\varpi_P(a,z,u-\alpha,w)\dif\alpha
		\lesssim_P
		\varpi_P(a,z,u,w).
		\label{eq:large-jump-ratio-preservation}
	\end{align}
\end{lemma}
\begin{proof}
{By \eqref{eq:ratio-weight-characteristic-form}, the output and
input weights are equivalent, respectively, to}
\(
(1+(|z-aw|/a+|u-w|)/\langle w\rangle)^{-P}
\)
and
\(
(1+(|z-aw|/a+|u-\alpha-w|)/\langle w\rangle)^{-P}.
\)
{The finite large-jump rate follows from
\eqref{eq:finite-large-jump-rate} with} \(a=\varkappa\). {If}
$|z-aw|/a+|u-\alpha-w|\geq(|z-aw|/a+|u-w|)/8$,
{then the input weight is bounded by the output weight. {Estimate
\eqref{eq:large-jump-ratio-preservation} also follows from \eqref{eq:finite-large-jump-rate}}
when} $|z-aw|/a+|u-w|\leq16\langle w\rangle$. {It remains to
consider only the region}
\begin{align}\label{eq:large-jump-ratio-region}
 \frac{|z-aw|}{a}+|u-\alpha-w|
 &<\frac18\left(\frac{|z-aw|}{a}+|u-w|\right),&
 \frac{|z-aw|}{a}+|u-w|&>16\langle w\rangle.
\end{align}
{In this region,}
\[
 |u-\alpha|\leq\frac3{16}\left(\frac{|z-aw|}{a}+|u-w|\right),
 \qquad
 |\alpha|\geq\frac58\left(\frac{|z-aw|}{a}+|u-w|\right).
\]
{On the support of} $\Omega_1$,
$|\alpha|\lesssim\langle\Pi_{\widehat\alpha}u\rangle$,
$\Pi_{\widehat\alpha}u=\Pi_{\widehat\alpha}(u-\alpha)$, and
$|\mathrm P_{\widehat\alpha}u|\geq|\alpha|-|u-\alpha|$.
{Integrating over the region
\eqref{eq:large-jump-ratio-region} gives}
\begin{align*}
 \int\frac{\Omega_1(u,\alpha)}{|\alpha|^{d+2s}}\dif\alpha
 &\lesssim
 \left(\frac{|z-aw|}{a}+|u-w|\right)^{(1-\kappa)_++d-1-\varkappa-2s},\\
 \int\frac{\Omega_2(u,\alpha)}{|\alpha|^{d+2s}}\dif\alpha
 &\lesssim
 \left(\frac{|z-aw|}{a}+|u-w|\right)^{d-\kappa-\varkappa-2s}.
\end{align*}
{{The bound for the $\Omega_2$ integral follows by radial integration.} By the
definition of} $N_{\mathrm{tail}}$, {both terms are bounded by}
$(|z-aw|/a+|u-w|)^{-N_{\mathrm{tail}}}$; {here}
$(1-\kappa)_++\min\{\kappa,1\}=1$ {and}
$\min\{\kappa,1\}\leq\kappa$. {Since}
$P\leq N_{\mathrm{tail}}$, {the last display is bounded by}
\(
C_P(1+(|z-aw|/a+|u-w|)/\langle w\rangle)^{-P}.
\)
{This proves \eqref{eq:large-jump-ratio-preservation}.}
\end{proof}

\begin{lemma}\label{lemGR}
	{Let} $M_*\geq0$, $0\leq P_*\leq N_{\mathrm{tail}}$,
	{and} $T\in(0,1)$. {Assume that} $T$
	{lies within the time ranges of Propositions~\ref{propsmalljump}
	and \ref{prop:small-jump-source-holder} for some rapid exponent}
	{$L>P_*+2n_0$}. {Suppose that, for every}
	$0<t-\tau\leq T$, {with} $\sigma=t-\tau$,
	\begin{align}
		&\bigl(
		|F|+{[F]^{\mathrm{int}}_{v,b_0}}+{[F]^{\mathrm{int}}_{x,\beta_0}}
		\bigr)(t,x,v;\tau,y,w)
		\lesssim
		\langle w\rangle^{M_*}
		\varpi_{P_*}(\sigma,x-y,v,w)
		\sigma^{2+2s}.
		\label{eq:regular-ratio-remainder-input}
	\end{align}
	{For $T$ sufficiently small,}
	\begin{align}
		&\bigl(
		|\G^{(0)}\oast\mathcal B_gF|
		+{[\G^{(0)}\oast\mathcal B_gF]^{\mathrm{int}}_{v,b_0}}
		+{[\G^{(0)}\oast\mathcal B_gF]^{\mathrm{int}}_{x,\beta_0}}
		\bigr)(t,x,v;\tau,y,w)
		\notag\\
		&\qquad\lesssim
		T^{b_0/(2s)}
		\langle w\rangle^{M_*}
		\varpi_{P_*}(\sigma,x-y,v,w)
		\sigma^{2+2s}.
		\label{eq:regular-ratio-remainder-output}
	\end{align}
\end{lemma}
\begin{proof}
	{Use the notation from \eqref{notaxv}.
	Lemma~\ref{lem:large-jump-ratio-preservation}, the finite rate
	\eqref{eq:finite-large-jump-rate}, and the large-}$\alpha$
	{cancellation rate \eqref{eq:qrem-large-finite-rate} give}
	\begin{align*}
		|\mathcal L_g^>F|
		+|\mathcal Q_{s,\mathrm{rem}}^>(g,F)|
		\lesssim
		\langle w\rangle^{M_*}
		\varpi_{P_*}(a_2,\zeta-y,u,w)
		a_2^{2+2s}.
	\end{align*}
	{Fix the exponent} {$L>P_*+2n_0$}
	{from the hypotheses.
	The rapid factor in Proposition~\ref{propsmalljump}, the weighted}
	$L^1$ {bound for the small-jump kernel, and
	\eqref{eq:ratio-submultiplicative} yield}
	\begin{align*}
		&\iint_{\mathbb R^{2d}}
		\bK_{v,L}(a_1,X_1,V_1)
		\varpi_{P_*}(a_2,\zeta-y,u,w)\dif\zeta\dif u\lesssim
		\varpi_{P_*}(\sigma,x-y,v,w).
	\end{align*}
	{The corresponding Duhamel terms are therefore bounded by}
	\begin{align*}
		\int_\tau^t a_2^{2+2s}\dif r
		\lesssim T\sigma^{2+2s}.
	\end{align*}
	{For the small-}$\alpha$ {cancellation term, use the
	weak identity \eqref{eq:qrem-small-discrete-ibp} and the product expansion
	\eqref{eq:qrem-small-product-difference}. Estimate the coefficient
	difference by \eqref{eq:qrem-small-coefficient-difference} and the two
	velocity differences by the normalized H\"older bounds for the outer kernel
	and for} $F$. {The moment estimate
	\eqref{eq:qrem-small-moments} reduces the time integral to}
	\begin{align*}
		&\int_\tau^t
		a_1^{-1+\frac{b_0}{2s}}a_2^{2+2s}\dif r
		+\int_\tau^t
		a_2^{1+2s+\frac{b_0}{2s}}\dif r\lesssim
		\sigma^{2+2s+\frac{b_0}{2s}}
		\le
		T^{\frac{b_0}{2s}}\sigma^{2+2s}.
	\end{align*}
\begingroup
For terminal differences, repeat the mixed-difference argument in
Lemma~\ref{lem:one-step-cancellation-remainder}, using terminal
exponent $b>b_0$ and source exponent $b_0$, with all endpoint profiles.
The rapid exponent $L$ absorbs reference-velocity and ratio-weight
comparisons. Weighted convolution, \eqref{eq:qrem-small-moments}
and \eqref{eq:qrem-terminal-endpoints} give the same two time
integrals in the original exponents $b_0,\beta_0$.
The finite-rate terms use nonsingular endpoint integrals and
$T\leq T^{b_0/(2s)}$. This proves
\eqref{eq:regular-ratio-remainder-output}.
\par\endgroup
\end{proof}

				\noindent
\subsubsection{Proof of Theorem \ref{Greshort}}
	{By \eqref{eq:def-full-parametrix-terms},}
	\(\G^{(j)}=\mathscr K^j\G^{(0)}\), {where} \(\mathscr K\)
	{is defined in \eqref{eq:def-full-duhamel-operator}. Set}
\begin{align*}
	\sigma:=t-\tau,
	\qquad
	(X,V):=(x-y-\sigma w,v-w).
\end{align*}
{Choose an integer} $J$ {such that}
\begin{align}\label{eq:choice-J-regular}
	\frac{(J+1)b_0}{2s}
	\geq
	d+\frac ds+2+2s.
\end{align}
Since $b_0>s$, it suffices to choose $(J+1)/2\geq d+d/s+2+2s$.
Thus $J$ can be fixed independently of the margin $b_0-s$; the constants in the
H\"older estimates may still depend on that margin. Set
\begin{align}\label{eq:correct-M-sharp}
	{M_\sharp
	:=(J+3)n_0+4d+6s-\min\{\kappa,1\}+\Lambda_*}.
\end{align}
{The assumption} $\varkappa>M_\sharp$
{in Theorem~\ref{Greshort} then gives}
\begin{align*}
	{N^{(J+2)}}
	&=\varkappa-2d-2s+\min\{\kappa,1\}-(J+3)n_0
	{>2(d+2s)+\Lambda_*},
\end{align*}
{The mass induction uses
\eqref{eq:def-fixed-increment-mass}. The displayed reserve covers
the tail moment and the single conversion to physical seminorms
after the intrinsic estimates.}
{
Proposition~\ref{prop:finite-parametrix-bounds},
\eqref{eq:global-profile-comparison}, and
\eqref{eq:weight-transfer-ratio} imply that, for} $0\leq j\leq J$,
\begin{align}
	&\bigl(
	|\G^{(j)}|+{[\G^{(j)}]^{\mathrm{int}}_{v,b_0}}
	+{[\G^{(j)}]^{\mathrm{int}}_{x,\beta_0}}
	\bigr)(t,x,v;\tau,y,w)
	\notag\\
	&\qquad\leq
	C_J\langle w\rangle^{M_J}
	\varpi_{{N^{(J)}}} (\sigma,x-y,v,w)
	\bT_{v}(\sigma,X,V)\notag\\
	&\qquad\leq C_J\langle w\rangle^{M_{J+1}}
	\varpi_{{N^{(J+1)}}} (\sigma,x-y,v,w)
	\bT(\sigma,X,V),
	\label{eq:finite-terms-final-envelope}
\end{align}
{where}
\begin{align*}
	\bT(\sigma,X,V)
	:=
	\bT_{v_0}(\sigma,X,V)\big|_{v_0=0}
	=
	\sigma^{-d-d/s}
	\mathcal N\left(
	\sigma^{-1-\frac1{2s}}X,
	\sigma^{-\frac1{2s}}V
	\right).
\end{align*}

{We apply Proposition~\ref{prop:finite-parametrix-bounds} once more,
now with truncation order} $J+1$, {to estimate the first omitted
term. Since}
\begin{align*}
	\mathfrak P_v(\sigma)
	=
	\sigma^{-d-d/s}
	\langle v\rangle^{\kappa d/s+2},
\end{align*}
{the choice of} $n_0$, {the characteristic form
\eqref{eq:ratio-weight-characteristic-form}, and
\eqref{eq:weight-transfer-ratio} imply}
\begin{align}
	&\bigl(
	|\G^{(J+1)}|+{[\G^{(J+1)}]^{\mathrm{int}}_{v,b_0}}
	+{[\G^{(J+1)}]^{\mathrm{int}}_{x,\beta_0}}
	\bigr)(t,x,v;\tau,y,w)
\lesssim
	C_{J+1}\langle w\rangle^{M_{J+2}}
	\varpi_{{N^{(J+2)}}}(\sigma,x-y,v,w)
	\sigma^{2+2s}.
	\label{eq:first-regular-remainder}
\end{align}
{Indeed, \eqref{eq:weight-transfer-ratio} transfers the factor}
$\langle v\rangle^{\kappa d/s+2}$ {to}
$\langle w\rangle$, {and the choice of} $n_0$
{absorbs the resulting loss. The time exponent follows from
\eqref{eq:choice-J-regular}.}

{Consider the remainder series}
\begin{align}
	\mathfrak R^{(J+1)}
	:=
	\sum_{\ell=0}^{\infty}
	\mathscr K^\ell\G^{(J+1)}.
	\label{eq:def-full-remainder-series}
\end{align}
{Fix} {\(L>{N^{(J+2)}}+2n_0\)}.
{Choose $c_0,r_0$ once for the finite
collection of rapid exponents needed in the small-jump and one-step
estimates, including each source exponent increased by
$M_{\mathrm{pc}}$.} {Let}
\(T_{\rm par}\in(0,1]\) {be the minimum of the thresholds furnished by
Propositions~\ref{propsmalljump} and
\ref{prop:small-jump-source-holder} (with exponent} \(L\){),
Proposition~\ref{prop:finite-parametrix-bounds} (at order}
\(J+1\){), and the associated one-step lemmas,
{including Lemma~\ref{lemGR}.} Let}
\(C_{\rm GR}\) {be the constant in Lemma~\ref{lemGR} for}
\(P_*={N^{(J+2)}}\) {and} \(M_*=M_{J+2}\), {and set}
\[
	T_*:=\min\left\{T,T_{\rm par},(2C_{\rm GR})^{-2s/b_0}\right\}.
\]
{For} \(0<\sigma\le T_*\), {let}
\(\mathfrak X_{T_*}\) {be the Banach space with norm}
\[
 \|F\|_{\mathfrak X_{T_*}}
 :=\sup\frac{|F|+{[F]^{\mathrm{int}}_{v,b_0}}+{[F]^{\mathrm{int}}_{x,\beta_0}}}
 {\langle w\rangle^{M_{J+2}}\varpi_{{N^{(J+2)}}}
 (\sigma,x-y,v,w)\sigma^{2+2s}} .
\]
{Then}
\(\|\mathscr K\|_{\mathfrak X_{T_*}\to\mathfrak X_{T_*}}\le1/2\).
{The same} \(T_*\) {works for every base time because}
\(g_\tau(r,x,v):=g(\tau+r,x,v)\) {satisfies the same uniform
bounds as} \(g\).

{Set} \(S_J:=\sum_{j=0}^{J}\G^{(j)}\) {and}
\(\mathfrak A_{T_*}:=S_J+\mathfrak X_{T_*}\).
By \eqref{eq:first-regular-remainder}, the series
$\mathfrak R^{(J+1)}$ in \eqref{eq:def-full-remainder-series} converges in
$\mathfrak X_{T_*}$ and uniquely solves
\[
\mathfrak R^{(J+1)}=\G^{(J+1)}+\mathscr K\mathfrak R^{(J+1)}.
\]
Hence $\G=S_J+\mathfrak R^{(J+1)}$ is the unique kernel in
$\mathfrak A_{T_*}$ satisfying $\G=\G^{(0)}+\mathscr K\G$.
{It solves
\eqref{eq:full-green-via-small} distributionally.}

{For} \(1\le j\le J\),
{\eqref{eq:parametrix-term-mass} gives}
\(\|\G^{(j)}\|_{L^1_{x,v}}\lesssim
\langle w\rangle^{M_j}\sigma^{jb_0/(2s)}\to0\).
{Moreover,} \({N^{(J+2)}}>2d\) {and}
\(\iint\varpi_{{N^{(J+2)}}}\,\dif x\dif v
\lesssim\sigma^d\langle w\rangle^{2d}\), {so every}
\(R\in\mathfrak X_{T_*}\) {satisfies}
\[
 \|R\|_{L^1_{x,v}}\lesssim
 \|R\|_{\mathfrak X_{T_*}}\langle w\rangle^{M_{J+2}+2d}
 \sigma^{d+2+2s}\to0.
\]
{Thus} \(\G\) {has the same delta trace as}
\(\G^{(0)}\). {If}
\(H=S_J+R\in\mathfrak A_{T_*}\) {is another mild kernel, then}
\(D:=R-\mathfrak R^{(J+1)}\) {satisfies} \(D=\mathscr KD\),
{and hence} \(D=0\). {Finally,}
\(\mathfrak R^{(J+1)}=\G-S_J=\mathscr K^{J+1}\G\),
{and}
\begin{align*}
	&\bigl(
	|\mathfrak R^{(J+1)}|
	+{[\mathfrak R^{(J+1)}]^{\mathrm{int}}_{v,b_0}}
	+{[\mathfrak R^{(J+1)}]^{\mathrm{int}}_{x,\beta_0}}
	\bigr)(t,x,v;\tau,y,w)\lesssim
	\langle w\rangle^{M_{J+2}}
	\varpi_{{N^{(J+2)}}}(\sigma,x-y,v,w)
	\sigma^{2+2s}.
\end{align*}
{By \eqref{eq:ratio-weight-characteristic-form} and}
${N^{(J+2)}}>2(d+2s)$,
\begin{align*}
	\varpi_{{N^{(J+2)}}}(\sigma,x-y,v,w)
	&\lesssim
	\langle w\rangle^{2(d+2s)}
	\varpi_{{N^{(J+2)}}-2(d+2s)}(\sigma,x-y,v,w)
	\left\langle R_\sigma(X,V)\right\rangle^{-2(d+2s)}.
\end{align*}
{Splitting into the regions} $R_\sigma(X,V)\leq1$
{and} $R_\sigma(X,V)>1$, {and using the two polynomial
factors in the definition of} $\mathcal N$, {we obtain, for}
$0<\sigma\leq1$,
\begin{align*}
	\sigma^{2+2s}
	\left\langle R_\sigma(X,V)\right\rangle^{-2(d+2s)}
	\lesssim
	\bT(\sigma,X,V).
\end{align*}
\begingroup
The envelopes $\varpi_N\mathfrak T_v^{(j)}$ and
$\varpi_N\sigma^q$ satisfy the shift hypothesis of
Lemma~\ref{lem:A34-physical-transfer}: $\varpi_N$ changes by a bounded
factor, and \eqref{eq:A10-profile-moderation} controls the tail profile.
The peak term in $\mathfrak T_v^{(j)}$ is independent of $(X,V)$.
Apply that lemma to the assembled intrinsic estimates.
The factor $\langle v\rangle^{\Lambda_*}$ is transferred by
\eqref{eq:weight-transfer-ratio} to
$\langle w\rangle^{\Lambda_*}$ and consumes $\Lambda_*$ powers of
the ratio weight.  Since ${N^{(J+2)}}-\Lambda_*>2(d+2s)$, the preceding
tail-moment argument then gives the physical estimate
\eqref{ptwGreen}, with
\[
 N_\sharp:={N^{(J+2)}}-\Lambda_*-2(d+2s)
 =\varkappa-M_\sharp,
 \qquad
 M_{J+2}+\Lambda_*+2(d+2s)\le M_\sharp.
\]
\par\endgroup

                \vspace{0.5cm}

                \noindent\textbf{Acknowledgements}~~ {Q.-H. Nguyen
was supported by the CAS Project for Young Scientists in Basic Research (Grant
No.~YSBR-031) and the NSFC (Grant Nos.~1251101538 and 12595282). T. Yang was
supported by the General Research Fund of Hong Kong (Project No.~15303924), a
start-up grant from The Hong Kong Polytechnic University (Project
No.~P0043962), and the Kuok Group Foundation. K. Chen and T. Yang also
acknowledge support from the Research Center for Nonlinear Analysis at The Hong
Kong Polytechnic University (Project No.~P0046121).}


\begin{thebibliography}{99}
					\bibitem{ADVW}
					Alexandre, R., Desvillettes, L., Villani, C., Wennberg, B.,  {\em Entropy dissipation and long-range interactions.} Arch. Ration. Mech. Anal. 152(4):327--355, (2000). DOI: 10.1007/s002050000083.

\bibitem{AlexandreVillani2002} Alexandre, R., and Villani, C., {\em On the Boltzmann equation for long-range interactions.} Comm. Pure Appl. Math. 55(1), 30--70 (2002). DOI: 10.1002/cpa.10012.

\bibitem{AMUXY}
					Alexandre, R., Morimoto, Y., Ukai, S., Xu, C.-J., Yang, T., {\em  Regularizing Effect and Local Existence for the Non-Cutoff Boltzmann Equation,}  Arch. Rational Mech. Anal., 198,  39--123, (2010). DOI: 10.1007/s00205-010-0290-1.

\bibitem{AMUXY-1}
					Alexandre, R., Morimoto, Y., Ukai, S., Xu, C-J, Yang, T.,  {\em The Boltzmann Equation Without Angular Cutoff in the Whole Space: Qualitative Properties of Solutions,} Arch. Rational Mech. Anal., 202(2), 599--661, (2011).

\bibitem{AMUXY2011-AA}
					Alexandre, R. and Morimoto, Y. and Ukai, S. and Xu, C. and Yang, T., {\em The Boltzmann Equation Without Angular Cutoff in the Whole Space: II, global existence for hard potential,} Anal. Appl., 9(2), 113--134 , (2011).

\bibitem{AMUXYKJM}	Alexandre, R. and Morimoto, Y. and Ukai, S. and Xu, C. and Yang, T., {\em Smoothing effect of weak solutions for the spatially homogeneous Boltzmann equation without angular cutoff.} Kyoto J. Math. 52 (3) 433--463, 2012.

\bibitem{AMUXY2012JFA}
					Alexandre, R. and Morimoto, Y. and Ukai, S. and Xu, C.-J. and Yang, T., {\em The Boltzmann equation without angular cutoff in the whole space I: Global existence for soft potential}, Journal of Functional Analysis, 262(3), 915--1010, (2012).

\bibitem{AMUXY2011-CMP}
					Alexandre, R. and Morimoto, Y. and Ukai, S. and Xu, C. and Yang, T., {\em Global existence and full regularity of the Boltzmann equation without angular cutoff}, Comm. Math. Phys., 304(2), 513--581,  (2011).

\bibitem{AMSY} Alonso, R., Morimoto, Y., Sun, W., and Yang, T., {\em Non-cutoff Boltzmann equation with polynomial decay perturbations.} Rev. Mat. Iberoam. 37(1), 189--292 (2021). DOI: 10.4171/RMI/1206.

\bibitem{AuscherImbertNiebel} Auscher, P., Imbert, C., and Niebel, L., {\em Fundamental solutions to Kolmogorov--Fokker--Planck equations with rough coefficients: existence, uniqueness, upper estimates.} SIAM J. Math. Anal. 57(2), 2114--2137 (2025). DOI: 10.1137/24M1649241.

\bibitem{BCD}
					Bedrossian, J., Coti Zelati, M., and Dolce, M., {\em Taylor dispersion and phase mixing in the non‐cutoff Boltzmann equation on the whole space.} Proceedings of the London Mathematical Society, 129(1), e12616.  (2024).

\bibitem{Caf}Caflisch, R. E., {\em The fluid dynamic limit of the nonlinear Boltzmann equation.} Comm. Pure Appl. Math. 33(5), 651--666 (1980). DOI: 10.1002/cpa.3160330506.

\bibitem{CDL24} Cao, C., Deng, D., and Li, X., {\em The Vlasov--Poisson--Boltzmann/Landau system with polynomial perturbation near Maxwellian.} SIAM J. Math. Anal. 56(1), 820--876 (2024). DOI: 10.1137/23M1567060.

\bibitem{CLXX}Cao, H., Li, H.-G., Xu, C.-J., and Xu, J., {\em Well-posedness of Cauchy problem for Landau equation in critical Besov space.} Kinetic and Related Models,  12(4): 829-884, (2019).

\bibitem{CM17}Carrapatoso, K., and Mischler, S., {\em Landau equation for very soft and Coulomb potentials near Maxwellians.} Annals of PDE, 3, 1-65, (2017).

\bibitem{CTW16} Carrapatoso, K., Tristani, I.,  and  Wu, K.-C. {\em Cauchy problem and exponential stability for the inhomogeneous Landau equation.} Archive for Rational Mechanics and Analysis, 221(1):363–418, (2016).

\bibitem{CG24}Carrapatoso, K., and Gervais, P., {\em Noncutoff Boltzmann equation with soft potentials in the whole space.} Pure and Applied Analysis 6(1), 253–303, 2024.

\bibitem{CHN}Chen, K., Hu, R., and Nguyen, Q.-H., {\em Schauder-type Estimates and Well-posedness for Non-local Quasilinear Evolution Equations in Fluid Dynamics.} arXiv:2604.10682.

\bibitem{CHNfp}Chen, K., Hu, R., and Nguyen, Q.-H., {\em Well-posedness of the fractional Fokker--Planck equation.} Proc. Amer. Math. Soc. 154, 1991--2004 (2026). DOI: 10.1090/proc/17411.

\bibitem{CKW} Chen Z. Q.,  Kumagai T. and Wang J., {\em Stability of heat kernel estimates for symmetric non-local Dirichlet forms.}
					Mem. Amer. Math. Soc. 271 (2021), no. 1330.

\bibitem{Ch-Zh4} Chen Z.-Q.  and  Zhang X., {\em $L^p$-maximal hypoelliptic regularity of nonlocal kinetic Fokker-Planck operators}, {\it J. Math. Pures Appl.} \textbf{116} (2018), 52--87.

\bibitem{DL}
					Degond, P., Lemou, M., {\em Dispersion relations for the linearized Fokker–Planck equation.} Arch. Ration. Mech. Anal. 138, 137–167, (1997).

\bibitem{DiPernaLions1989} DiPerna, R. J., and Lions, P.-L., {\em On the Cauchy problem for Boltzmann equations: Global existence and weak stability.} Ann. of Math. (2) 130(2), 321--366 (1989). DOI: 10.2307/1971423.

\bibitem{DV1}Desvillettes, L. and Villani, C., {\em On the spatially homogeneous Landau equation for hard potentials. I. Existence, uniqueness and smoothness.} Comm. Partial Differential Equations, 25(1-2):179–259, (2000).

\bibitem{Duan11}Duan, R., Strain, R.M. {\em Optimal Time Decay of the Vlasov–Poisson–Boltzmann System in $\mathbb{R}^d$.}
					Arch. Ration. Mech. Anal. 199, 291–328 (2011). https://doi.org/10.1007/s00205-010-0318-6

\bibitem{DuanHyp2011} Duan, R., {\em Hypocoercivity of linear degenerately dissipative kinetic equations.} {Nonlinearity, 24(8), 2165--2189 (2011).}

\bibitem{DLX} 
					Duan, R., Liu, S., Xu, J., {\em Global well-posedness in spatially critical Besov space for the Boltzmann equation,}
					Arch. Ration. Mech. Anal., 220(2): 711-745, 2016.

\bibitem{DL26} Duan, R., and Li, Z., {\em Polynomial tail solutions of the non-cutoff Boltzmann equation near local Maxwellians.} Pure Appl. Anal. 8, 385--457 (2026). DOI: 10.2140/paa.2026.8.385.

\bibitem{DLL26} Duan, R., Li, Z., and Liu, S., {\em Global mild solution with polynomial tail for the Boltzmann equation in the whole space.} Math. Models Methods Appl. Sci. 36(10), 2113--2162 (2026). DOI: 10.1142/S0218202526500338.

\bibitem{DLSS21} Duan, R., Liu, S., Sakamoto, S., and Strain, R. M., {\em Global mild solutions of the Landau and non-cutoff Boltzmann equations.} Comm. Pure Appl. Math. 74(5), 932--1020 (2021). DOI: 10.1002/cpa.21920.

\bibitem{Friedman1964}
Friedman, A.,
{\em Partial Differential Equations of Parabolic Type.}
Prentice-Hall, Englewood Cliffs, NJ, 1964.

\bibitem{GGL25} Golding, W., Gualdani, M., and Loher, A., {\em Global smooth solutions to the Landau--Coulomb equation in $L^{3/2}$.} Arch. Ration. Mech. Anal. 249, article 34 (2025). DOI: 10.1007/s00205-025-02107-x.

\bibitem{GS2011}
					Gressman, P. T.  and Strain, R. M., {\em Global Classical Solutions of The Boltzmann Equation Without Angular Cut-Off, } J. Amer. Math. Soc., 24(3), 771--847, (2011).

\bibitem{Grube}F. Grube, {\em Pointwise estimates of the fundamental solution to the fractional Kolmogorov equation,} arXiv:2411.00687.

\bibitem{GuillenSilvestre2025} Guillen, N., and Silvestre, L., {\em The Landau equation does not blow up.} Acta Math. 234(2), 315--375 (2025). DOI: 10.4310/ACTA.2025.v234.n2.a2.

\bibitem{Guo}  Guo, Y.,  {\em The Landau equation in a periodic box}, Comm. Math. Phys. 231,  391--434, (2002).

\bibitem{Guo-0}
					Guo, Y., {\em The Boltzmann equation in the whole space}, Indiana Univ. Math. J., 53(4): 1081-1094, (2004).

\bibitem{Guo06}
					Guo, Y., {\em Boltzmann diffusive limit beyond the Navier-Stokes approximation}, Comm. Pure Appl. Math. 59, 626–687, (2006).

\bibitem{HaoWuZhang} Hao, Z., Wu, M., and Zhang, X., {\em Schauder estimates for nonlocal kinetic equations and applications.} J. Math. Pures Appl. 140, 139--184 (2020). DOI: 10.1016/j.matpur.2020.06.003.

\bibitem{He2018}{He, L.-B., {\em Sharp bounds for Boltzmann and Landau collision operators.} Ann. Sci. \'Ec. Norm. Sup\'er. (4), 51(5), 1253--1341, (2018).}

\bibitem{HeJi} He, L.-B., and Ji, J., {\em Regularity estimates for the non-cutoff soft potential Boltzmann equation with typical rough and slowly decaying data.} arXiv:2305.05856.

\bibitem{HsiaoYu}Hsiao, L., and Yu, H. {\em On the Cauchy problem of the Boltzmann and Landau equations with soft potentials. }Quarterly of applied mathematics, 65(2), 281--315, (2007).

\bibitem{HSTLandau}
Henderson, C., Snelson, S., and Tarfulea, A.,
{\em Local solutions of the Landau equation with rough, slowly decaying
initial data.} Ann. Inst. H. Poincar\'e C Anal. Non Lin\'eaire
37(6), 1345--1377 (2020).

\bibitem{HSTBoltzmann}
Henderson, C., Snelson, S., and Tarfulea, A.,
{\em Classical solutions of the Boltzmann equation with irregular
initial data.} Ann. Sci. \'Ec. Norm. Sup\'er. (4)
58(1), 107--201 (2025).

\bibitem{HW22} Henderson, C., and Wang, W., {\em Local well-posedness for the Boltzmann equation with very soft potential and polynomially decaying initial data.} SIAM J. Math. Anal. 54(3), 2845--2875 (2022). DOI: 10.1137/21M1427504.

\bibitem{HouZhang} Hou, H., and Zhang, X., {\em Heat kernel estimates for nonlocal kinetic operators.} Probab. Theory Relat. Fields (2026). DOI: 10.1007/s00440-026-01523-8.

\bibitem{ImbertSilvestreVillani2026} Imbert, C., Silvestre, L., and Villani, C., {\em On the monotonicity of the Fisher information for the Boltzmann equation.} Invent. Math. 243, 127--179 (2026). DOI: 10.1007/s00222-025-01376-3.

\bibitem{ImbertSilvestre-JEMS}
					Imbert, C., Silvestre, L., {\em The weak Harnack inequality for the Boltzmann equation without cut-off,}
					J. Eur. Math. Soc. (JEMS), 22(2), 507--592, (2020).

\bibitem{ImbertSilvestre-APDE}
					Imbert, C., Silvestre, L., {\em The Schauder estimate for kinetic integral equations,}
					Anal. \& PDE, 14(1), 171--204, (2021).

\bibitem{IS0}
					Imbert, C., Silvestre, L., {\em Global regularity estimates for the Boltzmann equation without cut-off,} J. Amer. Math. Soc. 35, 625--703, (2022).

\bibitem{Kawa}Kawashima, S. {\em The Boltzmann equation and thirteen moments.} North-Holl Math Stud,  148: 157--172. (1987).

\bibitem{haina} Li, H., and Xu, Y.,
					{\em Pointwise upper bound for the fundamental solution of fractional Fokker-Planck equation.} arXiv:2501.13771.

\bibitem{LiuYu04} Liu, T.-P., and Yu, S.-H., {\em The Green's function and large-time behavior of solutions for the one-dimensional Boltzmann equation.} Comm. Pure Appl. Math. 57(12), 1543--1608 (2004). DOI: 10.1002/cpa.20011.

\bibitem{LiuYu06} Liu, T.-P., and Yu, S.-H., {\em Green's function of Boltzmann equation, 3-D waves.} Bull. Inst. Math. Acad. Sin. (N.S.) 1(1), 1--78 (2006).

\bibitem{LZ}Liu, L., and Zhang, L. {\em Decay of the Boltzmann equation in spatial critical Besov space.} Journal of Differential Equations, 286, 751-784. (2021).

\bibitem{MoSa}
					Morimoto, Y., Sakamoto, S., {\em Global solutions in the critical Besov space for the non-cutoff Boltzmann equation, } J. Differential Equations, 261(7), 4073-4134, (2016).

\bibitem{Mou06}Mouhot, C., {\em Explicit coercivity estimates for the linearized Boltzmann and Landau operators.} Commun. Part. Differ. Equ. 31, 1321–1348 (2006).

\bibitem{MS07}Mouhot, C., Strain, R., {\em Spectral gap and coercivity estimates for the linearized Boltzmann collision operator without angular cutoff.} J. Math. Pures Appl. 87, 515–535 (2007).

\bibitem{Hung}  Nguyen Q.--H., {\em Potential estimates and quasilinear parabolic equations with measure data,}  Memoirs of the AMS, 291, 1449, (2023).

\bibitem{Sato} Sato, K.-I., {\it L\'{e}vy Processes and Infinitely Divisible Distributions}, volume 68, Cambridge University Press, 1999.

\bibitem{S}
					Silvestre, L.,  {\em A new regularization mechanism for the Boltzmann equation without cut-off,} Commun. Math. Phys., 348(1), 69--100, (2016).

\bibitem{Luis2017}Silvestre, L.  {\em Upper bounds for parabolic equations and the Landau equation. }Journal of Differential Equations, 262(3), 3034-3055, (2017).

\bibitem{SnelsonSolomon} Snelson, S., and Solomon, C., {\em A continuation criterion for the Landau equation with very soft and Coulomb potentials.} Arch. Ration. Mech. Anal. 250, article 20 (2026). DOI: 10.1007/s00205-026-02182-8.

\bibitem{Villani02}
			Villani, C.,
			{\em A review of mathematical topics in collisional kinetic
				theory},
			in {\em Handbook of Mathematical Fluid Dynamics}, Vol.~I,
			pp.~71--305, North-Holland, Amsterdam, 2002.

\bibitem{YZ}Yang, T., and Zhou, Y., {\em On the Boltzmann equation with strong kinetic singularity and its grazing limit from a new perspective.} Math. Ann. 391(4): 4911--4995, (2025).
\end{thebibliography}
			\end{document}